\documentclass[reqno,english,11pt]{amsart}

\usepackage{amsmath,amsfonts,amssymb,graphicx,amsthm,url}
\usepackage[noadjust]{cite}
\usepackage{stmaryrd}
\usepackage{mathrsfs,booktabs,tabularx}
\usepackage{xifthen,xcolor,tikz,setspace}
\usetikzlibrary{decorations.pathmorphing,patterns,shapes,calc,decorations}
\usetikzlibrary{decorations.pathreplacing}
\usepackage{mathtools,upgreek,bbm}
\usepackage[final]{showkeys}

\definecolor{refkey}{gray}{.75}
\definecolor{labelkey}{gray}{.5}
\usepackage[algo2e,boxed,vlined,algoruled]{algorithm2e}
\usepackage{comment}
\usepackage[shortlabels]{enumitem}

\usepackage[letterpaper]{geometry}
\usepackage[colorinlistoftodos]{todonotes}
\presetkeys{todonotes}{inline, color=green}{}
\usepackage{accents}
\usepackage[algo2e,boxed,vlined,algoruled]{algorithm2e}

\usepackage[small]{caption}
\usepackage[colorlinks=true]{hyperref}
\colorlet{DarkGreen}{green!50!black}
\colorlet{DarkGray}{gray!60!black}

\numberwithin{equation}{section}

\renewcommand{\restriction}{\mathord{\upharpoonright}}
\renewcommand{\epsilon}{\varepsilon}

\newcommand{\one}{\mathbbm{1}}
\newcommand\norm[1]{\lVert#1\rVert}

\usepackage{color}
\definecolor{refkey}{gray}{.5}
\definecolor{labelkey}{gray}{.5}
\definecolor{light}{gray}{.9}

\usepackage{soul}

\usepackage[capitalise]{cleveref}

\newtheorem{theoremBG}{Theorem}

\newtheorem{theorem}{Theorem}[section]
\newtheorem*{theorem*}{Theorem}
\newtheorem{lemma}[theorem]{Lemma}
\newtheorem{claim}[theorem]{Claim}
\newtheorem{proposition}[theorem]{Proposition}
\newtheorem{observation}[theorem]{Observation}

\newtheorem{corollary}[theorem]{Corollary}

\theoremstyle{definition}{
	
	\newtheorem{definition}[theorem]{Definition}
	\newtheorem*{definition*}{Definition}
	
	\newtheorem{remark}[theorem]{Remark}
	\newtheorem*{remark*}{Remark}
}

\AddToHook{env/lemma/begin}{\crefalias{theorem}{lemma}}
\AddToHook{env/claim/begin}{\crefalias{theorem}{claim}}
\AddToHook{env/proposition/begin}{\crefalias{theorem}{proposition}}
\AddToHook{env/corollary/begin}{\crefalias{theorem}{corollary}}
\AddToHook{env/definition/begin}{\crefalias{theorem}{definition}}
\AddToHook{env/example/begin}{\crefalias{theorem}{example}}
\AddToHook{env/question/begin}{\crefalias{theorem}{question}}
\AddToHook{env/remark/begin}{\crefalias{theorem}{remark}}
\AddToHook{env/observation/begin}{\crefalias{theorem}{observation}}
\AddToHook{env/fact/begin}{\crefalias{theorem}{fact}}
\crefname{step}{Step}{Steps}
\crefname{step}{Step}{Steps}
\crefname{case}{Case}{Cases}
\crefname{claim}{Claim}{Claims}

\newcommand{\E}{\mathbb E}

\renewcommand{\P}{\mathbb P}

\newcommand{\R}{\mathbb R}
\newcommand{\Z}{\mathbb Z}

\newcommand{\sfh}{\mathsf h}
\newcommand{\sfv}{\mathsf v}
\newcommand{\sfw}{\mathsf w}

\newcommand{\sfy}{\mathsf y}

\newcommand{\cA}{\ensuremath{\mathcal A}}
\newcommand{\cB}{\ensuremath{\mathcal B}}
\newcommand{\cC}{\ensuremath{\mathcal C}}
\newcommand{\cD}{\ensuremath{\mathcal D}}
\newcommand{\cE}{\ensuremath{\mathcal E}}
\newcommand{\cF}{\ensuremath{\mathcal F}}
\newcommand{\cG}{\ensuremath{\mathcal G}}
\newcommand{\cH}{\ensuremath{\mathcal H}}

\newcommand{\cK}{\ensuremath{\mathcal K}}

\newcommand{\cP}{\ensuremath{\mathcal P}}

\newcommand{\cS}{\ensuremath{\mathcal S}}

\newcommand{\cW}{\ensuremath{\mathcal W}}

\newcommand{\llb }{\llbracket}
\newcommand{\rrb }{\rrbracket}

\newcommand{\fp}{\mathfrak{p}}
\newcommand{\fq}{\mathfrak{q}}
\newcommand{\fu}{\mathfrak{u}}

\newcommand{\fI}{\mathfrak{I}}
\newcommand{\fL}{\mathfrak{L}}

\newcommand{\fR}{\mathfrak{R}}

\newcommand{\fS}{\mathfrak{S}}

\newcommand{\sfL}{{\ensuremath{\mathsf L}}}

\newcommand{\sfW}{{\ensuremath{\mathsf W}}}
\newcommand{\sfu}{{\ensuremath{\mathsf u}}}

\newcommand{\sA}{{\ensuremath{\mathscr A}}}
\newcommand{\sB}{{\ensuremath{\mathscr B}}}
\newcommand{\sC}{{\ensuremath{\mathscr C}}}

\newcommand{\sE}{{\ensuremath{\mathscr E}}}
\newcommand{\sF}{{\ensuremath{\mathscr F}}}

\newcommand{\sL}{{\ensuremath{\mathscr L}}}

\newcommand{\sN}{{\ensuremath{\mathscr N}}}

\newcommand{\sS}{{\ensuremath{\mathscr S}}}

\newcommand{\fD}{\mathfrak{D}}

\newcommand{\bE}{{\ensuremath{\mathbf E}}}

\newcommand{\bd}{{\ensuremath{\mathbf d}}}

\newcommand{\bY}{{\ensuremath{\mathbf Y}}}

\newcommand{\bP}{{\ensuremath{\mathbf P}}}

\renewcommand{\epsilon}{\varepsilon}

\newcommand{\GFF}{\textsc{gff}\xspace}
\newcommand{\ZGFF}{\ensuremath{\Z\textsc{gff}}\xspace}
\newcommand{\SOS}{\textsc{sos}\xspace}
\newcommand{\Dim}{\textsc{d} }
\newcommand{\hatpi}{{\widehat\pi}}
\newcommand{\hatmu}{{\widehat\mu}}
\newcommand{\hatnu}{{\widehat\nu}}

\newcommand{\tildeq}{{\widetilde q}}
\newcommand{\hatZ}{{\widehat Z}}
\newcommand{\tildeZ}{{\widetilde Z}}
\newcommand{\partialvtx}{\partial_{\mathtt v}}
\newcommand{\Lh}[1][h]{L_{c}^{(#1)}}

\DeclareMathOperator{\var}{Var}
\DeclareMathOperator{\dist}{dist}

\DeclareMathOperator{\Int}{Int}

\renewcommand{\d}{\mathrm{d}}

\newcommand{\n}{{\vec{\mathsf{n}}}}
\renewcommand{\o}{{\mathsf{o}^*}}

\newcommand{\tv}{{\textsc{tv}}}

\makeatletter
\newcommand{\superimpose}[2]{%
	{\ooalign{$#1\@firstoftwo#2$\cr\hfil$#1\@secondoftwo#2$\hfil\cr}}}

\newcommand{\sbullet}{%
	\hbox{\fontfamily{lmr}\fontsize{.4\dimexpr(\f@size pt)}{0}\selectfont\textbullet}}

\makeatother

\title[Limit shape and emergence of Discrete Gaussian level lines]{The limit shape and emergence of the \\ Discrete Gaussian level lines}

\author{Joseph Chen}
\address{J.\ Chen\hfill\break
LPSM\\ Sorbonne Universit\'e\\ 
4 Place Jussieu 75252\\ Paris, France.}
\email{jchen@lpsm.paris}

\author{Eyal Lubetzky}
\address{E.\ Lubetzky\hfill\break
	Courant Institute\\ New York University\\
	251 Mercer Street\\ New York, NY 10012, USA.}
\email{eyal@courant.nyu.edu}

\begin{document}
	
	\begin{abstract}
		Consider the $(2+1)$\Dim Discrete Gaussian model (\ZGFF, integer-valued Gaussian free~field) on an $L\times L$ box with a hard floor at height zero and zero boundary conditions, at low temperature. 
        The second author, Martinelli and Sly (2016) showed that the surface has a plateau, filling nearly the full square, at height either $H$ or $H+1$ for an explicit function $H(L)$. In a companion paper, we studied the local laws of the top level lines near the four sides of the box, and showed that after rescaling each by $(L^{2/3-o(1)},L^{1/3-o(1)})$, they converge to a product of Ferrari--Spohn~diffusions. Two key features of the top level lines remained unaddressed: their global limit shape, and the critical window marking the transition from a top plateau at height $H$ to one at height $H+1$. 
         These features are intrinsically linked: deriving the global limit of the top level line is needed for determining whether it is preferable to be at height $H$ or $H+1$ near criticality.
     
       This work completes this picture as follows. First, we obtain the global limit of the top level lines: for every fixed $n$, the $n$-th from-the-top level line converges in Hausdorff distance to a deterministic shape $\sL_n$ that features the Wulff shape at scale $N_n=L^{1-o(1)}$ near the four corners of the box. Second, we identify, for every $h$, the point of emergence of a macroscopic $h$ level line: the probability of this event is monotone increasing in $L$ (up to a $o(1)$ error), and undergoes a sharp transition from near $0$ to near $1$ in a critical window of width $\leq L^{1/2+o(1)}$ around a side length $L=L_c^{(h)}$. 
       This transition is discontinuous in that, once a macroscopic level $h$ emerges, it immediately occupies nearly all the box, and the above global and local scaling limits (Wulff, Ferrari--Spohn) hold for it. 
       The new results extend to the $(2+1)$\Dim $|\nabla\phi|^p$-models (\ZGFF is the case $p=2$) for every fixed $p> 1$.
	\end{abstract}
	
	\maketitle
    \vspace{-0.1in}	
	\section{Introduction}

The $(2+1)$-dimensional Discrete Gaussian model (\ZGFF), also known as the integer-valued Gaussian free field, is a random surface model extensively studied in the context of the roughening transition in crystals (see the work of Chui and Weeks~\cite{ChuiWeeks76} in 1976, and the related models in~\cite{BCF51} dating back to the 1950's). It is dual to the Villain XY model~\cite{Villain75}, and as such undergoes a Berezinskii--Kosterlitz--Thouless (\textsc{bkt}) phase transition~\cite{Berezinskii71,KosterlitzThouless73}---as established, along with this result for the Solid-On-Solid (\SOS) model, by Fr\"ohlich and Spencer in their celebrated papers~\cite{FrohlichSpencer81a,FrohlichSpencer81b}.

At inverse temperature $\beta>0$, with zero boundary conditions and a hard floor at height zero, the model is the probability measure over height functions $\phi:\Lambda\to\Z_+$ on $\Lambda=\llb 1,L\rrb^2$ given by
\begin{equation}\label{eq:pi-def}
	\pi^0_\Lambda(\phi) \propto \exp\!\Big(-\beta\sum_{x\sim y}|\phi_x-\phi_y|^2\Big)\,,
\end{equation}
where $x\sim y$ denotes nearest neighbors in $\Z^2$, and the boundary conditions are $\phi_x=0$ for $x\notin\Lambda$. Let $\hatpi^0_\Lambda$ be the floor-free analogue ($\phi:\Lambda\to\Z$).
The \ZGFF phase transition for $\phi\sim\hatpi_\Lambda$ occurs at a critical $\beta_{\textsc{r}}>0$ (empirically, $\beta_{\textsc{r}}\approx 0.665$) as follows: 
\begin{itemize}
\item for $\beta\leq\beta_{\textsc{r}}$ (delocalized regime) the surface is rough, in that $\lim_{L\to\infty}\var(\phi_o)=\infty$; \item for $\beta>\beta_{\textsc{r}}$ (localized regime) the surface is rigid, with $\var(\phi_o)=O(1)$ (moreover, $|\phi_o|$ has a finite exponential moment). 
\end{itemize}
The localization result for $\beta\gg 1$ was established in~\cite{BrandenbergerWayne82} via an elementary Peierls argument. 
The delocalization for $\beta\ll 1$ was shown  in the aforementioned works of Fr\"ohlich and Spencer~\cite{FrohlichSpencer81a,FrohlichSpencer81b} (see also \cite{Lammers22,LammersOtt24,vanEngelenburgLis23,OttSchweiger25}).
The regimes extend to $(0,\beta_{\textsc{r}}),(\beta_{\textsc{r}},\infty)$ via the monotonicity of $\beta\mapsto \var(\phi_0)$ due to \cite{AHPS21} (this includes the exponential tail in the localized regime by log concavity of the marginals, see \cite[\S8.2]{Sheffield05}). Delocalization at the critical $\beta=\beta_{\textsc{r}}$ was thereafter proved by Lammers~\cite{Lammers22}.

Throughout this paper, we focus on a fixed large enough $\beta$, where the surface is localized, and we let $\hatpi_\infty$ denote the infinite-volume weak limit of $\hatpi_\Lambda$ as $L\to\infty$ (well-known to exist for such~$\beta$).
The~presence of a hard floor creates a nontrivial competition: the boundary conditions favor a flat surface at height zero in the localized regime, but the floor limits the entropy of surfaces near it, prohibiting some downward fluctuations.
This induces \emph{entropic repulsion}: the surface lifts itself to gain entropy from downward spikes, producing a macroscopic plateau far above the floor (\cref{fig:corner-sim}). In a recent companion paper~\cite{ChenLubetzky25}, we studied the \emph{local geometry} of this plateau---the mesoscopic law of the top level lines near the sides of the box. Our goal here is to study the \emph{global geometry} of the plateau---its macroscopic shape and the sharp transition at which each new layer appears.

\begin{figure}
    \vspace{-0.175in}
\centering
    \begin{tikzpicture}
 \node at (2,5.5) {\includegraphics[width=0.295\textwidth]{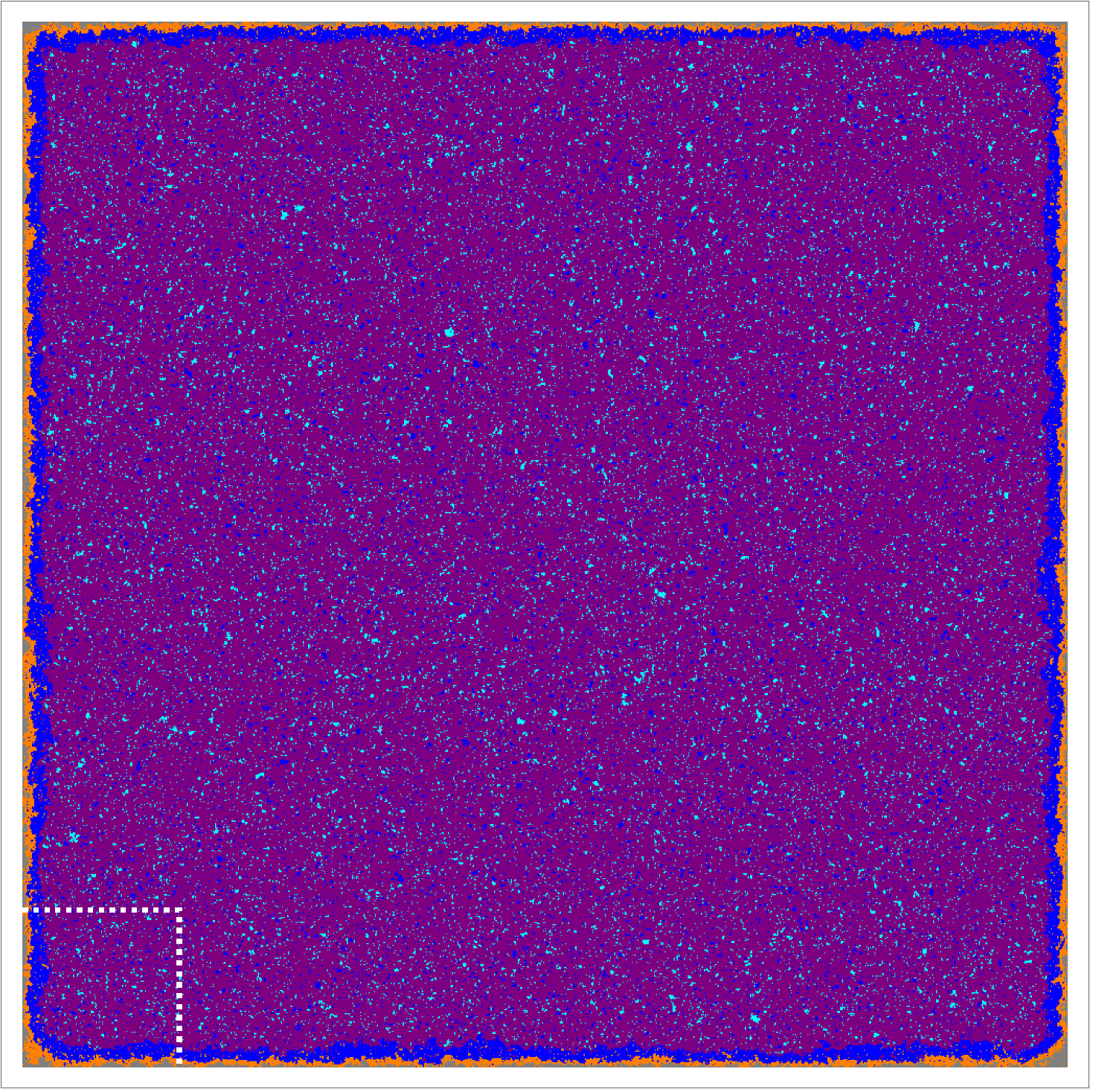}};

     \node at (8.5,5.5) {\includegraphics[width=0.45\textwidth]{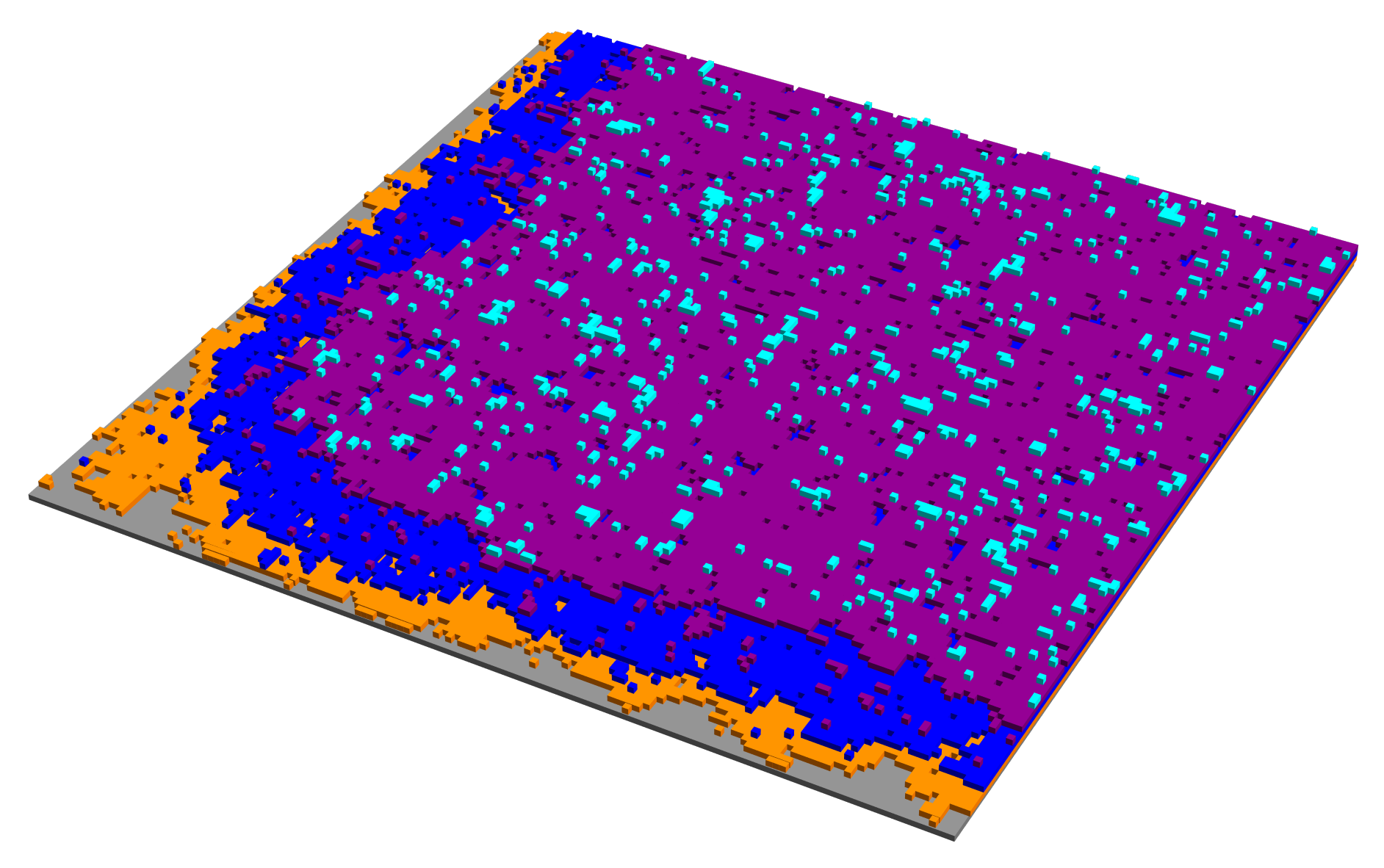}};
\end{tikzpicture}
\vspace{-0.2in}
\caption{Simulation of the low temperature $(2+1)$\Dim \ZGFF on a box $\Lambda$ of side length $L=1000$, zooming in on the corner of the box, where the macroscopic limit shape is visible.}
\label{fig:corner-sim}
\vspace{-0.15in}
\end{figure}

\subsection{Previous work on the level lines above a hard floor} 
Following is an account of the works most relevant to the current paper; see~\cite{ChenLubetzky25} for a more detailed review of the related literature.

\subsubsection{Entropic repulsion and the  surface plateau}\label{subsec:prev-work-shape}

 Bricmont, El-Mellouki, and Fr\"ohlich, in their breakthrough paper~\cite{BEF86}, were the first to rigorously analyze the effect of a hard floor on two closely related $\Z$-valued height functions above a hard floor: the \ZGFF from \cref{eq:pi-def}, and the (absolute-value) \SOS model, where the gradient cost $|\phi_x-\phi_y|^2$ in~\cref{eq:pi-def} is replaced by $|\phi_x-\phi_y|$. They established that, in both models at low enough temperature, the floor pushes the average surface height to $\operatorname{polylog}(L)$, with the heuristic being that a surface at height $h$ gains entropy from the extra downward fluctuations (of order $h$) that become accessible. Specifically, they showed that $\phi\sim\pi_\Lambda^0$ has $\frac1{L^2}\sum_x\E[\phi_x] \geq \frac{c}{\beta}\sqrt{\log L}$, and that the analogous \SOS model has $\frac1{L^2}\sum_x\E[\phi_x] \geq \frac{c}{\beta} \log L$. (This is the correct surface height for \SOS, while for \ZGFF it is off by a $\sqrt{\log\log L}$ factor; see \cref{eq:H-def}.)

For \SOS, Caputo et al.~\cite{CLMST12,CLMST14,CLMST16} made significant progress in understanding the geometry of the surface above a floor at low temperature. They showed that the surface typically forms a plateau: all but an $\epsilon$-fraction of the sites are all at the same height, which is either $\lfloor\frac{1}{4\beta}\log L\rfloor-1$ or $\lfloor\frac
{1}{4\beta}\log L\rfloor$.
They further obtained a shape theorem for the \SOS surface. The $h$ \emph{level lines} are the loops formed by dual-edges separating sites with $\phi<h$ from ones with $\phi\geq h$ (see \cref{def:level-lines} for more details). It was shown in \cite{CLMST16} that there exists $c_*(\beta)$ tending to $4$ as $\beta\to\infty$, so that for every fixed $\epsilon>0$, if \[ (1+\epsilon) c_* \beta e^{4\beta h} < L < (1-\epsilon)c_* \beta e^{4\beta (h+1)} \] 
then, with high probability (w.h.p.) as $L\to\infty$ (equivalently, as $h\to\infty$), one has that:
\begin{enumerate}[(i)]
\item\label{it:sos-h} At least $\frac{9}{10}$ of the sites have height exactly $h$ (\cite[Thm.~1]{CLMST16}).
\item\label{it:sos-shape} For each $j=0,\ldots,h$ there is a unique $j$ level-line loop of length $\geq(\log L)^2$ (none for $j>h$). These loops are nested, and after rescaling $\llb 1,L\rrb^2$ to $[0,1]^2$ they converge in Hausdorff distance to a deterministic limit obtained by translates of Wulff shapes (\cite[Thm.~2]{CLMST16}).
\end{enumerate}
This delineated for the $(2+1)$\Dim \SOS the sequence of critical side-lengths $L_c^{(h)}=c_* \beta e^{4\beta h}$ that mark the emergence of each new layer in the surface, where the plateau increases from height $h-1$ to $h$. 
(The \SOS behavior at the critical side-lengths $L=(1+o(1))L_c^{(h)}$ was not addressed by \cite{CLMST16}.)

\begin{figure}
\centering
    \vspace{-0.1in}
    \begin{tikzpicture}
    \begin{scope}[shift={(0,-.9)}]
    \filldraw[pattern=north east lines, preaction={fill=gray!15}, pattern color=red!75] (7.75,0.12) rectangle (8.75,0.02);

    \filldraw[pattern=north east lines, preaction={fill=gray!15}, pattern color=red!75] (0.25,0.12) rectangle (1.25,0.02);

    \filldraw[pattern=north east lines, preaction={fill=blue!15}, pattern color=blue!75] (1.27,0.22)--(1.75,0.22)--(1.75,0)--(2.5,0)--(2.5,-0.42)--(6.5,-0.42)--(6.5,0)--(7.25,0)--(7.25,0.22)--(7.73,0.22)--(7.73,-0.52)--(1.27,-0.52)--cycle;

    \draw[ultra thick, black] (0,0)--(10,0);
    
    \filldraw[fill=green!15] (2.5,-0.02) rectangle (6.5,-0.42);

    \filldraw[fill=gray!15] (1.75,0.02) rectangle (7.25,0.42);

    \draw[green!50!black,dotted,very thick] (2.5,-0.45)--node[below,yshift=-3pt,xshift=-2pt]{\tiny$1.1 L_c^{(h)}$}(2.5,-1);
   
    \draw[green!50!black,dotted,thick] (6.5,-0.45)--node[below,yshift=-3pt,xshift=2pt]{\tiny$0.9 L_c^{(h+1)}$}(6.5,-1);
   \node[circle,scale=0.4,draw=black,fill=gray,label={[label distance=3pt]below:{\small$L_c^{(h)}$}}] (Lc1) at (.75,0){};

    \draw[gray!75!black,dotted,thick] (1.75,0.45)--node[above,xshift=2pt,yshift=5pt]{\tiny$(1+o(1)) L_c^{(h)}$}(1.75,0.9);
   
    \draw[gray!75!black,dotted,thick] (7.25,0.45)--node[above,xshift=2pt,yshift=5pt]{\tiny$(1-o(1)) L_c^{(h+1)}$}(7.25,0.9);

    \node[font=\small,black] (sos) at (4.5,0.21) {\SOS: height $h$; limit shape known};

    \node[font=\small,black] (zgff) at (4.5,-0.23) {\ZGFF: height $h$; no shape};

    \node[circle,scale=0.4,draw=black,fill=gray,label={[label distance=3pt]below:{\small$L_c^{(h+1)}$}}] (Lc2) at (8.25,0){};
    \node[circle,scale=0.4,draw=black,fill=gray,label={[label distance=3pt]below:{\small$L_c^{(h)}$}}] (Lc1) at (.75,0){};

    \coordinate (p1) at (4.5,0.45);
    \coordinate (p2) at (4.5,-0.45);
    \end{scope}

    \begin{scope}[scale=.25,shift={(15,1)}]
    \draw [color=black,very thick] (0,0)--(5,0)--(5,5)--(-0,5)--cycle;
    \draw [thick, rounded corners=6,gray] (0,0) rectangle ++(5,5);
    cycle;
        \draw [thick, rounded corners=10,orange!90!white] (0,0) rectangle ++(5,5);
    cycle;
    \draw [thick, rounded corners=14,blue!90!black] (0,0) rectangle ++(5,5);

    \draw [thick, rounded corners=18,draw={rgb:red,128;green,0;blue,128},fill=purple!10] (0,0) rectangle ++(5,5);
    \node[text={rgb:red,128;green,0;blue,128},font=\tiny] at (2.5,2.5) {$h$};
    \node[text=black,font=\tiny] at (2.5,5.6) {\SOS};
    
    \draw[gray] (2.5,-0.25)--(p1);
    \end{scope}

     \begin{scope}[scale=.25,shift={(15,-13)}]
    \draw [color=black,very thick] (0,0)--(5,0)--(5,5)--(-0,5)--cycle;
         
    \pgfmathsetseed{314159}
\filldraw [thin, fill=purple!10,draw={rgb:red,128;green,0;blue,128}] (0.25,0.25) decorate[decoration={random steps,segment length=1.25mm,amplitude=0.5mm}]{--(0.25,4.75)} decorate[decoration={random steps,segment length=1.25mm,amplitude=0.5mm}]{-- (4.75,4.75)} decorate[decoration={random steps,segment length=1.25mm,amplitude=0.5mm}]{--(4.75,0.25)} decorate[decoration={random steps,segment length=1.25mm,amplitude=0.5mm}]{--(0.25,0.25)};

        \draw [dotted,thick,draw={rgb:red,128;green,0;blue,128},fill=purple!25] (0.6,0.6) rectangle ++(3.8,3.8);
    \node[text={rgb:red,128;green,0;blue,128},font=\tiny] at (2.5,2.5) {$h$};
    \node[text=black,font=\tiny] at (2.5,-0.6) {\ZGFF};
    
    \draw[gray] (2.5,5.25)--(p2);
    \end{scope}
    \end{tikzpicture}
    \vspace{-0.1in}
    \caption{Previous results for \SOS (gray region) in~\cite{CLMST16} and \ZGFF (green region) in \cite{LMS16}.
    In \SOS, the height and limit shape of the top level was identified outside a $1+o(1)$ window around $\lambda_*\beta/\hatpi_\infty(\phi_o=h)$, the natural candidate for $L_c^{(h)}$. The \ZGFF results excluded a larger $1+\delta/\beta$ window and missed the limit shape. \cref{thm:main-thm-crit-window,thm:grad-phi-p} extend the range where the top height is identified (to the blue region), excluding now a $1+(L_c^{(h)})^{-1/2+o(1)}$ window. \cref{thm:main-thm-limit-shape} gives the \ZGFF limit shape for all $L$, regardless of whether the top height can be identified. \cref{thm:grad-phi-p} provides this latter extension also for \SOS. 
    }
    \label{fig:known-vs-new-results}
    \vspace{-0.1in}
\end{figure}
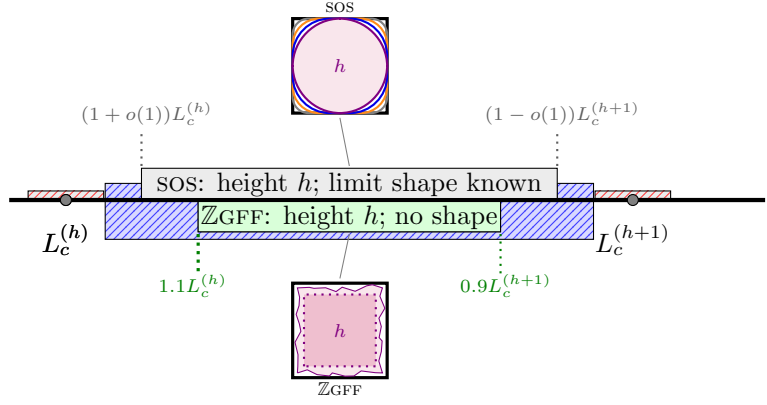

For \ZGFF, the second author, Martinelli, and Sly~\cite{LMS16} proved an analogue of the \SOS \cref{it:sos-h}. Let
\begin{equation}\label{eq:H-def}
	H(L) := \max\bigg\{h : \hatpi_\infty(\phi_o = h) \geq \frac{5\beta}{L}\bigg\} \asymp \sqrt{(1/\beta) \log L \log\log L}\,.
\end{equation}
Call a loop $\gamma$ \texttt{large} if its length is at least $\log L$ (one could even use a threshold of $(C/\beta)\log L$). The following result of \cite{LMS16} shows that the surface typically forms a sequence of nested \texttt{large} loops, one per level $h=0,\ldots,H$ and none at height $H+2$ or above:
\begin{theoremBG}[{\cite[Thm.~2]{LMS16}}\footnote{The cutoff for \texttt{large} loops in \cite{LMS16} (called \emph{macroscopic} loops there) was $\log^2 L$, though the proofs hold also for a threshold of $\log L$ provided $\beta$ is large enough; see the footnote below the definition of macroscopic loops in \cite[\S1.2]{LMS16}.}]
	Fix $\beta$ large and consider the \ZGFF $\phi\sim\pi^0_\Lambda$ per \cref{eq:pi-def} and $H(L)$ per \cref{eq:H-def}. Then, w.h.p, 
		there is a unique \texttt{large} $h$ level-line loop for each $h=0,\ldots,H$, and there are no \texttt{large} $h$ level-line loops for any $h\geq H+2$. This sequence of loops is nested; the loops for $h\leq H-1$ have area $(1-o(1))L^2$, and the loop for $h=H$ has area at least $(1-\epsilon_\beta)L^2$.
        \label{thm:LMS}
\end{theoremBG}
The proof argument in \cite{LMS16} (see \S4.4 there) further showed that if $\hatpi_\infty(\phi_o = H+1) > 4.1 \beta / L$ then there is an $(H+1)$ level-line loop with area at least $(1-\epsilon_\beta) L^2$, and that if $\hatpi_\infty(\phi_o = H+1) < 3.9 \beta / L$ then there are no \texttt{large} $(H+1)$ level-line loops. (The constants $3.9$ and $4.1$ can be taken to be $4-\delta$ and $4+\delta$ if $\beta$ is large enough as a function of $\delta$; namely, the aforementioned argument holds when replacing them by $4-5/\beta$ and $4+2/\beta$, respectively.) Together, these results of~\cite{LMS16} show that the critical side-length $L_c^{(h)}$ marking the onset of level $h$ behaves as $\approx (4\pm\epsilon_\beta)\beta / \hatpi_\infty(\phi_o=h)$. (See \cref{fig:known-vs-new-results} comparing these results with the \SOS picture summarized in \cref{it:sos-h,it:sos-shape} above.)

For all but an exceptional set $\sB$ of side lengths of zero logarithmic density (the values near which the plateau transitions from height $h$ to $h+1$; see \cref{eq:exceptional-set-def}), the top loop at height $H$ also has area $(1-o(1))L^2$ and is the unique macroscopic level line at that height.

\subsubsection{Local limit at the flat boundary}
\label{subsec:prev-work-FS}
If the low temperature \ZGFF model is to take after its \SOS counterpart, then the macroscopic limit shapes of its top level lines should press against all four sides of the box, producing \emph{flat portions} coinciding with each side. Along these flat portions, the level line can be viewed as a 1\Dim  interface pinned near the side wall, and the relevant question is the scale and law of its transversal fluctuations. Near the four corners, one wishes to both characterize the curved limit, and obtain the law of the random fluctuations around it.

In the \SOS model, the top level line was shown in~\cite{CLMST16} to have random fluctuations of at most $L^{1/3+o(1)}$ from the flat portions of its limit. A key ingredient was showing that the law of the top level line in the relevant region resembles that of a random walk, conditioned to be nonnegative, and penalized exponentially in the area below it. Such area-tilted walks are known~\cite{ISV15} to have $L^{1/3}$ fluctuations and their scaling limit is a Ferrari--Spohn diffusion~\cite{FerrariSpohn05}. It was recently established~\cite{CKL24} that the order of the \SOS top level line fluctuations is at least $L^{1/3}$, and one expects that to be the correct order. Notably, the conjectured \SOS scaling limit for an individual level line in the presence of other level lines is \emph{not} Ferrari--Spohn but a variant that accounts for the interaction between the different level lines. Rather, the conjectured \SOS limit should be that of an ensemble of non-crossing random walks with geometric area tilts (see, e.g.,~\cite{BCG25,CaputoGanguly23,CIW19a,CIW19b,DLZ23,HKS25,Serio23} for studies of this line ensemble in the discrete and continuous settings, which for a single curve reverts to a Ferrari--Spohn limit).

For the \ZGFF, the picture turns out to be different. It was shown in the companion paper~\cite{ChenLubetzky25} that if one excludes a set $\sB$ of near-critical side lengths (akin to the set excluded in \cite{LMS16}; see \cref{eq:exceptional-set-def}), then the following separation of scales occurs for the top level lines. The $n$-th from the top level line $\fL_n$ lies typically at distance $N_n^{1/3}$ from the flat portion of the boundary for 
\begin{equation}\label{eq:Nn-def}	N_n := \frac{1}{\hatpi_\infty(\phi_o = H+1-n)}\quad(=L^{1-o(1)})\qquad(n=0,1,\ldots)\,.\,,\end{equation}
which further satisfies
$ N_n = N_{n-1}e^{-\Theta(\sqrt{\beta\log L/\log\log L})}$ (so $\fL_{n}$ is supported in  ``a scale of its own''). 
This separation of scales allowed the authors to decouple the different level lines, and show that the top $m$ level lines, each after an $(N_n^{2/3},N_n^{1/3})$ space-time rescaling of its distance from the side boundary, converge to \emph{independent} Ferrari--Spohn diffusions. Specifically, the Ferrari--Spohn diffusion $\mathsf{FS}_\sigma$ on $(0,\infty)$ is the diffusion with Dirichlet boundary condition at $0$ and generator
\begin{equation}\label{eq:fs-generator-intro}
	\sfL_\sigma = \tfrac{\sigma^2}{2}\tfrac{\d^2}{\d x^2} + \sigma^2\tfrac{\varphi_\sigma'(x)}{\varphi_\sigma(x)}\tfrac{\d}{\d x}\,,
\end{equation}
where $\varphi_\sigma(x) = \mathsf{Ai}((2/\sigma^2)^{1/3}x - \omega_1)$ for  $\mathsf{Ai}$ the Airy function and $\omega_1=\min\{x>0:\mathsf{Ai}(-x)=0\}$. This process is ergodic and reversible against the probability density $\frac{(2/\sigma^2)^{1/3}}{\mathrm{Ai'(-\omega_1)^2}}  \varphi_\sigma(x)^2\one_{\{x>0\}}$.
Let \begin{equation}\label{eq:exceptional-set-def} \sB = \bigcup_{h\geq 1} \llb \tfrac34 \overline L^{(h)}, \overline L^{(h)}\rrb \qquad\mbox{where}\qquad
		\overline L^{(h)} = \lceil 5\beta /\hatpi_\infty(\phi_o=h)\rceil \qquad(h=1,2,\ldots)\end{equation} 
be the set of excluded side length. Recalling that the transition from a plateau at height $h-1$ to one at height $h$ was known to occur at $(4\pm\epsilon_\beta)\beta/\hatpi_\infty(\phi_o=h)$, by excluding $\sB$ one avoids this delicate window: for $L\notin\sB$, the plateau is w.h.p.\ at height $H$ from \cref{eq:H-def} (rather than $H+1$).
(The previous work \cite{LMS16} excluded $\llb 3.9\beta /\hatpi_\infty(\phi_o=h), 4.1\beta /\hatpi_\infty(\phi_o=h) \rrb$ with the exact same effect.)

\begin{theoremBG}[{\cite[Thm.~1.1]{ChenLubetzky25}}]\label{thm:CL25}
	Fix $\beta>0$ large enough and $m\geq 1$. Let $L\notin\sB$ for $\sB$ as in \cref{eq:exceptional-set-def}, and let $\fL_n$ ($n=1,\ldots,m$) be the \texttt{large} $(H+1-n)$ level line\footnote{Excluding the set $\sB$ meant that w.h.p.\ there is no \texttt{large} $H+1$ level line, so $\fL_n$ is the $n$-th highest \texttt{large} level line.} of $\phi\sim\pi^0_\Lambda$ on $\Lambda=\llb 1,L\rrb^2$. Set~$N_n$ as in \cref{eq:Nn-def}, and let $I_n$ be the interval of length $2N_n^{2/3}$ co-centered on the bottom side of $\Lambda$. Denote by $\psi_n(x) = \min\{y\geq 0:(\frac{L}{2}+x,y)\in\fL_n\}$ the vertical distance of $\fL_n$ from~$I_n$.
	Then the joint law of the rescaled distances 
    $ Y_n(t) := N_n^{-1/3}\psi_n(t N_n^{2/3})$ of the top $m$ level lines ($n=1,\ldots,m$) converges weakly to $m$ i.i.d.\ stationary Ferrari--Spohn diffusions $\mathsf{FS}_{\sigma}$ on~$[-1,1]$ for an explicit fixed $\sigma>0$.
\end{theoremBG}

The above theorem showed, in particular, that the top level line fluctuations at the center of the side are of order $N_1^{1/3} = L^{1/3-o(1)}$, as previously conjectured in~\cite{LMS16}. As stressed in \cite[Rem.~1.2]{ChenLubetzky25}, the $L^{o(1)}$ correction in the estimate for the scale $N_n$ is necessary: even though $N_n\asymp L$ for infinitely many values of $L$, for infinitely many others one has $N_n \approx L \exp(-c\sqrt{\beta\log L\log\log L})$. 

While \cref{thm:CL25} addressed only the local law in the flat portion of the limit shape (where it coincides with the side boundaries), the $N_n^{1/3}$ fluctuations are also connected to the  global limit shape near the corners of the box. As in \SOS, the droplet delimited by the \ZGFF level line $\fL_n$ behaves in an $N_n^{2/3}\times N_n^{1/3}$ rectangle as an area-tilted random walk, induced to either advance or retreat as per the behavior of the corresponding Wulff shape in the rectangle. This was leveraged in \cite{ChenLubetzky25} (following \cite{CLMST12} for \SOS) via a ``growth gadget'' (\cite[Thm.~4.4]{ChenLubetzky25}; see \cref{thm:old-level-line-contains-Wulff} below) with two ramifications: (i)~$\fL_n$ is pressed against the flat boundary to (nearly) the correct scale; and (ii)~$\fL_n$ contains a Wulff shape near the corners of the box. In the first---the focus of \cref{thm:CL25}---there was no need for a ``retreat gadget'' as $\fL_n$ is trivially bounded by the boundary of the box in its flat portion. However, in the second, without a retreat gadget, \cite{ChenLubetzky25}  gave only an ``inner bound'' on the global limit shape. A matching ``outer bound'' seemed highly nontrivial: the growth gadget for \ZGFF already required significant new ideas compared to the \SOS setting, and several of these relied crucially on monotonicity arguments that go in the wrong direction for a retreat gadget. 

\subsection{Main results}\label{sec:main-results}
Two central open problems left in~\cite{LMS16} (see \S1.5 there) were to determine the \emph{global limit shape} of the \ZGFF level lines---in particular, whether the Wulff construction that governs the \SOS limit shape also governs the \ZGFF---and to prove that the top level line has fluctuations of order $L^{1/3+o(1)}$ along the flat portions of its limit. We resolved the latter in the companion work~\cite{ChenLubetzky25} (see \cref{thm:CL25} above), except at side lengths $L\in\sB$ that are near the  critical points~$\{L_c^{(h)}: h\geq 1\}$. In this paper we address the former, as well as extend the analysis of \cite{ChenLubetzky25} to cover every $L$.

The first theorem provides the global limit shape of the top finitely many level lines, namely, the $H+1-n$ level line for every fixed $n\geq 0$ (as per \cite{LMS16}, w.h.p.\ there are no  \texttt{large} $H+2$ level lines, but that there may or may not be such a loop at height $H+1$; the prequel \cite{ChenLubetzky25} looked only at the latter $L$'s, whence the top height was $H$).
To construct the limit shape for $\fL_n$, we begin with the Wulff shape $\cW_1$ which is the convex body with unit area whose boundary minimizes the surface tension integral of the model.
Let $\cW$ denote $\cW_1$ scaled by $\sfw_1/2$, 
where $\sfw_1$ is the value of said surface-tension integral along $\partial \cW_1$. (See \cref{sec:poly-st-wulff} for the definitions related to the Wulff shape. The proofs will use a slightly more accurate rescaling of $\cW$, relevant if one wishes to improve the quantitative convergence estimates; see \cref{eq:def-ell*}.) The Wulff shape $\cW$ will appear in the limit shape of each $\fL_n$ after a rescaling by $N_n$; see \cref{fig:limit-shape-schematic} for a depiction of this self-similar limit.
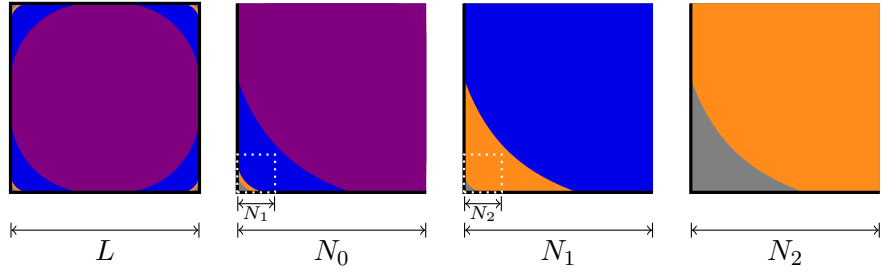
\begin{figure}
\vspace{-0.1in}
\centering
    \begin{tikzpicture}
  \begin{scope}[scale=.5]
    \fill [color=black,very thick,fill=gray] (0,0)--(5,0)--(5,5)--(-0,5)--cycle;
    \fill [thick, rounded corners=5,fill=orange!90!white] (0,0) rectangle ++(5,5);
    cycle;
    \fill [thick, rounded corners=7,fill=blue!90!black] (0,0) rectangle ++(5,5);
    \fill [thick, rounded corners=30,fill={rgb:red,128;green,0;blue,128}] (0,0) rectangle ++(5,5);
    \draw [color=black,very thick] (0,0)--(5,0)--(5,5)--(-0,5)--cycle;

    \draw [|<->|] (0,-1)--node[below] {$L$}(5,-1);
    
    \end{scope}

    \begin{scope}[scale=.5,shift={(6,0)}]
    \fill[color=gray] (0,0)--(5,0)--(5,5)--(-0,5)--cycle;
    \fill [thick, rounded corners=8,fill=orange!90!white] (0,0) rectangle ++(5,5);
    cycle;
    \fill [thick, rounded corners=12,fill=blue!90!black] (0,0) rectangle ++(5,5);
    \fill [fill={rgb:red,128;green,0;blue,128}] (0,5)--(0,3) to[bend right=25] (3,0)--(5,0)--(5,5)--cycle;
    \draw[color=black, very thick] (5,0)--(0,0)--(0,5);
    \draw[color=white, dotted, thick] (0,0)--(0,1)--(1,1)--(1,0)--cycle;

    \draw [|<->|] (0,-0.3)--node[below,yshift=3pt] {\tiny$N_1$}(1,-0.3);
    \draw [|<->|] (0,-1)--node[below] {$N_0$}(5,-1);
    
    \end{scope}

    \begin{scope}[scale=.5,shift={(12,0)}]
    \fill[color=gray] (0,0)--(5,0)--(5,5)--(-0,5)--cycle;
    \fill [thick, rounded corners=8,fill=orange!90!white] (0,0) rectangle ++(5,5);
    \fill [fill=blue!90!black] (0,5)--(0,3) to[bend right=25] (3,0)--(5,0)--(5,5)--cycle;
    \draw[color=black, very thick] (5,0)--(0,0)--(0,5);
    \draw[color=white, dotted, thick] (0,0)--(0,1)--(1,1)--(1,0)--cycle;
    \draw [|<->|] (0,-0.3)--node[below,yshift=3pt] {\tiny$N_2$}(1,-0.3);
    \draw [|<->|] (0,-1)--node[below] {$N_1$}(5,-1);
    
    \end{scope}

    \begin{scope}[scale=.5,shift={(18,0)}]
    \fill[color=gray] (0,0)--(5,0)--(5,5)--(-0,5)--cycle;
    \fill [fill=orange!90!white] (0,5)--(0,3) to[bend right=25] (3,0)--(5,0)--(5,5)--cycle;
    \draw[color=black, very thick] (5,0)--(0,0)--(0,5);
    \draw [|<->|] (0,-1)--node[below] {$N_2$}(5,-1);
    
    \end{scope}
\end{tikzpicture}
\vspace{-0.1in}
\caption{Schematic of the macroscopic limit shape of the top level lines $n=0,1,\ldots$ as established in \cref{thm:main-thm-limit-shape}, with the associated scales $N_n$ near the corner of the box $\Lambda$.}
\vspace{-0.1in}
\label{fig:limit-shape-schematic}
\end{figure}
\begin{theorem}[Limit Shape]\label{thm:main-thm-limit-shape}

Fix $\beta>0$ large enough and consider $\phi\sim \pi^0_\Lambda$, the $(2+1)$\Dim \ZGFF on $\Lambda=\llb 1,L\rrb^2$ with a floor and zero boundary conditions. Set $H(L)$ and $N_n$ per \cref{eq:H-def,eq:Nn-def}. Fix $m\geq 1$, and let $\sL_n$ ($n=0,\ldots,m$) be obtained by placing the 4 quadrants of the Wulff shape $\cW$ defined above, rescaled by $N_n$, in the 4 corners of $\Lambda$, and joining them via straight lines along  $\partial \Lambda$. Let $\fL_n$ be the \texttt{large} $(H+1-n)$ level line loop(s) of $\phi$.
Then, with probability $1-O(L^{-10})$:
\begin{enumerate}
    \item{} For each $1\leq n\leq m$, $\fL_n$ is a single loop that satisfies $d_\cH(\fL_n,\sL_n)< N_n e^{-\sqrt{\log L}}$.

\item{} For $n=0$, either $\fL_0$ is empty (no \texttt{large} $(H+1$) level-line loops), or it is a single loop that satisfies $d_\cH(\fL_0, \sL_0) < N_0 e^{-\sqrt{\log L}}$.
\end{enumerate}
\end{theorem}
An important aspect of the above theorem is that it holds for all $L$, rather than avoiding an exceptional set of $L$'s around the emergence of the $H+1$ level line. Consequentially, for any $L$, w.h.p.\ there cannot be a loop $\fL_0$ which only occupies, say, half of $\Lambda$. That is, there are no values of $L$ where, with constant probability, the surface features a plateau at height~$H+1$  coexisting with a plateau at height~$H$.
This byproduct of \cref{thm:main-thm-limit-shape} is reminiscent of a discontinuous phase transition, an interpretation that can be made precise as per the following remark.

\begin{remark*} By FKG, one may couple all \ZGFF models $\phi^{(L)}\sim \pi^0_{\Lambda_L}$ on $\Lambda_L=\llb 1,L\rrb^2$ for $L=1,2,\ldots$ such that if $L_1 < L_2$ then $\phi^{(L_1)} \leq \phi^{(L_2)}\restriction_{\Lambda_{L_1}}
$ pointwise. As we increase $L$ starting from 1 and sample within this coupling, for every height $h \in \Z_+$ there will be a random side length $L(h)$ which is the minimum $L$ such that $\phi^{(L)}$ contains a \texttt{large} $h$ level line. \Cref{thm:main-thm-limit-shape} shows that, eventually a.s.\ for $h$, every $L>L(h)$ will also feature a \texttt{large} $h$ level line (so the phase transition is well-defined), and, immediately upon its emergence in $\phi^{(L(h))}$, the $h$ level line will have the same ``super-critical'' features (asymptotically, after an appropriate rescaling) of the $h$ level line at $\phi^{(L(h+1)-1)}$. \end{remark*}

Although the $H+1$ and $H$ phases cannot coexist, there may still be a critical window of $L$ where both phases have a constant probability to exist; e.g., for every $h$ we can set $L_1(h)$ as the first $L$ where the probability of a \texttt{large} $h$ is at least $\frac1{10}$, and $L_2(h)$ as the last $L$ where it is at most~$\frac9{10}$. Our second result states that the width of this critical window is at most  $L^{1/2+o(1)}$, and the probability for $\phi$ to feature a \texttt{large} $h$ level line undergoes a sharp transition from $o(1)$ to $1-o(1)$ (see \cref{fig:sharp-transition}). Fixing $h$, let $\underline{L}^{(h)} := \frac34\overline{L}^{(h)}$ from \cref{eq:exceptional-set-def}, and define $\lambda_* := \tfrac{1}{\beta}(2\tau(0) + \tfrac{\sfw_1}{2}) \approx 4(1\pm\epsilon_\beta)$ where $\tau$ and $\sfw_1$ are the surface tension at angle 0 and surface-tension integral along the unit area Wulff shape.

\begin{theorem}[Critical window]\label{thm:main-thm-crit-window}
Fix any $h \in \Z_+$, and define with respect to $\underline{L}^{(h)}$ the probability
\[p_h(L) := \pi^0_{\Lambda}(\text{there exists an $h$ level-line loop with interior area $\geq .8\underline{L}^{(h)}$})\,.\]
Then the function $L\mapsto p_h(L)$ is increasing, and the probability that there exists a \texttt{large} $h$ level-line loop is $p_h(L)+o(1)$, where the $o(1)$-term goes to $0$ as $L\to\infty$.
Furthermore, if we define \[ \Lh = \min\{L : p_h(L) \geq 1/2 \}\,,\] 
then we have
\[ \Lh = (1+o(1))\frac{\lambda_* \beta}{\hatpi_\infty(\phi_o = h)}\]
and the following holds for every $h$:
\begin{enumerate}[(i)]
\item if $L < \Lh - (\Lh)^{1/2 + o(1)}$, then $p_h(L) = o(1)$;
\item if $L > \Lh + (\Lh)^{1/2 + o(1)}$, then $p_h(L) = 1-o(1)$.
\end{enumerate}
\end{theorem}

\begin{figure}
\vspace{-0.1in}
\centering
    \begin{tikzpicture}

    \begin{scope}[shift={(0,-.9)}]
    
    \filldraw[pattern=north east lines, preaction={fill=gray!15}, pattern color=red!75] (7.5,0.12) rectangle (8.5,0.02);
    \draw[|<->|] (7.5,0.3)--node[above=2pt] {\tiny$L^{1/2+o(1)}$}(8.5,0.3);

    \filldraw[pattern=north east lines, preaction={fill=gray!15}, pattern color=red!75] (0.25,0.12) rectangle (1.25,0.02);

    \filldraw[pattern=north east lines, preaction={fill=blue!15}, pattern color=blue!75] (1.25,0.12) rectangle (2.5,-0.12);

    \filldraw[pattern=north east lines, preaction={fill=green!15}, pattern color=green!75] (2.5,-0.12) rectangle (6.5,-0.02);

    \filldraw[pattern=north east lines, preaction={fill=blue!15}, pattern color=blue!75] (6.5,0.12) rectangle (7.5,-0.12);

    \draw[ultra thick, black] (0,0)--(10,0);
    
    \draw[green!50!black] (2.5,-0.15)--node[below,xshift=-2pt]{\tiny$1.1 L_c^{(h)}$}(2.5,-0.35);
   
    \draw[green!50!black] (6.5,-0.15)--node[below,xshift=2pt]{\tiny$0.9 L_c^{(h+1)}$}(6.5,-0.35);
   \node[circle,scale=0.4,draw=black,fill=gray,label={[label distance=3pt]below:{\small$L_c^{(h)}$}}] (Lc1) at (.75,0){};

    \node[circle,scale=0.3,draw=black,fill=gray] (p1) at (1.05,0){};

    \node[circle,scale=0.3,draw=black,fill=gray] (p2) at (1.7,0){};

    \node[circle,scale=0.3,draw=black,fill=gray] (p3) at (2.1,0){};
    
    \node[circle,scale=0.3,draw=black,fill=gray] (p4) at (5,0){};
    
    \node[circle,scale=0.3,draw=black,fill=gray] (p5) at (7.25,0){};
    \node[circle,scale=0.4,draw=black,fill=gray,label={[label distance=3pt]below:{\small$L_c^{(h+1)}$}}] (Lc2) at (8,0){};

    \node[circle,scale=0.4,draw=black,fill=gray,label={[label distance=3pt]below:{\small$L_c^{(h)}$}}] (Lc1) at (.75,0){};
    
    \node[circle,scale=0.3,draw=black,fill=gray] (p6) at (8.75,0){};

    \end{scope}

    \begin{scope}[scale=.3,shift={(-6,0)}]
    \fill [color=black,very thick,fill=gray] (0,0)--(5,0)--(5,5)--(-0,5)--cycle;
    \fill [thick, rounded corners=7,fill=orange!90!white] (0,0) rectangle ++(5,5); cycle;
    \draw [color=black,very thick] (0,0)--(5,0)--(5,5)--(-0,5)--cycle;
    \coordinate (p11) at (2.5, -0.25);
    \end{scope}

    \begin{scope}[scale=.3,shift={(0,0)}]
    \fill [color=black,very thick,fill=gray] (0,0)--(5,0)--(5,5)--(-0,5)--cycle;
    \fill [thick, rounded corners=7,fill=orange!90!white] (0,0) rectangle ++(5,5);
    cycle;
    \fill [thick, rounded corners=20,fill=blue!90!black] (0,0) rectangle ++(5,5);
    \draw [color=black,very thick] (0,0)--(5,0)--(5,5)--(-0,5)--cycle;
    \coordinate (p12) at (2.5, -0.25);
    \end{scope}

    \draw[gray] (p11) [bend right=30]to node[midway] (p11p12){} (p12);
    \draw[gray] (p11p12)--(p1);

    \begin{scope}[scale=.3,shift={(6,0)}]
    \fill [color=black,very thick,fill=gray] (0,0)--(5,0)--(5,5)--(-0,5)--cycle;
    \fill [thick, rounded corners=5,fill=orange!90!white] (0,0) rectangle ++(5,5);
    cycle;
    \fill [thick, rounded corners=13,fill=blue!90!black] (0,0) rectangle ++(5,5);
    \draw [color=black,very thick] (0,0)--(5,0)--(5,5)--(-0,5)--cycle;
    \draw[gray] (2.5,-0.25)--(p2);
    \end{scope}

    \begin{scope}[scale=.3,shift={(13,0)}]
    \fill [color=black,very thick,fill=gray] (0,0)--(5,0)--(5,5)--(-0,5)--cycle;
    \fill [thick, rounded corners=5,fill=orange!90!white] (0,0) rectangle ++(5,5);
    cycle;
    \fill [thick, rounded corners=7,fill=blue!90!black] (0,0) rectangle ++(5,5);
    \draw [color=black,very thick] (0,0)--(5,0)--(5,5)--(-0,5)--cycle;
    \draw[gray] (2.5,-0.25)--(p3);
    \draw[gray] (2.5,-0.25)--(p4);
    \draw[gray] (2.5,-0.25)--(p5);
    \end{scope}

    \begin{scope}[scale=.3,shift={(28,0)}]
    \fill [color=black,very thick,fill=gray] (0,0)--(5,0)--(5,5)--(-0,5)--cycle;
    \fill [thick, rounded corners=5,fill=orange!90!white] (0,0) rectangle ++(5,5);
    cycle;
    \fill [thick, rounded corners=7,fill=blue!90!black] (0,0) rectangle ++(5,5);
    \fill [thick, rounded corners=20,fill={rgb:red,128;green,0;blue,128}] (0,0) rectangle ++(5,5);
    \draw [color=black,very thick] (0,0)--(5,0)--(5,5)--(-0,5)--cycle;
    \draw[gray] (2.5,-0.25)--(p6);
    \end{scope}
\end{tikzpicture}
\vspace{-0.14in}
\caption{Evolution of the surface, as a new layer (plateau at height $h$) emerges. In the red interval around $L_c^{(h)}$, the top level may be $h$ (dark blue) or $h-1$ (orange), possibly both with constant probability. In the blue and green intervals, the top level is always $h$, with a larger and larger limit shape as $L$ increases until it occupies a $1-o(1)$ fraction of the box. This was established in the green interval (without identifying the limit shape) in \cite{LMS16}.}
\vspace{-0.08in}
\end{figure}
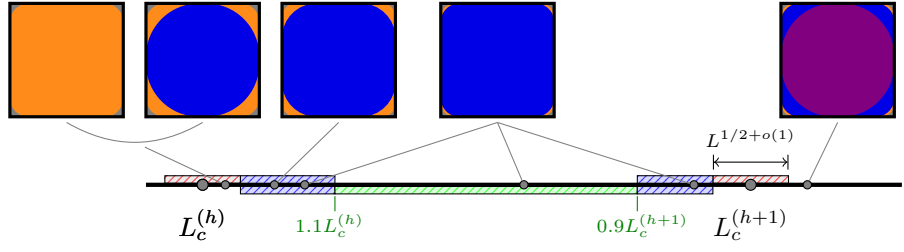

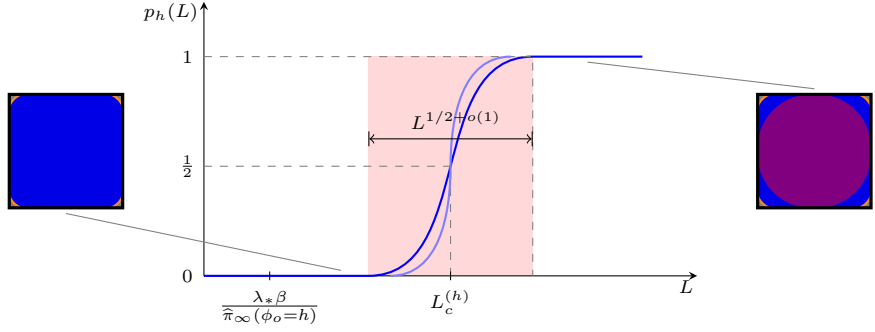
\begin{figure}
\vspace{-0.1in}
    \begin{tikzpicture}
    \begin{scope}[scale=1.45]

    \pgfmathsetmacro{\xstart}{-0.5};
    \pgfmathsetmacro{\xend}{4};
        
    \fill[fill=red!15] (1,0) rectangle (2.5,2);

    \draw[-stealth] (\xstart,0)--(\xstart,2.5);
    \draw[-stealth] (\xstart,0)--(\xend,0);
    \draw[thick,blue] (\xstart,0)--(1,0) to[out=0, in=180] (2.5,2) -- (3.5,2);
    \draw[gray,dashed] (\xstart,1)--(1.75,1);
    \draw[gray,dashed] (1.75,0)--(1.75,1);
    \draw[gray,dashed] (\xstart,2)--(2.5,2)--(2.5,0);
    \node[font=\tiny] at (\xstart-0.15,2) {$1$};
    \node[font=\tiny] at (\xstart-0.15,1) {$\frac12$};
    \node[font=\tiny] at (\xstart-0.15,0) {$0$};
    
    \draw (1.75, 0.04) -- ++(0, -0.08) node[below, font=\tiny] {$L_c^{(h)}$};

    \draw[thick,blue!50] (1.2,0) to[out=0, in=-90] (1.75,1) to[out=90, in=180] (2.3,2);

    \node[font=\tiny] at (\xend-0.1,-0.1) {$L$};
    \node[font=\tiny] at (\xstart-0.3,2.4) {$p_h(L)$};

    \draw (\xstart+0.6, 0.04) -- ++(0, -0.08) node[below, font=\scriptsize] {$\frac{\lambda_*\beta}{\hatpi_\infty(\phi_o=h)}$};

    \draw[|<->|] (1,1.25)--node[above,text=black,xshift=2pt,font=\tiny] {$L^{1/2+o(1)}$}(2.5,1.25);

    \coordinate (p1) at (0.75,0.05);
    \coordinate (p2) at (3,1.95);

    \end{scope}

     \begin{scope}[scale=.3,shift={(-11,3)}]
    \fill [color=black,very thick,fill=gray] (0,0)--(5,0)--(5,5)--(-0,5)--cycle;
    \fill [thick, rounded corners=5,fill=orange!90!white] (0,0) rectangle ++(5,5);
    cycle;
    \fill [thick, rounded corners=7,fill=blue!90!black] (0,0) rectangle ++(5,5);
    \draw [color=black,very thick] (0,0)--(5,0)--(5,5)--(-0,5)--cycle;
    \draw[gray] (2.5,-0.25)--(p1);
    \end{scope}

    \begin{scope}[scale=.3,shift={(22,3)}]
    \fill [color=black,very thick,fill=gray] (0,0)--(5,0)--(5,5)--(-0,5)--cycle;
    \fill [thick, rounded corners=5,fill=orange!90!white] (0,0) rectangle ++(5,5);
    cycle;
    \fill [thick, rounded corners=7,fill=blue!90!black] (0,0) rectangle ++(5,5);
    \fill [thick, rounded corners=20,fill={rgb:red,128;green,0;blue,128}] (0,0) rectangle ++(5,5);
    \draw [color=black,very thick] (0,0)--(5,0)--(5,5)--(-0,5)--cycle;
    \draw[gray] (2.5,5.25)--(p2);
    \end{scope}
    \end{tikzpicture}
    \vspace{-0.15in}
    \caption{The sharp transition of $p_h(L)$, a proxy for the probability that there exists a \texttt{large} $h$ level line. $p_h(L)$ is monotone, and increases from near $o(1)$ to near $1-o(1)$ within a window of size at most $(L_c^{(h)})^{1/2 + o(1)}$. It is possible that the window is even smaller (depicted in light blue), and identifying the real window size is still open (see \cref{sec:open-problems}). The  prediction $\frac{\lambda_* \beta}{\hatpi_\infty(\phi_o = h)}$ for $L_c^{(h)}$ from previous work on \SOS lies outside the critical window.}
    \label{fig:sharp-transition}
    \vspace{-0.15in}
\end{figure}
\begin{remark*}
    We further prove that the $(1+o(1))$ correction is necessary in the expression for $L_c^{(h)}$:
\begin{equation}\label{eq:lambda*-correction}
    L_c^{(h)} > \frac{\lambda_* \beta}{\hatpi_\infty(\phi_o = h)} + L^{1-o(1)}\,;
    \end{equation}
    in particular, $L = \lfloor\frac{\lambda_* \beta}{\hatpi_\infty(\phi_o = h)}\rfloor$ lies in the ``sub-critical'' regime where 
    $p_h(L) = o(1)$, as \cref{fig:sharp-transition} depicts. 
\end{remark*}
    To give context to the last remark, let us sketch why $ \frac{\lambda_* \beta}{\hatpi_\infty(\phi_o = h)}$ was the natural prediction for~$L_c^{(h)}$, appearing also in the previous work \cite{CLMST16} on the \SOS model. As argued in the pioneering work of Bricmont, El-Mellouki, and Fr\"ohlich~\cite{BEF86}, an $h$ level line $\fL$, with length $|\fL|$ and interior $\Int(\fL)$, costs $\beta|\fL|$ in energy, and gains an entropy of $\log( 1+\mathsf{p}_h)\Int(\fL) \approx \mathsf{p}_h\Int(\fL)$ where $\mathsf{p}_h$ is the probability of a spike of depth $h$, thanks to an independent choice of placing such a spike in every $x\in\Int(x)$. If  $\fL$ were an $L\times L$ square, then the energetic cost would be $4\beta L$ and, for \SOS, the entropic gain would be $e^{-4\beta h} L^{2}$, leading to the prediction of $L_c^{(h)} = \frac{4\beta}{e^{-4\beta h}}$ as the value of $L$ balancing these. This heuristic was refined in \cite{CLMST16} via replacing (a) the ansatz on the shape of $\fL$ from a box to its Wulff-type limit shape, changing $4$ into $\lambda_*$; and (b) the probability of a downward spike $\mathsf{p}_h$ by the large deviation shape under $\hatpi_\infty$ (which becomes a spike after $O(1)$ levels), changing $e^{-4\beta h}$ into $\hatpi_\infty(\phi_o = h) \sim c_0 e^{-4\beta h}$.
    The gap between $L_c^{(h)}$ and $\frac{\lambda_* \beta}{\hatpi_\infty(\phi_o = h)}$ for the \ZGFF highlights a more subtle aspect of this phenomenon. The entropy term $\hatpi_\infty(\phi_o=h)|\Int(\fL)|$ treated every $x$ in the interior of $\fL$ as having an independent choice for a downward deviation, ignoring  local correlations... Since the large deviations of the \ZGFF are shown in \cite{LMS16} to resemble harmonic functions, a downward deviation drags with it nearby points, and this effect is significant enough to change the location of $L_c^{(h)}$.
    See \cref{rem:bad-prediction} for the technical details.

\medskip

The next theorem extends \cref{thm:CL25} to all $L$, removing the exceptional set. As we discussed, for $L$ within the critical window, both events $\{\fL_0 \text{ exists}\}$ and $\{\fL_0 \text{ does not exist}\}$ may have constant probabilities. Nevertheless, when we condition on them, the local limit law is still Ferrari--Spohn. That is, in the same flavor of the discussion following \cref{thm:main-thm-limit-shape}, there are no $L$ for which the $H+1$ level line exists yet exhibits a different local behavior along the flat parts of the limit shape.

\begin{theorem} \label{thm:FS-no-exceptional}
		Fix $\beta>0$ large enough and consider $\phi\sim \pi^0_\Lambda$, the $(2+1)$\Dim \ZGFF model on $\Lambda=\llb 1,L\rrb^2$ 
        with a floor and zero boundary conditions. Set $H(L)$ and $N_n$ per \cref{eq:H-def,eq:Nn-def}.
For fixed~$m$, let $\fL_n$ ($n=0,\ldots,m$) be the \texttt{large} $(H+1-n)$ level lines, where possibly $\fL_0$ does not exist.
Let $I_n = \llb \frac{L}2- N_n^{2/3},\frac{L}2+N_n^{2/3}\rrb$, let $\rho_n(x) := \min\{y\geq 0 \,:\; (\frac{L}2+x,y)\in\fL_n\}$ be the vertical distance of~$\fL_n$ from $I_n$, and 
$Y_n(t) := N_n^{-1/3}\rho_n(t N_n^{2/3})$.
Then there exists a constant $\sigma>0$  such that:
\begin{enumerate}[1.,leftmargin=2em]
\item The law of $(Y_n(t))_{n=1}^m$ under $\pi_\Lambda^0$ converges weakly to the law of 
$m$ i.i.d.\ copies of the stationary Ferrari--Spohn diffusion $\mathsf{FS}_{\sigma}$ on~$[-1,1]$.
\item For $L$ with $\pi_\Lambda^0(\fL_0\mbox{ exists})>L^{-10}$, the law of $(Y_n(t))_{n=0}^m$ under $\pi_\Lambda^0(\cdot\mid\fL_0\mbox{ exists})$ converges weakly to that of $m+1$ i.i.d.\ copies of 
$\mathsf{FS}_{\sigma}$.
\item 
     For $L$ with $\pi_\Lambda^0(\fL_0\mbox{ does not exist}) > L^{-10}$, the law of $(Y_n(t))_{n=1}^m$ under $\pi_\Lambda^0(\cdot\mid \fL_0\mbox{ does not exist})$ converges weakly to that of $m$ i.i.d.\ copies of $\mathsf{FS}_{\sigma}$.
\end{enumerate}
The same statements hold for $\bar{Y}_n(t)$ corresponding to $\bar{\rho}_n(x)=\max\{y\leq \frac{L}2:(\frac{L}2+x,y)\in\fL_n\}$.
\end{theorem}

Finally, our proofs hold more generally for the family of $|\nabla\phi|^p$ models for $p \geq 1$, of which the \ZGFF ($p = 2$) and \SOS ($p = 1$) are special cases, given by
\begin{equation}\label{eq:grad-phi-def}
	\pi^{(p),0}_\Lambda(\phi) \propto \exp\Big(-\beta \sum_{x\sim y}|\phi_x-\phi_y|^p\Big)\,.
	\end{equation}
As before, the infinite-volume weak limit $\pi^{(p)}_\infty$ is well-defined, and leads to analogous definitions of $H(L)$ and $N_n$ per \cref{eq:H-def,eq:Nn-def}. Similarly, the surface tension is well-defined, leading to analogous definitions of the Wulff shape $\cW$ from \cref{thm:main-thm-limit-shape} and $\lambda_*$ from \cref{thm:main-thm-crit-window}.
\begin{theorem}\label{thm:grad-phi-p}
There exist absolute constants $\delta_p, \delta, c >0$ such that the following hold. The statements of \cref{thm:main-thm-limit-shape,thm:main-thm-crit-window,thm:FS-no-exceptional} extend to the $|\nabla\phi|^p$ model for any fixed $p>1$: \cref{thm:main-thm-crit-window,thm:FS-no-exceptional} extend verbatim and \cref{thm:main-thm-limit-shape} extends after modifying the Hausdorff distance bound to 
\[ d_{\cH}(\fL_n,\sL_n) < \begin{cases} N_n^{1-\delta_p} & \mbox{if }1 < p < 2\,,\\ N_ne^{-c\sqrt{\beta\log L}} & \mbox{if }p > 2\,.\end{cases}\] Finally, the statements of \cref{thm:main-thm-limit-shape,thm:main-thm-crit-window} also extend to $p=1$ (\SOS): \cref{thm:main-thm-crit-window} extends verbatim and \cref{thm:main-thm-limit-shape} extends after modifying the bound on the Hausdorff distance to $N_n^{1-\delta}$.
\end{theorem}
For \SOS, the previous work \cite{CLMST16} only considered $L\notin \bigcup_h (1\pm\epsilon) L_c^{(h)}$, avoiding the near-critical (``exceptional'') side lengths.
The extension of the new results to \SOS separates the question of establishing a limit shape from that of determining whether the top level is at $H$ or $H+1$. For the former, the extension of \cref{thm:main-thm-limit-shape} establishes the \SOS limit shape for all $L$ (no exceptional side lengths); for the latter, the extension of \cref{thm:main-thm-crit-window} places the transition point in a window of width $L^{1/2+o(1)}$ as opposed to just $o(L)$. 
Note that \Cref{thm:FS-no-exceptional} does not extend to $p = 1$, and as discussed above, it is not expected to (the conjectured local limit law for \SOS is \emph{not} Ferrari--Spohn).

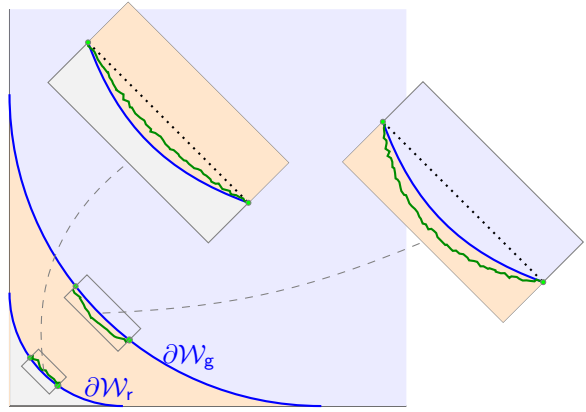
\begin{figure}
    \centering
    \vspace{-0.1in}
    \begin{tikzpicture}[scale=0.75]
    
    \begin{scope}
        \draw[color=black, fill=gray!10] (7,0)--(0,0)--(0,7);

        \pgfmathsetmacro{\innerarc}{5.5};
        \pgfmathsetmacro{\outerarc}{2};
        \fill[orange!20, even odd rule]
            (\innerarc,0) arc[radius=\innerarc,start angle=270,end angle=180] -- (0,0) -- cycle
            (\outerarc,0) arc[radius=\outerarc,start angle=270,end angle=180] -- (0,0) -- cycle;

        \fill[blue!8] (7,0)--(\innerarc,0) arc[radius=\innerarc,start angle=270,end angle=180] -- (0,7) -- (7,7) -- cycle;
        \draw[thick,draw=blue]
            (\innerarc,0) arc[radius=\innerarc,start angle=270,end angle=180];

        \draw[thick,draw=blue]
            (\outerarc,0) arc[radius=\outerarc,start angle=270,end angle=180];

        \coordinate (A1) at ($(\innerarc,\innerarc)+(232:\innerarc)$);
        \coordinate (B1) at ($(\innerarc,\innerarc)+(218:\innerarc)$);
        \coordinate (M1) at ($(A1)!0.5!(B1)$);
        \node[circle,scale=0.22,draw=gray,fill=green] at (A1) {};
        \node[circle,scale=0.18,draw=gray,fill=green] at (B1) {};
        \pgfmathsetseed{271828}
        \draw[thick,draw=green!55!black] ($(A1)+(-0.05,-0.05)$)
            decorate[decoration={random steps,segment length=0.6mm,amplitude=0.3mm}]{to[bend left=15]($(B1)+(-0.05,-0.05)$)};

        \coordinate (A2) at ($(\outerarc,\outerarc)+(235:\outerarc)$);
        \coordinate (B2) at ($(\outerarc,\outerarc)+(215:\outerarc)$);
        \coordinate (M2) at ($(A2)!0.5!(B2)$);
        \node[circle,scale=0.22,draw=gray,fill=green] at (A2) {};
        \node[circle,scale=0.18,draw=gray,fill=green] at (B2) {};
        \pgfmathsetseed{16180}
        \draw[thick,draw=green!55!black] ($(A2)+(0.05,0.05)$)
            decorate[decoration={random steps,segment length=0.6mm,amplitude=0.3mm}]{to[bend left=15]($(B2)+(0.05,0.05)$)};
    \end{scope}

    \draw[gray] ($(A1)+(-0.2,-0.2)$) -- ($(B1)+(-0.2,-0.2)$) -- ($(B1)+(0.2,0.2)$) -- ($(A1)+(0.2,0.2)$) -- cycle;

    \draw[gray] ($(A2)+(-0.15,-0.15)$) -- ($(B2)+(-0.15,-0.15)$) -- ($(B2)+(0.15,0.15)$) -- ($(A2)+(0.15,0.15)$) -- cycle;

    \node[color=blue,font=\small] at ($(B1)+(2,-1.25)$) {$\partial\cW_{\mathsf{g}}$};

    \node[color=blue,font=\small] at ($(B2)+(1.4,-0.55)$) {$\partial \cW_{\mathsf{r}}$}; 

    \begin{scope}[shift={(2.8,5)},scale=2, rotate=-45]
        \coordinate (Az1) at (-1,0);
        \coordinate (Bz1) at ( 1,0);
        \filldraw [fill=gray!10, draw=gray] (-1,-0.5) rectangle (1,0.5);
        \coordinate (r1) at (0,-0.5);

        \fill[orange!20] (Az1) to[bend right=25] (Bz1) --(1,0.5)--(-1,0.5)--(Az1);
        
        \draw[thick,dotted] (Az1) -- (Bz1);
        \draw[thick, draw=blue] (Az1) to[bend right=25] (Bz1);
        \pgfmathsetseed{271828}
        \draw[thick, draw=green!55!black] (Az1)
            decorate[decoration={random steps,segment length=0.5mm,amplitude=0.25mm}]{to[bend right=15](Bz1)};
        \node[circle,scale=0.18,draw=gray,fill=green] at (Az1) {};
        \node[circle,scale=0.18,draw=gray,fill=green] at (Bz1) {};
        
    \end{scope}

    \begin{scope}[shift={(8,3.6)},scale=2, rotate=-45]
        \coordinate (Az2) at (-1,0);
        \coordinate (Bz2) at (1,0);
        \draw (Az2) -- (Bz2);
        \filldraw [fill=blue!8, draw=gray] (-1,-0.5) rectangle (1,0.5);
        \fill[orange!20] (Az2) to[bend right=25] (Bz2) --(1,-0.5)--(-1,-0.5)--(Az2);
        
        \coordinate (r2) at (0,-0.5);
        \draw[thick, dotted] (Az2) -- (Bz2);
        \draw[thick, draw=blue] (Az2) to[bend right=25] (Bz2);
        
        \pgfmathsetseed{16180}
        \draw[thick, draw=green!55!black] (Az2)
            decorate[decoration={random steps,segment length=0.5mm,amplitude=0.25mm}]{to[bend right=40](Bz2)};
        \node[circle,scale=0.18,draw=gray,fill=green] at (Az2) {};
        \node[circle,scale=0.18,draw=gray,fill=green] at (Bz2) {};
        
    \end{scope}

    \draw[dashed, gray] (M2) to[bend left=30] (r1);
    \draw[dashed, gray] (M1) to[bend right=10] (r2);
    \end{tikzpicture}
    \vspace{-0.1in}
    \caption{
Identifying the size $\ell$ of the Wulff shape featured in $\fL_n$ near the corners of box: A smaller $\ell$ leads to a bigger area of the limit shape. 
A retreat mechanism compares the Wulff boundary $\partial \cW_{\mathsf{r}}$ (blue) to the location of an area-tilted random walk (green). When the random walk does not drop past the Wulff boundary, neither will $\fL_n$. Similarly, a growth mechanism says when the random walk stays below the Wulff boundary $\partial \cW_{\mathsf{g}}$, so will $\fL_n$. As a result, we can sandwich the corner of $\fL_n$ between two arcs of Wulff shapes.
    }
    \label{fig:hausdorff-dist-1}
    \vspace{-0.1in}
\end{figure}

\subsection{Proof ideas}\label{sec:proof-ideas}
In what follows, we outline the approach used to prove the new results.

\subsubsection*{The limit shape}
    The level lines of the \ZGFF above a floor feature a tradeoff between an energy cost for the size of the level line and an entropic reward for its interior area. Variational problems of this type have been well-studied, and the optimal shape in a square domain $\Lambda$ is precisely the shape $\sL_n$ from \cref{thm:main-thm-limit-shape}, featuring flat sides and four quadrants of a Wulff shape (defined via the surface tension associated to the model) at the corners (see, e.g., \cite{schonmann1996constrained}). To identify the limit shape, one must determine the size $\ell$ of the Wulff shape featured at the four corners (see \cref{fig:hausdorff-dist-1}).

 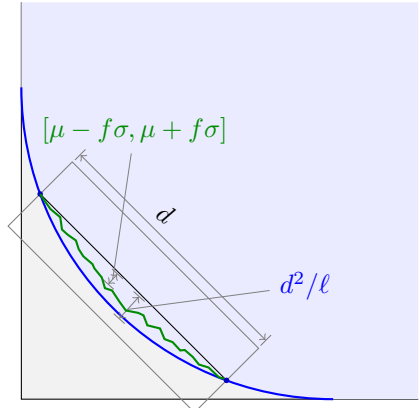
\begin{figure}
 \vspace{-0.1in}
    \begin{tikzpicture}
    \begin{scope}[scale=0.75]

    \pgfmathsetmacro{\innerarc}{5.5};    
    \coordinate (a) at ($(\innerarc,\innerarc)+(200:\innerarc)$);
    \coordinate (b) at ($(\innerarc,\innerarc)+(250:\innerarc)$);
        
    \draw [color=black,fill=gray!10] (7,0)--(0,0)--(0,7);

    \coordinate (x) at (1.53,1.52);

    \fill[blue!8] (7,0)--(\innerarc,0) arc[radius=\innerarc,start angle=270,end angle=180] -- (0,7) -- (7,7) -- cycle;
        \draw[thick,draw=blue]
            (\innerarc,0) arc[radius=\innerarc,start angle=270,end angle=180];
        
    \draw [draw=gray, rotate=-45] ($(a)-(0,0.8)$) rectangle ($(b)+(0,0.8)$);

    \node[circle,scale=0.2,fill=blue!75!black] at (a) {};
    \node[circle,scale=0.2,fill=blue!75!black] (b) at (b) {};

    \draw (a)--(b) node[midway] (m) {};
         
    \draw [|<->|,gray,shorten <= 1.25pt, shorten >= 0.75pt] ($(x)+(0.15,-0.15)$)--coordinate[pos=0.45](dev1) ($(m)+(0.15,-0.15)$);

    \draw [|<->|,gray,shorten <= 7.5pt, shorten >= 0.5pt] ($(x)+(-0.24,0.24)$)--coordinate[pos=0.75] (dev2) ($(m)+(-0.24,0.24)$);

    \draw [|<->|,gray] ($(a)+(0.7,0.7)$)--($(b)+(0.7,0.7)$) node[pos=0.4,above,rotate=-35,text=black,font=\small] {$d$};

    \pgfmathsetseed{314159}
    \draw [thick, draw=green!55!black] (a) decorate[decoration={random steps,segment length=1.25mm,amplitude=0.5mm}]{--($(x)+(0.2,0.2)$)} decorate[decoration={random steps,segment length=1.25mm,amplitude=0.5mm}]{-- (b)};

    \node [text=blue,font=\small] (dev1label) at (5,2) {$d^2/\ell$};

    \node[text=green!55!black,font=\small] (dev2label) at (2,4.75) {$[\mu-f\sigma, \mu+f\sigma]$};

    \draw[very thin, gray] (dev1label)--(dev1);

    \draw[very thin, gray] (dev2label)--(dev2);
    
        \end{scope}
    \end{tikzpicture}
    \vspace{-0.15in}
    \caption{Take a chord of length $d = N_n^{2/3}f$ for some $f = f(L) \to \infty$ as $L \to \infty$. The Wulff shape of length $\ell$ dips $\asymp d^2/\ell$ below the midpoint of the chord, while the area-tilted random walk dips to mean $\mu \asymp N_n^{1/3} f^2$ with fluctuations $\sigma = N_n^{1/3}\sqrt{f} = o(\mu)$. For $\cW_{\mathsf{g}},\cW_{\mathsf{r}}$
    depicted in \cref{fig:hausdorff-dist-1}, the optimal choices of their sizes $\ell_{\mathsf{g}},\ell_{\mathsf{r}}$ satisfy $d^2/\ell_{\mathsf{g}} = \mu-f\sigma$ and $d^2/\ell_{\mathsf{r}} = \mu+f\sigma$.}
    \label{fig:hausdorff-dist-3}
    \vspace{-0.1in}
\end{figure}
    The following heuristic was at the heart of \cite{CLMST16}, both answering the question above on the scale of the Wulff shape, and yielding the $L^{1/3+o(1)}$ upper bound on the fluctuations at the flat portions of the \SOS level-line limit.
    Suppose we know that locally, between two points $A, B$ distance $d$ apart, the $(H-n)$ level line $\fL_n$ can be coupled to a random walk with area tilt of strength $1/N_n$. Without the area tilt, the normalized area above $\fL_n$ is $O(d^{3/2})$. Hence when $d^{3/2} \gg N_n$, the area tilt will push $\fL_n$ downwards, so that the mean distance $\mu$ of $\fL_n$ from $\overline{AB}$ is of a larger order than its fluctuations $\sigma$. Explicitly, if $d = N_n^{2/3}f$ for $f = f(L) \to \infty$ as $L \to \infty$, then $\mu = O(N_n^{1/3}f^2)$ and $\sigma = O(N_n^{1/3}f^{1/2})$. Hence, the larger $d$ is, the smaller $\sigma/\mu$ will be.

    Now place $A, B$ somewhere along the bottom left boundary of a Wulff shape $\cW$ of length $\ell$, as in \cref{fig:hausdorff-dist-3}. The distance between $\partial \cW$ and $\overline{AB}$ in the middle is $\asymp d^2/\ell$. Hence, if $d^2/\ell > \mu+f\sigma$, then even if at $A, B$ the level line $\fL_n$ reaches all the way up to $\partial \cW$, it still will stay above $\partial \cW$ in the middle. Hence, such a Wulff shape overestimates how close $\fL_n$ is to the corner of $\Lambda$, and we can retreat the Wulff shape away from the corners by taking a larger and larger $\ell$ until we reach $\cW_{\mathsf{r}}$ with radius $\ell_{\mathsf{r}} := d^2/(\mu+f\sigma)$. Similarly, we can grow any Wulff shape of length $\ell$ such that $d^2/\ell < \mu-f\sigma$ towards the corner, until we reach $\cW_{\mathsf{g}}$ with radius $\ell_{\mathsf{g}} := d^2/(\mu - f\sigma)$. 

    The growth procedure was carried out rigorously in \cite{ChenLubetzky25} for the \ZGFF, giving ``half the limit.'' However, a serious obstacle stood in the way of deriving a matching retreat gadget. In the next subsection we explain this obstacle in more detail, as well as our new approach to circumvent it.
    
    Provided a matching retreat gadget can be established, one would find, as depicted in \cref{fig:hausdorff-dist-1}, that~$\fL_n$ is sandwiched between two arcs of $\partial \cW_{\mathsf{g}}$ and $\partial \cW_{\mathsf{r}}$, which gives the limit shape if the Hausdorff distance $X$ between the two arcs is $o(L)$. To compute the distance $X$, as per \cref{fig:hausdorff-dist-2} we see that $X = O(\ell_{\mathsf{g}} - \ell_{\mathsf{r}})
    $. Since $\sigma = o(\mu)$, up to first order, we have $\ell_{\mathsf{g}} = O(d^2/\mu) = O(N_n)$, so that $X = O((1-\frac{\ell_{\mathsf{r}}}{\ell_{\mathsf{g}}})N_n)$. Then, the expressions for $d^2/\ell_i$ give $\frac{\ell_{\mathsf{r}}}{\ell_{\mathsf{g}}} = \frac{\mu-f\sigma}{\mu+f\sigma} = 1 - O(\frac{f\sigma}{\mu})$. Combining these, $X = O(\frac{f\sigma}{\mu}N_n) = o(N_n)$, and showing $\frac{\sigma}\mu < e^{-\sqrt{\log L}}$ leads to the bounds of \cref{thm:main-thm-limit-shape}.

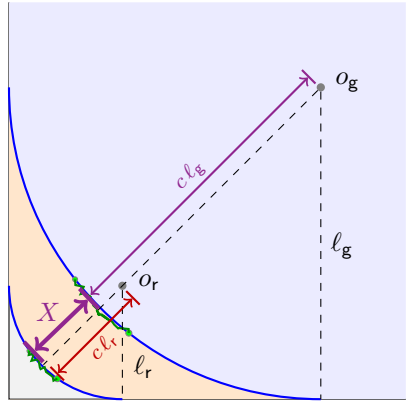
\begin{figure}
    \centering
    \vspace{-0.1in}
    \begin{tikzpicture}[scale=0.75,font=\small]
        
        \pgfmathsetmacro{\innerarc}{5.5};
        \pgfmathsetmacro{\outerarc}{2};
    
    \begin{scope}
        \draw[color=black, fill=gray!10] (7,0)--(0,0)--(0,7);

        \fill[orange!20, even odd rule]
            (\innerarc,0) arc[radius=\innerarc,start angle=270,end angle=180] -- (0,0) -- cycle
            (\outerarc,0) arc[radius=\outerarc,start angle=270,end angle=180] -- (0,0) -- cycle;

        \fill[blue!8] (7,0)--(\innerarc,0) arc[radius=\innerarc,start angle=270,end angle=180] -- (0,7) -- (7,7) -- cycle;
        \draw[thick,draw=blue]
            (\innerarc,0) arc[radius=\innerarc,start angle=270,end angle=180];
        \draw[thick,draw=blue]
            (\outerarc,0) arc[radius=\outerarc,start angle=270,end angle=180];

        \coordinate (A1) at ($(\innerarc,\innerarc)+(232:\innerarc)$);
        \coordinate (B1) at ($(\innerarc,\innerarc)+(218:\innerarc)$);
        \coordinate (M1) at ($(A1)!0.5!(B1)$);
        \node[circle,scale=0.22,draw=gray,fill=green] at (A1) {};
        \node[circle,scale=0.18,draw=gray,fill=green] at (B1) {};
        \pgfmathsetseed{271828}
        \draw[thick, draw=green!55!black] (A1)
            decorate[decoration={random steps,segment length=0.5mm,amplitude=0.5mm}]{to[bend left=15](B1)};

        \coordinate (A2) at ($(\outerarc,\outerarc)+(235:\outerarc)$);
        \coordinate (B2) at ($(\outerarc,\outerarc)+(215:\outerarc)$);
        \coordinate (M2) at ($(A2)!0.5!(B2)$);
        \node[circle,scale=0.22,draw=gray,fill=green] at (A2) {};
        \node[circle,scale=0.18,draw=gray,fill=green] at (B2) {};
        \pgfmathsetseed{16180}
        \draw[thick, draw=green!55!black] (A2)
            decorate[decoration={random steps,segment length=0.5mm,amplitude=0.5mm}]{to[bend left=15](B2)};
    \end{scope}

    \node[circle,scale=0.3,fill=gray,label={[yshift=-7pt, xshift=10pt]$o_{\textsf{g}}$} ] (o1) at (\innerarc,\innerarc) {};

    \draw[dashed,black] (o1)--(\innerarc,0) node[midway,xshift=8pt] () {$\ell_{\mathsf{g}}$};

    \node[circle,scale=0.3,fill=gray,label={[yshift=-7pt, xshift=10pt]$o_{\textsf{r}}$} ] (o2) at (\outerarc,\outerarc) {};

    \draw[dashed,black] (o2)--(\outerarc,0) node[pos=0.7,xshift=8pt] () {$\ell_{\mathsf{r}}$};

    \draw[dashed] (o1)--(M2);

    \draw[thick,|<->,magenta!50!blue] ($(o1)+(-0.2,0.2)$)--($(M1)+(-0.2,0.2)$) node[midway,xshift=-4pt,yshift=4pt,rotate=45,font=\scriptsize] () {$c\, \ell_{\mathsf{g}}$};

    \draw[ultra thick,|<->|,magenta!50!blue] ($(M1)+(-0.2,0.16)$)--($(M2)+(-0.2,0.2)$) node[midway,xshift=-5pt,yshift=5pt] () {$X$};

    \draw[thick,|<->|,red!75!black] ($(o2)+(0.2,-0.2)$)--($(M2)+(0.23,-0.17)$) node[midway,xshift=3pt,yshift=-4pt,rotate=45,font=\scriptsize] () {$c\, \ell_{\mathsf{r}}$};

    \end{tikzpicture}
    \vspace{-0.1in}
    \caption{The Hausdorff distance $X$ between the Wulff shapes $\cW_{\mathsf{g}}$ and $\cW_{\mathsf{r}}$ driven by the growth and retreat mechanisms, respectively. If $\cW_{\mathsf{g}},\cW_{\mathsf{r}}$ have sizes $\ell_{\mathsf{g}}, \ell_{\mathsf{r}}$, then for some constant $c=c(\beta)$ with $1 < c < \sqrt{2}$ we have $X = (\ell_{\mathsf{g}} - \ell_{\mathsf{r}})(\sqrt{2}-c)$. For an optimal choice of $\ell_{\mathsf{g}},\ell_{\mathsf{r}}$ via the mechanism depicted in \cref{fig:hausdorff-dist-3}, one gets $X \asymp N_n\frac {f\sigma}\mu = N_nf^{-1/2} = o(N_n)$.}
    \label{fig:hausdorff-dist-2}
    \vspace{-0.1in}
\end{figure}

\subsubsection*{Area-tilted random walk representation on larger domains} As mentioned above, it is crucial to show that locally (i.e., in a mesoscopic box), $\fL_n$ behaves like a random walk with an area tilt of strength $1/N_n$. The main challenge of this paper is to extend this picture to larger scales, even up to the macroscopic scale. This directly depends on estimating the probability that $\phi_x \geq 0$ for all $x$ inside a domain $D$ with boundary conditions $H+1-n$. In the low temperature setting, decorrelation estimates imply that this floor event is realized nearly independently for each $x$, so \begin{align}\label{eq:pf-ideas-estimate}\hatpi_D^{H+1-n}(\phi_x \geq 0,\forall x \in D) \approx \prod_{x \in D}\hatpi_D^{H+1-n}(\phi_x \geq 0)\,.
\end{align}However, the error in this approximation grows with the size of the domain $|D|$, and the error produced by the standard inclusion-exclusion analysis becomes too large once $|D| \geq L^{1+o(1)}$.

From the above discussion on the limit shape, we need to exhibit this random walk coupling in a box $R$ of width $d\gg N_n^{2/3}$, and to avoid pinning issues we require the height of $R$ to be at least the same order (stemming from the fact that the pinning effect can be  $\exp(c d)$, which, if the height is $o(d)$, would overwhelm the large deviation probability of the random walk). 
This leads to a minimum area of $|D| \gg N_n^{4/3}$, far exceeding the threshold $L^{1+o(1)}$ allowed by prior arguments. For comparison:
\begin{enumerate}
    \item the same requirement of $|D|\gg N_n^{4/3}$ would have been needed in \cite{ChenLubetzky25} were it not for fortuitous monotonicity arguments: by FKG, there we could relax the floor constraints on some of the area (reducing $|D|$ down to $L^{1+o(1)}$) in the context of a growth gadget, but this would go in the wrong direction in the context of a retreat gadget.
\item In the \SOS model, the error in \cref{eq:pf-ideas-estimate} is much smaller, leading to a less restrictive domain requirement of $|D| \geq L^{3/2}$, easily allowing for $D$ with $|D|\approx L^{4/3}$ in~\cite{CLMST16,CKL24}.
\end{enumerate}
In fact, not only do we need to allow $|D| \approx L^{4/3}$, but to study the critical window about $L_c^{(h)}$ where the $h$ level line first appears, we will need to treat $|D|$ all the way up to $O(L^2)$. 

To this end, we need a finer analysis of the probability $\hatpi_D^{H-n}(\phi_x \geq 0,\forall x \in D)$. Rather than attempt to reduce the error in \cref{eq:pf-ideas-estimate}, we take a different approach and say that the floor event occurs nearly independently across (but not within) mesoscopic $\ell \times \ell$ boxes. (A choice of $\ell =1$ corresponds to the old approach described above.) That is, we express the above probability via
\[\xi_{\ell,h} := -\frac{1}{\ell^2}\log\hatpi_\infty\left( \phi_x \geq -h
    ,\,\forall x\in Q_{\ell} := \llb 1,\ell\rrb^2 
    \right)\,,\]
for an appropriately chosen $\ell = \ell_* \asymp \sqrt L$ (see \cref{thm:key-area} for the precise statement). Handling mesoscopic vs.\ macroscopic domains $D$ calls for different choices of $\ell$ to control the associated errors. We do so by first obtaining a generic bound applicable to a range of $\ell$'s (\cref{prop:key-area-estimate}), and thereafter showing that $\xi_{\ell,h}$ does not change much, so we can use the specific $\xi_{\ell_*,h}$ across these various domains $D$. 
From there, we conclude that the area tilt felt by $\fL_n$ on larger scales is given by $\uprho_n/N_n$ for a new parameter $\uprho_n = 1-o(1)$, vs.\ the old approach of $\ell=1$ that had a tilt of $1/N_n$. (See below for a further discussion on $\uprho_n$ and its effect on the location of $L_c^{(h)}$.)

\subsubsection*{Discontinuity of the transition} Once the key estimate of \cref{thm:key-area} is obtained, the proof of the limit shape of the level lines for $L$ away from $L_c^{(h)}$, as well as a preliminary estimate on $L_c^{(h)}$, follows by combining the methods used for the \SOS model in \cite{CLMST16} with the disagreement polymer machinery established in \cite{ChenLubetzky25} to treat the \ZGFF (and the $|\nabla\phi|^p$ for $p > 1$). We push these results further in two directions (new even for \SOS as far as the critical window and global limit are concerned):
\begin{enumerate}[(i)]
    \item We tighten the estimate on the location of $L_c^{(h)}$ to a near square root window. This involves a quantitative estimate on a functional $\cF$ on loops, which captures a tradeoff between the loop length and interior area (see \cref{lem:linear-critical-window}). \item \label{it:extend-global-local-no-window} We extend both the global limit shape and local scaling limit to all $L$, no longer excluding a window around $L_c^{(h)}$. For $L$ close to $L_c^{(h)}$, we do not know if the $h$ level line exists, so we must work on the measures $\pi_L(\cdot \mid \text{$h$ level line exists})$ and $\pi_L(\cdot \mid \text{$h$ level line does not exist})$. These conditional measures no longer have FKG, so the previous proofs need to be modified so as to be less reliant on monotonicity.

\end{enumerate}
Our analysis in \cref{it:extend-global-local-no-window} implies that the moment the $h$ level line appears, it attains the limit shape $\sL_n$ occupying nearly the whole box $\Lambda$, and features Ferrari--Spohn fluctuations along the flat sides of $\sL_n$. In particular, there is no critical phase around $L_c^{(h)}$ where the $h$ level line only occupies, say, a third of the space. Roughly speaking, this is because if the area inside an $h$ level line $\fL$ is smaller than a threshold $C_1L^2$, it cannot exist as the gain from the interior of $\fL$ is too small compared to the gradient cost. At the same time, if the interior area of $\fL$ is larger than another threshold $C_2L^2$, then a growth procedure shows that $\fL$ must contain the entire aforementioned limit shape. (More precisely, these thresholds relate not just to the interior area but also to the shape contained by $\fL$. See \cref{lem:prelim-ll-bounds} and the definition of $\cE_{x, \ell}$ in \cref{clm:crit-window-growth} for details.) The dichotomy on the $h$ level line is due to the fact that these thresholds satisfy $C_1 > C_2$. (In fact, $C_2 \leq \frac14+\epsilon_\beta$, pertaining to the level line $\fL_0$ encompassing a Wulff shape of diameter $\approx \frac12 \sfw_1 N_0/L $,
whereas $C_1 \geq 1-\epsilon_\beta$.) 

\subsubsection*{Analysis of the area tilt $\uprho_n$} 
As we mentioned above, the area tilt $\uprho_n/N_n$ as opposed to $1/N_n$ was a result of the new approach of controlling \cref{eq:pf-ideas-estimate} via the tail for the minimum in a mesoscopic square of side length $\sqrt{L}$, as opposed to the singleton rate $\hatpi_\infty(\phi_o \geq -h)$. Initially, the authors
viewed this strategy as merely a technical workaround for the large error associated with \cref{eq:pf-ideas-estimate}. 
However, as it turns out, $ (1-\uprho_n)L \geq L^{1-o(1)}$. In view of the upper bound $L^{1/2+o(1)}$ on the critical window, this positions the original prediction for $L_c^{(h)}$ (using an area tilt of $1/N_0$) in the subcritical regime, where the probability of a \texttt{large} $h$ level line is $o(1)$. We further extended this separation of the a-priori prediction for $L_c^{(h)}$ from the critical window to the $|\nabla\phi|^p$ model for all $1 < p \leq 2$.

Along the way, we prove two results which are essential for the estimates on $\uprho_n$, and may be of independent interest.
Firstly, for the \ZGFF ($p = 2$), we show that $\hatpi_\infty(\cdot \mid \phi_x = h, \forall x \in B_r(o))$ has a weak form of rigidity about the discrete harmonic function $\phi^*$ which is $h$ on $B_r(o)$ and 0 on $\partial B_R(o)$ for large~$R$. In the $\R$-valued setting---the discrete \GFF---one can look at $\tilde\phi =  \phi - \phi^*$, which is nothing but a discrete \GFF with $0$ boundary conditions. If the same were true in the \ZGFF, one would be able to infer rigidity of $\phi$ about $\phi^*$ for large $\beta$. Unfortunately, 
in the $\Z$-valued setting, said $\tilde\phi$ becomes a field taking values in $\{a_i + \Z\}_{i \in \Lambda}$ for some shifts $a_i \in [0, 1)$. Since the fractional parts $a_i$ are complicated (the complement of those of $\phi^*$), they could destroy the full rigidity of~$\tilde\phi$ (see, e.g.,~\cite{garban2023statistical}, where such fields are studied in detail). Nevertheless, we are able to show rigidity beyond height $h/\log h$ (see \cref{prop:pi-y-h-phi^*-rigid}) by controlling the corresponding integer rounding errors. Secondly, for $1 < p < 2$, we show (see \cref{lem:LD-capacity}) that for any finite set $A$, the large deviation rate function for the event that $\phi\restriction_A = h$ is given by the $p$-capacity of $A$.

\subsection{Organization/Reader's guide}
The paper is organized as follows.
In \cref{sec:prelim} we list the required preliminaries, mostly focusing on disagreement polymers and the Wulff shape. \cref{sec:xi-bbq} proves estimates on the probability of the floor event, leading to the area prefactor of $\uprho_n/N_n$. \cref{sec:refined-CE} uses these area estimates and cluster expansion to obtain the law of the disagreement polymers. \cref{sec:transition-window} proves \cref{thm:main-thm-crit-window}, showing the critical window has width $\leq L^{1/2+o(1)}$. \cref{sec:limit-shape} proves \cref{thm:main-thm-limit-shape}, showing the limit shape of the level lines is a translation of Wulff shapes. \cref{sec:limit-law} proves \cref{thm:FS-no-exceptional}, extending the Ferrari--Spohn limit law to all $L$. \cref{sec:extend-to-p} proves \cref{thm:grad-phi-p}, extending the main results to the $|\nabla\phi|^p$ model for $p \geq 1$.

Several of the main results are of fairly different flavors---targeting different phenomena in the \ZGFF and featuring proofs that build on different methods. We include the following guide for the benefit of a reader interested in one of these results in particular.
\begin{itemize}
    \item For the proof of \cref{thm:main-thm-limit-shape,thm:FS-no-exceptional} and the resulting discontinuous transition, one should read  \cref{sec:refined-CE} for the estimates on the polymer law, and then read \cref{sec:limit-shape,sec:limit-law}.

    \item  For the proof of \cref{thm:main-thm-crit-window} that the critical window has size at most $L^{1/2+o(1)}$, one should read particularly \cref{lem:prob-of-contour,lem:linear-critical-window} from the preliminaries, and then \cref{sec:refined-CE,sec:transition-window}.

    \item For the discussion on the new area tilt of $\uprho_n/N_n$ vs.\ $1/N_n$ and its effect on the value of $L_c^{(h)}$, one should read \cref{sec:xi-bbq} and then the start of \cref{sec:transition-window} (in particular \cref{rem:bad-prediction}). Thereafter, one should read \cref{sec:extend-to-p} to see how this varies for different $p$.  

    \item For the extension of the area-tilted random walk representation to larger domains, one should read \cref{sec:xi-bbq}, and may skip \cref{sec:sharp-UB-uprho} (the effect of a $\uprho_n/N_n$ vs.\ $1/N_n$ area term).

    \item For a proof that the \ZGFF conditioned on $\phi_o = h$---or more generally, conditioned on $\phi_x=h$ for all $x\in B_r(o)$, a ball  around the origin---is rigid beyond scale $h/\log h$ about the harmonic function solving the corresponding $\R$-valued Dirichlet problem, see \cref{prop:pi-y-h-phi^*-rigid}.
    
    \item For the extension to the $|\nabla\phi|^p$ models for $1 \leq p < \infty$, see \cref{sec:extend-to-p}.
\end{itemize}

\subsection{Open problems}\label{sec:open-problems}
The results in \cref{thm:main-thm-limit-shape,thm:FS-no-exceptional} yield, for all $L$, the global limit shape of the \ZGFF measure $\pi_\Lambda^0$ and the Ferrari--Spohn law of the local fluctuations around said limit in its flat portion. It thus remains to establish the local limit law of the level lines about the corner of the limit shape. The long-standing conjecture (for \ZGFF as well as other related models) is that the fluctuations $O(L^{1/2})$, and moreover the scaling limit should be Brownian motion. For versions of this conjecture see, e.g., \cite[Eq.~(3.5) and the paragraph below it]{SchnmannShlosman95} in the 2\Dim Ising model with an external field, \cite[\S1.2]{CLMST16} in \SOS and \cite[\S1.5]{LMS16} in \ZGFF.
 
In light of \cref{thm:main-thm-crit-window}, it would also be interesting to establish the correct order of the critical window for the \ZGFF, or more generally for the $|\nabla\phi|^p$ models at any $p\geq 1$. Improving our upper bound of $L^{1/2+o(1)}$ on the width to $O(\sqrt L)$ for \ZGFF or \SOS would seem to require a  new idea to overcome the $L^{1/2+o(1)}$ error in the law of the level lines in a macroscopic scale (see \cref{cor:CE-macro-uniformBC-smallL}). Any lower bound on the width of critical window would be interesting.

\section{Preliminaries}\label{sec:prelim}
In this section, we will define level lines and disagreement polymers for height functions, review prior results needed for this paper, and set up general notation. We will consider various subgraphs of $\Z^2$, as well as dual-edges (referred to as bonds) of $(\Z^2)^*$. If $\Lambda$ is a subgraph of $\Z^2$ and $\gamma$ is a collection of bonds, then $\Lambda \setminus \gamma$ refers to the result of removing from $\Lambda$ every edge for which its dual-edge is in $\gamma$. For any point $u\in \R^2$, we will denote its $x, y$ coordinates by $u_1, u_2$ respectively. For a set of vertices $U \subset \Z^2$, let $\partial U$ be the boundary bonds of $U$, i.e., the set of bonds dual to $uv$ with $u\in U$ and $v\notin U$, and let 
$\partialvtx U$ be the external vertex boundary of $U$, i.e., every vertex $v\notin U$ adjacent to some $u\in U$. Occasionally we will refer to two vertices as $*$-adjacent, which will mean that their $L^\infty(\R^2)$ distance is 1. Finally, any object, such as a level line, disagreement polymer or a connected component, will be called \texttt{large} if its size is at least $\log L$. 

\subsection{Level lines and disagreement polymers}
We begin by formally defining the level lines.
\begin{definition}[level lines] \label{def:level-lines}
    Given a height function $\phi$ on a domain $V \subset \Z^2$, its $h$ level lines are defined as the collection of loops obtained from bonds that are dual to nearest-neighbors $x\sim y$ such that $\phi_x < h$ and $\phi_y \geq h$. As standard in the literature, we separate these into a collection of self-avoiding paths and loops via a \textsc{northeast} splitting rule (in case $4$ dual-edges share an endpoint, one splits them apart along the \textsc{northeast} diagonal). 
  For an $h$ level-line loop $\fL$, denote its length (number of bonds) by $|\fL|$, its interior (the finite connected region in $\R^2 \setminus \fL$) by $\Int(\fL)$, and the area of its interior by $A(\fL)=|\Int(\fL)|$. We call $\fL$ an \emph{up-loop} if $\phi_x \geq h$ for every $x\in \Int(\fL)$ adjacent to an edge dual to some $b\in\fL$, and otherwise ($\phi_x <h$ for all such $x$) we call it a \emph{down-loop}.
    \end{definition}
The key object which we will study are not the level lines themselves but rather ``disagreement polymers'' as introduced in \cite{ChenLubetzky25}, which are connected components of dual edges corresponding to height differences. The primary issue with studying an $h$ level line $\fL$ directly is that it induces boundary conditions of $\geq h$ and $\leq h-1$ above and below $\fL$, as opposed to $=h$ and $=h-1$ boundaries. We would like to use a cluster expansion analysis which studies the law of $\fL$ by paying a cost along the length of $\fL$, and then separately analyzing the cost of height configurations on the domains above and below $\fL$. In the case of the \SOS model ($p = 1$), the linear penalty on gradients allows one to split up a gradient cost of $|x-y|$ for $x \geq h$ and $y \leq h-1$ into $|x-h| + |h - (h-1)|+ |(h-1) - y|$. This fits the described framework -- every edge of $\fL$ contributes $|h - (h-1)| = 1$, and then the terms $|x-h|$ and $|(h-1) - y|$ can be analyzed separately. However, this decomposition clearly fails when the cost of the gradient is nonlinear, highlighting why boundaries of $\geq h$ and $\leq h-1$ are difficult to work with for the $|\nabla\phi|^p$ models with $p > 1$.

We next summarize the definition and key properties of disagreement polymers as proven in \cite{ChenLubetzky25}.
	\begin{definition}[Disagreement polymer]\label{def:gamma}
		Let $\phi$ be any $\Z$-valued height function on a domain $V \subset \Z^2$. Associate to each bond $e\in (\Z^2)^*$, dual to some edge $(x,y)\in\Z^2$, the gradient $(\nabla \phi)_e := \phi_x - \phi_y$, where $x$ is taken to be the \textsc{north} vertex if $(x,y)$ is vertical and the \textsc{west} vertex if $(x,y)$ is horizontal. Call $e$ a \emph{disagreement bond} of $\phi$ if $(\nabla\phi)_e\neq 0$. A \emph{disagreement polymer}~$\gamma$ is a (maximal) connected component of the disagreement bonds of $\phi$. For such a polymer, let $D_i$ be the connected components of $V \setminus \gamma$. By construction, within each $D_i$, all the vertices that are $*$-adjacent to $V\setminus D_i$ must have the same height $h_i$ in~$\phi$. Hence, we can view each $\gamma$ as a triple $(\gamma, \{D_i\}, \{h_i\})$, and it makes sense to talk about the components $D_i$ of $\gamma$ and their corresponding heights $h_i$.
	\end{definition}

\begin{figure}
    \centering
    \includegraphics[width=0.48\linewidth]{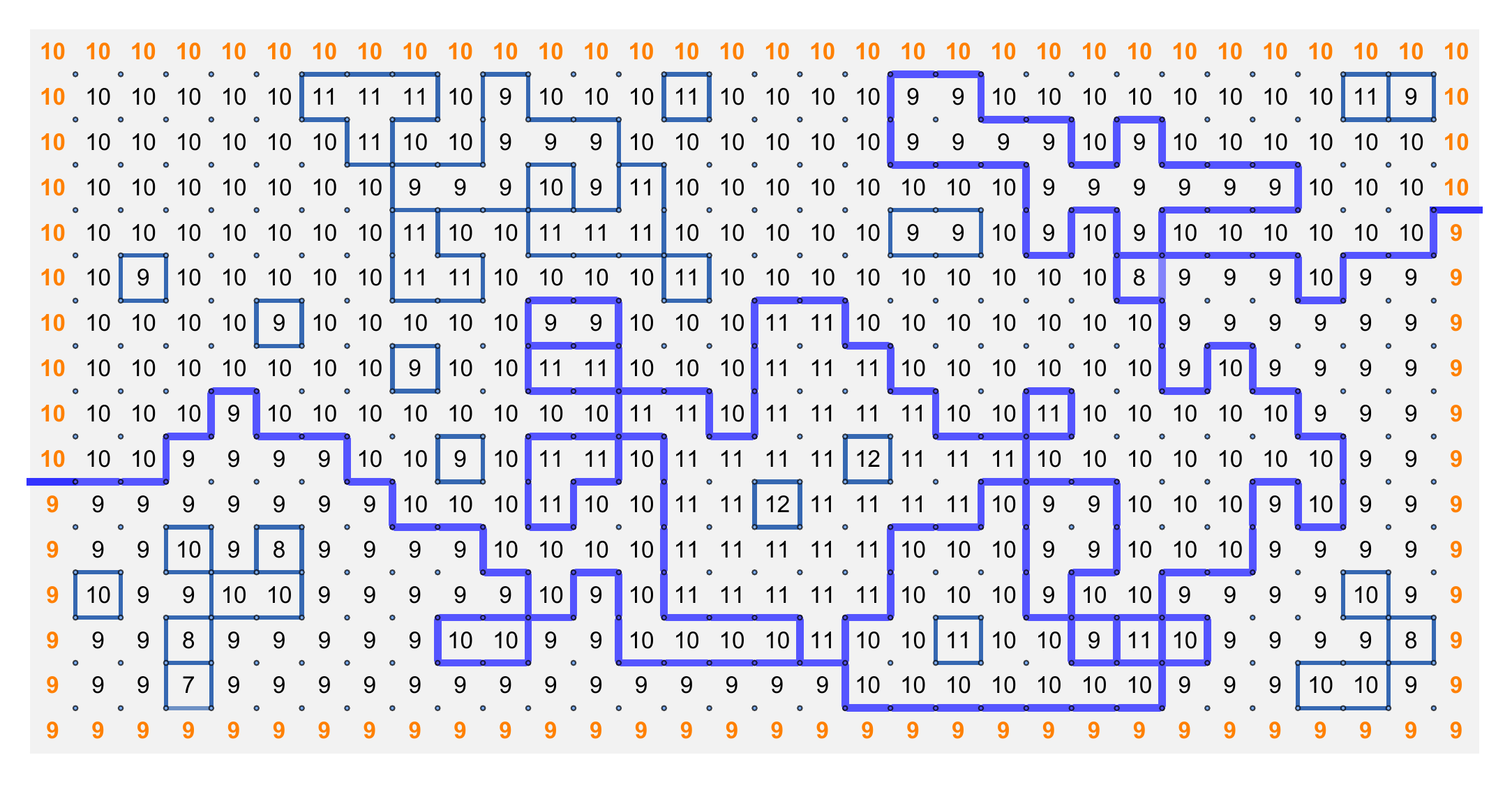}
    \includegraphics[width=0.48\linewidth]{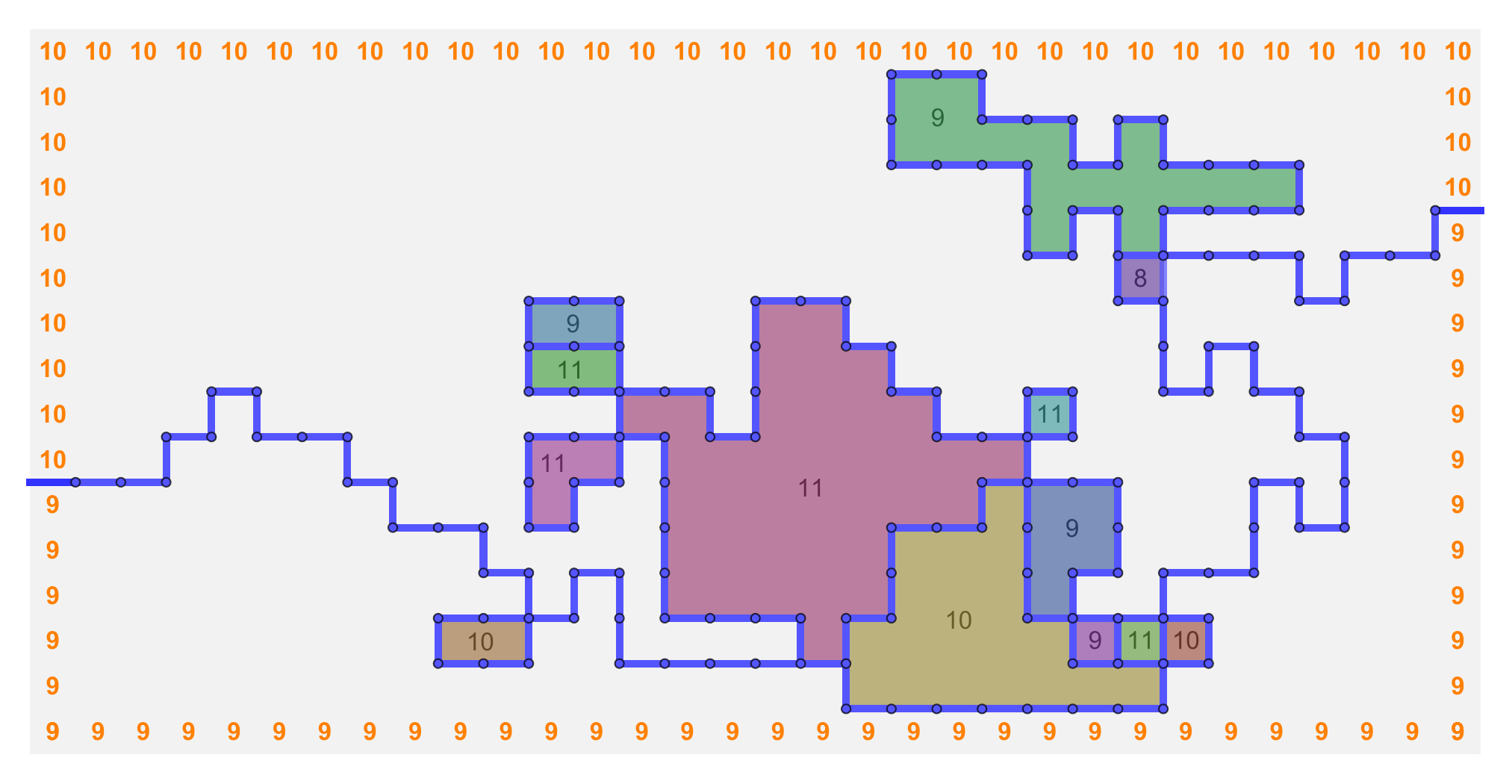}
    \includegraphics[width=0.48\linewidth]{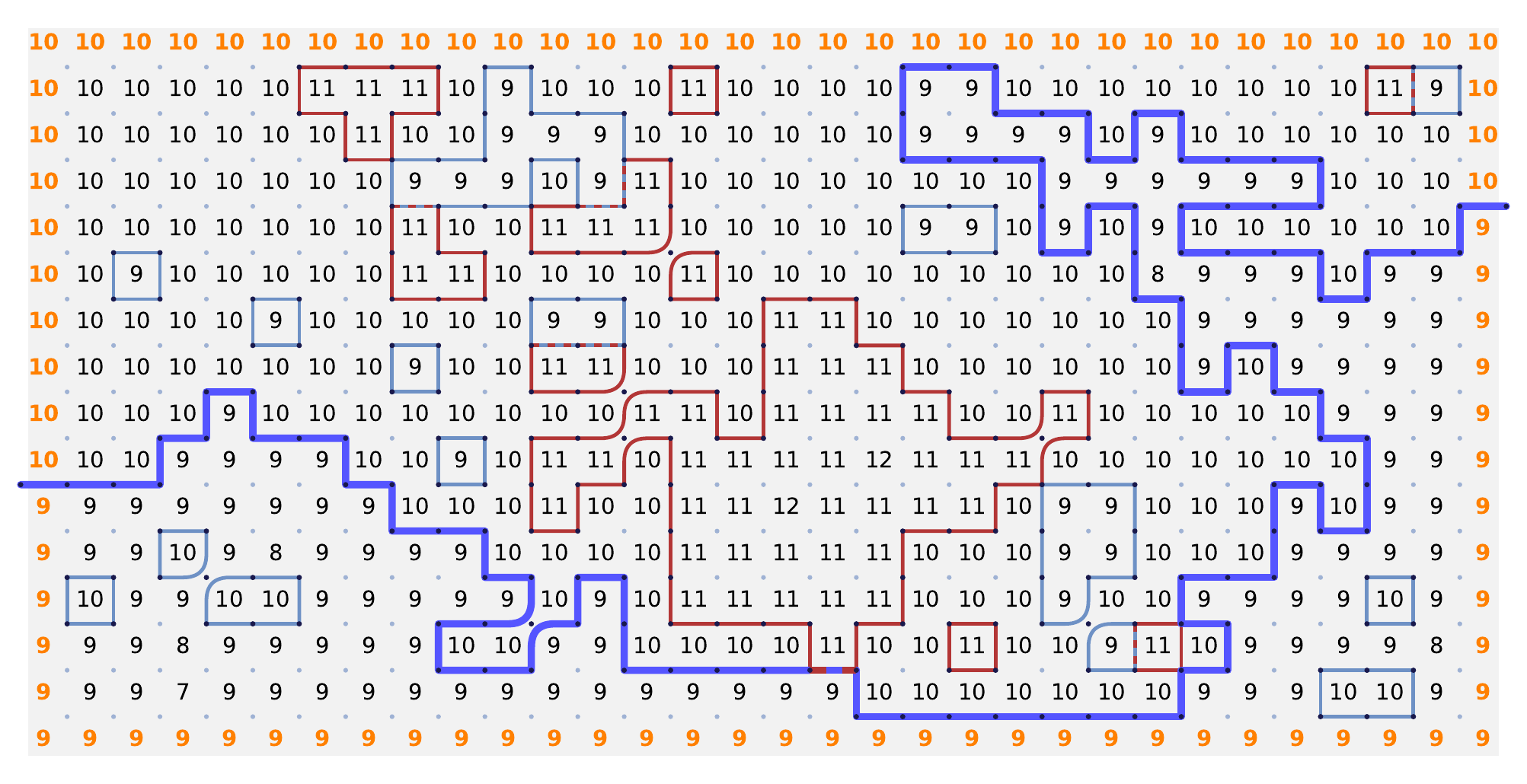}
    \vspace{-0.1in}
    \caption{Comparison of the disagreement polymer (top left: in thick blue, the bonds that $\gamma$ consists of; top right: the regions $D_i$ are marked along with the corresponding $h_i$) and the level lines (bottom: the $10$ level lines in blue, the $11$ level lines in red) in the same $\phi$.}
    \label{fig:disagree-polymer-level-line}
\end{figure}

Under Dobrushin-type boundary conditions $h+1$ and $h$, the $h$ level lines $\fL$ from \cref{def:level-lines} are a subset of the corresponding disagreement polymer $\gamma$ from \cref{def:gamma}; see \cref{fig:disagree-polymer-level-line}.

    It will be convenient to reference the sets
	$\Delta^*_\gamma := \{ u\in V\,:\; \dist(u,\gamma) \leq 1/\sqrt{2}\}$, and $D_i^\circ := D_i \setminus \Delta^*_\gamma$.
\begin{definition}[Energy, length and decorations of a disagreement polymer]\label{def:disagree-polymer-len-energy-decor}
		Let $\gamma$ be a disagreement polymer. Its \emph{length} $\sN(\gamma)$ and \emph{energy} 
		$\sE_\beta(\gamma)$ are defined as
		\[\sN(\gamma)=\sum_{e\in\gamma}|(\nabla \phi)_e|\quad,\quad 
		\sE_\beta(\gamma) := \beta\sum_{e \in \gamma} |(\nabla \phi)_e|^2\,.\] 
        
		To describe the law of the disagreement polymer, we will use a family of decoration functions $\Phi$ on connected subsets $\sfW\subset \Z^2$ satisfying the following properties.
		\begin{enumerate}[(i)]
			\item \label{it:phi(W)-symmetric}
			The function $\Phi$ is invariant under translations, rotations by $\pi/2$, and reflection across the $x$ and $y$ axes. 
			
			\item There exists a constant $C > 0$ such that for every $\sfW$, we have 
			$
			|\Phi(\sfW)| \leq \exp(-(\beta - C)\bd(\sfW))
			$, 
			where $\bd(\sfW)$ is the size of the smallest connected set of bonds in $(\Z^2)^*$ containing all the boundary bonds of $\sfW$.\label{it:phi(W)-decay-bound}
		\end{enumerate}
	\end{definition}
    When the boundary conditions induce a unique disagreement polymer $\gamma$, the law of $\gamma$ in the $\ZGFF$ model with no floor can be written in terms of $\sE_\beta$ and $\Phi$. The effect of $\Phi$ will be in the form of a sum over $\Phi(\mathsf W)$ for $\mathsf W \subset V$ that intersects $\Delta^*_\gamma$. That is, for a general set of dual bonds $\gamma$, let
\begin{equation}\label{eq:IU(gamma)-def}
    \fI_U(\gamma) := \sum_{\substack{\sfW \subset U\\  \sfW \cap \Delta^*_\gamma \neq \emptyset}}\Phi(\sfW)\,.
\end{equation}
We have the following representation for $\hatpi_V^\eta(\gamma)$ under $0/1$ Dobrushin-type boundary conditions $\eta$.  
    \begin{proposition}[{\cite[Prop.~2.12]{ChenLubetzky25}}]\label{prop:CE-law-without-area}
		Let $V\subset \Z^2$ be a connected domain, and consider the \ZGFF model $\hatpi_V^\xi$ with boundary conditions $\xi$ that are $1$ on a $*$-connected path in $\partialvtx V$ and $0$ elsewhere so that they induce a unique disagreement polymer $(\gamma,\{D_i\},\{h_i\})$ in $V\cup \partialvtx V$ that contains boundary disagreements. Then for $\beta\geq \beta_0$, the law of this unique disagreement polymer is given by 
		\begin{equation}\label{eq:CE-without-area}
			\hatpi^\eta_V(\gamma) =\frac1{\widehat Z^\eta_V} \exp\Big(-\sE_\beta(\gamma) + \fI_V(\gamma)\Big)
		\end{equation}
		for $\widehat Z^\eta_V = \widehat Z^\eta_V(\beta)$ and $\fI_V$ from \cref{eq:IU(gamma)-def} featuring a decoration function $\Phi$ as per \cref{def:disagree-polymer-len-energy-decor}.    
	\end{proposition}
    The case of $\pi^\eta_V$, the measure with the floor, is harder and requires conditions on the domain and on $\gamma$ to obtain a similar form for the law. In fact, precisely studying how these conditions change as the domain size varies from mesoscopic to macroscopic is the primary technical contribution of this paper, and will be the goal of \cref{sec:xi-bbq,sec:refined-CE}. For reference, we first provide here the preliminary result used in \cite{ChenLubetzky25}.
\begin{proposition}[{\cite[Prop.~2.3]{ChenLubetzky25}}]\label{prop:CE1_old}
    Fix $n\geq 1$. Let $V\subset \Z^2$ be a connected domain, and consider the \ZGFF model $\pi_{V;F}^{\xi}$ with a floor at~0 imposed only on a subset $F \subset V$, and boundary conditions $\xi$ that are $H+1-n$ on a $*$-connected path in $\partialvtx V$ and $H-n$ elsewhere so that they induce a unique disagreement polymer $(\gamma,\{D_i\},\{h_i\})$ in $V\cup \partialvtx V$ that contains boundary disagreements. Then for $\beta\geq \beta_0$, the law of this unique disagreement polymer $\gamma$ is given by 
		\begin{align}\label{eq:CE-with-area}
			\pi^{\eta}_{V;F}(\gamma) =\frac1{Z_{V;F}^\eta} \exp\Big(-\sE_\beta(\gamma) + \fI_V(\gamma)
			\Big)\prod_{i\geq 0}\hatpi_{D_i^\circ}^{h_i}\big(\phi_x \geq 0,\, \forall x \in D_i^{\circ}\cap F\big)\,,
		\end{align}
		for $Z_{V;F}^\eta = Z_{V;F}^\eta(\beta,n)$ 
        and $\fI_V$ from \cref{eq:IU(gamma)-def} featuring a decoration function $\Phi$ as per \cref{def:disagree-polymer-len-energy-decor}.
		
		Moreover, if we further have $|F| \leq L e^{\kappa\sqrt{\log L}}$ and $|\partial F|\leq L^{1-\delta}$ for fixed $\kappa>0$ and $0<\delta<\frac13$, denote by $D_0$ and $D_1$ the regions of $\gamma$ containing the boundary vertices of $V$ at heights $H+1-n$ and $H-n$, respectively, and let
		\begin{align*} E:=\left\{|\gamma|\leq L^{1-\delta}\,,\;
			\;|F| - |D_0 \cap F| - |D_1 \cap F| \leq L^{1-\delta}\right\}\,,\end{align*}
		then the following holds
		for all $\beta\geq \beta_0$. The probability distribution given~by 
		\begin{equation}\label{eq:pVF-old-expression}
        \fp_{V;F}^{\eta}(\gamma) := 
		\frac1{\widetilde{Z}^\eta_{V;F}} \exp\bigg(-\sE_\beta(\gamma) + \frac{|D_0\cap F|}{N_n}+ \fI_V(\gamma)\bigg)\prod_{i \geq 2}\hatpi_{D_i^\circ}^{h_i}(\phi_x \geq 0,\, \forall x \in D_i^{\circ} \cap F)
		\end{equation}
        for $\gamma\in E$, 
		with $N_n$ from \cref{eq:Nn-def} and $\widetilde{Z}^\eta_{V;F}$ a normalizer, satisfies that for every $\gamma\in E$,
		\begin{align}\label{eq:CE-with-area-cond-old}
			\pi^\eta_{V;F}(\gamma \mid E) =(1+o(1))\fp_{V;F}^{\eta}(\gamma)\,.
		\end{align}
\end{proposition}
It will be convenient, for brevity, to fold the product in \cref{eq:pVF-old-expression} into the energy of $\gamma$, defining\footnote{In \cite[Eq. 3.5]{ChenLubetzky25}, there is an extra term $3c(\beta)|\gamma|$ which was absorbed into the definition of the $\Phi$ functions to make the latter non-negative. This was to set up for a random walk coupling, which is not used here, hence we do not write it. Of course, the law of $\gamma$ is not changed -- we just add and subtract $3c(\beta)|\gamma|$.}
\begin{equation}
    \label{eq:E*-def}
\sE^*_\beta(\gamma) := \sE_\beta(\gamma) - \sum_{i\geq 2} \log \hatpi_{D_i^\circ}^{h_i}(\phi_x \geq 0,\, \forall x \in D_i^\circ)\,,
\end{equation}
so $\fp_{V;F}^\eta(\gamma) \propto \exp\big(-\sE^*_\beta(\gamma)+\frac{|D_0\cap F|}{N_n} + \fI_V(\gamma)\big)$, an area-tilted law of a low-temperature polymer. 

In \cref{sec:refined-CE}, we will also be interested in studying the probability of a disagreement polymer existing when the domain has uniform boundary conditions. We begin here with a preliminary bound on the probability of a level line. 

\begin{definition}\label{def:exist-level-line}
    Let $\cC_{\fL, h}$ denote the event that $\fL$ is an $h$ level line. Similarly, let $\cC_{\gamma, h}$ denote the event that $\gamma$ is a disagreement polymer such that an $h$ level line can be formed using a subset of the bonds of $\gamma$. (To differentiate between the two, $\fL$ is reserved for level lines and $\gamma$ is reserved for disagreement polymers.)
\end{definition}
    
\begin{lemma}[{\cite[Prop~4.1]{LMS16}}\footnote{The original result was for $F = V$, but the proof extends to $F \subseteq V$.}]\label{lem:prob-of-contour}
    For sufficiently large $\beta$, and $h$ such that $\hatpi_\infty(\phi_o\geq h) \leq (1/\log L)^2$, we have for any constant boundary conditions $j \geq 0$ that 
    \begin{equation}
        \pi^j_{V;F}(\cC_{\fL, h+1}) \leq \exp\big(-(\beta-o(1))|\fL| + \hatpi_\infty(\phi_o > h)|\Int(\fL) \cap F| \big)\,.
    \end{equation}
\end{lemma}

\begin{corollary}\label{cor:no-macro-contours}
    For sufficiently large $\beta$, and $h$ such that $\hatpi_\infty(\phi_o\geq h) \leq (1/\log L)^2$, suppose $F \subset V \subset \Z^2$ and
    \begin{equation}\label{eq:no-macro-contours}
        |F| \leq \Big(\frac{3\beta}{\hatpi_\infty(\phi_o > h)}\Big)^2\,.
    \end{equation}
    Then, under the measure $\pi^{h}_{V;F}$, there are no \texttt{large} disagreement polymers in $V$ except with probability $O(L^{-10})$.
\end{corollary}
\begin{proof}
    It suffices to show that w.h.p.\ there are no \texttt{large} $h+1$ level lines. (By definition, this also rules out higher level lines, and level line contours which have lower heights on the interior as opposed to the exterior can always be ruled out by a standard Peierls argument.) This then follows from \cref{lem:prob-of-contour} because the area term cannot make up for the cost of $|\fL|$. Indeed, we know by the isoperimetric inequality that $|\Int(\fL) \cap F| \leq \sqrt{|F|}\tfrac{|\fL|}{4}$. For an $h+1$ level line, the area term is then bounded above by $\hatpi_\infty(\phi_o > h)|\Int(\fL) \cap F| \leq \frac34\beta|\fL|$, so that \cref{lem:prob-of-contour} gives $\pi^h_{V;F}(\cC_{\fL, h+1}) \leq \exp(-(\tfrac{\beta}{4}-o(1))|\fL|)$, whence the standard Peierls argument enumerating over $\fL$ concludes.
\end{proof}

As seen in the above results, the law of level lines and disagreement polymers is intimately related to the probability of large height deviations. The following theorem summarizes a few important results concerning such large deviations. These results were first proven in \cite[Thm.~3.1]{LMS16} for the case of $V = \Z^2$, and extended to more general $V$ in \cite{ChenLubetzky25}. (The latter work also sharpened the bound on the conditional probability in \cref{eq:LD-conditional}.)
\begin{theorem}[{\cite[Thm.~2.5]{ChenLubetzky25}}]\label{thm:LD-DG}
		There exist constants $\beta_0>0$ and $c>c'>0$ so that the following holds for every $\beta\geq \beta_0$ and integer $h\geq 2$. Let $V \subset \Z^2$ be a region containing $\cB_r(o)$, the ball of radius $r=\lceil 2c h/\log h\rceil $ centered at the origin $o$, as well as $\cB_{r+1}(z)$ for a site $z\in V$. Then
		\begin{align}
			\exp\Big(-c\beta \frac{h}{\log h}\Big) \leq &\;\frac{\hatpi_V^0(\phi_o = h)}{\hatpi_V^0(\phi_o = h-1)} \leq 
			\exp\Big(-c'\beta \frac{h}{\log h}\Big)
			\,,\label{eq:LD-ratio}\\
			\exp\Big(-2\pi\beta\frac{h^2}{\log h} - c \beta \frac{h^2}{\log^2h}\Big)
			\leq  &\quad\;\;\hatpi_V^0(\phi_o = h)\!\!\!\!\quad\;\leq \exp\Big(-2\pi\beta\frac{h^2}{\log h} + c \beta\frac{h^2}{\log^2h}\Big)\,,\label{eq:LD}\\
			&\!\!\!\!\!\!\!\!\!
			\hatpi_V^0(\phi_z = h \mid \phi_o = h) \leq \exp\Big(-c \beta \frac{h^2}{\log^2h}\Big)\,.\label{eq:LD-conditional}
		\end{align}
	\end{theorem}

\subsection{Polymer model, surface tension, Wulff shape}\label{sec:poly-st-wulff}
The form of the law of $\gamma$ in \cref{prop:CE-law-without-area,prop:CE1_old} motivate the general study of polymer models outside the context of the measures $\hatpi^\eta_V$ and $\pi^\eta_V$. We begin by defining the set of admissible polymers.

\begin{definition}[Polymers from $A$ to $B$]
    Let $V \subset \Z^2$ be finite and simply connected, with two marked points $A, B$ on $\partial V$. Assuming without loss of generality that $A$ is left of $B$, let $\xi$ be boundary conditions of $h$ along the upper\footnote{The choice of which arc is upper can be made precise via a conformal map sending $\partial V$ to the unit circle, $A$ to $(0, -1)$, and $B$ to $(0, 1)$, and selecting the arc which maps to the upper half plane.} arc of $\partial V$ from $A$ to $B$ and $h-1$ along the lower arc. For every $\Z$-height function on $V$ with boundary condition $\xi$, there is a unique disagreement polymer $\gamma$ which contains the boundary disagreements. Define $\cP_V(A, B)$ as the set of all such possible disagreement polymers in this setting.\footnote{Note that the choice of $h$ is irrelevant here, all that matters here is that the boundary heights differ by 1.} If $V$ is an infinite volume domain, then define $\cP_V(A, B) = \bigcup_{V' \subset V} \cP_{V'}(A, B)$, where the union is over all simply connected $V' \subset V$ which have finite volume. 
\end{definition}
Consider the polymer model with weights given by
\[\tildeq_U(\gamma) = \exp\big(-\sE_\beta(\gamma)+ \fI_U(\gamma)\big)\,,\]
and partition function 
\[\tildeZ_{V, U}(A, B) := \sum_{\gamma \in \cP_{V}(A, B)} \tildeq_U(\gamma)\,.\]
(Just as in the previous work \cite{ChenLubetzky25}, this will be different from the weights appearing in \cref{sec:refined-CE}, and the two will be compared at the level of the surface tensions via \cite[Prop.~4.14]{ChenLubetzky25}.)

Now fix a unit vector $\n$ with angle $\theta$. Let $N$ be such that the point $N\n$ lies on the lattice. In \cite[\S3.3]{ChenLubetzky25} it was shown that the surface tension for this polymer model is well-defined:
\begin{equation}\label{eq:surface-tension}
		\tau_\beta(\theta):=\tau_\beta(\n) := -\lim_{N \to \infty} \frac1{\norm{N\n}_1} \log \tildeZ_{\Z^2, \Z^2}(\o, N\n)\,.
\end{equation}
We can extend $\tau_\beta$ to a function over all of $\R^2$ by homogeneity, and in \cite[Prop.~3.11]{ChenLubetzky25} it was shown that $\tau_\beta$ is analytic and satisfies and strict convexity: for any $\sfu, \sfv$ not on the same line,
\[\tau_\beta(\sfu) + \tau_\beta(\sfv) > \tau_\beta(\sfu + \sfv)\,.\]
Moreover, $\tau_\beta$ is symmetric under rotations by $\pi/4$, reflections across the $x$ and $y$ axis and the diagonals $y = \pm x$.

We now move away from our specific $\tau_\beta$ and recall some general facts concerning the Wulff shape, see, e.g.,~\cite[\S2 and \S4]{schonmann1996constrained} for a more comprehensive overview. We can define the Wulff shape 
\begin{equation}\label{eq:def-Wulff-shape}
		\cW  = \cW(\tau) := \bigcap_{\sfy \in \R^2} \{\sfh \in \R^2: \sfh \cdot \sfy \leq \tau(\sfy)\}\,.
	\end{equation}
Define also the Wulff functional on closed rectifiable curves in $\R^2$ as $W(\gamma) :=\int_\gamma \tau(\theta_s)ds$. If $A(\gamma)$ denotes the area interior to $\gamma$, then $\partial\cW$ is the minimizer of the Wulff functional over all curves such that $A(\gamma) = A(\partial\cW)$. Let $\cW_1$ denote the Wulff shape with unit area, and $\sfw_1 = \sfw_1(\tau):= W(\partial\cW_1)$. Let $\ell_\tau$ denote the length of the smallest square which contains $\cW_1$. One can compute that $\ell_\tau = 4\tau(0)/\sfw_1$. Observe that $\ell_\tau \geq 1$, whence this immediately implies that $4\tau(0) \geq \sfw_1$.

Now suppose we wish to minimize the Wulff functional, but we require that $\gamma$ lies in the unit square and that the area $A(\gamma)$ is equal to some fixed constant $\alpha \in (0, 1)$. When $\alpha \leq \tfrac{1}{\ell_\tau^2}$, the above discussion immediately gives an answer of $\sqrt{\alpha}\cW_1$. When $\alpha > \tfrac{1}{\ell_\tau^2}$, the answer is given by a translation of Wulff shapes denoted $\sL(\lambda)$ defined below (for $\lambda$ such that $A(\sL(\lambda)) = \alpha$).
\begin{definition}\label{def:cL}
	Given $\tau$, let $\sL(\lambda)$ be the set obtained by first taking the union of all translates of $\tfrac{\sfw_1}{2\beta\lambda}\cW_1$ inside the unit square.\footnote{The factor of $\beta$ is for our application with $\tau = \tau_\beta$, as we chose the convention of not normalizing the surface tension by $\beta$. For applications where $\tau$ has already been normalized, take $\beta = 1$.}
\end{definition}

Next, fix $\beta, \lambda > 0$, and consider the functional
\[\cF_\lambda(\gamma) = \cF_\lambda(\gamma, \tau):= -\int_\gamma \tau(\theta_s)ds + \lambda\beta A(\gamma)\,.\]
This functional captures quantitatively a tradeoff between the length of a curve and its interior area, which are precisely the two main terms governing the law of a level line. Hence, in order to identify the point at which a new top level line forms, we need to know when $\cF_\lambda(\gamma) = 0$ and provide bounds on how sensitive $\cF_\lambda(\gamma)$ is to changes in $\lambda$. This is the content of the following lemma, which will be very useful in \cref{sec:transition-window} (a non-quantitative version was already known in \cite{schonmann1996constrained}).
\begin{lemma}\label{lem:linear-critical-window}
    Fix a constant $\alpha \in (0, 1)$. Set $\lambda_* = $ $\tfrac{1}{\beta}(2\tau(0) + \tfrac{\sfw_1}{2})$. For any loop $\gamma$ in the unit square with an interior area of $|A(\gamma)| \geq \alpha$, we have that
    \[\cF_\lambda(\gamma)\leq \begin{cases}
			\beta(\lambda - \lambda_*), & \text{if $\lambda \geq \lambda_*$}\\
            -\beta(\lambda_* - \lambda)\alpha, & \text{if $\lambda < \lambda_*$ and $1 - (4\tau(0)^2- \tfrac{\sfw_1^2}{4})\tfrac{1}{\beta^2\lambda^2} \geq \alpha$.}
		 \end{cases}\]
    Moreover, if $\lambda \geq \lambda_*$, then $\sup_\gamma\cF_\lambda(\gamma) = \cF_\lambda(\sL(\lambda))$ and $\cF_\lambda(\sL(\lambda)) \geq (\beta - (4\tau(0)^2- \tfrac{\sfw_1^2}{4})\tfrac{1}{\beta\lambda^2})(\lambda - \lambda_*)$. 
\end{lemma}
\begin{proof}
    Consider all $\gamma$ with a fixed interior area $A(\gamma) = \alpha$. As discussed above, when $\alpha \leq (\tfrac{\sfw_1}{4\tau(0)})^2$, $\cF_\lambda$ is maximized over such $\gamma$ when $\cF_\lambda(\sqrt{\alpha}\cW_1)$, and when $\alpha > (\tfrac{\sfw_1}{4\tau(0)})^2$ it is maximized at $\cF_\lambda(\sL(\fu))$ where $\fu= \fu(\alpha)$ is such that $\sL(\fu)$ has area $\alpha$. Hence, to study $\sup_\gamma \cF_\lambda(\gamma)$, it suffices to study these two terms for various $\alpha$.
    
    For the case of $\sqrt{\alpha} > \tfrac{\sfw_1}{4\tau(0)}$, we have
    \[\int_{\partial \sL(\fu)}\tau(\theta_s)ds = \tfrac{\sfw_1^2}{2\beta\fu} + 4\tau(0)(1 - \tfrac{\sfw_1}{2\beta\fu}\tfrac{4\tau(0)}{\sfw_1}) = \tfrac{\sfw_1^2}{2\beta\fu} + 4\tau(0)(1 - \tfrac{2\tau(0)}{\beta\fu})\,,\]
    \[\alpha = 1 + \tfrac{\sfw_1^2}{4\beta^2\fu^2} - \tfrac{4\tau(0)^2}{\beta^2\fu^2}\,.\]
    Hence, we have
    \begin{equation}\label{eq:F-lambda-u}\cF_\lambda(\partial\sL(\fu)) = -4\tau(0) + \lambda\beta + (4\tau(0)^2- \tfrac{\sfw_1^2}{4})\tfrac{2\fu-\lambda}{\beta\fu^2}\,.
    \end{equation}
    Since $\sfw_1 \leq 4\tau(0)$ implies that $4\tau(0)^2 - \tfrac{\sfw_1^2}{4} \geq 0$, we can compute the first and second derivatives of $\tfrac{2\fu-\lambda}{\fu^2}$ to get that $\cF_\lambda(\partial\sL(\fu))$ is maximized at $\fu = \lambda$. Hence it suffices to upper bound the function 
    \[F(\lambda) := -4\tau(0) + \lambda\beta + (4\tau(0)^2- \tfrac{\sfw_1^2}{4})\tfrac{1}{\beta\lambda}\,.\]
    We can compute that $F(\lambda) = 0$ when $\lambda = \lambda_*$. Moreover, $F'(\lambda) = \beta - (4\tau(0)^2- \tfrac{\sfw_1^2}{4})\tfrac{1}{\beta\lambda^2} \leq \beta$. Hence, for $\lambda > \lambda_*$ we have $F(\lambda) \leq \beta(\lambda - \lambda_*)$, and for $\lambda < \lambda_*$ we have $F(\lambda) \leq(\beta - (4\tau(0)^2- \tfrac{\sfw_1^2}{4})\tfrac{1}{\beta\lambda^2})(\lambda - \lambda_*)$.

    For the case of $\sqrt{\alpha} \leq \tfrac{\sfw_1}{4\tau(0)}$,  we have
    \[\cF_\lambda(\sqrt{\alpha}\cW_1) = -\sqrt{\alpha}\sfw_1 + \beta\lambda\alpha\,.\]
    Let $\lambda_0 := \sfw_1/(\beta\sqrt{\alpha})$, the value where $\cF_\lambda(\sqrt{\alpha}\cW_1) = 0$. Then, we can write $\cF_\lambda(\sqrt{\alpha}\cW_1) = \beta(\lambda - \lambda_0)\alpha$. By a direct computation, we obtain that $\lambda_0 > \lambda_*$ when $\sqrt{\alpha} \leq \tfrac{\sfw_1}{4\tau(0)}$. Thus, $\cF_\lambda(\sqrt{\alpha}\cW_1) \leq \beta(\lambda - \lambda_*)\alpha$. To summarize, we have proven that when $\lambda \geq \lambda_*$, we have $\sup_\gamma \cF_\lambda(\gamma) \leq \beta(\lambda - \lambda_*)$, and when $\lambda< \lambda_*$, we have $\sup_\gamma \cF_\lambda(\gamma) \leq (\lambda - \lambda_*)\min(\beta - (4\tau(0)^2- \tfrac{\sfw_1^2}{4})\tfrac{1}{\beta\lambda^2}, \beta|A(\gamma)|)$.

    Finally we prove the last statement. We wish to show that $F(\lambda) \geq \cF_\lambda(\sqrt{\alpha}\cW_1)$ for all $\sqrt{\alpha} \leq \tfrac{\sfw_1}{4\tau(0)}$. Fix such an $\alpha$, which defines a $\lambda_0$ as above. When $\lambda \in [\lambda_*, \lambda_0]$, we have $F(\lambda) \geq 0$ while $\cF_\lambda(\sqrt{\alpha}\cW_1) \leq 0$.
    When $\lambda > \lambda_0$, we lower bound $F(\lambda) \geq \tilde F(\lambda) := -4\tau(0) + \lambda\beta$. The assumption on $\alpha$ implies that $\tilde F(\lambda_0)=  -4\tau(0) + \tfrac{\sfw_1}{\sqrt{\alpha}} \geq 0$. We now conclude that $\tilde F(\lambda) \geq \cF_\lambda(\sqrt{\alpha}\cW_1)$ in this regime since both are linear functions of $\lambda$, and $\tilde F(\lambda)$ has a larger slope and value at $\lambda_0$. The lower bound on $F(\lambda)$ follows from the formula for the derivative computed above.
\end{proof}

\begin{remark}\label{rem:cL-alt-notation}
    In our applications in \cref{sec:transition-window,sec:limit-shape}, it will often be simpler to let the argument of $\sL$ be the scaling factor of $\cW_1(\tau)$ directly, so that $\sL(\ell)$ is the union of translates of $\ell\cW_1(\tau)$. By abuse of notation, we will use both, and it will be clear from context whether the argument refers to the area tilt or the scaling factor of the Wulff shape.  
\end{remark}

When the functional $\cF_\lambda(\gamma)$ is positive, it is maximized at $\cF_\lambda(\sL(\lambda))$, suggesting that the limit shape of the level lines is given by $\sL(\lambda)$ for some $\lambda$. We recall here the partial result proven in \cite{ChenLubetzky25} that the level lines contain such a shape, in the notation of the above remark. 
\begin{theorem}[{\cite[Thm.~4.4 revised as per Rem.~4.7]{ChenLubetzky25}}]\label{thm:old-level-line-contains-Wulff} 
		Consider the \ZGFF on an $L \times L$ box, for any $L$. Fix $n \geq 1$. Then, for any constant $C > 0$, w.h.p., the $H+1-n$ level line contains
		\[(1-\tfrac{N_n^{1/3}e^{3C\sqrt{\log L}}}{L})L\sL(\ell_n(1+e^{-\frac{C}3\sqrt{\log L}}))\,,\]
        where $\ell_n = \frac{\sfw_1(\tau_\beta)N_n}{2L}$.
\end{theorem}
In fact, in this work we will use a slightly different side length, defining
\begin{equation}\label{eq:def-ell*}
\ell^*_n = \frac{\sfw_1(\tau_\beta)N_n}{2L\uprho_n}
\end{equation}
for $\uprho_n$ defined immediately below in \cref{eq:rho-n-def}. Note that after we prove \cref{prop:uprho-bound}, the above theorem also holds with $\ell^*_n$ instead of $\ell_n$, because the difference between $\uprho_n$ and 1 can be absorbed into the $1+e^{-\frac C3\sqrt{\log L}}$ term. 

Finally, we conclude this section with a few definitions related to disagreement polymers which will be generally useful throughout this paper.

\begin{definition}\label{def:outer-envelope}
    For any disagreement polymer $\gamma$, define its outer envelope $\mathsf{OE}(\gamma)$ as the outermost loop that consists only of bonds in $\gamma$.
\end{definition}

\begin{definition}\label{def:gamma-components} For a bond $b \in \gamma$, call $\cD_b = \cD_b(\gamma)$ the connected component of regions in $\Lambda' \setminus (\gamma \cup D_0 \cup D_1)$ containing a region that has $b$ as part of its boundary. If there is no such region adjacent to $b$, then $\cD_b = \emptyset$. Denote by $\fD = \fD(\gamma)$ the collection of all nonempty sets $\cD_b$.
\end{definition}

Note that $\bigcup_{\cD \in \fD} \cD = \bigcup_{b \in \gamma}\cD_b = \bigcup_{i \geq 2} D_i$, and that each $\cD \in \fD$ can be written as a union of sets $\bigcup_{i \in I}D_i$ for some finite set $I \subset \{i \geq 2\}$.

\begin{definition}\label{def:cut-point}
    We call $v$ a cut-point of $\gamma$ if it is the end-vertex of a bond $b \in \gamma$ with $\cD_b = \emptyset$.
\end{definition}

\section{Minimum height in a region without a floor}\label{sec:xi-bbq}

In this section we will study the large deviation rates for the minimum height of the set of sites in a square of side length $\ell_* \in [\sqrt{L},2\sqrt{L})$ under the infinite-volume measure $\hatpi_\infty$:
\begin{align}
    \label{eq:xi-n-def}
    \upxi_n
    &:= -\frac{1}{\ell_*^2}\log\hatpi_\infty\left( \phi_x \geq -(H+1-n),\,\forall x\in Q_{\ell_*} := \llb 1,\ell_*\rrb^2 \right)\quad\mbox{for}\quad \ell_*:=2^{\lceil\frac12\log_2 L\rceil}\\
    \label{eq:rho-n-def}
\uprho_n
&:= \left(\upxi_{n+1} - \upxi_n\right) N_n\,. 
\end{align}
One of the main results of this section will  establish the following estimates for $\upxi_n$ and $\uprho_n$.
\begin{proposition}\label{prop:uprho-bound}
        Fix $n \geq 0$. There are absolute constants $c,C>0$ such that
        \begin{align}
        \label{eq:upxi-bounds}
         \frac{\upxi_n}{\hatpi_\infty(\phi_o > H+1-n)} &\in [1 - e^{-C\frac{\log L}{\log\log L}},  1 - e^{-c\frac{\log L}{\log\log L}}]\,,\\
         \label{eq:uprho-bounds}
        \uprho_n &\in [1 - e^{-C\frac{\log L}{\log\log L}},  1 - e^{-c \frac{\log L}{\log\log L}}]\,.
        \end{align}
\end{proposition}
We will see in \cref{sec:refined-CE} that $\uprho_n$ is the correct pre-factor of the area tilt in the cluster expansion for the disagreement polymers. 
Towards that, we will show in this section that $\upxi_n$ is the correct rate for the probability of a region to lie above height $-h$ for $h=H+1-n$, as per the next result.
In what follows, a connected subset $F$ of the vertices of $\Z^2$ is said to have $g\geq 0$ \emph{holes} if we can decompose its vertex boundary $\partialvtx F$ into $g+1$ disjoint $*$-connected closed paths.

\begin{theorem}\label{thm:key-area}
        Fix $g\geq 0$, $n \geq 0$, and set $h = H+1-n$. The following hold for every subsets $F\subseteq V \subset\Z^2$ where $F$ is connected with $g$ holes.
        \begin{enumerate}
            \item\label{it:xi-n-area-meso}[\emph{Mesoscopic range}] If  $|\partial F|\vee |\partial V| \leq O(L^{2/3}\exp(\sqrt{\log L}))$ then 
    \begin{align}\label{eq:xi-n-area-estimate-meso}\hatpi^0_V(\phi_x \geq -h,\, \forall x \in F) = (1+o(1))\exp\left(-\upxi_{n}|F|\right)\,.
    \end{align}
    \item\label{it:xi-n-area-macro}[\emph{Macroscopic range}] If $|\partial F| \vee |\partial V| \leq O( L\exp(\sqrt{\log L}))$ then 
    \begin{align}\label{eq:xi-n-area-estimate-macro-lower}\quad\;\hatpi^0_V(\phi_x \geq -h,\, \forall x \in F) \geq \exp\left(-\upxi_{n}|F| + O\big(\sqrt{L}e^{\frac{\log L}{\log\log L}}\big)\right)\,,
    \end{align}
    and if $\fS$ is the event that there are no disagreement polymers $\gamma$ in $\phi$ with $|\gamma| \geq \log L$, then        
\begin{align}\label{eq:xi-n-area-estimate-macro-upper}\hatpi^0_V(\phi_x \geq -h,\, \forall x \in F \mid \fS) \leq \exp\left(-\upxi_{n}|F| + O\big(\sqrt{L}e^{\frac{\log L}{\log\log L}}\big)\right)\,.
    \end{align}
If $|F| \leq  \big(\frac{3\beta}{\hatpi_\infty(\phi_o > h)}\big)^2 $, then the last upper bound also applies to $\hatpi_V^0(\phi_x\geq -h,\,\forall x\in F)$.
\end{enumerate}
The above statements also hold when replacing $\hatpi_V^0$ by $\hatpi_\infty$.

\end{theorem}

\subsection{Rates for the minimum height in a box}\label{sec:xi-and-xi-bar-bounds}
This section is devoted to proving lower and upper bounds on the following rates: for every $\ell\geq 1$ and $h\geq 0$, define
\begin{align} \label{eq:xi-l-h-def}
    \xi_{\ell,h} &:= -\frac{1}{\ell^2}\log\hatpi_\infty\left( \phi_x \geq -h
    ,\,\forall x\in Q_{\ell} := \llb 1,\ell\rrb^2 
    \right)\,,\\
 \label{eq:xi-bar-l-h-def}
    \bar\xi_{\ell,h} &:= -\frac{1}{\ell^2}\log\hatpi_\infty\left( \phi_x > -h
    ,\,\forall x\in Q_{\ell} 
    \;\big|\; \phi_x \geq -h
    ,\,\forall x\in Q_{\ell} \right)\,,
\end{align}
noting that 
\begin{equation}\label{eq:bar-xi-diff-xi}\bar\xi_{\ell,h} = \xi_{\ell,h-1} - \xi_{\ell,h}\,,\end{equation}
and that $\upxi_n$ and $\uprho_n$ from \cref{eq:xi-n-def,eq:rho-n-def} are simply $\xi_{\ell_*,H+1-n}$ and $\bar\xi_{\ell_*,H+1-n} N_n$, respectively.

\subsubsection{Estimating the rate \texorpdfstring{$\xi_{\ell,h}$}{xi}}\label{sec:xi-bounds}
By FKG, if $\ell = k m$ then the probability in the right-hand of \cref{eq:xi-l-h-def} is bounded from below by $\hatpi_\infty(\phi_x\geq -h,\,\forall x\in Q_m)^k$, leading to the following observation.\begin{observation}\label{obs:xi-monotone}
For all $k,m\geq 1$, $\xi_{k m,h} \leq \xi_{m,h} $. In particular, 
$ \xi_{\ell,h} \leq \xi_{1,h}= -\log\hatpi_\infty(\phi_o\geq -h)$.
\end{observation}
The next lemma readily infers an upper bound on $\xi_{\ell,h}$ from the above observation for the $\ell,h$ relevant to our application in \cref{sec:bbq}, and complements it with a sharper lower bound.

\begin{lemma}\label{lem:xi-lower-weak-upper}
There exist absolute constants $\beta_0,c_0>0$ such that the following holds for all $\beta>\beta_0$. Let $h = H+1-n$ for fixed $n\geq 0$. Then for all $1 \leq \ell \leq \sqrt{L} \exp(-\frac12\frac{\log L}{\log\log L})$ we have
\begin{equation}
\label{eq:xi-correct-lower-weak-upper}
 1 - e^{-c_0\frac{\log L}{\log \log L}} \leq \frac{\xi_{\ell,h}}{\hatpi_\infty(\phi_o < -h)} \leq 1 + L^{-1+o(1)}\,.
 \end{equation}
Moreover, this holds when replacing $\xi_{\ell,h}$ by $-\frac1{|S|}\log\hatpi_\infty(\phi_x\geq -h,\,\forall x\in S)$ for any set $S\subset \Z^2$, not necessarily connected, of size at most $\ell^2$. In addition, the upper bound holds for all $\ell \geq 1$.
\end{lemma}
\begin{proof}
We begin with the upper bound. Recall from \cref{obs:xi-monotone} that $\xi_{\ell,h}\leq \xi_{1,h}$. Writing $\xi_{1,h}=-\log(1-t)$ for $t  = \hatpi_\infty(\phi_o<-h)$, we use
the fact  that $1-t\geq e^{-t-t^2}$ for $0<t<\frac12$ (here $t < \epsilon_\beta<\frac12$ for every $h\geq 0$ by the standard Peierls estimate; cf.~\cite{BrandenbergerWayne82}), and find that for every $\ell,h$,
\[ \frac{\xi_{\ell,h}}{\hatpi_\infty(\phi_o < -h)}
 \leq 1 + \hatpi_\infty(\phi_o < -h)\,.\]
When $h=H+1-n$ for fixed $n\geq 0$, as in our hypothesis, we have $\hatpi_\infty(\phi_o<-h) = L^{-1+o(1)}$ (see, e.g., \cite[Eq.~(2.3)]{ChenLubetzky25} and the definition of $H$), establishing the upper bound in \cref{eq:xi-correct-lower-weak-upper}.

We proceed to the lower bound. The Bonferroni inequalities imply that
\begin{align*}
    \hatpi_\infty(\phi_x \geq -h,\,\forall x\in Q_\ell) &\leq 1 - |Q_\ell| \hatpi_\infty(\phi_o< -h) + \sum_{x,y\in Q_\ell} \hatpi_\infty(\phi_x< -h\,,\, \phi_y< -h)\,.\end{align*} 
If $\dist(x,y)>\log L$ then $\hatpi_\infty(\phi_y<-h\mid\phi_x<-h) < \hatpi_\infty(\phi_y<-h) + L^{-10}$ by the well-known decorrelation estimates $\hatpi_\infty$ at low temperature (stemming from a routine Peierls argument; see also \cref{eq:couple-pi-to-inf-vol} below), and  for all other distinct $x,y$ we have that
$\hatpi_\infty(\phi_y<-h\mid \phi_x<-h) < e^{-c\beta h^2/\log^2 h}$ by \cite[Thm.~2.5, Eq.~(2.8)]{ChenLubetzky25}. Hence,
\begin{align}\nonumber
\frac{\sum_{x,y\in Q_\ell} \hatpi_\infty(\phi_x< -h\,,\, \phi_y< -h)}{|Q_\ell| \hatpi_\infty(\phi_o<-h)} &\leq |Q_\ell|\left(\hatpi_\infty(\phi_o<-h)+ L^{-10}\right) + O((\log L)^2) e^{-c\beta h^2/\log^2 h} \\
&\leq e^{-c' \frac{\log L}{\log\log L}}\,,\label{eq:sum-pi(h-h)-ratio}    
\end{align}
using $h^2/\log^2 h\asymp \frac1\beta \frac{\log L}{\log\log L}$, so the last summand in the first line is at most $\exp(-c\frac{\log L}{\log\log L})$, as is the first summand since $|Q_\ell|=\ell^2\leq L \exp(-\frac{\log L}{\log\log L})$ and $\hatpi_\infty(\phi_o<-h) \leq L^{-1} \exp(o(\sqrt{\log L}))$.
Overall,
\begin{align*}
    \hatpi_\infty(\phi_x \geq -h,\,\forall x\in Q_\ell)^{1/\ell^2} &\leq \left(1 - \big(1-e^{-c'\frac{\log L}{\log\log L}}\big) \ell^2 \hatpi_\infty(\phi_o<-h)\right)^{1/\ell^2} \nonumber \\
    &\leq \exp\left(-\big(1-e^{-c'\frac{\log L}{\log\log L}}\big) \hatpi_\infty(\phi_o<-h)\right)\,,
\end{align*} 
establishing the lower bound in \cref{eq:xi-correct-lower-weak-upper}.
\end{proof}
Note the asymmetry between the lower and upper bounds on $\xi_{\ell,h}/\hatpi_\infty(\phi_o<-h)$ in \cref{eq:xi-correct-lower-weak-upper}. The following lemma refines the upper bound to show that  $\xi_{\ell,h} /\hatpi_\infty(\phi_o<-h)$ is in fact bounded from above by $1-e^{c_1\frac{\log L}{\log\log L}}$.

\begin{lemma}\label{lem:xi-refined-upper}
There exist absolute constants $\beta_0,c_1>0$ such that the following holds for all $\beta>\beta_0$. Fix $n\geq 0$ and $\delta>0$, and let $h = H+1-n$. Then for all $\ell \geq L^{\delta}$ we have
\[ \frac{\xi_{\ell,h}}{\hatpi_\infty(\phi_o<-h)} \leq 1 - e^{-c_1\frac{\log L}{\log \log L}} \,.\]
\end{lemma}
We need the following result, showing that the upper bound $\hatpi_V^0(\phi_x = h \mid \phi_o = h) \leq \exp(-c\beta \frac{h^2}{\log^2 h})$ for any $x\neq o$ (\cite[Thm.~2.5, Eq.~(2.8)]{ChenLubetzky25}) is tight when $x$ is a neighbor of the origin.
\begin{claim}\label{clm:pi(h|h)-lower}
Let $h\geq 1$, and let $V\subseteq \Z^2$ be a connected region with $V\supset B_R(o)$, where $R = \lfloor h/\log h\rfloor$. There is an absolute constant $c>0$ such that, for large $\beta$, if $o'$ is a neighbor of the origin $o$ then
\[ \hatpi_V^0\left(\phi_{o'} \geq h\mid \phi_o \geq h\right) \geq e^{-c\beta \frac{h^2}{\log^2 h}}\,.\]
\end{claim}
\begin{proof}
It was shown in \cite[Eq.~(3.6)]{LMS16} that, for a large enough absolute constant $C_1>0$, \[\hatpi_V^0\Big(\phi_{o'} \geq h - C_1\frac{h}{\log h} \;\big|\; \phi_o = h\Big) > \frac12\,.\]
Consequently,
\[\hatpi_V^0\Big(\phi_{o'} \geq h - C_1\frac{h}{\log h} \;\Big|\; \phi_o \geq h\Big) \geq \hatpi_V^0\Big(\phi_{o'} \geq h - C_1\frac{h}{\log h} \;\Big|\; \phi_o = h\Big) \frac{\hatpi_V^0(\phi_o = h)}{\hatpi_V^0(\phi_o \geq h)}
> \frac13\,,\]
where the last inequality used that $\hatpi_V^0(\phi_o=h) = (1-o(1))\hatpi_V^0(\phi_o\geq h)$ with the $o(1)$-term going to $0$ as $h\to\infty$ (see for instance \cite[Thm.~2.5, Eq.~(2.6)]{ChenLubetzky25}). Therefore, 
\begin{align*} \hatpi_V^0(\phi_{o'} \geq h \mid \phi_o \geq h) &\geq \frac13\,\hatpi_V^0\Big(\phi_{o'} \geq h \;\Big|\; \phi_{o'} \geq h-C_1\frac{h}{\log h}\,,\, \phi_o \geq h\Big) \\ 
&\geq \frac13\,\hatpi_V^0\Big(\phi_{o'} \geq h \;\Big|\; \phi_{o'} \geq h-C_1\frac{h}{\log h}\Big)  
\geq e^{-C_2\beta\frac{h^2}{\log^2 h}}\,, \end{align*}
where the inequality between the lines follows from FKG, and the last inequality used the fact that $\hatpi_V^0(\phi_o=h')/\hatpi_V^0(\phi_o=h'-1) \geq e^{-c'\beta \frac{h'}{\log h'}} \geq e^{-c'\beta \frac{h}{\log h}} $ for an absolute constant $c'$ and all $h'\in\llb h-C_1 \frac h{\log h},h\rrb$ (see, e.g., \cite[Thm.~2.5, Eq.~(2.6)]{ChenLubetzky25}), yielding the final constant $C_2 = c' C_1$.
\end{proof}
\begin{proof}[Proof of \cref{lem:xi-refined-upper}]
 By dividing $Q_\ell$ into $1\times 2$ and $2\times 1$ boxes (``dominoes'') and applying FKG, we see that, for a neighbor of the origin $o'\sim o$,
\begin{align*} 
\hatpi_\infty(\phi_x\geq -h,\,\forall x\in Q_\ell) &\geq 
\hatpi_\infty(\phi_o\geq -h,\,\phi_{o'} \geq -h)^{\lceil \ell^2/2\rceil} \\
&=
\Big(1-2\hatpi_\infty( \phi_o<-h) + \hatpi_\infty(\phi_o< -h,\,\phi_{o'} < -h)\Big)^{\lceil \ell^2/2\rceil}\,. 
\end{align*}
Again using that $1-t\geq e^{-t-t^2}$ for $0<t<\frac12$, if we denote
\[ q = \hatpi_\infty(\phi_o<-h)\Big(1-\frac12\hatpi_\infty(\phi_{o'}<-h\mid \phi_o<-h)\Big) = L^{-1+o(1)}\]
then we see that 
\[
-\log \hatpi_\infty(\phi_x\geq -h,\,\forall x\in Q_\ell) \leq \lceil \ell^2/2\rceil 2q(1+2q) \leq (\ell^2+1)q(1+2q)\,.
\]
Thus,
\begin{align*}
     \frac{\xi_{\ell,h}}{\hatpi_\infty(\phi_o<-h)} &\leq \Big(1+\frac1{\ell^2}\Big) (1+2q)\frac{q}{{\hatpi_\infty(\phi_o<-h)}} \\
     &\leq  \left(1+L^{-2\delta}\right) \left(1+2 L^{-1+o(1)}\right)
     \Big(1-\frac12\hatpi_\infty(\phi_{o'}<-h\mid \phi_o<-h)\Big)\,,
\end{align*}
and the desired result now follows from \cref{clm:pi(h|h)-lower}.
\end{proof}
\subsubsection{Estimating the rate \texorpdfstring{$\bar\xi_{\ell,h}$}{xi-bar}}\label{sec:xi-bar-bounds}
Throughout this subsection, denote
\[ \bP_{\ell,h} := \hatpi_\infty(\cdot \mid \sF_{\ell,h})\qquad\mbox{where}\qquad \sF_{\ell,h}:= \left\{\phi_x \geq -h,\,\forall x\in Q_{\ell}\right\}\]
for brevity, so that $\bar \xi_{\ell,h} = -\frac{1}{\ell^2}\log\bP_{\ell,h}(\phi_x > -h,\,\forall x\in Q_\ell)$.
Note that $\bP_{\ell,h}$ is simply a \ZGFF model with a floor at $-h$ on $Q_\ell$, and thus enjoys FKG (see, e.g., \cite[Claim~A.1]{ChenLubetzky25}). In particular, the same reasoning that led to \cref{obs:xi-monotone} for $\xi_{\ell,h}$ gives its following analogue for $\bar\xi_{\ell,h}$.
\begin{observation}
\label{obs:xi-bar-monotone}
For all $k,m\geq 1$, $\bar\xi_{k m,h} \leq \bar\xi_{m,h} $. In particular, 
$ \bar\xi_{\ell,h} \leq \bar\xi_{1,h}$.
\end{observation}
As $\xi_{\ell,h} \sim \hatpi_\infty(\phi_o < -h)$, it would be natural guess when looking at \cref{eq:bar-xi-diff-xi} that one would have $\bar\xi_{\ell,h}\sim \hatpi_\infty(\phi_o = -h)$. This is indeed the case, as shown in the following analogue of \cref{lem:xi-lower-weak-upper}.
\begin{lemma}
    \label{lem:xi-bar-lower-weak-upper}
In the setting of \cref{lem:xi-lower-weak-upper}, we have
\begin{equation}\label{eq:xi-bar-lower-weak-upper}
    1 - e^{-c_0 \frac{\log L}{\log\log L}} \leq 
    \frac{\bar\xi_{\ell,h}}{\hatpi_\infty(\phi_o = -h)} 
    \leq 1 + L^{-1+o(1)}\,.\end{equation}
\end{lemma}
\begin{proof}
The proof will be a simple modification of the argument used to prove \cref{lem:xi-lower-weak-upper}.

The upper bound follows directly from \cref{obs:xi-bar-monotone} and the fact that $\bar\xi_{1,h} = -\log(1-t)$ for $t = \hatpi_\infty(\phi_o = -h \mid \phi_o \geq -h)  = \hatpi_\infty(\phi_o=-h)/\hatpi_\infty(\phi_o\geq -h)$. As $t \leq (1+L^{-1+o(1)})\hatpi_\infty(\phi_o=-h)$ for our $h$, arguing as for the upper bound in \cref{lem:xi-lower-weak-upper} yields $\bar\xi_{1,h}\leq (1+L^{-1+o(1)})\hatpi_\infty(\phi_o=-h)$.

For the lower bound, we start, as in said lemma, with
\begin{align}\label{eq:P-in-ex}
    \bP_{\ell,h}(\phi_x > -h,\,\forall x\in Q_\ell) \leq 1 &- |Q_\ell| \bP_{\ell,h}(\phi_o = -h ) + \sum_{x,y\in Q_\ell} \bP_{\ell,h}(\phi_x = -h\,,\, \phi_y = -h)\,.\end{align} 
Recalling $\bP_{\ell,h}(\phi_x = -h\,,\,\phi_y = -h)=\hatpi_\infty(\phi_x \leq -h\,,\,\phi_y \leq -h \mid \sF_{\ell,h})$, we deduce from FKG, and the fact that $\hatpi_\infty(\phi_o=h)=(1-o(1))\hatpi_\infty(\phi_o\geq h)$ (see~\cite[Thm.~3.1]{LMS16}) that 
\begin{align}\label{eq:pairs-P-upper-bound}
\frac{\sum_{x,y\in Q_\ell} \bP_{\ell,h}(\phi_x = -h\,,\, \phi_y = -h)}{|Q_\ell| \hatpi_\infty(\phi_o = -h)} &\leq \frac{\sum_{x,y\in Q_\ell} \hatpi_\infty(\phi_x \leq -h\,,\, \phi_y \leq -h)}{(1-o(1))|Q_\ell| \hatpi_\infty(\phi_o \leq -h)} \leq e^{-c \frac{\log L}{\log\log L}}\,,
\end{align}
where the last transition is by \cref{eq:sum-pi(h-h)-ratio}.
Further note that 
\begin{align*}
\bP_{\ell,h}(\phi_o = -h) &
\geq  \hatpi_\infty(\phi_o=-h) - \hatpi_\infty(\phi_o=-h,\,\sF_{\ell,h}^c) \\ 
&= (1-\hatpi_\infty(\sF_{\ell,h}^c \mid \phi_o=-h))\hatpi_\infty(\phi_o=-h)\,,
\end{align*}
and as usual, we can control $\hatpi_\infty(\sF_{\ell,h}^c\mid \phi_o=-h)$ via a union bound, splitting the treatment of $x\in Q_\ell$ into $x$ at distance larger than $\log L$ from the origin and those within said distance. As argued above \cref{eq:sum-pi(h-h)-ratio}, in the former case, 
\[ |Q_\ell| \hatpi_\infty(\phi_x<-h\mid \phi_o=-h) < \ell^2 \left(\hatpi_\infty(\phi_o < -h) + L^{-10}\right) < e^{-(1-o(1))\frac{\log L}{\log\log L}}\,,\]
by the assumption on $\ell$ and the fact
$\hatpi_\infty(\phi_o<-h) \leq L^{-1} e^{c\beta \frac{h}{\log h}}\leq L^{-1} e^{o(\sqrt{\log L})}$; in the latter case,
\[O(\log^2 L) \hatpi_\infty(\phi_x<-h\mid \phi_o=-h) < O(\log ^2 L) \exp\Big(-c\beta \frac{h^2}{\log^2 h}\Big) \leq \exp\Big(-c'\frac{\log L}{\log\log L}\Big)\,.\] Combined, we find that
\begin{align}\label{eq:Fc-given-h-upper-bound}
\hatpi_\infty(\sF_{\ell,h}^c\mid \phi_o=-h) &\leq e^{-c' \frac{\log L}{\log\log L}}\,,
\end{align}
and infer that
\begin{equation}\label{eq:P-lower-bound}
\bP_{\ell,h}(\phi_o = -h) \geq \left(1-e^{-c\frac{\log L}{\log\log L}}\right)\hatpi_\infty(\phi_o=-h)\,.
\end{equation}
Plugging \cref{eq:pairs-P-upper-bound,eq:P-lower-bound}, into \cref{eq:P-in-ex} yields
\begin{align*}
\bP_{\ell,h}(\phi_x > -h,\,\forall x\in Q_\ell) &\leq 1 - \Big(1-e^{-c\frac{\log L}{\log\log L}}\Big)|Q_\ell|\hatpi_\infty(\phi_o=-h) \\
&\leq \exp\bigg(-\Big(1 -e^{-c\frac{\log L}{\log\log L}}\Big)|Q_\ell|\hatpi_\infty(\phi_o=-h)\bigg)\,,
\end{align*}
yielding the required lower bound on $\bar\xi_{\ell,h}$.
\end{proof}

\subsection{Minimum height under the no-floor measure}\label{sec:bbq}
The following proposition is our main tool for expressing the probability under $\hatpi_V^0$ that $\phi_x\geq -h$ on a subset $F$ to the rate $\xi_{\ell_0,h}$ for a smaller scale $\ell_0$. It will be used to derive \cref{thm:key-area} using appropriate choices of $\ell_0$ (namely, $\ell_0=L^{-1/3+o(1)}$ for the mesoscopic scale and $\ell_0 = L^{-1/2-o(1)}$ for the macroscopic scale), as well as to relate the corresponding $\xi_{\ell_0,h}$ to the rate $\xi_{\ell_*,h}$, for $\ell_*$ from \cref{eq:xi-n-def}, appearing in \cref{thm:key-area}.

\begin{proposition}\label{prop:key-area-estimate}
    Fix $g\geq 0$, $n \geq 0$, and set $h = H+1-n$. There are absolute constants $c,C>0$ such that the following holds for every $\ell_0 \in \llb \log^2 L, \sqrt{L}e^{-\frac12\frac{\log L}{\log\log L}}\rrb$ and subsets 
    $F \subseteq V\subset \Z^2$ where $F$ has $g$ holes and $|\partial F| \leq L^{3/2}$.
    Let $\Xi = \Xi_1+\Xi_2+\Xi_3$, where 
    \begin{equation}\label{eq:Xi123-def} \Xi_1 =  \frac{|\partial V|}L  e^{C\frac{\log L}{\log\log L}}\,,\qquad \Xi_2 = \frac{|F|}
     {\ell_0 L } e^{ \sqrt{\log L}} \,,\qquad \Xi_3 =\frac{|\partial F|\ell_0 }L e^{-c\frac{\log L}{\log\log L}}\,.\end{equation}
    \begin{enumerate}[(i)]
    \item If $|F|\leq \ell_0 L e^{-\sqrt{\log L}}$, then
    \begin{align}\label{eq:area-estimate-meso}\hatpi^0_V(\phi_x \geq -h,\, \forall x \in F) = \exp\left(-\xi_{\ell_0,h}|F| +O(\Xi) + o(L^{-5}) \right)\,.
    \end{align}
\item Otherwise, if $\fS$ is the event that there are no disagreement polymers $\gamma$ with $|\gamma| \geq \log L$, then
\begin{align}\label{eq:area-estimate-macro-lower}
    \hatpi_V^0(\phi_x \geq -h,\, \forall x \in F) &\geq \exp\left(-\xi_{\ell_0,h}|F| + O(\Xi)\right)\,,\\
    \label{eq:area-estimate-macro-upper}
 \hatpi_V^0(\phi_x \geq -h,\, \forall x \in F\mid\fS) &\leq \exp\left(-\xi_{\ell_0,h}|F| + O(\Xi)+o(L^{-5})\right)\,.
\end{align}
If $|F| \leq  \big(\frac{3\beta}{\hatpi_\infty(\phi_o > h)}\big)^2 $, then the last upper bound also applies to $\hatpi_V^0(\phi_x\geq -h,\,\forall x\in F)$.
\end{enumerate}
    The same holds under $\hatpi_\infty$, in which case the error term $\Xi_1$ can be omitted from $\Xi$. Separately, in the special case where $F = \llb1,\ell\rrb^2$ with $ \ell_0 + 10\lceil \log L\rceil \mid \ell$, the error term $\Xi_3$ can be omitted from $\Xi$.
\end{proposition}
\begin{proof}
   Let us first establish \cref{eq:area-estimate-meso,eq:area-estimate-macro-lower,eq:area-estimate-macro-upper}, then argue why the proof extends to $\hatpi_\infty$ and how to omit $\Xi_1$ in that case, as well as $\Xi_3$ when $F$ is a square of side length $\ell$ divisible by $\ell_0 + 10\lceil \log L\rceil$.

Partition $\Z^2$ into squares $R_i$ of side length $\ell_0 + 10\lceil \log L\rceil$, and let $Q_i \subset R_i$ be the concentric squares of side length $\ell_0$. 
 With this tiling of $\Z^2$ in hand, partition $F$ into the sets
    \begin{align*}
        F_1 &= \{ x\in F\,:\; \dist(x,\partial V) < 5\log L\} &\mbox{(near boundary)}\,,\\
        F_2 &= \bigcup_{i:\;R_i\subset F} R_i \setminus (Q_i \cup F_1) &\mbox{(square annuli)}\,,\\
        F_3 &= \bigcup_{i:\,R_i\cap F^c\neq \emptyset} Q_i \cap (F\setminus F_1) &\mbox{(cut off squares)}\,,\\
        F_4 &= F \setminus (F_1 \cup F_2 \cup F_3)&\mbox{(full squares)}\,.
    \end{align*}
We will first argue that
\begin{align}
    \label{eq:bbq-asymp-to-F3-F4}
    \hatpi_V^0(\phi_x\geq -h,\,\forall x\in F) &= \hatpi_V^0\left(\phi_x\geq -h,\,\forall x\in F_3\cup F_4\right) e^{O(\Xi_1) + O(\Xi_2)}\,,\\
    \label{eq:bbq-asymp-to-F3-F4-add-xi}
    \xi_{\ell_0,h}(|F_1|+|F_2|) &= O(\Xi_1) + O(\Xi_2)\,,
\end{align}
for $\Xi_1,\Xi_2$ from \cref{eq:Xi123-def}, both assumed to be $o(1)$ or the bound is trivial.
To see this, note first that the right-hand of \cref{eq:bbq-asymp-to-F3-F4} is clearly an upper bound on the left-hand (without any error term) by event inclusion. For a lower bound, let us denote
\[ p_i := \max_{x\in F_i} \hatpi_V^0\left(\phi_x < -h\right)\qquad\mbox{for $i=1,2$}\,,\]
and apply FKG to see that
\begin{align*} \hatpi_V^0(\phi_x \geq -h\,,\forall x\in F) &\geq \hatpi_V^0\left(\phi_x \geq -h,\,\forall x\in F_3\cup F_4\right) \prod_{i=1,2}(1-p_i)^{|F_i|}\\
&\geq \hatpi_V^0\left(\phi_x \geq -h,\,\forall x\in F_3\cup F_4\right) \prod_{i=1,2}(1-p_i|F_i|)\,, 	
\end{align*}
and it remains to show that  $p_i |F_i| \leq \Xi_i$ and $\xi_{\ell_0,h} |F_i| \leq \Xi_i$ for $i=1,2$ and the $\Xi_i$'s as per \cref{eq:Xi123-def}. 
\begin{enumerate}[(1)]
\item 	Consider $i=1$. It was shown in \cite[\S2.1]{ChenLubetzky25} (specifically, combining Lemma 2.8 with Eq.~(2.7) there) that there exists some absolute constant $c>0$ such that, for every $V$, $h\geq 2$ and large enough $\beta$,
\[ \hatpi_V^0(\phi_x = h) \leq \hatpi_\infty(\phi_o = h) e^{ c \beta \frac{h^2}{\log^2 h}}\,.\]
As $\hatpi_\infty(\phi_o = h) / \hatpi_\infty(\phi_o = h-1) \leq \exp(-c\beta h/\log h)$ for all $h\geq 2$ (\cite[Eq.~(3.1)]{LMS16}), 
in our setting of $h=H+1-n$ 
we have $\hatpi_\infty(\phi_o<-h) = L^{-1} \exp(O(\sqrt{\frac{\log L}{\log\log L}}))$ and so, for some other absolute constant $c>0$, 
\[ p_1 \leq  L^{-1} e^{c \frac{\log L}{\log\log L}}\,.\]
The fact that $|F_1| = O(|\partial V|\log^2 L)$ now shows $p_1 F_1 \leq \Xi_1$ if the constant $C$ is  large enough. The bound $\xi_{\ell,h} \leq (1+o(1))\hatpi_\infty(\phi_o<-h)$ as per \cref{lem:xi-lower-weak-upper} shows that $\xi_{\ell_0,h} |F_1| \leq \Xi_1$ as well (with room to spare, as  $\hatpi_\infty(\phi_o=h)$ is better controlled than $\hatpi_V^0(\phi_x=h)$).

\item Consider $i=2$. The upper bound we gave on $p_1$ also holds for $p_2$, however it is imperative that in this case we improve the term $\exp(c \frac{\log L}{\log\log L})$ given there.  To this end, we use the standard decay of correlations in the low temperature \ZGFF model (e.g.,~\cite{BrandenbergerWayne82}), whereby  \begin{align}\label{eq:couple-pi-to-inf-vol}
\|\hatpi_{W}^0(\phi\restriction_{U}\in\cdot) - \hatpi_\infty(\phi\restriction_{U}\in\cdot)\|_{\tv} &\leq |\partial U|e^{-c \beta r}\quad\mbox{ if $U\subset W\subset \Z^2$ has $\dist(U,\partial W)> r$} 
\end{align}
(see, e.g., \cite[Eq.~(2.20)]{ChenLubetzky25} where this was stated for $U=\cB_r(o)$, a ball of radius $r$ about the origin; the exact same reasoning holds when $U$ is a general domain, as it hinges on a Peierls argument to encapsulate $U$ in a loop of height $0$ in $W\setminus U$). By definition, we  excluded $F_1$ from the rectangular annuli $R_i\setminus Q_i$ within $F_2$, and thus, for large enough $\beta$,
\[ p_2 \leq \hatpi_\infty(\phi_o<-h) + e^{-c \beta\log L} \leq (1+o(1))\hatpi_\infty(\phi_o<-h)\,. \]
Using $\hatpi_\infty(\phi_o = h) / \hatpi_\infty(\phi_o = h-1) \leq \exp(-c\beta h/\log h)$, we get
\[ p_2 \leq (1+o(1))\hatpi_\infty(\phi_o >h) \leq 
\frac1L e^{c (n+1)
\sqrt{\beta\frac{\log L}{\log \log L}}} = \frac1L e^{o(\sqrt{\log L})}\,. \]
Combining this with the fact that 
\[ |F_2| \leq 2 \frac{10\lceil\log L\rceil }{\ell_0 + 10\lceil\log L\rceil} |F| = O\Big(\frac{|F| \log L}{\ell_0}\Big)\]
shows $p_2 |F_2| = o(\Xi_2)$ if $L$ is large enough. Since $\xi_{\ell_0,h} \leq (1+o(1))\hatpi_\infty(\phi_o<-h)$ as per \cref{lem:xi-lower-weak-upper}, we obtain the exact same bound for $\xi_{\ell_0,h} |F_2|$.
\end{enumerate}
We have thus established \cref{eq:bbq-asymp-to-F3-F4,eq:bbq-asymp-to-F3-F4-add-xi}, reducing the proof of \cref{eq:area-estimate-macro-lower} to proving that
\begin{equation}
    \label{eq:bbq-reduce-to-F3-F4-LB}
    \hatpi^0_V(\phi_x \geq -h,\, \forall x \in F_3\cup F_4) \geq \exp\Big(-\xi_{\ell_0,h}\left(|F_3|+|F_4|\right) + o(\Xi_3)\Big)\,,
\end{equation}
with $\Xi_3$ as per \cref{eq:Xi123-def} with the following specific constant: letting $c_0>0$ denote the absolute constant from \cref{lem:xi-lower-weak-upper}, we set
\begin{equation}\label{eq:Xi3}
\Xi_3 := \frac{|\partial F| \ell_0}{L} e^{-(c_0/2) \frac{\log L}{\log\log L}}\,.
\end{equation}
Similarly, once \cref{eq:bbq-reduce-to-F3-F4-LB} is established, the proof of \cref{eq:area-estimate-meso} will be reduced to showing that, for $|F|\leq \ell_0 L e^{-\sqrt{\log L}}$,
\begin{equation}
    \label{eq:bbq-reduce-to-F3-F4-UB-meso}
    \hatpi^0_V(\phi_x \geq -h,\, \forall x \in F_3\cup F_4) \leq \exp\Big(-\xi_{\ell_0,h}\left(|F_3|+|F_4|\right) + o(\Xi_3)+o(L^{-5})\Big)\,,
\end{equation}
and the proof of \cref{eq:area-estimate-macro-upper} will be reduced to showing that, without this restriction on $|F|$,
\begin{equation}
    \label{eq:bbq-reduce-to-F3-F4-UB-macro}
    \hatpi^0_V(\phi_x \geq -h,\, \forall x \in F_3\cup F_4 \mid \fS) \leq \exp\Big(-\xi_{\ell_0,h}\left(|F_3|+|F_4|\right) + o(\Xi_3)+o(L^{-5})\Big)\,.
\end{equation}

We begin with the lower bound in \cref{eq:bbq-reduce-to-F3-F4-LB}, which will be a direct consequence of FKG:
\[ 
\hatpi_V^0(\phi_x\geq -h,\,\forall x\in F_3\cup F_4) \geq 
\Big(\min_{x\in F_3}\hatpi_V^0(\phi_x\geq -h\Big)^{|F_3|}
\prod_{Q_i \subset F_4}\hatpi_V^0(\phi_x\geq -h,\,\forall x\in Q_i)\,.
\]
Appealing once more to the coupling in \cref{eq:couple-pi-to-inf-vol} (bearing in mind that every $x\in F_3\cup F_4$ is at distance at least $\log L$ from $\partial F$ by construction), it follows that, for every $x\in F_3$, \[ \hatpi_V^0(\phi_x\geq -h) \geq 1-\hatpi_\infty(\phi_o<-h)-L^{-10} \geq \exp\left(-\hatpi_\infty(\phi_o<-h) - L^{-2}e^{o(\sqrt{\log L
})}\right)\,,\]
provided that $\beta$ is large enough (using $\hatpi_\infty(\phi_o<-h)\leq L^{-1}\exp(o(\sqrt{\log L}))$ in the last inequality).
Similarly, for every $Q_i$ accounted for in $F_4$,
\[
\hatpi_{V}^0(\phi_x \geq -h,\,\forall x\in Q_i) \geq \hatpi_\infty^0(\phi_x \geq -h,\,\forall x\in Q_i) - L^{-10} \geq e^{-\xi_{\ell_0,h}|Q_i| - L^{-9-o(1)})}\,. 
\]
Therefore, summing the $L^{-9-o(1)}$ error over a crude bound on $|F_4|\leq |F|\leq |\partial F|^2 \leq L^3$,
\[
\hatpi_V^0(\phi_x\geq -h,\,\forall x\in F_3\cup F_4) \geq \exp\left(-\left(\hatpi_\infty(\phi_o<-h) + L^{-2}e^{o(\sqrt{\log L})}\right)|F_3| -\xi_{\ell_0,h} |F_4| - o(L^{-5})\right)\,.
\]
To infer \cref{eq:bbq-reduce-to-F3-F4-LB} from the last display, it remains to show that \begin{equation}
\label{eq:F3-upper-Xi3}
\left| \hatpi_\infty(\phi_o<-h)+L^{-2}e^{o(\sqrt{\log L})} - \xi_{\ell_0,h}\right| |F_3| \leq o(\Xi_3)\,.
\end{equation}
By \cref{lem:xi-lower-weak-upper},  $\xi_{\ell_0,h} = (1+O(e^{-c_0\frac{\log L}{\log\log L}}))\hatpi_\infty(\phi_o<-h)$
for the same $c_0>0$ from above; thus, 
\begin{align*}
\left| \hatpi_\infty(\phi_o<-h)+ L^{-2}e^{o(\sqrt{\log L})}- \xi_{\ell_0,h}\right| |F_3| &\leq
e^{-(c_0-o(1))\frac{\log L}{\log\log L}} \frac{|F_3|}{L}\,,
\end{align*}
where we again plugged in the fact that $\hatpi_\infty(\phi_o<-h) = L^{-1} \exp(o(\sqrt{\log L}))$ for $h=H+1-n$ with fixed $n \geq 0$. Revisiting the definition of $\Xi_3$ in \cref{eq:Xi3}, it is thus left to show that
\begin{equation}
\label{eq:F3-bound}
|F_3| =  O(|\partial F| \ell_0)\,.
\end{equation}
(This is in lieu of the trivial bound 
$|F_3|\leq O(|\partial F|\ell_0^2)$, which would not suffice.)
We will argue that
\[ |F_3| \leq 4 (|\partial F|\ell_0 + (g+1)\ell_0^2) = O(|\partial F| \ell_0)\,. \]
To see this, view the squares $R_i$ partitioning the plane $\Z^2$ as the squares of a chessboard, and let $R_{i_j}$ be their subset corresponding to selecting only the light squares that lie in even files. By symmetry, the claim will follow once we show $|F_3 \cap \bigcup R_{i_j}| \leq |\partial F|\ell_0+(g+1)\ell_0^2$. Recall our hypothesis that $F$ has at most $g$ holes (this is the only point in the proof that requires that assumption); that is, $\partial_\sfv F$ can be decomposed into at most $g+1$ disjoint $*$-connected closed paths $P_0,\ldots,P_g$. Each $P_i$ visits at most $1+ |P_i|/\ell_0$ distinct squares $R_{i_j}$ (the shortest path from $R_{i_{j_1}}$ to the next $R_{i_{j_2}}$ has length at least $\ell_0+10\lceil\log L\rceil > \ell_0$), and each such square contributes at most $\ell_0^2$ sites to $F_3$ via its corresponding $Q_{i_j}$. Thus, $|F_3\cap \bigcup R_{i_j}| \leq (g+1 + |\partial F|/\ell_0) \ell_0^2 $, as claimed.
Having obtained \cref{eq:F3-bound}, we arrive at \cref{eq:F3-upper-Xi3} and thus conclude the proof of the lower bound as per \cref{eq:bbq-reduce-to-F3-F4-LB}. Consequently, \cref{eq:area-estimate-macro-lower} is established. 

We now turn to the proofs of the upper bounds as per \cref{eq:bbq-reduce-to-F3-F4-UB-meso,eq:bbq-reduce-to-F3-F4-UB-macro}.
To that end, we wish to expose the disagreement polymers along a grid of sites that will separate the $Q_i$'s in $F_3\cup F_4$, thereby inducing a product measure on $Q_i$'s that are each surrounded by a loop of sites at height~$0$. Recalling that every $Q_i$ that is part of $F_3\cup F_4$ has $\dist(Q_i,F_3\cup F_4 \setminus Q_i)\geq 5\log L$, and moreover, distinct such $Q_i,Q_j$ are at distance at least $10\log L$ from each other, we
define \[ U := \left\{ x \in F \,: 2 \log L \leq \dist(x,F_3\cup F_4) \leq 3 \log L\right\}\,,\]
and let 
\[ \Gamma_U := \left\{\mbox{$\gamma \in \phi$ incident to some $x\in U$}\right\}\,,\quad\Gamma_U^\dagger := \left\{\gamma\in\Gamma_U\,:\; |\gamma|\geq \log L\right\}\,.\]
If we further define the events 
\[ A = \left\{\mbox{$\exists \gamma\in\Gamma_U^\dagger$ with $|\gamma|\geq \frac1{10}\ell_0$}\right\}\,,\quad B_k=\left\{|\Gamma_U^\dagger| = k\right\}\] then the standard Peierls estimate yields that
\[ \hatpi_V^0(A) \leq e^{-(\beta-C)\ell_0}\,,\]
and, for every $k\geq 1$, 
\[\hatpi_V^0(B_k) \leq |F|^{k} e^{-(\beta-C) k\cdot \frac1{10} \log L} \leq L^{-(\beta-C') k/10 }\,.\]

For \cref{eq:bbq-reduce-to-F3-F4-UB-meso}, recall $\hatpi_\infty(\phi_o< -h) \leq L^{-1}e^{o(\sqrt{\log L})}$, and note that if $|F| \leq \ell_0 L e^{-\sqrt{\log L}}$ then
\begin{equation}\label{eq:area-estimate-small-F}|F|\hatpi_\infty(\phi_o<-h) \leq  e^{-(1-o(1))\sqrt{\log L}} \ell_0 =  o(\ell_0)\,,
\end{equation}
in which case we can absorb the additive $\hatpi_V^0(A)$ term as a multiplicative $1+O(e^{-(\beta-C-o(1))\ell_0})$ factor in front of our main term of 
$e^{-\xi_{\ell_0,h}|F|}$; that is, this adds a negligible term of $O(e^{-c \ell_0}) =o(L^{-5})$ to our exponent (using here that $\ell_0 \geq \log^2 L$). Therefore, it will suffice to prove \cref{eq:bbq-reduce-to-F3-F4-UB-meso} conditional on $A^c$. (NB.\ that the same holds for $\hatpi_V^0(\bigcup\{ B_k:\,k\geq \ell_0 / \log L\})$, so in principle we could have restricted also to $\bigcup_{k\leq \ell_0} B_k$, but the bound on $\hatpi_V^0(B_k)$ is good enough to allow to sum over all $k$'s without needing this threshold.)

Consider some $Q_i$ intersecting $F_3\cup F_4$. We say that $Q_i$ is \texttt{good} if one can find in $U$, for each connected component $S$ of $Q_i \cap (F_3\cup F_4)$, a $*$-adjacent loop of sites $\cC$ at distance at most $3\log L$ from $S$ such that $\phi\restriction_{\cC}=0$. (NB.\ If $Q_i$ is not \texttt{good} then there must be some $\gamma\in \Gamma_U^\dagger$ intersecting $R_i$.)

By \cref{eq:couple-pi-to-inf-vol} and the definition of $\xi_{\ell_0,h}$ (as well as $e^{-\xi_{\ell_0,h}|Q_i|} \geq 1 - \ell_0^2 L^{-1}e^{o(\sqrt{\log L})} \geq 1-o(1) $, so we can write the additive $O(L^{-10})$ coupling error  as a $1+O(L^{-10})$ factor), for every $Q_i\subset F_4$,
\[ \hatpi_V^0(\phi_x\geq -h,\,\forall x\in Q_i \mid \Gamma_U,\,\mbox{$Q_i$ is \texttt{good}}) = \exp\left(-\xi_{\ell_0,h}|Q_i| + O(L^{-10})\right)\,.\]
Similarly, if $Q_i\cap F_3\neq\emptyset$ then by \cref{lem:xi-lower-weak-upper} (noting that indeed $\ell_0 \leq \sqrt{L} \exp(-\frac{\log L}{\log\log L})$ and appealing to the remark in that lemma that allows its application to $Q_i \cap F_3$, a region that is not necessarily connected), for the constant $c_0>0$ from that lemma we have
\[ \hatpi_V^0(\phi_x\geq -h,\,\forall x\in Q_i \cap F_3 \mid \Gamma_U,\,\mbox{$Q_i$ is \texttt{good}}) \leq \exp\left(-\hatpi_\infty(\phi_o<-h)\Big(1+ O(e^{-c_0 \frac{\log L}{\log \log L}}\Big)|Q_i \cap F_3|\right)\,.\]
All the \texttt{good} $Q_i$'s are conditionally independent given $\Gamma_U$ (stemming from the fact that their pairwise distances are all at least $10\log L$).
Finally, for each $k=0,\ldots,\ell_0/\log L$ and every realization of $\Gamma_U$ consistent with $A^c \cap B_k$, all the $Q_i$'s accounted in $F_3\cup F_4$ are \texttt{good} except for at most $4k$.

Note that  the two terms that we have encountered in the last two displays (when conditioning on $Q_i$ being \texttt{good}), namely $\xi_{\ell_0,h}|Q_i|$ or $\hatpi_\infty(\phi_o<-h)|Q_i\cap F_3|$, are each at most 
\[ (1+o(1))\hatpi_\infty(\phi_o<-h)\ell_0^2 \leq e^{- (1-o(1))\frac{\log L}{\log\log L}}< e^{-\frac34\frac{\log L}{\log\log L}}\]
for any sufficiently large $L$, using here our upper bound on $\ell_0$. We can therefore account for the (at most $4k$) missing $Q_i$'s that are not \texttt{good} from each of the last two displays---increasing $\sum |Q_i|$ to $|F_4|$ in the first and $\sum |Q_i\cap F_3|$ to $|F_3|$ in the second---by adding a term $4k e^{-\frac34 \frac{\log L}{\log\log L}}$ in the exponent. That is, summing the $O(L^{-10})$ error over $|F_4|\leq|F|\leq L^3$ we can infer that
\begin{align*}
\hatpi_V^0(\phi_x\geq -h,\,\forall x\in F_3\cup F_4\mid \Gamma_U, A^c\cap B_k) \leq \exp\bigg(& -\xi_{\ell_0,h}|F_4| +O(L^{-7})\\
&-\hatpi_\infty(\phi_o<-h)\big(1
+O(e^{-c_0\frac{\log L}{\log\log L}})\big)|F_3|\\
&+ 4k e^{-\frac34\frac{\log L}{\log\log L}}\bigg)\,.
\end{align*}
Turning our attention to the term $4 k e^{-\frac34\frac{\log L}{\log\log L}}$, we have
\[ \sum_{k \geq 0}
\hatpi_V^0(B_k) e^{4k e^{-\frac34\frac{\log L}{\log\log L}}} \leq \sum_{k\geq 0}
\exp\bigg(-\Big[(\beta-C)\log L - 4 e^{-\frac34\frac{\log L}{\log\log L}}\Big]k\bigg) \leq 1+O(L^{-10})\,,\]
provided that $\beta$ is large enough. Combining the last two displays, we see that
\begin{align}
\hatpi_V^0\big(\phi_x\geq -h,\,&\forall x\in F_3\cup F_4\mid \Gamma_U, A^c\big) \nonumber\\ &\leq \exp\Big( -\hatpi_\infty(\phi_o<-h)\big(1+ O(e^{-c_0\frac{\log L}{\log\log L}})\big)|F_3| -\xi_{\ell_0,h} |F_4| + o(L^{-5})\Big)\,.
\label{eq:upper-bound-F3-F4}
\end{align}
We have already seen in \cref{eq:F3-upper-Xi3} that we may replace $\hatpi_\infty(\phi_o<-h)|F_3|$ by $\xi_{\ell_0,h}|F_3| + o(\Xi_3)$. The same reasoning allows us to neglect the term $\exp(-c_0 \frac{\log L}{\log\log L})\hatpi_\infty(\phi_o<-h)|F_3|$, as \cref{eq:F3-bound} gives
\[ e^{-c_0\frac{\log L}{\log\log L}}\hatpi_\infty(\phi_o<-h)|F_3| \leq e^{-(c_0-o(1))\frac{\log L}{\log\log L}} \frac{|\partial F|\ell_0}{L} = o(\Xi_3)\,.\]
again by the definition of $\Xi_3$ in \cref{eq:Xi3} and the choice of the constant in the exponent there. We have therefore established \cref{eq:bbq-reduce-to-F3-F4-UB-meso}, which, when combined with \cref{eq:bbq-reduce-to-F3-F4-LB}, completes \cref{eq:area-estimate-meso}.

To obtain \cref{eq:bbq-reduce-to-F3-F4-UB-macro}, the situation is simpler, since $A$ is precluded by the conditioning on $\fS$, which in addition implies that $B_0$ holds (that is, $\Gamma_U^\dagger = \emptyset$). Each $Q_i$ intersecting $F_3\cup F_4$ is thus guaranteed to be \texttt{good} conditionally on $\Gamma_U,\fS$ (and these $Q_i$'s are conditionally independent). Thus, the only modification we need to make to the preceding argument is to control the effect that the conditioning on $\fS$ has on the probability that $\phi_x\geq -h$ in $Q_i$. To this end, we may bound
\begin{align*} \hatpi_V^0(\phi_x \geq -h,\,\forall x\in Q_i\mid \Gamma_U,\,\fS) &\leq \frac{\hatpi_V^0(\phi_x \geq -h,\,\forall x\in Q_i\mid \Gamma_U,\,\mbox{$Q_i$ is \texttt{good}})}{\hatpi_V^0(\fS)} \\
    &\leq  \left(1+O(\hatpi_V^0(\fS^c)\right) \hatpi_V^0(\phi_x \geq -h,\,\forall x\in Q_i\mid \Gamma_U,\,\mbox{$Q_i$ is \texttt{good}})
\,,\end{align*}
and the same holds for $\hatpi_V^0(\phi_x\geq -h,\,\forall x\in Q_i\cap F_3\mid \Gamma_U,\,\fS)$. Since $\hatpi_V^0(\fS^c) < L^{-10}$, accumulating this error over all the $Q_i$'s results in an additive term of $O(L^{-8})$ in the exponent, hence the analog of \cref{eq:upper-bound-F3-F4} holds true for $\hatpi_V^0(\cdot\mid\Gamma_U,\,\fS)$. This establishes \cref{eq:bbq-reduce-to-F3-F4-UB-macro}, and hence also \cref{eq:area-estimate-macro-upper}. 

In the special case where $|F| \leq  \big(\frac{3\beta}{\hatpi_\infty(\phi_o > h)}\big)^2 $, we obtain the unconditional version of \cref{eq:area-estimate-macro-upper} as a consequence of \cref{cor:no-macro-contours}. Let us apply the elementary bound $\P(A)\leq \frac{\P(A)}{\P(B)} =\frac{\P(A\mid B)}{\P(B\mid A)}$, valid whenever $\P(A\cap B)>0$, to the events $A=\{\phi_x\geq -h,\,\forall x\in F\}$ and $B = \fS$ under the measure $\hatpi_V^0$. In that notation, we wish to extend our bound on $\P(A\mid B)$ to $\P(A)$; 
it thus suffices to show that
\[ \hatpi_V^0(\fS \mid \phi_x\geq -h,\,\forall x\in F) \geq 1 - o(L^{-5})\,.\]
Indeed, the quantity on the left is nothing but $\pi^h_{V;F}(\fS)$, which is $1-O(L^{-10})$ by \cref{cor:no-macro-contours}.

It remains to address the remark in the proposition pertaining to the improved error terms. 
When working under $\hatpi_\infty$, as opposed to $\hatpi_V^0$, we may apply the same proof only with $F_1=\emptyset$, as there is no longer a need to exclude the sites near $\partial V$ in order to support a coupling with $\hatpi_\infty$. Hence, we avoid the error term $\Xi_1$ formerly associated with $F_1$.
Similarly, when $F$ is a square of side length $\ell$ such that $\ell_0 + 10\lceil\log L \rceil \mid \ell$, we can partition $F$ into squares $R_i$'s without any residue, so $F_3=\emptyset$, forgoing the error term $\Xi_3$ that was formerly associated with $F_3$, as claimed.
\end{proof}

\subsection{Proofs of the bound on \texorpdfstring{$\upxi_n$}{xi} in \cref{prop:uprho-bound}, and \cref{thm:key-area}}
Recall that $\upxi_n$ is $\xi_{\ell_*,h}$ for $\ell_*$ as per \cref{eq:xi-n-def}. In \cref{sec:xi-and-xi-bar-bounds} we gave bounds on $\xi_{\ell,h}$ for $\ell < \sqrt{L}\exp(-C\frac{\log L}{\log\log L})$, and with the help of \cref{prop:key-area-estimate}, we can now extend those to $\ell_* \asymp \sqrt{L}$ to derive \cref{eq:upxi-bounds} from \cref{prop:uprho-bound}.
Thereafter, a more delicate application of that proposition will yield \cref{thm:key-area}.
\begin{proof}[Proof of \cref{prop:uprho-bound}, \cref{eq:upxi-bounds}]
We will relate $\xi_{\ell_*,h}$ to $\xi_{\ell_0,h}$ for 
\[ \ell_0 = \lfloor L^{1/4} \rfloor\,.\]
Applying \cref{prop:key-area-estimate} for $\hatpi_\infty$ and $F=\llb1,\ell_*\rrb^2$, we get from \cref{eq:area-estimate-meso} (noting that $|F| = \ell_*^2 \asymp L$, which is indeed smaller than  $\ell_0 L e^{-\sqrt{\log L}}=L^{5/4-o(1)}$, fulfilling the  hypothesis\footnote{This application of \cref{prop:key-area-estimate} would be valid for any  $\ell_* \leq L^{5/8-o(1)}$. More generally, if we set $\ell_0 = L^{1/2-o(1)}$, we could have extended the result up to $\ell_*\leq L^{3/4-o(1)}$.}) that
\[ \hatpi_\infty(\phi_x\geq -h,\,\forall x \in F) = \exp\bigg(-\xi_{\ell_0,h}\ell_*^2 + O\Big(\frac{\ell_*^2}{\ell_0} L^{-1+o(1)}+ \ell_* \ell_0 L^{-1-o(1)}\Big)+
o(L^{-5})\bigg)\,.\]
The dominant error term is $(\ell_*^2/\ell_0) L^{-1+o(1)}$ (originating from $\Xi_2$ in the above proposition), and so
\begin{equation}\label{eq:xi*-bound}\xi_{\ell_*,h} = -\frac1{\ell_*^2}\log\hatpi_\infty(\phi_x\geq -h,\,\forall x\in F) = \xi_{\ell_0,h} + L^{-5/4+o(1)}\,.\end{equation}
In particular, since $\xi_{\ell_0,h} = (1+o(1))\hatpi_\infty(\phi_o<-h)=L^{-1+o(1)}$, we infer that 
\[ \frac{\upxi_n}{\hatpi_\infty(\phi_o<-h)} = \frac{\xi_{\ell_*,h}}{\hatpi_\infty(\phi_o<-h)} = \left(1+L^{-1/4+o(1)}\right)\frac{\xi_{\ell_0,h}}{\hatpi_\infty(\phi_o<-h)}\,;\]
thus, \cref{eq:upxi-bounds} follows from \cref{lem:xi-lower-weak-upper,lem:xi-refined-upper} for $\xi_{\ell_0,h}$.

Having established \cref{eq:upxi-bounds}, we note that at this point we can also infer the following weaker version of the bound in \cref{eq:uprho-bounds} on $\uprho_n$:
\begin{align}         \label{eq:uprho-bounds-weak}
        \uprho_n &\in [1 - e^{-C\frac{\log L}{\log\log L}},  1 + L^{-1+o(1)}]\,.
        \end{align}
Indeed,
\cref{eq:xi*-bound} implies that $\bar\xi_{\ell_*,h} = \xi_{\ell_*,h-1} - \xi_{\ell_*,h}$ satisfies
\[ \bar\xi_{\ell_*,h} = \bar\xi_{\ell_0,h} + L^{-5/4+o(1)}\,,\]
and the fact that $\bar\xi_{\ell_0,h} =L^{-1+o(1)}$ shows that
\begin{equation}\label{eq:rho_n_via_xibar_ell0} \uprho_n = \frac{\bar\xi_{\ell_*,h}}{\hatpi_\infty(\phi_o=-h)} = \left(1+L^{-1/4+o(1)}\right)\frac{\bar\xi_{\ell_0,h}}{\hatpi_\infty(\phi_o=-h)}\,;\end{equation}
hence, \cref{eq:uprho-bounds-weak} follows from \cref{lem:xi-bar-lower-weak-upper} applied to $\bar\xi_{\ell_0,h}$.
\end{proof} 

\begin{proof}[Proof of \cref{thm:key-area}]
We aim to apply \cref{prop:key-area-estimate} for $\ell_0 = L^{1/3+o(1)}$ to obtain \cref{eq:xi-n-area-estimate-meso} (mesoscopic range), and for $\ell_0 = L^{1/2+o(1)}$ to establish \cref{eq:xi-n-area-estimate-macro-lower,eq:xi-n-area-estimate-macro-upper} (macroscopic range). In order to relate the corresponding $\xi_{\ell_0,h}$ rates to $\upxi_n$, unlike the proof of \cref{prop:uprho-bound}, it will be imperative to fulfill the divisibility condition of \cref{prop:key-area-estimate} and avoid the error term $\Xi_3$ there.

  Let $\ell_1 \in \llb L^{1/3}, \sqrt{L}e^{-\frac12\frac{\log L}{\log\log L}}\rrb$ be an integer satisfying $\ell_1 \mid \ell_*$, and let $\ell_0 = \ell_1 - 10\lceil \log L\rceil$. 
  Applying \cref{prop:key-area-estimate} for $\hatpi_\infty$ and $F=\llb1,\ell_*\rrb^2$ (noting that $\ell_0 L e^{-\sqrt{\log L}} \geq L^{4/3-o(1)}$, whereas $\ell_*^2\asymp L$, qualifying an application of \cref{eq:area-estimate-meso}), we find that
 \[ \hatpi_\infty(\phi_x\geq -h,\,\forall x \in F) = \exp\bigg(-\xi_{\ell_0,h}\ell_*^2 + O\Big(\frac{\ell_*^2}{\ell_0 L} e^{\sqrt{\log L}}\Big)+ o(L^{-5})\bigg)\,,\]
 and therefore
 \begin{equation}\label{eq:xi*-bound-refined}\left|\xi_{\ell_*,h} - \xi_{\ell_0,h}\right| = O\Big( \frac1{\ell_0 L} e^{\sqrt{\log L}}\Big) \,.\end{equation}

 For \cref{eq:xi-n-area-estimate-meso}, let $\ell_0 = \ell_1 - \lceil 10\log L\rceil$ for
 \[  \ell_1 = 2^{\lceil \frac13 \log L + (4 \log 2)\sqrt{\log L} \rceil} \asymp L^{1/3} e^{4\sqrt{\log L}}\,.\]
 Our assumption that $|\partial F|\vee|\partial V| \leq O(L^{2/3} e^{\sqrt{\log L}})$ implies  $ |\partial F| \leq 10 L$ and $|F|\leq O(L^{4/3} e^{2\sqrt{\log L}})$, whereas $\ell_0 L e^{-\sqrt{\log L}} \geq L^{4/3} e^{3\sqrt{\log L}}$. Hence, we may appeal to \cref{eq:area-estimate-meso} of \cref{prop:key-area-estimate}, so
 \begin{align*}
 \hatpi_V^0(\phi_x\geq -h,\,\forall x \in F) &= \exp\Big(-\xi_{\ell_0,h}|F| + O(\Xi) + o(L^{-5})\Big)\,,
 \end{align*}
 with $\Xi=\Xi_1+\Xi_2+\Xi_3$ for
 \begin{align*}
 \Xi_1 &= \frac{|\partial V|}{L} e^{C\frac{\log L}{\log\log L}} \leq L^{-1/3+o(1)}\,,\\
 \Xi_2 &\leq \frac{|\partial F|^2}{\ell_0 L} e^{\sqrt{\log L}} \leq e^{-\sqrt{\log L}}\,,\\
 \Xi_3 &= \frac{|\partial F|\ell_0}{L} e^{-c \frac{\log L}{\log\log L}} \leq  e^{-(c-o(1))\frac{\log L}{\log\log L}}\,,
 \end{align*}
 whence 
 \[ \Xi=o(1)\,.\] 
 Appealing to \cref{eq:xi*-bound-refined}, we may replace $\xi_{\ell_0,h}$ by $\upxi_n = \xi_{\ell_*,h}$ and incur an error of at most
 \[ \frac{|F|}{\ell_0 L} e^{\sqrt{\log L}} \leq e^{-\sqrt{\log L}}\,,\]
 thus establishing \cref{eq:xi-n-area-estimate-meso}.
 
 For \cref{eq:xi-n-area-estimate-macro-lower,eq:xi-n-area-estimate-macro-upper}, let $\ell_0 = \ell_1 - \lceil 10\log L\rceil$ for
 \[  \ell_1 = 2^{\lceil \frac12 \log L - (\frac34\log 2)\frac{\log L}{\log\log L} \rceil} \asymp \sqrt{L} e^{-\frac34\frac{\log L}{\log\log L}}\,.\]
 By assumption, $|\partial F| \vee|\partial V|\leq O(Le^{\sqrt{\log L}})$, and \cref{eq:area-estimate-macro-lower,eq:area-estimate-macro-upper} yield
 \begin{align*}
 \hatpi_V^0(\phi_x\geq -h,\,\forall x \in F) &\geq \exp\Big(-\xi_{\ell_0,h}|F| + O(\Xi) + o(L^{-5})\Big)\,,\\
\hatpi_V^0(\phi_x\geq -h,\,\forall x \in F\mid\fS) &\leq \exp\Big(-\xi_{\ell_0,h}|F| + O(\Xi) + o(L^{-5})\Big)\,,
 \end{align*}
 with $\Xi=\Xi_1+\Xi_2+\Xi_3$ for
 \begin{align*}
 \Xi_1 &= \frac{|\partial V|}{L} e^{C\frac{\log L}{\log\log L}} \leq e^{(C+o(1))\frac{\log L}{\log\log L}} \,,\\
 \Xi_2 &\leq \frac{|\partial F|^2}{\ell_0 L} e^{\sqrt{\log L}} \leq \sqrt L e^{(\frac34-o(1))\frac{\log L}{\log\log L}}\,,\\
 \Xi_3 &= \frac{|\partial F|\ell_0}{L} e^{-c \frac{\log L}{\log\log L}} \leq  \sqrt{L} e^{-(c+\frac34-o(1))\frac{\log L}{\log\log L}}\,,
 \end{align*}
so that, for any sufficiently large $L$,
\[ \Xi \leq \sqrt{L} e^{\frac{\log L}{\log\log L}}\,.\]
(The remark regarding the validity of the upper bound without conditioning on $\fS$ in the special case where 
$|F|\leq  \big(\frac{3\beta}{\hatpi_\infty(\phi_o > h)}\big)^2$ follows from the statement below \cref{eq:area-estimate-macro-upper} in \cref{prop:key-area-estimate}.)
Revisiting \cref{eq:xi*-bound-refined}, we replace $\xi_{\ell_0,h}$ by $\upxi_n = \ell_{\ell_*,h}$ at a cost of 
 \[ \frac{|F|}{\ell_0 L} e^{\sqrt{\log L}} \leq \sqrt{L} e^{(\frac34-o(1))\frac{\log L}{\log\log L}}\,,\]
 obtaining \cref{eq:xi-n-area-estimate-macro-lower,eq:xi-n-area-estimate-macro-upper} and completing the proof (NB.\ Had we chosen $\ell_1 \nmid \ell_*$, \cref{eq:xi*-bound-refined} would have instead resulted in an untenable cost of $(\ell_0/\ell_*)|F|L^{-1+o(1)}=L^{1-o(1)}$.)
\end{proof}

\subsection{Sharper upper bound on \texorpdfstring{$\uprho_n$}{rho}}\label{sec:sharp-UB-uprho}
We conclude this section with the proof of \cref{eq:uprho-bounds}, featuring a sharper upper bound compared to the one in \cref{eq:uprho-bounds-weak}, which already established the sought lower bound. 
Although this sharper bound is not needed for the proof of \cref{thm:main-thm-crit-window}, it is necessary for the remark following it, that the critical window from \cref{thm:main-thm-crit-window} excludes the natural prediction $\lambda_*\beta/\hatpi_\infty(\phi_o=h)$; see \cref{rem:bad-prediction} for more details. 

We will need the following result (which may be of independent interest) establishing that the \ZGFF, conditional on $\phi\restriction_{B_r(o)}=h$, is rigid, beyond scale $h/\log h$, about the corresponding real-valued harmonic solution $\phi^*$.

\begin{proposition}
    \label{prop:pi-y-h-phi^*-rigid}
Fix $r \geq 0$. Let $h \geq 1$, $R = \lceil 20 h/\log h\rceil$, and let $\phi^*$ be the $\R$-valued harmonic function in $A_{r,R} := B_R(o)\setminus B_r(o)$ with boundary conditions $h$ on $\partial B_r(o)$ and $0$ on $\partial B_R(o)$. There exist an absolute constant $C>0$ such that, for every sufficiently large $\beta$ and every $z\notin B_r(o)$,
\[ \hatpi_\infty\left(\phi_z  \notin \Big[\phi^*_z-C \frac{h}{\log h}, \phi^*_z+C\frac{h}{\log h}\Big]\;\Big|\; \phi_x = h,\, \forall x\in B_r(o)\right) \leq e^{-\frac15(\beta-C) \frac{h}{\log h}}\,.\]
\end{proposition}
\begin{proof}
From the explicit solution to the Dirichlet problem stated in the proposition in terms of hitting times of simple random walk $S_t$ in $\Z^2$ (see, e.g., \cite[\S1.4]{Lawler_RW}), one has that
\begin{equation}\label{eq:phi*-expression-1} \phi^*_x=h \P_x(\tau_{\partial B_r(o)}<\tau_{\partial B_R(o)})\qquad\mbox{for all $x\in A_{r,R}$}\,,\end{equation}
where $\P_x$ denote the law of the random walk $S_t$ started at $x$ and $\tau_A$ is the first time it hits a set $A$. It is well-known that the potential kernel $a(x)$ satisfies that $a(S_t)$ is a martingale, and from known estimates on $a(x)$ (see, e.g.,~\cite[\S1.6]{Lawler_RW}, and in particular Prop.~1.6.7 and Example~1.6.8 there), 
\begin{equation}\label{eq:phi*-expression-2} \P_x(\tau_{\partial B_r(o)}>\tau_{\partial B_R(o)}) = \frac{\log |x| - \log r + O(1/r)}{\log R - \log r}\,,\end{equation}
where $|x|$ denotes the Euclidean distance of $x$ from the origin. We will use the following fact on $\phi^*$, which, just like the last display, follows from the Optional Stopping Theorem and known estimates for $a(x)$ (e.g., those in \cite[p.~125--127] {Spitzer_PrinciplesRW}, combined with a Taylor approximation to account for the lattice error $|S_{\tau_{B_r(o)}}|\in [r,r+1]$, whence $\log|S_{\tau_{B_r(o)}}|\in [\log r, \log r + \frac1r]$, and similarly for $B_R(o)$):
\begin{equation}\label{eq:p-near-boundary}
\phi_x^* \leq \frac1{64}\qquad\mbox{for all $x\in A_{r,R}$ adjacent to some $y\in \partial B_R(o)$}\,.
\end{equation}

Several steps will be needed to deal with boundary effects and different large deviation scenarios.

\begin{enumerate}[label=\textbf{Step~\arabic*.}, ref=\arabic*, wide=0pt, itemsep=1ex,
listparindent=\parindent, parsep=0pt
]
    \item \label[step]{st:1-subtract-phi*} (\emph{Subtracting the real-valued solution $\phi^*$}.)
Let \[ \sigma := \phi - \phi^*\,.\]
Denoting the discrete Laplacian by 
$(\Delta \phi)_x = \frac14 \sum_{y\sim x} (\nabla \phi)_{xy}$ for $(\nabla \phi)_{xy} = \phi_y-\phi_x$, we see that
\[ \sum_{x\sim y}|(\nabla\phi)_{xy}|^2 
=  \sum_{x\sim y}|(\nabla\sigma)_{xy}|^2
 + \sum_{x\sim y}|(\nabla\phi^*)_{xy}|^2
 - 8 \sum_{x} \sigma_x (\Delta \phi^*)_x\,.
\]
Since $\sum_{x\sim y}|(\nabla\phi^*)_{xy}|^2$ is but a constant, $(\Delta \phi^*)_x=0$ for all $x\notin \partialvtx A_{r,R}$ (the external vertex boundary of $A_{r,R}$), and $\sigma_x=0$ for $x\in B_r(o)$, we see that the law of $\sigma$ under $\hatpi_\infty$ can be written as
\begin{equation}\label{eq:mu-def} \hatmu_\infty(\sigma) = \frac1{Z_{\hatmu}} \exp\bigg(-\beta \sum_{x\sim y} |(\nabla\sigma)_{xy}|^2 + 8\beta \sum_{x\in\partialvtx B_R(o)} (\Delta\phi^*)_x \sigma_x \bigg)\,,\end{equation}
where $Z_\hatmu$ sums the exponent on the right-hand over  all configurations $\sigma:\Z^2\to \R$ with
\begin{enumerate}
    [(i)]
    \item $\sigma_x = 0$ for all $x\in B_r(o)$;
    \item $\sigma_x \in -\phi^*_x + \Z$ for all $x$.
\end{enumerate}
Using \cref{eq:p-near-boundary} (as well as that $\phi^* \geq 0$ and $\phi^*\restriction_{B_R(o)^c}=0$) we further have
\begin{equation}
    \label{eq:Delta-phi*-max}0 \leq (\Delta\phi^*)_x \leq \frac1{64}\quad\mbox{for all $x\in \partialvtx B_R(o)$}\,,
\end{equation} 
and see that the distribution $\hatmu_\infty$ is nothing but a \ZGFF distribution with a small positive external field of $\Delta\phi^*$ on $\partialvtx B_R(o)$ (and values that are the integers translated by the fractional parts of $-\phi^*$).

\item \label[step]{st:2-steep-phi*} (\emph{Ruling out many \texttt{steep} level lines in $\lfloor\sigma\rfloor$.}) 
The fact that $\sigma$ is $\R$-valued hinders the Peierls approach for establishing rigidity: for instance, if one were to look at $\lfloor\sigma\rfloor$ (point-wise floor of $\sigma$), erasing a level-line loop via an appropriate vertical shift of its interior may actually increase the energetic cost if (consider, for instance, the situation when the gradients along it have fractional parts $0.1,0.9$). A similar phenomenon occurs when rounding $\sigma$ to the nearest integers (e.g., consider fractional parts $0.4, 0.6$ in that setting). To deal with this, we introduce the following notion of a \texttt{steep} level line.

\begin{definition}\label{def:level-line-annuli} We say $\sA$ is a $j$ level-line \emph{annulus} in  $\lfloor\sigma\rfloor$, for some $j\in\Z$, if it consists of two $j$~level-line loops $\sA^{\textsf{in}},\sA^{\textsf{out}}$ in $\lfloor\sigma\rfloor$ as per \cref{def:level-lines} with reversed types---$\sA^{\textsf{out}}$ is an up-loop and $\sA^{\textsf{in}}$ is a down-loop (whence $\sA$ is an up-annulus) or vice versa (a down-annulus)---so that either
\begin{enumerate}[(a)]
\item{}(trivial annulus, one external loop) $\sA^{\textsf{in}}=\emptyset$ and $B_r(o)\cap \Int(\sA^{\textsf{out}}) = \emptyset$; or
    \item{}(nontrivial annulus) $B_r(o)\subset \Int( \sA^{\textsf{in}}) \subset \Int(\sA^{\textsf{out}})$.
\end{enumerate}
Write $|\sA| := |\sA^{\textsf{in}}|+|\sA^{\textsf{out}}|$
and
$\Int(\sA) := \Int(\sA^{\textsf{out}})\setminus\Int(\sA^{\textsf{in}})$.
\end{definition}
Note that for any level line $\fL$ of $\lfloor\sigma\rfloor$, either $B_r(o)\subset \Int(\fL)$ or $B_r(o)\subset \Int(\fL)^c$ since $\sigma\restriction_{B_r(o)}=0$. The above definition requires that if $B_r(o)\subset \Int(\sA^{\textsf{out}})$ then there must be an internal loop, of a reversed type, separating $B_r(o)$ from $\sA^{\textsf{out}}$. Further note that, since $\sA^{\textsf{in}},\sA^{\textsf{out}}$ are two $j$ level-line loops that are nested ($\Int(\sA^{\textsf{in}}) \subset \Int(\sA^{\textsf{out}})$) and of reversed types, then they must be disjoint. 

Observe that if $\lfloor \sigma_z \rfloor = k > 0$, then there must exist level line up-annuli $\sA_1,\ldots,\sA_k$ in $\lfloor\sigma\rfloor$ corresponding to heights $k,k-1,\ldots,1$ such that $z\in\Int(\sA_1)\subset\ldots\subset\Int(\sA_k)$. (This would be easy to see if in lieu of $\hatpi_\infty$ we considered $\hatpi^0_{B_M(o)}$ for $M\gg R$ (and the same for $\hatmu_\infty$), and one can couple the two measures to agree on $B_{M/2}(o)$ up to a probability $e^{-c\beta M}$, vanishing as $M\to\infty$.) Similarly, if $\lfloor \sigma_z \rfloor = -k$ for $k>0$ we can find such down-annuli $\sA_1,\ldots,\sA_k$ at heights $-k+1,-k+2,\ldots,0$.

\begin{definition}\label{def:steep-level-line}
We say a $j$ level-line loop $\fL$ in $\lfloor \sigma\rfloor$ is \texttt{steep} if at least $\frac23 |\fL|$ of its bonds $b$ satisfy $|(\nabla \lfloor\sigma\rfloor)_{xy}|\geq 2$, where $xy$ is the edge of $\Z^2$ dual to $b$. Similarly, we say that a $j$ level-line annulus~$\sA$ in $\lfloor\sigma\rfloor$ is \texttt{steep} if $\frac23|\sA|$ of the bonds $b$ in $\sA^{\textsf{in}}\cup \sA^{\textsf{out}}$ satisfy $|(\nabla \lfloor\sigma\rfloor)_{xy}|\geq 2$ for the $xy$ dual to $b$. 
\end{definition}
The following claim will control the probability of the occurrence of $k$ \texttt{steep} level lines annuli.
\begin{claim}\label{clm:k-steep-level-lines}
Let $\mathscr{S}_{z,k}^\uparrow$ be the event that there is a sequence $\{\sA_i\}_{i=1}^k$ of \texttt{steep} $j_i$ level line up-annuli in $\lfloor\sigma\rfloor$ for some $j_1> \ldots > j_k > 0$, such that $z \in \Int(\sA_1) \subset \Int(\sA_2)\subset\ldots \subset \Int(\sA_k)$. There exists an absolute constant $C>0$ such that, for every large enough $\beta>0$ and every $k\geq 1$,
\begin{equation}\label{eq:mu-rigid-upper-steep} 
\hatmu_\infty(\sS_{z,k}^\uparrow) \leq e^{-\frac15( \beta-C) k}\,.\end{equation}
The same holds for the event $\sS_{z,k}^\downarrow$ corresponding to $j_i$ level line down-annuli for $j_1<\ldots<j_k<0$.
\end{claim}
\begin{proof}
Fix a sequence $\{\sA_i\}_{i=1}^k$ of level-line annuli with $z\in\Int(\sA_1)\subset\ldots\subset\Int(\sA_k)$, and let $\sS^\uparrow_{z,\{\sA_i\}}$ denote the event that these occur as \texttt{steep} $j_i$ level line up-annuli in $\lfloor\sigma\rfloor$ for some $j_1>\ldots >j_k>0$. The sought \cref{eq:mu-rigid-upper-steep} will follow from showing that 
\begin{equation}\label{eq:steep-level-line-prob}
    \hatmu_\infty(\sS^\uparrow_{z,\{\sA_i\}}) \leq e^{-\frac15 \beta \sum_{i=1}^k|\sA_i|}\,.\end{equation}
This is a consequence of a routine enumeration over the $\sA_i$'s: for an upper bound on the number of such sequences, we may ignore the consistency constraints between the $2k$ loops $\{\sA_i^{\texttt{in}},\sA_i^{\texttt{out}}\}_{i=1}^k$; every external loop $\sA_i^{\texttt{out}}$ must cross the horizontal line through $y$ at distance at most $|\sA_i^{\texttt{out}}|/2$ from~$y$, and every internal loop $\sA_i^{\texttt{in}}$ must cross the horizontal line through $o$ at distance at most $|\sA_i^{\texttt{in}}|/2$ from $o$. We see that the enumeration cost is at most $\exp(C \sum_i |\sA_i|)$ for some absolute constant $C>0$, and hence \cref{eq:mu-rigid-upper-steep} will follow from \cref{eq:steep-level-line-prob} by a union bound.

To obtain \cref{eq:steep-level-line-prob}, consider the Peierls map $T$ that subtracts the height of $x$ by the number of level-line annuli that contain $x$ in their interior:
\begin{equation}\label{eq:def-T-annuli} (T \sigma)_x = \sigma_x - \sum_{i=1}^k\one_{\{x\in \Int(\sA_i)\}}\,.\end{equation}
Observe that, by definition, $T$ does not modify $\sigma_x$ for $x\in B_r(o)$, and elsewhere, it only modifies $\sigma$ via an integer shift (specifically, a shift by $1$), and therefore $T \sigma$ remains a legal configuration; clearly, as we have fixed $\{\sA_i\}_{i=1}^k$ ahead of time, it is an injection, and thus \cref{eq:steep-level-line-prob} will follow once we establish that, for every $\sigma\in\sS^\uparrow_{z,\{\sA_i\}}$,
\begin{equation}\label{eq:mu-sigma-mu-Tsigma} \hatmu_\infty(\sigma) \leq e^{-\frac14\beta\sum_{i=1}^k|\sA_i|}\hatmu_\infty(T\sigma)\,.\end{equation}
In view of the expression for $\hatmu$ in \cref{eq:mu-def}, 
\[\frac{\hatmu_\infty(\sigma)}{\hatmu_\infty(T\sigma)}=e^{-\beta(\Upsilon_1-\Upsilon_2)}\,,\]
where
\begin{equation}\label{eq:def-Upsilon12-annuli} \Upsilon_1 = \sum_{xy\text{ dual to }b\in \bigcup \sA_i} |(\nabla\sigma)_{xy}|^2 - |(\nabla (T\sigma))_{xy}|^2\qquad,\qquad
\Upsilon_2 = 8\sum_{x\in\partialvtx B_R(o)}(\Delta\phi^*)_x (\sigma_x- (T\sigma)_x)\,.\end{equation}
The handling of the gradient term $\Upsilon_1$ shows the rationale behind the definition of a \texttt{steep} level line. Within this argument, consider $xy$ dual to some $b\in\bigcup_i\sA_i$. W.l.o.g., $(\nabla\lfloor\sigma\rfloor)_{xy}\geq 1$ (otherwise reverse the roles of $x,y$; the gradient is nonzero as $xy$ is dual to a bond in one of the $\sA_i$'s). Since
\[ (\nabla\sigma)_{xy} > (\nabla\lfloor\sigma\rfloor)_{xy} - 1\,, \]
we see that every bond $b\in\bigcup_i\sA_i$ dual to $xy$ with $(\nabla \lfloor\sigma\rfloor)_{xy}\geq 2$ contributes at least $1$ to $\Upsilon_1$ for each $\sA_i$ it belongs to (and more than that if it belongs to multiple $\sA_i$'s, due to the quadratic effect). On the other hand, if $b\in\bigcup_i\sA_i$ is dual to $xy$ with $(\nabla\lfloor \sigma\rfloor)_{xy} = 1$ then it decreases $\Upsilon_1$ by at most $1$ (and by definition belongs to just a single $\sA_i$).  
Thus, the fact that every $\sA_i$ is \texttt{steep} implies that
\[ \Upsilon_1 \geq \frac13 \sum_i |\sA_i|\,.\]
To handle the external field term $\Upsilon_2$, we use \cref{eq:Delta-phi*-max} to write
\[ \Upsilon_2 \leq \frac18 \sum_i |\Int(\sA_i)\cap \partialvtx B_R(o)|\,, \]
and proceed to make the elementary yet important observation that, for every $i$,  \begin{equation*}
|\sA_i| \geq |\Int(\sA_i) \cap \partialvtx B_R(o)|\,.\end{equation*} 
(One way to see this would send disjoint paths $P_x$ from every $x\in\partial B_R(o)$ to $\infty$---e.g., assigning the $4$ sectors $(\frac\pi4,\frac{3\pi}4)$, $(\frac{3\pi}4,\frac{5\pi}4)$, $(\frac{5\pi}4,\frac{7\pi}4)$, $(\frac{7\pi}4,\frac{\pi}4)$ paths going north, west, south, east, respectively. This would map each $x\in \Int(\sA_i)\cap\partialvtx B_R(o)$ to a distinct bond of $ \sA_i$, dual to some $e\in P_x$.)    
    
Combining the last three displays shows that 
\[ \Upsilon_1 - \Upsilon_2 \geq \frac15\sum_i |\sA_i|\,,\]
establishing \cref{eq:mu-sigma-mu-Tsigma} and thus completing the proof of \cref{eq:mu-rigid-upper-steep}. The analogous bound for $\sS_{z,k}^\downarrow$ follows via the Peierls map $(T\sigma)_x=\sigma_x + \sum_{i=1}^k\one_{\{x\in\Int(\sA_i)\}}$, where the exact same argument holds for analyzing the gradient change $\Upsilon_1$ (and one can use the simple bound $\Upsilon_2 \leq 0$, as the external field makes it only more preferable to increase the heights of the configuration).
\end{proof}

\item \label[step]{st:3-non-steep-energy} (\emph{Bounding the energy of non-\texttt{steep} level lines.}) To treat a large deviation stemming from many non-\texttt{steep} level lines, we will show establish the following lower bound on the number of gradient-$1$ bonds in a collection of non-\texttt{steep} level lines surrounding a site $y$.

\begin{claim}\label{clm:non-steep-level-lines-energy}
Every sequence $\sA_1,\ldots,\sA_k$ of level-line annuli with $z\in\Int(\sA_1)\subset \ldots\subset\Int(\sA_k)$ that are non-\texttt{steep} 
in $\lfloor\sigma\rfloor$ contains at least $\frac29 k^2$  bonds $b$ whose dual $xy$ satisfies $|(\nabla \lfloor\sigma\rfloor)_{xy}|=1$.
\end{claim}
\begin{proof}
We will prove by induction that, $|\Int(\sA_j)| \geq j^2/9$ for all $j\geq 1$. The base case is trivial (one always has $|\Int(\sA_j)|\geq 1$ since the loops $\sA_j^{\texttt{in}},\sA_j^{\texttt{out}}$ are disjoint). For the inductive step, 
recall \cref{def:level-line-annuli}, and observe that the fact that $\Int(\sA_j)\subset \Int(\sA_{j+1})$ implies that 
\[ \Int(\sA_j^\texttt{out}) \subset \Int(\sA_{j+1}^\texttt{out})\qquad\mbox{and}\qquad 
\Int(\sA_{j+1}^\texttt{in}) \subset \Int(\sA_{j}^\texttt{in})\,.
\]
By the isoperimetric inequality in $\Z^2$ and the induction hypothesis, $|\sA_j| \geq 4\sqrt{|\Int(\sA_j)|} \geq 4j/3$,
where we recall that $|\sA_j| = |\sA_j^{\texttt{in}}|+|\sA_j^{\texttt{out}}|$. Recalling \cref{def:steep-level-line}, the fact that  $\sA_j$ is non-\texttt{steep} (and that it is a level line of $\lfloor\sigma\rfloor$, whence each of its bonds $b$ has a dual $xy$ with $(\nabla\lfloor\sigma\rfloor)_{xy}\neq 0$) implies that it has at least $|\sA_j|/3 \geq 4j/9$ bonds $b$ whose dual $xy$ has  $(\nabla\lfloor\sigma\rfloor)_{xy} = 1$. Each such bond $b\in\sA_j^{\texttt{out}}$ corresponds to a distinct unit square in 
$\Int(\sA_{j+1}^{\texttt{out}})\setminus\Int(\sA_j^{\texttt{out}})$, whereas each such $b\in \sA_j^{\texttt{in}}$ corresponds to a distinct unit square in $\Int(\sA_{j}^{\texttt{in}})\setminus\Int(\sA_{j+1}^{\texttt{in}})$. Hence,
\[ \left|\Int(\sA_{j+1})\right| = \left|\Int(\sA_{j+1}^{\texttt{out}})\setminus \Int(\sA_{j+1}^{\texttt{in}})\right| \geq \left|\Int(\sA_j)\right| + \frac{4j}9 \geq \frac{j (j + 4)}9 \geq \frac{(j+1)^2}9\,,\]
completing the proof of the induction. At the same time, since the gradient-$1$ bonds counted above in the collection $\{\sA_j\}_{j=1}^k$ are all distinct by definition, we find that this collection of loops contains at least $\frac49 \sum_{j=1}^k j \geq \frac29 k^2$ such bonds, as required. 
\end{proof}

\item \label[step]{st:4-non-steep-ld} (\emph{Ruling out many non-\texttt{steep} level lines in $\lfloor\sigma\rfloor$.}) Let $\hatnu_\infty$ be the \ZGFF on configurations $\bar\sigma:\Z^2\to\Z$ satisfying $\bar\sigma_x=0$ for all $x\in B_r(o)$, with the external field from \cref{eq:mu-def}, and a modified interaction of $\beta/2$ (as opposed to $\beta$) for edges within $B_R(o)$ or intersecting it:
\begin{equation}\label{eq:nu-def} \hatnu_\infty(\bar\sigma) = \frac1{Z_\hatnu} \exp\bigg(-\frac{\beta}2 \!\!\! \sum_{\substack{x\sim y \\ \{x,y\}\cap B_R(o)\neq \emptyset}} \!\!\! |(\nabla\bar\sigma)_{xy}|^2 
-\beta\!\!\!\sum_{\substack{x\sim y \\ \{x,y\}\cap B_R(o)= \emptyset}}\!\!\! |(\nabla\bar\sigma)_{xy}|^2
+ 8\beta \!\!\! \sum_{x\in\partialvtx B_R(o)} (\Delta\phi^*)_x \bar\sigma_x \bigg)\,.\end{equation}
The analysis of the Peierls map in $T$ from the proof of \cref{clm:k-steep-level-lines} will readily show that $\hatnu_\infty$ is rigid:
\begin{claim}
    \label{clm:nu-rigid} The following holds for $\beta$ large enough. Let $\{\sA_i\}_{i=1}^k$ be a sequence of level-line annuli with $z\in\Int(\sA_1)\subset\ldots\subset\Int(\sA_k)$, and let $\sS_{z,\{\sA_i\}}^\uparrow$ denote the event that these occur as $j_i$ level line up-annuli for some $j_1>\ldots>j_k>0$. There exists an absolute constant $C>0$ such that
    \[ \hatnu_\infty(\sS_{z,\{\sA_i\}}^\uparrow) \leq e^{-\frac14\beta\sum_{1}^k|\sA_i|}\,.\]
    The same holds for $\sS_{z,\{\sA_i\}}^\downarrow$ corresponding to down-annuli for some $j_1<\ldots<j_k<0$.
\end{claim}
\begin{proof}
Let $T$ be the map from \cref{eq:def-T-annuli}, decrementing the height of each of the regions $\Int(\sA_i)$, and define $\Upsilon_1,\Upsilon_2$ as in \cref{eq:def-Upsilon12-annuli}.
Unlike the situation in the proof of \cref{clm:k-steep-level-lines}, where the $\R$-valued nature of the configuration $\sigma\sim\hatmu_\infty$ called for the notion of \texttt{steep} level lines to handle the energy gain in the gradients of $T\sigma$ compared $\sigma$, here $\bar\sigma\sim\hatnu_\infty$ is $\Z$-valued; thus, there is no loss in the gradient term $\Upsilon_1$, and
\[ \Upsilon_1 \geq \sum_i |\sA_i|\,.\]
The analysis of the external field term $\Upsilon_2$ holds verbatim as in said proof, giving
\[ \Upsilon_2 \leq \frac18 \sum_i |\sA_i|\,.\]
Combining these gives the sought bound on $\sS_{z,\{\sA_i\}}^\uparrow$, as here we have 
\[ \frac{\hatnu_\infty(\bar\sigma)}{\hatnu_\infty(T\bar\sigma)} \leq e^{-\beta(\frac12\Upsilon_1-\Upsilon_2)}\]
due to the modified interaction of $\beta/2$ within $B_R(o)$. As in the proof of \cref{clm:k-steep-level-lines}, 
the same applies to $\sS_{z,\{\sA_i\}}^\downarrow$ via the counterpart of $T$ that increments the height of each region $\Int(\sA_i)$.
\end{proof}
Let $\overline\sS_{z,k}^\uparrow$ be the event that there is a sequence $\{\sA_i\}_{i=1}^k$ of non-\texttt{steep} level line up-annuli with $z\in\Int(\sA_1)\subset\ldots\subset\Int(\sA_k)$ (the analogue of $\sS_{z,k}^\uparrow$ from \cref{clm:k-steep-level-lines} for non-\texttt{steep} annuli). By combining \cref{clm:non-steep-level-lines-energy,clm:nu-rigid} with the usual enumeration over the annuli $\{\sA_i\}_{i=1}^k$, we find that
\begin{equation}
    \label{eq:non-steep-nu-bound}
    \hatnu_\infty(\overline\sS_{z,k}^\uparrow) \leq e^{-\tfrac1{18}(\beta - C) k^2}\,.
\end{equation}
The same applies for the analogously defined $\overline\sS_{z,k}^\downarrow$ pertaining down-annuli.

The next standard result (see, e.g., \cite[Claim~3.5]{LMS16}, an analogue when there is no external field on $\partialvtx B_R(o)$ as we encounter here) controls the Radon--Nikodym derivative between $\hatmu_\infty$ and $\hatnu_\infty$.
\begin{claim}
    \label{clm:mu-nu-RND}
There exists an absolute constant $c_0>0$ such that
\begin{equation}
    \label{eq:mu-nu-RND}
 \sup_\sigma \frac{\hatmu_\infty(\sigma)}{ \hatnu_\infty(\lfloor\sigma\rfloor)} \leq e^{c_0 \beta R^2}\,.\end{equation}
In particular,  $\hatmu_\infty(\overline\sS_{z,k}^\uparrow) \leq e^{-\frac1{25}(\beta-C)k^2}$ for $k \geq 10 \sqrt{c_0} R$, and similarly for the analogous $\overline\sS_{z,k}^\downarrow$.
\end{claim}
\begin{proof}
Let $\phi^*_\textsf{f} := \phi^* - \lfloor\phi^*\rfloor$ denote the fractional part of $\phi^*$, and recall that by definition $\sigma = \lfloor\sigma\rfloor - \phi^*_\textsf{f}$. For every $x\sim y$ we have
\[ |(\nabla\sigma)_{xy}|^2 = |(\nabla\lfloor\sigma\rfloor)_{xy}|^2 + |(\nabla\phi^*_\textsf{f})_{xy}|^2- 2 (\nabla\lfloor\sigma\rfloor)_{xy}(\nabla\phi^*_\textsf{f})_{xy}\,.\]
As $\phi^*_\textsf{f}$ is fixed, we may rewrite $\hatmu_\infty$ from \cref{eq:mu-def} via another partition function $\tilde Z_\hatmu$ 
so that
\[ \hatmu_\infty(\sigma) = \frac1{\tilde Z_\hatmu} \exp\bigg(-\beta \sum_{x\sim y} |(\nabla\lfloor\sigma\rfloor)_{xy}|^2 + 2\beta\!\!\!\sum_{\substack{x\sim y \\ 
\{x,y\}\cap B_R(o)\neq \emptyset}\!\!\!
} (\nabla\lfloor\sigma\rfloor)_{xy}(\nabla\phi_\textsf{f}^*)_{xy}
+ 8\beta \sum_{x\in\partialvtx B_R(o)} (\Delta\phi^*)_x \lfloor\sigma\rfloor_x \bigg)\,,
\]
where we were allowed to sum the cross-term $(\nabla\lfloor\sigma\rfloor)_{xy}(\nabla\phi^*_\textsf{f})_{xy}$ only on $x,y$ with at least one of the sites falling within $B_R(o)$ as otherwise $\phi^*_x=\phi^*_y=0$  (and so $(\nabla\phi_\textsf{f}^*)_{xy}=0$) by construction.
Comparing this with \cref{eq:nu-def} for $\bar\sigma := \lfloor\sigma\rfloor$, we get
\begin{align*} \frac{\hatmu_\infty(\sigma)}{\hatnu_\infty(\lfloor\sigma\rfloor)} &= \frac{Z_\hatnu}{\tilde Z_\hatmu} \exp\bigg( -\frac{\beta}2 \!\!\! \sum_{\substack{x\sim y \\ \{x,y\}\cap B_R(o)\neq \emptyset}} \!\!\! \Big(|(\nabla\lfloor\sigma\rfloor)_{xy}|^2 - 4 (\nabla\lfloor\sigma\rfloor)_{xy}(\nabla\phi^*_\textsf{f})_{xy} \Big) \bigg)\\
&\leq 
\frac{Z_\hatnu}{\tilde Z_\hatmu} \exp\bigg( 2\beta \!\!\! \sum_{\substack{x\sim y \\ \{x,y\}\cap B_R(o)\neq \emptyset}} \!\!\! |(\nabla \phi^*_\textsf{f})_{xy}|^2
\bigg) \\
&\leq \frac{Z_\hatnu}{\tilde Z_\hatmu} \exp\left(c \beta R^2\right)
\end{align*}
for an absolute constant $c>0$,
using that $4ab \leq a^2 + 4b^2$ for any real $a,b$ for the inequality in the second line, and that $|\nabla\phi^*_\textsf{f}|\in[0,1)$ for the inequality in the third line.

To bound $Z_\hatnu/\tilde Z_\hatmu$ from below, we infer from the identity in the first line above that, for $\bar\sigma\sim \hatnu_\infty$ and $\E_\hatnu$ denoting expectation w.r.t.\ $\hatnu_\infty$,
\begin{align*} \frac{\tilde Z_\hatmu}{Z_\hatnu} &= \E_\hatnu\bigg[
\exp\bigg( -\frac{\beta}2 \!\!\! \sum_{\substack{x\sim y \\ \{x,y\}\cap B_R(o)\neq \emptyset}} \!\!\! \Big(|(\nabla\bar\sigma)_{xy}|^2 - 4 (\nabla\bar\sigma)_{xy}(\nabla\phi^*_\textsf{f})_{xy} \Big) \bigg)\bigg] \\
& \geq 
\exp\bigg( -\frac{\beta}2 \!\!\! \sum_{\substack{x\sim y \\ \{x,y\}\cap B_R(o)\neq \emptyset}} \!\!\! \Big(\E_\hatnu\left[|(\nabla\bar\sigma)_{xy}|^2\right] - 4 (\nabla\phi^*_\textsf{f})_{xy}\E_\hatnu\left[(\nabla\bar\sigma)_{xy}\right] \Big) \bigg)\,,
 \end{align*}
 using here Jensen's inequality. The proof will be concluded once we show that for every $x\sim y$ we have $\E_\hatnu[|(\nabla\bar\sigma)_{xy}|^2] \leq C$ for $C>0$ (thus also $\E_\hatnu[|(\nabla\bar\sigma)_{xy}|] \leq \sqrt C$) of the form $C=\epsilon_\beta$.
This is a consequence of \cref{clm:nu-rigid}: if $\bar\sigma_{y} = \bar\sigma_x + k$ for some $k>0$ then we there must be $k$ level-line annuli going through the bond $b$ dual to $xy$---either up-annuli with $y$ in their interior, or down-annuli with $x$ in their interior (or a combination of these). Enumerating over their types and geometry is a factor of at most $\exp(C\sum_{i=1}^k|\sA_i|)$, and hence the probability of this event, in light of \cref{clm:nu-rigid}, is at most $\exp(-\frac14(\beta - C)k)$. This yields that $\E[\exp(\frac15(|\nabla\bar\sigma)_{xy}] < \epsilon_\beta$, thus establishing \cref{eq:mu-nu-RND}. The proof of the final assertion of the claim follows from combining \cref{eq:non-steep-nu-bound,eq:mu-nu-RND}.
\end{proof} 

\end{enumerate}
We  now conclude the proof of the proposition by combining the result on \texttt{steep} annuli from \cref{st:2-steep-phi*} with the one on non-\texttt{steep} annuli from \cref{st:4-non-steep-ld}. Recall that our aim is to rule out $\{|\sigma_z| > C \frac{h}{\log h}\}$ for $\sigma\sim\hatmu_\infty$.
Let $C = 250\sqrt{c_0}$ for  the absolute constant $c_0>0$ from \cref{clm:mu-nu-RND}, and
$k := \lceil (C/2) \frac{h}{\log h}\rceil$. Recalling $R=\lceil 20 h/\log h\rceil $, for this choice we have $k > 10\sqrt{c_0}R$. 

By definition, $\{\sigma_z > C \frac{h}{\log h}\}$ implies $\sS_{z,k}^\uparrow \cup \overline\sS_{z,k}^\uparrow$.
To treat $\sS_{z,k}^\uparrow$, we appeal to \cref{clm:k-steep-level-lines} (the dominant term in the final estimate) and to treat $\overline\sS_{z,k}^\uparrow$ we appeal to \cref{clm:mu-nu-RND}. The same holds for $\{\sigma_z < -C\frac{h}{\log h}\}$ via the analogous argument for $\sS_{z,k}^\downarrow \cup \overline\sS_{z,k}^\downarrow$. This completes the proof of the proposition.
\end{proof}

We will be interested in bounding the probability that, given that $\phi$ exceeds height $h$ at the origin and its neighbor, we have $\phi_y\geq h$ at some other site $y$. The proposition above will provide such a bound when $y$ is such that $\phi^*_y < h - C\frac{h}{\log h}$. The next claim will explain this is the case for all $y$ with $|y|> r_1$ for some absolute constant $r_1>0$, and then extend that bound also to $|y| \leq r_1$.

\begin{claim}\label{clm:pi(h|h_h)-upper}
Let $o$ be the origin and $o'$ be a neighbor of $o$. There exists an absolute constant $c>0$ such that, for large enough $\beta$, every $h\geq 1$ and every $y\neq o,o'$,
\[ \hatpi_\infty(\phi_y \geq h \mid \phi_o\geq h,\, \phi_{o'}\geq h) \leq e^{-c\beta \frac{h}{\log h}}\,.\]
\end{claim}
\begin{proof}
Put $o'=(-1,0)$.
Let $C>0$ be the absolute constant from \cref{prop:pi-y-h-phi^*-rigid} and $\phi^*$ be the harmonic function defined there. We first tweak the conditioning to the one in said proposition:
\begin{align*}
\hatpi_\infty(\phi_y\geq h\mid\phi_o\geq h,\, \phi_{o'}\geq h) 
 &\leq \hatpi_\infty(\phi_y\geq h\mid \phi_x \geq h,\forall x\in B_1(o)) \\
 &\leq (1+\epsilon_\beta)\hatpi_\infty(\phi_y\geq h - 1\mid \phi_x = h,\forall x\in B_1(o))\,,
\end{align*} 
where the first transition is by FKG, and the second one is by a Peierls argument via the map that decreases $\phi_z$ by $1$ for each $z$ in a connected component of $\{v : \phi_v > h\}$ that intersects some $z\in B_1(o)$ (this may also decrease $\phi_y$ by $1$, hence the weaker inequality $\phi_y \geq h-1$).
The asymptotic behavior of $\phi^*_{x}$ in terms of $|x|$ given by \cref{eq:phi*-expression-1,eq:phi*-expression-2} shows that there some absolute constant $r_1>0$ (behaving as $e^C$ for the above $C>0$) such that $\phi^*_y < h - (C+1)\frac{h}{\log h}$ for every $y$ with $|y|\geq r_1$. Thus, applying \cref{prop:pi-y-h-phi^*-rigid} to the last display shows that
\begin{equation}\label{eq:pi(h|hh)-beyond-r1}
\sup_{y:\,|y| \geq r_1}  \hatpi_\infty\Big(\phi_y \geq h - \frac{h}{\log h} \;\Big|\; \phi_o\geq h,\, \phi_{o'}\geq h\Big) \leq \exp\Big(-\frac15(\beta-C) \frac{h}{\log h}\Big)\,.   
\end{equation}
Next, fix $y\neq o,o'$ with $|y|< r_1$. We aim to show that 
\begin{equation}\label{eq:pi(h|hh)-O(r1)}\hatpi^0_\infty(\phi_y \geq h\mid \phi_o, \phi_{o'} \geq h) \leq \exp\Big(-\frac{1}{10\lceil r_1\rceil}(\beta-C)\frac{h}{\log h}\Big)\,.\end{equation} 
In the special case where $y\sim o$ (or $y\sim o'$, symmetrically) this is simple, and will demonstrate the basic principle of the argument. Recalling $o'=(-1,0)$, suppose $y=(1,0)$. Letting $y_k = (k,0)$ for $k\in \Z$ (so that $o',o,y$ are $y_{-1},y_0,y_1$, respectively), and setting $K = \lceil r_1\rceil$, one has
\begin{align} \hatpi_\infty(\phi_{y_K} \geq h \mid \phi_o \geq h,\,\phi_{o'}\geq h) &\geq 
\hatpi_\infty\bigg(\bigcap_{k=1}^{K}\phi_{y_k} \geq h \;\Big|\; \phi_o\geq h,\,\phi_{o'}\geq h\bigg) \nonumber\\
&\geq 
\prod_{k=1}^{K}\hatpi_\infty\left(\phi_{y_k} \geq h \mid \phi_{y_{k-1}}\geq h,\,\phi_{y_{k-2}}\geq h\right) \nonumber\\ & = \hatpi_\infty\left(\phi_y \geq h\mid\phi_o\geq h,\,\phi_{o'}\geq h\right)^{K}\,,\label{eq:pi(yk)-iteration}
\end{align}
where the second inequality is by FKG (replacing the conditioning on $\bigcap_{i=-1}^{k-1}\{\phi_{y_{i}}\geq h\}$ by the one on $\{\phi_{y_{k-1}}\geq h\}\cap\{\phi_{y_{k-2}}\geq h\}$). As $|y_K|\geq r_1$, it is covered by \cref{eq:pi(h|hh)-beyond-r1}, whose upper bound thus applies to the left-hand of \cref{eq:pi(yk)-iteration}, thereby yielding \cref{eq:pi(h|hh)-O(r1)} (moreover, with the slightly better constant $1/(5\lceil r_1\rceil)$ in the exponent).
The case $y=(0,1)$ is very similar: letting $y_k = (\lfloor k/2 \rfloor  , \lceil k/2 \rceil )$ (so again $o',o,y$ are $y_{-1},y_0,y_1$, respectively), and (say) $K = 2\lceil r_1 \rceil$ the above display holds unchanged, and as $|y_K|\geq r_1$ (with room to spare), we again get \cref{eq:pi(h|hh)-O(r1)}. 

It remains to show \cref{eq:pi(h|hh)-O(r1)} for $|y|<r_1$ that is not incident to $o,o'$. The strategy would be to first show that, for any (arbitrarily chosen) neighbor $y'\sim y$, 
\begin{equation}\label{eq:pi(yy'|hh)}\hatpi_\infty\Big(\phi_y \wedge \phi_{y'}\geq 
h - \frac{h}{\log h}\;\Big|\; \phi_o\wedge\phi_{o'} \geq h\Big) \leq \exp\Big(-\frac{1}{5\lceil r_1\rceil}(\beta-C)\frac{h}{\log h}\Big)\,,\end{equation} 
and thereafter use this for the sought bound on $y$. Towards establishing \cref{eq:pi(yy'|hh)}, we will set aside the fact that $y$ was chosen ahead of $y'$, and treat these as two interchangeable adjacent sites.

If $yy'$ is a horizontal edge---denoting $y=(a,b)$ and $y'=(a-1,b)$---then we let 
\[ y_k:=(k a,k b)\,,\qquad y'_k:=y_k-(1,0)\,.\] If it is a vertical edge---denoting $y=(a,b)$ and $y'=(a,b-1)$, and here we assume w.l.o.g.\ that if $a>0$ then $b\geq 0$ and if $a<0$ then $b\leq 0$ (otherwise, reflect $\phi$ about the $x$-axis)---then we let 
\[ y_k:= \begin{cases}(\frac{k}2(a+b),\frac{k}2(b-a))& \mbox{$k$ is even}\\
(\frac{k-1}2(a+b)+a,\frac{k-1}2(b-a)+b)& \mbox{$k$ is odd}\\
\end{cases}\,,\qquad
y'_k := \begin{cases}
    y_k - (1,0) & \mbox{$k$ is even}\\
    y_k - (0,1) & \mbox{$k$ is odd}
\end{cases}
\,.
\]
By $\Z^2$ symmetry (translation, $\frac\pi2$-rotation), in both cases the law of $\phi\restriction_{\{y_k,y'_k\}}$ given $\phi\restriction_{\{y_{k-1},y'_{k-1}\}}$ is equal to that of $\phi\restriction_{\{y,y'\}}$ given $\phi\restriction_{\{o,o'\}}$. Moreover, $|y_k| \geq k$ (in the second case, $|a+b|\geq 2$ by the choice of $\operatorname{sign}(b)$ and since $y\not\sim o,o'$). In particular, setting $K= \lceil r_1\rceil$ we guarantee that $|y_K|\geq r_1$. 

Revisiting \cref{eq:pi(yk)-iteration}, we replace the events $\{\phi_{y_k}\geq h\}$ there by $\{\phi_{y_k}\wedge\phi_{y'_k} \geq h-\frac{k h}{K\log h}\}$, to find that by the exact same reasoning,
\begin{align*}
\hatpi_\infty\Big( \phi_{y_K}\wedge\phi_{y'_K} \geq h-\frac{h}{\log h}& \;\Big|\; \phi_{o}\wedge\phi_{o'} \geq h \Big) \\ &\geq \prod_{k=1}^K \hatpi_\infty\Big(
\phi_{y_k}\wedge\phi_{y'_k} \geq h-\frac{k h}{K\log h}\;\Big|\; \phi_{y_{k-1}}\wedge\phi_{y'_{k-1}} \geq h-\frac{(k-1) h}{K\log h}\Big)\,.
\end{align*}
Notice that for every $k\geq 1$ and $M$, chosen sufficiently large such that $y_{k-1},y_{k-1}',y_k,y_k'\in B_M(o)$, we have by shift invariance and monotonicity that
\begin{align*}
\hatpi^0_{B_M(o)}\Big(\phi_{y_k} \wedge \phi_{y'_k}\geq h - \frac{k h}{K \log h} &\;\Big|\; \phi_{y_{k-1}}\wedge \phi_{y'_{k-1}} \geq h-\frac{(k-1)h}{K\log h}\Big) \\
&= \hatpi^{\frac{(k-1)h}{K \log h}}_{B_M(o)}\Big(\phi_{y_k}\wedge \phi_{y'_k}\geq h - \frac{h}{K\log h} \;\Big|\; \phi_{y_{k-1}}\wedge \phi_{y'_{k-1}} \geq h\Big) \\ 
&\geq \hatpi^{0}_{B_M(o)}\Big(\phi_{y_k}\wedge \phi_{y'_k}\geq h - \frac{h}{K\log h} \;\Big|\; \phi_{y_{k-1}}\wedge \phi_{y'_{k-1}} \geq h\Big)\,.
\end{align*}
Taking $M\to\infty$ yields this inequality under $\hatpi_\infty$, and combined with the previous display we get
\begin{align*}
\hatpi_\infty\Big( \phi_{y_K}\wedge\phi_{y'_K} \geq h-\frac{h}{\log h} \;\Big|\; \phi_{o}\wedge\phi_{o'} \geq h \Big) \geq \hatpi_\infty\Big(\phi_y \wedge \phi_{y'} \geq h-\frac{h}{\log h} \;\Big|\; \phi_o\wedge\phi_{o'}\geq h\Big)^K\,.
\end{align*}
The left-hand is at most $\hatpi_\infty\big(\phi_{y_K} \geq h-\frac{h}{\log h}\mid \phi_o \wedge \phi_{o'}\geq h\big)$, which we may bound from above via \cref{eq:pi(h|hh)-beyond-r1} since $|y_K|\geq r_1$, thus establishing \cref{eq:pi(yy'|hh)}.

To derive the required bound on $\phi_y$, first note that one can infer from \cref{eq:pi(yy'|hh)} that either $y$ satisfies \cref{eq:pi(h|hh)-O(r1)} or $y'$ does, after using FKG to bound the arithmetic mean of these probabilities.
In case $y$ satisfies \cref{eq:pi(h|hh)-O(r1)} for any choice of $y'$ as one of its neighbors, we are done. Otherwise, every $y'\sim y$ has $\hatpi_\infty(\phi_{y'}\geq h-\frac{h}{\log h}\mid \phi_o\wedge\phi_{o'}\geq h) \leq \exp(-\frac1{10\lceil r_1\rceil}(\beta-C)\frac{h}{\log h})$, whence
\begin{align*}
\hatpi_\infty(\phi_y\geq h &\mid \phi_o\wedge\phi_{o'}\geq h) \\ &\leq
\hatpi_\infty\Big(\max_{y'\sim y}\phi_{y'}\geq h-\frac{h}{\log h}\;\Big|\; \phi_o\wedge\phi_{o'}\geq h\Big) + 
\hatpi_\infty\Big(\phi_y \geq h \;\Big|\; \max_{y'\sim y}\phi_{y'}\leq h-\frac{h}{\log h}\Big) \\
&\leq 4 e^{-\frac1{10\lceil r_1\rceil}(\beta-C)\frac{h}{\log h}} + e^{-4\beta (\frac{h}{\log h})^2}
\end{align*}
(using for the last line a Peierls map that decreases $\phi_y$ by $\lceil\frac{h}{\log h}\rceil$), giving the required \cref{eq:pi(h|hh)-O(r1)}.
\end{proof}

As a consequence of \cref{clm:pi(h|h_h)-upper}, we can now infer an analogue of \cref{clm:pi(h|h)-lower} for $\bP_{\ell,h}$, which will be the key to improving the upper bound in \cref{lem:xi-bar-lower-weak-upper}.

\begin{corollary}
\label{cor:pi(Fc|_h_h)-lower}
Let $h = H+1-n$ for $n$ fixed, and $\ell \leq \sqrt{L} \exp(-\frac12\frac{\log L}{\log\log L})$. If $o'$ is a neighbor of the origin $o$, then there exists an absolute constant $c>0$ such that
\[ \bP_{\ell,h}\left(\phi_{o'} = -h\mid \phi_o = -h\right) \geq e^{-c \frac{\log L}{{\log\log L}}}\,.\]
\end{corollary}
\begin{proof}
    Recalling the definition $\bP_{\ell,h} = \hatpi_\infty(\cdot\mid \sF_{\ell,h})$, we have
 \begin{align*}
\bP_{\ell,h}(\phi_{o'} = -h \mid \phi_o = -h)
&\geq 	\hatpi_\infty\left(\phi_{o'}=-h,\, \sF_{\ell,h} \mid \phi_o=-h\right)
  \\ &=	\hatpi_\infty\left(\sF_{\ell,h} \mid \phi_{o'} = \phi_o = -h\right) \hatpi_\infty\left(\phi_{o'}=-h \mid \phi_o = -h\right)\,.
 \end{align*}
 Since $\hatpi_\infty(\phi_{o'} = -h \mid \phi_o = -h)\geq 
 e^{-c\beta \frac{h^2}{\log^2 h}}$ by \cref{clm:pi(h|h)-lower}, which for the value of $h$ here is $e^{-c \frac{\log L}{\log\log L}}$,  it will suffice to show that
\begin{equation}\label{eq:F-whp} \hatpi_\infty\left(\sF_{\ell,h} \mid \phi_o=\phi_{o'}=-h\right) = 1-o(1)\,.\end{equation}
To this end, recall that in the proof of \cref{lem:xi-bar-lower-weak-upper}, we bounded $\hatpi_\infty(\sF_{\ell,h}^c\mid\phi_o=-h)$ by splitting the treatment of $x\in Q_\ell$ into those at distance larger than $\log L$ from the origin and those within said distance, leading to \cref{eq:Fc-given-h-upper-bound}. The same reasoning applies here, showing 
\[ \hatpi_\infty(\sF_{\ell,h}^c\mid \phi_o=\phi_{o'}=-h) \leq e^{-(1-o(1))\frac{\log L}{\log\log L}} + O(\log^2 L) \max_{y\in B_{\log L}(o)} \hatpi_\infty(\phi_y < -h \mid \phi_o=\phi_{o'}=-h) \,.\]
(Note that here we used the assumption on $\ell$ to control  $y\in Q_\ell$ at distance larger than $\log L$ from the origin via a union bound over all $|Q_\ell|\leq L\exp(-\frac{\log L}{\log\log L})$ such sites). 

    It remains to bound $\hatpi_\infty(\phi_y>h\mid\phi_o=\phi_{o'}=h)$. To this end, 
can use a Peierls argument via the map that decreases by $1$ the heights of all sites in the connected component of $\{x:\phi_x>h\}$ intersecting $\{o,o'\}$, to find that $\hatpi_\infty(\phi_o = \phi_{o'} = h) \geq (1-o(1)) \hatpi_\infty(\phi_o \geq h,\, \phi_{o'} \geq h)$. Hence,
\begin{align*}
 \hatpi_\infty(\phi_y>h\mid\phi_o=\phi_{o'}=h) &\leq (1+o(1))
 \hatpi_\infty(\phi_y > h\mid\phi_o\geq h,\, \phi_{o'}\geq h) \\
&\leq  e^{-c \beta \frac{h}{\log h}}\,, 
\end{align*}
 where the last inequality is by \cref{clm:pi(h|hh)-upper-p}.
For $h$ as in our hypothesis, this at most $\exp(-c\sqrt{\beta \frac{\log L}{\log\log L}})$, which outweighs the $O(\log^2 L)$ factor, yielding \cref{eq:F-whp} and thus completing the proof.
\end{proof}

Equipped with the above result, we can now establish the desired upper bound on $\bar\xi_{\ell,h}$ via the following analog of \cref{lem:xi-refined-upper}.

\begin{lemma}\label{lem:xi-bar-upper}
There exist absolute constants $\beta_0,c_1>0$ such that the following holds for all $\beta>\beta_0$. For $h=H+1-n$ with fixed $n\geq 0$ and $\ell$ with $\ell \geq L^{\delta} $ for fixed $\delta>0$, 
\begin{equation*} \frac{\bar\xi_{\ell,h}}{\hatpi_\infty(\phi_o = -h)} \leq 1 - e^{-c_1 \frac{\log L}{\log\log L}}\,.\end{equation*}
\end{lemma}
\begin{proof}
The proof of \cref{lem:xi-refined-upper} extends verbatim to our setting of $\bP_{\ell,h}$: by FKG,
\begin{align*}
\bP_{\ell,h}(\phi_x> -h,\,\forall x\in Q_\ell) &\geq  \Big(1-2\bP_{\ell,h}(\phi_o=-h) + \bP_{\ell,h}(\phi_o=\phi_{o'}=-h) \Big)^{\lceil\ell/2\rceil}\,,
\end{align*} 
and following the same argument in that proof, after replacing in it $\hatpi_\infty(\phi_o<-h)$ by $\bP_{\ell,h}(\phi_o=-h)$ and the application of \cref{clm:pi(h|h)-lower} by that of \cref{cor:pi(Fc|_h_h)-lower}, we find that there exists some absolute constant $c>0$ (given by said corollary) such that
\begin{align*}
\frac{\bar\xi_{\ell,h}}{\bP_{\ell,h}(\phi_o=-h)} &\leq  1-e^{-c \frac{\log L}{\log\log L}}\,.
\end{align*}
Now, $\bP_{\ell,h}(\phi_o=-h) = \bP_{\ell,h}(\phi_o\leq -h)$, which, by FKG and our assumption on $h$, satisfies
\[\bP_{\ell, h}(\phi_o \leq -h) \leq \hatpi_\infty(\phi_o \leq -h \mid \phi_o \geq -h) \leq \frac{\hatpi_\infty(\phi_o = -h)}{1 - \hatpi_\infty(\phi_o < -h)} \leq (1+L^{-1+o(1)})\hatpi_\infty(\phi_o = -h)\,.\] 
Combining the last two displays concludes the proof. 
\end{proof}

\begin{proof}[Proof of \cref{prop:uprho-bound}, \cref{eq:uprho-bounds}]
As mentioned above, we already established the lower bound as part of \cref{eq:uprho-bounds-weak}. For the upper bound, set $\ell_0 = \lfloor L^{1/4} \rfloor$ as in the proof of \cref{prop:uprho-bound}. The sought bound now follows from substituting the bound on $\bar\xi_{\ell_0,h}$ from \cref{lem:xi-bar-upper} in \cref{eq:rho_n_via_xibar_ell0}.
\end{proof}

\section{Refined estimates on the law of the disagreement polymer}\label{sec:refined-CE}
This section is devoted to establishing several refinements of \cref{prop:CE1_old}. As mentioned in the beginning of \cref{sec:xi-bbq}, these will involve the rate $\uprho_n$ from \cref{eq:rho-n-def}.

The following result establishes the asymptotic law of the disagreement polymer in a box of area $L^{4/3}e^{O(\sqrt{\log L})}$ as opposed to $L (\log L)^{O(1)}$. Throughout the rest of this paper, the letter $\cG$ will be used to denote the conditions on $\gamma$ needed to apply the relevant polymer law proposition, which will change depending on the context.
\begin{proposition}[Law of a mesoscopic disagreement polymer]\label{prop:CE-meso}
    Fix $n\geq 0$ and $g\geq 0$. There exists $\beta_0$ such that the following holds for all $\beta\geq \beta_0$. Let $V\subset \Z^2$ be a connected domain with $g$ holes, and consider the \ZGFF model $\pi_{V;F}^{\eta}$ with a floor at $0$ imposed only on a subset $F\subset V$, and boundary conditions $\eta$ that are $H+1-n$ on a $*$-connected path in $\partialvtx V$ and $H-n$ elsewhere so that they induce a unique disagreement polymer $(\gamma,\{D_i\},\{h_i\})$ in $V\cup \partialvtx V$ that contains boundary disagreements. Further suppose $|\partialvtx F|\vee |\partialvtx V| \leq  O(L^{2/3}e^{\sqrt{\log L}})$, denote by $D_0$ and $D_1$ the regions of $\gamma$ containing the boundary vertices of $V$ at heights $H+1-n$ and $H-n$, respectively, and let
		\begin{align*} \cG&:=\left\{|\gamma|\leq L^{2/3}e^{\sqrt{\log L}}\right\}\,,\\
        \cG^\dagger &:= \cG \cap \left\{\max_{i\geq 2} |\partial D_i| < \log L \right\}\,.\end{align*}
        Then 
        \begin{align}\label{eq:CE-with-area-cond-1}
        \pi_{V;F}^\eta( \cG^\dagger \mid \cG) &= 1-o(1)\,,\\
        \label{eq:CE-with-area-cond-2}
			\pi^\eta_{V;F}(\gamma \mid \cG^\dagger) &=(1+o(1))\fq_{V;F}^{\eta}(\gamma)\quad\mbox{uniformly over all} \quad \gamma\in \cG^\dagger\,,
		\end{align}
        where $\fq_{V;F}^\eta$ is the probability distribution on $\gamma\in\cG^\dagger$ given~by 
		\begin{equation*}
        \fq_{V;F}^{\eta}(\gamma) := 
		\frac1{Z^\eta_{V;F}} \exp\bigg(-\sE^*_\beta(\gamma) + \frac{\uprho_{n}}{N_n}|D_0\cap F|+ \fI_V(\gamma) \bigg)
		\end{equation*}
        with $N_n$ from \cref{eq:Nn-def}, $\fI_V$ from \cref{eq:IU(gamma)-def}, $\uprho_{n}$ from \cref{eq:rho-n-def}, and a normalizer $Z^\eta_{V;F}$.
\end{proposition}
\begin{proof}
Let $(\gamma,\{D_i\},\{h_i\})$ be a disagreement polymer. 
We start with the expression for $\pi_{V;F}^\eta(\gamma)$ given in \cref{eq:CE-with-area} from \cref{prop:CE1_old}, which holds without any restrictions on the $D_i$'s or $V,F$, and recall the notation $\sE^*_\beta(\gamma)$ from \cref{eq:E*-def} that absorbed the $\hatpi_{D_i^\circ}^{h_i}$ terms for $i>1$, so that \[\pi^\eta_{V;F}(\gamma ) \propto
		\exp\Big(-\sE^*_\beta(\gamma) + \fI_V(\gamma)\Big) \prod_{i=0,1} \hatpi_{D_i^\circ}^{H+1-n-i}\big(\phi_x \geq 0,\, \forall x \in D_i^{\circ}\cap F\big)
        \,.\]
In view of this representation for $\pi_{V;F}^\eta$, we now further impose $\gamma \in \cG$, and turn our attention to 
\[ \hatpi_{D_i^\circ}^{h_i}\big(\phi_x \geq 0,\, \forall x \in D_i^{\circ}\cap F\big)\qquad\mbox{for $i=0,1$}\,.\]
The event $\cG$ will enable us to appeal to the mesoscopic range of \cref{thm:key-area} for $\hatpi_{D_i^\circ}^0$ ($i=0,1$). 
Let us verify the required hypotheses:
\begin{enumerate}
    \item The condition on the domain: $|\partial D_i^\circ| \leq O(|\gamma| + |\partial V|) =O(L^{2/3}e^{\sqrt{\log L}})$ as per $\cG$.
    \item The condition on the subset where we wish to control the minimum height: again we have that
    $ |\partial (D_i^\circ \cap F)| \leq O(|\gamma|+|\partial V|+|\partial F|) = O(L^{2/3} e^{\sqrt{\log L}})$ as per  $\cG$.
\end{enumerate}
Having qualified for an application of \cref{eq:xi-n-area-estimate-meso} from \cref{thm:key-area}, we deduce that, for $i=0,1$,
		\begin{align*}
			\hatpi_{D_i^\circ}^{H+1-n-i}\big(\phi_x \geq 0,\, \forall x \in D_i^\circ\cap F\big) &=
			\hatpi_{D_i^\circ}^{0}\big(\phi_x \geq (H+1-n-i),\, \forall x \in D_i^\circ\cap F\big) 
			\\ &=(1+ o(1))\exp\big(-\upxi_{n+i}|D_i \cap F| \big)\,,
		\end{align*}
using here that $\upxi_{n+i}=(1+o(1))\hatpi_\infty(\phi_o > H+1-n-i) \leq L^{-1+o(1)}$, and so 
\[ \upxi_{n+i}|(D_i \setminus D_i^\circ)\cap F| \leq L^{-1+o(1)} O(|\gamma|) = L^{-1/3+o(1)}=o(1)\,.\]  
We now move to multiply $\prod_{i\geq 0}\hatpi_{D_i^\circ}^{h_i}(\phi_x\geq0,\,\forall x\in D_i^\circ\cap F)$ by $\exp(\upxi_{n+1}|F|)$, which can be absorbed into the partition function for $\hatpi_{V;F}^\eta$, being independent of $\gamma$. This cancels the term $-\upxi_{n+1} |D_1 \cap F|$ in the exponent, whereas $(\upxi_{n+1}-\upxi_n)|D_0\cap F| = (\uprho_n/N_n) |D_0\cap F|$ by the definition of $\uprho_n$ in \cref{eq:rho-n-def}; thus,
    \[\pi^\eta_{V;F}(\gamma \mid \cG) = (1+o(1))\fp_{V;F}^\eta(\gamma)\]
    uniformly over $\gamma\in \cG$, where $\fp_{V;F}^\eta$ is the probability distribution on $\gamma\in\cG$ given by 
    \begin{equation}\label{eq:p-law-meso}\fp_{V;F}^\eta(\gamma) := \frac{1}{\tilde Z_{V;F}^\eta} 
		\exp\bigg(-\sE^*_\beta(\gamma) + \frac{\uprho_{n}}{N_n}|D_0\cap F|+ \fI_V(\gamma) +\upxi_{n+1}\sum_{i\geq 2}|D_i \cap F|\bigg)\,,\end{equation}
        in which $\tilde Z_{V;F}^\eta$ is a normalizer.
We now move to show \cref{eq:CE-with-area-cond-1} using this intermediary form of the law of $\gamma$, which throughout this proof we denote by $\fp(\gamma)$ in lieu of $\fp_{V;F}^\eta(\gamma)$ for brevity.
Recall \cref{def:gamma-components}, according to which
\[ \cG \setminus \cG^\dagger = \left\{\gamma\in \cG\,:\; \max_{\cD \in \fD(\gamma)}|\partial \cD| > \log L\right\}\,. \]
We will use a Peierls argument to show that, for some constant $c > 0$,
\begin{equation}\label{eq:control-finite-comp}\sum_{\gamma \in \cG \setminus \cG^\dagger} \fp(\gamma) \leq L^{-c \beta} \sum_{\gamma \in \cG} \fp(\gamma)\,.
\end{equation}
There are only $L^2$ locations for a bond $b$, so fix $b$ and consider the map $T_b$ on $\gamma$ defined as follows. If $b \notin\gamma$ or if $\cD_b = \emptyset$, then $T_b$ is the identity map. Otherwise, let $u, v$ be the two cut-points delineating $\cD_b$. Let $T_b(\gamma)$ be the disagreement polymer which replaces $\cD_b$ by a (arbitrarily chosen, say the outermost such) minimal length path from $u$ to $v$ that remains within $\cD_b$. Then, \[\sE_\beta(T_b(\gamma)) \leq \sE_\beta(\gamma) -\tfrac\beta2|\partial \cD_b|\,.\] 
Moreover, we only removed components $D_i$ for $i \geq 2$, and each such component contributes a nonnegative term (namely $-\log \hatpi_{D_i^\circ}^{h_i}(\phi_x\geq 0,\,\forall x\in D_i^\circ\cap F)$) to $\sE^*_\beta$, hence we also have
\[\sE^*_\beta(T_b(\gamma)) \leq \sE^*_\beta(\gamma) -\tfrac\beta2|\partial \cD_b|\,.\] 
We also could only have increased $|D_0|$ in $T_b(\gamma)$. However, $T_b(\gamma)$ did in fact lose the contribution of $\cD_b$ to $\upxi_{n+1} \sum_{i\geq 2} |D_i \cap F|$. By the decay properties of $\Phi$, we therefore have for some $C$ that
\[\fp(\gamma) \leq \fp(T_b(\gamma))e^{-\frac{\beta - C}2|\partial \cD_b|+\upxi_{n+1}|\cD_b|}\,.\]
Note that we necessarily have the crude bound that
$\max_{\cD \in \fD(\gamma)}|\partial \cD| \leq L^{2/3}e^{\sqrt{\log L}}$. By isoperimetry and this crude bound, we have $|\cD_b| \leq |\partial \cD_b|^2/16 \leq |\partial \cD_b|L^{2/3}e^{\sqrt{\log L}}/16$ which after multiplying by $\upxi_{n+1}$ is $o(1)$. Hence, the above display implies that for some $c > 0$, 
\[\fp(\gamma) \leq \fp(T_b(\gamma))e^{-c\beta|\partial \cD_b|}\,.\]

Moreover, the number of preimages $\gamma$ under $T_b$ for a given disagreement polymer $\gamma'$ such that $|\partial \cD_b(\gamma)| = k$ is bounded above by the number of connected components of bonds of size $k$ times the number of possible points $u$. There are $C^k$ choices for the connected component, and $k^2$ choices for the location of the cut-point $u$ at which to attach the component.

Putting the above all together proves \cref{eq:control-finite-comp}:
\begin{align}
    \sum_{\gamma \in \cG \setminus \cG^\dagger} \fp(\gamma) &\leq \sum_{\gamma' \in \cG}\sum_{\substack{k \geq \log L\\b \in \Lambda}}\sum_{\substack{\gamma \in \cG \cap T_b^{-1}(\gamma')\\|\partial \cD_b(\gamma)| = k}}\fp(\gamma)\nonumber\\
& \leq \sum_{\gamma' \in \cG}\sum_{k \geq \log L}L^2k^2C^k\fp(\gamma')e^{-c\beta k}\nonumber\\
&\leq
e^{-c\beta\log L}\sum_{\gamma' \in \cG} \fp(\gamma')\,.\label{eq:map-argument-Di}
\end{align}
This proves \cref{eq:CE-with-area-cond-1}. 

Finally, for all $\gamma \in \cG^\dagger$, the total area of $F \cap (\bigcup_{i \geq 2}D_i)$ is at most $L^{2/3}e^{\sqrt{\log L}}\log L$, hence we find that $\upxi_{n+1}\sum_{i\geq 2}|D_i \cap F| = L^{-1/3+o(1)} = o(1)$.
\end{proof}
\begin{proposition}[Law of a macroscopic disagreement polymer]\label{prop:CE-macro-DobrushinBC}
    Fix $n\geq 0$ and $g\geq 0$. In the setting of \cref{prop:CE-meso}, under the relaxed condition $|\partialvtx F|\vee |\partialvtx V| \leq  O(L e^{\sqrt{\log L}})$ and the modified definition of the event $\cG$ as
    \begin{align*} \cG := \left\{|\gamma| \leq Le^{\sqrt{\log L}}\right\}\cap \bigcap_{i=0,1}\left\{|D_i \cap F| \leq \Big(\frac{3\beta}{\hatpi_\infty(\phi_o > H+1-n-i)}\Big)^2\right\}\,,\end{align*}
the following holds. The probability distribution $\fp_{V;F}^\eta$ on $\gamma\in\cG$ given by 
		\begin{equation*}
        \fp_{V;F}^{\eta}(\gamma) := 
		\frac1{{Z}^\eta_{V;F}} \exp\bigg(-\sE^*_\beta(\gamma) + \frac{\uprho_{n}}{N_n}|D_0 \cap F|+ \fI_V(\gamma) + \upxi_{n+1}\sum_{i\geq 2}|D_i \cap F|\bigg)
		\end{equation*}
		with $N_n$ from \cref{eq:Nn-def}, $\uprho_{n}$ from \cref{eq:rho-n-def}, and a normalizer ${Z}^\eta_{V;F}$, satisfies that
		\begin{align}\label{eq:CE-with-area-cond-new}
			\pi^\eta_{V;F}(\gamma \mid \cG) = \exp\left(O\Big(\sqrt L e^{\frac{\log L}{\log\log L}}\Big)\right)\fp_{V;F}^{\eta}(\gamma)
		\end{align}
        uniformly over all $\gamma\in \cG$.
\end{proposition}

\begin{proof}
 The proof will follow via exactly the same argument that produced \cref{eq:p-law-meso} in the proof of \cref{prop:CE-meso}.
 As argued there, we have
  \[\pi^\eta_{V;F}(\gamma ) \propto
		\exp\Big(-\sE^*_\beta(\gamma) + \fI_V(\gamma)\Big) \prod_{i=0,1} \hatpi_{D_i^\circ}^{H+1-n-i}\big(\phi_x \geq 0,\, \forall x \in D_i^{\circ}\cap F\big)
        \,,\]
and next wish to qualify for an application of \cref{thm:key-area}, albeit this time in the macroscopic range.
Let us indeed verify the required hypotheses for applying said theorem for $\hatpi_{D_i^\circ}^0$ ($i=0,1$).
\begin{enumerate}
    \item The condition on the domain boundary: $|\partial D_i^\circ| \leq O(|\gamma| + |\partial V|) =O(L e^{\sqrt{\log L}})$ as per $\cG$.
    \item The stronger requirement needed to forgo the conditioning on $\fS$: we wish to satisfy
    \begin{equation}\label{eq:Di-F-bound} |D_i^\circ \cap F| \leq \Big(\frac{3\beta}{\hatpi_\infty(\phi_o > H+1-n-i)}\Big)^2\,,\end{equation}
    which is guaranteed (moreover for $D_i\cap F$) as per $\cG$.
\end{enumerate}
Therefore, \cref{eq:xi-n-area-estimate-macro-lower,eq:xi-n-area-estimate-macro-upper} are applicable, where in the latter we have the unconditional version secured by that theorem via the additional requirement \cref{eq:Di-F-bound}, and we infer that for $i=0,1$,
		\begin{align*}
			\hatpi_{D_i^\circ}^{H+1-n-i}\big(\phi_x \geq 0,\, \forall x \in D_i^\circ\cap F\big) &=
			\exp\left(-\upxi_{n+i}|D_i \cap F| + O\big(\sqrt L e^{\frac{\log L}{\log\log L}}\big) \right)\,.
		\end{align*}
Note that, in the above display, we further replaced $\upxi_{n+i}|D_i^\circ \cap F|$ by $\upxi_{n+i}|D_i \cap F|$, absorbing the difference between the two into the $O(L^{1/2+o(1)})$ error term since
\[ \upxi_{n+i}|D_i\setminus D_i^\circ| \leq L^{-1+o(1)} O(|\gamma|) \leq L^{o(1)}\,.\]

The $1+o(1)$ factor that was associated with \eqref{eq:xi-n-area-estimate-meso} from that theorem in the mesoscopic range, in the proof of \cref{prop:CE-meso}, will thereby be replaced by the $\exp(L^{1/2+o(1))}$ from \cref{eq:xi-n-area-estimate-macro-lower,eq:xi-n-area-estimate-macro-upper} in the macroscopic range. Proceeding as in the proof of 
\cref{eq:p-law-meso} by introducing a factor of $\exp(\upxi_{n+1}|F|)$, that is shifted into the partition function, we arrive at \cref{eq:CE-with-area-cond-new}, as required.
\end{proof}

\begin{proposition}[probability of a given disagreement polymer with uniform boundary conditions] \label{prop:CE-macro-uniformBC}
Fix $g\geq 0$, $n \geq 0$ and set $h = H+1-n$. Let $F\subseteq V \subset \Z^2$ where $V$ is connected, contains the origin and has $g$ holes. Suppose  $|\partial F|\vee|\partial V| \leq O(Le^{\sqrt{\log L}})$. Let $\gamma$ be a disagreement polymer where $D_0$ is the region containing the origin with $h_0=h$, and $D_1$ is the region containing the boundary vertices $\partial V$ with $h_1=h-1$. Moreover, assume that $|D_1 \cap F| \leq \big( \frac{3\beta}{\hatpi_\infty(\phi_o > h-1)}\big)^2$  and $|\gamma| \leq Le^{\sqrt{\log L}}$. Then, letting 
\[ \fp(\gamma) = \exp\bigg(-\sE^*_\beta(\gamma) + \frac{\uprho_n}{N_n}|D_0 \cap F| + \fI_{V}(\gamma) + \upxi_{n+1}\sum_{i\geq 2}|D_i \cap F|\bigg)\,, \]
we have
    \begin{align}
        \label{eq:p-gamma-lower-upper}
    \pi_{V;F}^{h-1}(\fS) e^{O(L^{1/2+o(1)})} 
   \leq 
 \frac{\pi^{h-1}_{V; F}(\cC_{\gamma, h})} {\fp(\gamma)} \leq 
 \frac{1}{\pi_{D_0^\circ\cap F}^h(\fS)} e^{O(L^{1/2+o(1)})}\,.
\end{align}
\end{proposition}
\begin{proof}
By the cluster expansion for disagreement polymers (see~\cite[Eq.~(2.28)]{ChenLubetzky25}), we know that if $w(\phi) := \prod_{\gamma'\in \phi}\exp(-\sE_\beta(\gamma'))$ and $\widehat Z_{V}^h = \sum_{\phi} w(\phi)$, then 
\begin{equation}\label{eq:ratio-Zhats} \frac{\prod_i \widehat Z_{D_i^\circ}^{h_i}}{\widehat Z_V^0} =\exp\bigg( \sum_{\substack{\sfW\subset V\\ \sfW\cap\Delta^*_\gamma\neq\emptyset}} \Phi(\sfW)\bigg) = \exp\left(\fI_V(\gamma)\right)\,.\end{equation}
Letting $Z_{V;F}^{h-1} = \sum_\phi w(\phi) \one_{\{\phi_x\geq 0,\,\forall x\in F\}}$ for the $\ZGFF$ with boundary conditions $h-1$ and a zero floor in $F$ (as well as defining $Z_{D_i^\circ;F}^{h_i}$ analogously),  we have
    \begin{align*}
        \pi^{h-1}_{V;F}(\cC_{\gamma, h}) &= \frac{1}{Z_{V;F}^{h-1}}\exp\left(-\sE_\beta(\gamma)\right)\prod_{i\geq 0}Z_{D_i^\circ;F}^{h_i}\,,
        \end{align*}
  and since $Z_{V;F}^{h-1} = \widehat Z_{V;F}^h \hatpi_V^{h-1}(\phi_x\geq 0,\,\forall x\in F)$ (and analogously for $Z_{D_i^\circ;F}^{h_i}$), we infer that
        \begin{align}
      \pi^{h-1}_{V;F}(\cC_{\gamma, h})  &= \frac{1}{\hatZ_{V}^{h-1}\hatpi_V^{h-1}(\phi_x \geq 0,\, \forall x \in F)}\exp\left(-\sE_\beta(\gamma)\right)\prod_{i\geq 0}\hatZ_{D_i^\circ}^{h_i}\hatpi_{D_i^\circ}^{h_i}(\phi_x \geq 0,\, \forall x \in D_i^\circ\cap F)\nonumber\\
        &=\exp\Big(-\sE^*_\beta(\gamma) + \fI_V(\gamma)\Big)\frac{\prod_{i =0,1}\hatpi_{D_i^\circ}^{h-i}(\phi_x \geq 0,\, \forall x \in F \cap D_i^\circ)}{\hatpi_V^{h-1}(\phi_x \geq 0,\, \forall x \in F)}\,,\label{eq:pi(gamma)-uniform}
    \end{align}
    where the last transition used \cref{eq:ratio-Zhats} and the definition of $\sE^*_\beta$ from \cref{eq:E*-def}. 

    For the upper bound on \cref{eq:p-gamma-lower-upper}, we 
 can lower bound the denominator in the right of \cref{eq:pi(gamma)-uniform} using \cref{eq:xi-n-area-estimate-macro-lower} of \cref{thm:key-area}, showing that
\begin{align}\pi^{h-1}_{V;F}(\cC_{\gamma, h})  \leq & \exp\bigg(-\sE^*_\beta(\gamma)+ \fI_V(\gamma) + \upxi_{n+1}\sum_{i\geq 0}|D_i\cap F| + O(L^{1/2+o(1)})\bigg) \nonumber\\ & \cdot \prod_{i =0,1} \hatpi_{D_i^\circ}^{h-i}\left(\phi_x \geq 0,\, \forall x \in  D_i^\circ\cap F\right) \,,\label{eq:piC-gamma-h-upper-bound}
\end{align}
and we move our attention to the two terms in the product over $i=0,1$. 

For $i=1$, the assumption on $|D_1\cap F|$ allows us to infer from \cref{eq:xi-n-area-estimate-macro-upper} the unconditional bound
\[
\hatpi_{D_1^\circ}^{h-1}\left(\phi_x \geq 0,\, \forall x \in  D_1^\circ\cap F\right) \leq \exp\Big(-\upxi_{n+1}|D_1^\circ \cap F| + O(L^{1/2+o(1)})\Big) \,.
\]
Since $\upxi_{n+1}|D_1\setminus D_1^\circ| \leq L^{-1+o(1)}O(|\gamma|) = L^{o(1)}$, we can replace $\exp(-\upxi_{n+1}|D_1^\circ \cap F|)$ in the above by $\exp(-\upxi_{n+1}|D_1 \cap F|)$, which cancels the corresponding term in the right hand of \cref{eq:piC-gamma-h-upper-bound}.

For $i=0$, by \cref{eq:xi-n-area-estimate-macro-upper} (in the conditional version, as we did not assume that $|D_0^\circ\cap F|$ is small):
\[ \hatpi_{D_0^\circ}^{h}\left(\phi_x \geq 0,\, \forall x \in  D_0^\circ\cap F\mid \fS \right) \leq \exp\Big(-\upxi_{n}|D_0^\circ \cap F| + O(L^{1/2+o(1)})\Big) \,.
\]
As in the proof of \cref{prop:key-area-estimate}, the routine bound $\P(A) \leq \frac{\P(A\mid B)}{\P(B\mid A)}$ for  $A = \{\phi_x\geq -h,\,\forall x\in F\}$ and $B = \fS$ shows that 
\[ \hatpi_{D_0^\circ}^{h}\left(\phi_x \geq 0,\, \forall x \in  D_0^\circ\cap F \right) \leq \frac{\exp\Big(-\upxi_{n}|D_0 \cap F| + O(L^{1/2+o(1)})\Big)}{ \hatpi_{D_0^\circ;F}^h(\fS)} \,,
\]
where again we absorbed the move from $\upxi_n|D_0^\circ\cap F|$ to $\upxi_n|D_0\cap F|$ into the error term.
Together with the $\upxi_{n+1}|D_0 \cap F|$ from \cref{eq:piC-gamma-h-upper-bound}, we see that the prefactor of $|D_0\cap F|$ is $\upxi_{n+1}-\upxi_n = \uprho_n/N_n$ as per \cref{eq:rho-n-def}, thus arriving at the upper bound of \cref{eq:p-gamma-lower-upper}.

For the lower bound in \cref{eq:p-gamma-lower-upper}, since $\pi_{V;F}^{h-1}(\fS)$ is nothing but $\hatpi_V^{h-1}(\fS\mid\phi_x\geq 0,\,\forall x\in F)$, we see from \cref{eq:pi(gamma)-uniform} that 
\begin{align}
      \frac{\pi^{h-1}_{V;F}(\cC_{\gamma, h})}{\pi_{V;F}^{h-1}(\fS)}  &=\exp\Big(-\sE^*_\beta(\gamma) + \fI_V(\gamma)\Big)\frac{\prod_{i =0,1}\hatpi_{D_i^\circ}^{h-i}(\phi_x \geq 0,\, \forall x \in F \cap D_i^\circ)}{\hatpi_V^{h-1}(\{\phi_x \geq 0,\, \forall x \in F\}\cap \fS)} \nonumber\\
      &\geq \exp\Big(-\sE^*_\beta(\gamma) + \fI_V(\gamma)\Big)\frac{\prod_{i =0,1}\hatpi_{D_i^\circ}^{h-i}(\phi_x \geq 0,\, \forall x \in F \cap D_i^\circ)}{\hatpi_V^{h-1}(\phi_x \geq 0,\, \forall x \in F\mid \fS)}\,. 
      \label{eq:piC-gamma-h-lower-bound}
    \end{align}
   By \cref{eq:xi-n-area-estimate-macro-lower,eq:xi-n-area-estimate-macro-upper} of \cref{thm:key-area},
   \begin{align*}
   \hatpi_V^{h-1}(\phi_x\geq0,\,\forall x\in F\mid \fS) &\leq \exp\Big(-\upxi_{n+1}|F|+O(L^{1/2+o(1))})\Big)\,,\\
   \hatpi_V^{h-1}(\phi_x\geq0,\,\forall x\in D_i^\circ \cap F) &\geq \exp\Big(-\upxi_{n+1}|D_i \cap F|+O(L^{1/2+o(1))})\Big)\quad \mbox{for $i=0,1$}\,,
   \end{align*} 
where, as before, we moved from $|D_i\cap F|$ from $|D_i^\circ \cap F|$ via an additive (negligible) error of $O(L^{o(1)})$.
Substituting this in \cref{eq:piC-gamma-h-lower-bound} establishes the lower bound in \cref{eq:p-gamma-lower-upper}, completing the proof.
\end{proof}

\begin{corollary}\label{cor:CE-macro-uniformBC-smallL}
Consider the setting of \cref{prop:CE-macro-uniformBC} for $n=0$, where the assumption on $|D_1\cap F|$ is implied by $|D_1\cap F|\leq L^2/3$. If we further assume that
$|F|\leq L^2$ then we have
    \begin{align}
        \label{eq:p-gamma-lower-upper-n=0}
    \pi_{V;F}^{H}(\fS)\exp\Big(O(L^{1/2+o(1)})\Big) 
   \leq 
 \frac{\pi^{H}_{V; F}(\cC_{\gamma, H+1})}{\fp(\gamma)} \leq  \exp\Big( O(L^{1/2+o(1)})\Big) \,.
\end{align}
\end{corollary}
\begin{proof}
By the definition of $H$, we have $\hatpi_\infty(\phi_o = H) \geq \frac{5\beta}L > \hatpi_\infty(\phi_o = H+1)$, and we immediately observe that the condition $|D_1\cap F|^{1/2}\leq 3\beta / \hatpi_\infty(\phi_o 
> H)$ would be satisfied if $|D_1\cap F|^{1/2} \leq \frac35 L$. The lower bound in \cref{eq:p-gamma-lower-upper-n=0} is verbatim the one given in \cref{eq:p-gamma-lower-upper}, and it remains to show that $\pi_{D_0^\circ\cap F}^{H+1}(\fS) = O(L^{-1})$, say, so that the factor $1/\pi_{D_0^\circ\cap F}^{H+1}(\fS)$ in the upper bound in \cref{eq:p-gamma-lower-upper} would be overtaken by the $O(L^{1/2+o(1)})$ in its exponent.
To see this, recall as usual that ruling out \texttt{large} $h-$level lines for $h<H+1$ is immediate from the standard Peierls argument, which is unperturbed by the presence of a floor (see, e.g., \cite[Eq.~(4.2) of Prop.~4.1]{LMS16}).
Consider an $(H+2)$ level line~$\fL$.
Recalling that $\hatpi_\infty(\phi_o>h) \leq e^{-c \beta \frac{h}{\log h}} \hatpi_\infty(\phi_o = h) = o(\hatpi_\infty(\phi_o=h))$ by \cref{eq:LD-ratio}, we see that $\hatpi_\infty(\phi_o = H+2) = o(1/L)$.
Therefore, by \cref{lem:prob-of-contour}, combined  with the isoperimetric inequality $|\Int(\fL)\cap F|\leq |\fL| \sqrt{|F|}/4 \leq O(|\fL| L)$ (using here our assumption on $|F|$), it follows that 
\begin{align*}
   \pi_{D_0^\circ\cap F}^{H+1}(\cC_{\fL,H+2}) &\leq \exp\left(-\Big(\beta  -  \hatpi_\infty(\phi_o > H+1)O(L)  - o(1)\Big) |\fL| \right) 
    \leq \exp\left(-\left(\beta   - o(1)\right) |\fL| \right)\,.
\end{align*}
At this point, a union bound over all \texttt{large} level-line loops $\fL$ shows that, say, $\pi_{D_0^\circ\cap F}^{H+1}(\fS) \geq 1 - L^{-10}$, and the proof is complete.
\end{proof}

\section{Asymptotically small critical window}\label{sec:transition-window}
In this section we prove \cref{thm:main-thm-crit-window}, showing that except for an asymptotically small critical window about $L_c^{(h)}$, we can determine w.h.p.\ whether the top level line is at height $H$ or height $H+1$. We broadly follow the proof strategy of \cite[Sections 4.2, 4.3]{CLMST16}. However, we will require some extra quantitative estimates utilizing \cref{lem:linear-critical-window}, pushing towards a critical window of width at most $(L_*^{(h)})^{1/2+o(1)}$. We will also need the machinery on disagreement polymers developed in \cite{ChenLubetzky25}, and most importantly the law of disagreement polymers in macroscopic domains obtained in \cref{sec:refined-CE}. 

Define $L_*^{(h)} = \lceil\tfrac{\lambda_*\beta}{\uprho_0\hatpi_\infty(\phi_o = h)}\rceil$, recalling that $\lambda_* \approx 4(1\pm \epsilon_\beta)$. Throughout this section, we will fix $h \geq h_0$ for some $h_0$ sufficiently large depending on $\beta$, and define $\lambda = \lambda(L)$ to be such that
$L = \frac{\lambda \beta}{\hatpi_\infty(\phi_o = h)}$, or equivalently $\hatpi_\infty(\phi_o = h) = \tfrac{\lambda\beta}{L}$. We split \cref{thm:main-thm-crit-window} into two propositions, showing separately that below the window $[L_*^{(h)} - (L_*^{(h)})^{1/2+o(1)}, L_*^{(h)} + (L_*^{(h)})^{1/2+o(1)}]$ the probability of having a \texttt{large} $h$ level line is $o(1)$, and above the window this probability is $1-o(1)$.
\begin{proposition}\label{prop:below-window}
    There exists $h_0 > 0$ such that for all $h \geq h_0$, for all $L \leq L_*^{(h)} - (L_*^{(h)})^{1/2+o(1)}$, the \ZGFF model on $ $ w.h.p.\ has no \texttt{large} level line at height $h$. 
\end{proposition}
\begin{proposition}\label{prop:above-window}
    There exists $h_0 > 0$ such that for all $h \geq h_0$, for all $L \geq L_*^{(h)} + (L_*^{(h)})^{1/2+o(1)}$, the \ZGFF model on $\Lambda$ w.h.p.\ has a \texttt{large} level line at height $h$.
\end{proposition}
The exponents of $1/2+o(1)$ come from the error term in \cref{prop:CE-macro-uniformBC}, which ultimately comes from  \cref{prop:key-area-estimate}.

\begin{remark}\label{rem:bad-prediction}
The term $L_*^{(h)}$ predicts where $L_c^{(h)}$ is by solving an optimization problem using this prefactor of $\uprho_0$ on the area tilt. One could define analogously a $\widetilde L_*^{(h)} := \lfloor \frac{\lambda_*\beta}{\hatpi_\infty(\phi_o = h)}\rfloor$, which predicts the location of $L_c^{(h)}$ using the same analysis but using an area tilt of $\frac{1}{N_0}$, as done in previous works (e.g. \cite{CLMST16,ChenLubetzky25}). Indeed, \cite{CLMST16} proves for \SOS that $L_c^{(h)} \in \widetilde L_*^{(h)} \pm (\widetilde L_*^{(h)})^{o(1)}$, and at first glance it appears that perhaps $L_c^{(h)} = \widetilde L_*^{(h)}$ since the errors leading to this estimate seem purely technical. 

However, at least for the \ZGFF, this is not the case. The difference between $L_*^{(h)}$ and $\widetilde L_*^{(h)}$ is of order $L_*^{(h)}(1 - \uprho_n)$. At the same time, in this section we prove that the window around $L_*^{(h)}$ where we can guarantee $L_c^{(h)}$ is has size $(L_*^{(h)})^{1/2+o(1)}$ (primarily coming from the error term in the macroscopic range results of \cref{thm:key-area}). Hence, if $1 - \uprho_n \leq L^{-1/2 + o(1)}$, then the gap between $L_*^{(h)}$ and $\widetilde L_*^{(h)}$ is larger than the window around $L_*^{(h)}$. This implies not only that $L_c^{(h)} \neq \widetilde L_*^{(h)}$, but moreover that the gap between the two is wide enough such that at $\widetilde L_*^{(h)}$ there is w.h.p.\ no \texttt{large} $h$ level line. In other words, using an area tilt of $\frac{1}{N_0}$ gives an erroneous prediction for the real transition point for the $h$ level line.
\end{remark}

To connect to the notation of $H(L)$, note that for every $L$ in the interval $[\tfrac{5\beta}{\hatpi_\infty(\phi_o= h-1)}, \tfrac{5\beta}{\hatpi_\infty(\phi_o= h)})$ we have $H(L) = h-1$. Hence, the above propositions are equivalent to saying that for any $h \geq h_0$,
    \begin{align*}&L \in \big[\tfrac{5\beta}{\hatpi_\infty(\phi_o = h-1)}, L_*^{(h)} - (L_*^{(h)})^{1/2+o(1)}\big]\leftrightarrow \mbox{ top level line at height } H(L)\,,\\
    &L \in \big(L_*^{(h)} - (L_*^{(h)})^{1/2+o(1)}, L_*^{(h)} + (L_*^{(h)})^{1/2+o(1)}\big)\leftrightarrow \mbox{ top l.l. at height } H(L) \mbox{ or } H(L) + 1\,,\\
    &L \in \big[L_*^{(h)} + (L_*^{(h)})^{1/2+o(1)}, \tfrac{5\beta}{\hatpi_\infty(\phi_o = h)}\big)\leftrightarrow \mbox{ top l.l. at height } H(L)+1\,.
\end{align*}
We will henceforth write $H+1$ instead of $h$ to be consistent with the rest of the paper.

\subsection{Preliminary lemmas}
We begin with preliminary bounds on the level lines. When $L$ is too small, we can rule out the existence of $\fL_0$ using the crude estimate of \cref{lem:prob-of-contour}. When $L$ too big, we can guarantee the existence of $\fL_0$ by the stitching procedure of \cite{LMS16}, and ensure $\fL_0$ contains a large square. For $L$ in between, we cannot a-priori say anything about the existence of $\fL_0$ (indeed refining this is the whole goal of this section), but we can still ensure that if $\fL_0$ exists it must still contain a large square.
\begin{lemma}\label{lem:prelim-ll-bounds}
    Recall that $\lambda$ satisfies $L = \frac{\lambda\beta}{\hatpi_\infty(\phi_o = H+1)}$. For $\beta$ large, there exists $\epsilon_\beta^\square \downarrow 0$ such that the following holds w.h.p.:
    \begin{enumerate}
        \item\label{it:small-lambda} If $\lambda \leq 4(1-\frac{3}{\beta})$, there are no \texttt{large} $H+1$ level lines.
        \item\label{it:big-lambda} If $\lambda \geq 4(1+\frac{3}{4\beta})$, there is a \texttt{large} $H+1$ level line containing a square of length $L(1-\epsilon_\beta^\square)$.
        \item\label{it:middle-lambda} For $\lambda \in [4(1-\frac{3}{\beta}), 4(1+\frac{3}{4\beta})]$, if a \texttt{large} $H+1$ level line exists, then it must contain a square of length $L(1-\epsilon_\beta^\square)$.
    \end{enumerate}
\end{lemma}
\begin{proof}
For any loop $\fL$, clearly $|\Int(\fL)| \leq L^2$. So, we have that if $|\fL| \geq \lambda L(1 - \frac{3}{\beta})^{-1}$, then \cref{lem:prob-of-contour} gives
\begin{equation}\label{eq:isoperimetry-bound-1}
    \pi_V^0(\cC_{\fL, H+1}) \leq \exp(-(\beta+o(1))|\fL| + \lambda\beta L) \leq \exp(-(3+o(1))|\fL|)\,.
\end{equation}
Additionally, since $|\Int(\fL)| \leq |\fL|^2/16$ by the isoperimetric inequality on $\Z^2$, we have that if $|\fL| \leq \frac{16L}{\lambda}(1 - \frac{3}{\beta})$, then
\[\pi_V^0(\cC_{\fL, H+1}) \leq \exp(-(\beta+o(1))|\fL| + \beta(1 - \tfrac{3}{\beta})|\fL|) \leq \exp(-(3+o(1))|\fL|)\,.\]
Hence we can rule out the existence of any $\fL$ such that $|\fL| \geq \lambda L(1 - \frac{3}{\beta})^{-1}$ or $\log L \leq |\fL| \leq \frac{16L}{\lambda}(1 - \frac{3}{\beta})$ by the standard Peierls argument. Hence, if $\lambda \leq 4(1 - \frac{3}{\beta})$, there is not gap between the two ranges on $|\fL|$, so there are w.h.p.\ no \texttt{large} $H+1$ level lines. 

Next, \cref{it:big-lambda} was essentially already proved via the stitching argument of \cite[\S4.4]{LMS16} (this was already summarized in \cite[Clm.~4.19]{ChenLubetzky25}, so we will not provide the details again here). Indeed, for such $L$, we have $\frac{4\beta + 4}{\hatpi_\infty(\phi_o \geq H)} \leq L^{1-o(1)}$, so either using the stitching argument of \cite{LMS16} or just \cref{thm:old-level-line-contains-Wulff} gives that in $\pi^0_{\Lambda}$ there is w.h.p.\ an $H$ level line distance at most $L^{1-o(1)}$ from $\partial \Lambda$. Then, the stitching argument says that as long as there exists $\frac{4\beta+2}{\hatpi_\infty(\phi_o \geq H+1)} \leq \ell \leq \frac{4\beta+4}{\hatpi_\infty(\phi_o \geq H+1)}$ such that a box of length $\ell + (\log \ell)^2$ fits inside the $H$ level line, then w.h.p.\ there exists a \texttt{large} $H+1$ level line surrounding $\Lambda_{L(1-\epsilon_\beta^\square)}$. So we are done as long as $L - L^{1-o(1)} > \frac{4\beta+2}{\hatpi_\infty(\phi_o \geq H+1)} + (\log \frac{4\beta+2}{\hatpi_\infty(\phi_o \geq H+1)})^2$. Plugging in the relation between $\lambda$ and $L$, this is equivalent to $1 - L^{-o(1)} > \frac{4\beta + 2}{\lambda \beta}$, which is satisfied by $\lambda \geq 4(1+\frac{3}{4\beta})$.

Finally to show \cref{it:middle-lambda}, recall from above that we only need to consider $\fL$ such that $|\fL| \geq \frac{16L}{\lambda}(1 - \frac{3}{\beta})$. Plugging this and the bound $\lambda \leq 4(1+\frac{3}{4\beta})$ into \cref{lem:prob-of-contour} gives that for some choice of $\epsilon_\beta$, we can assume that $|\Int(\fL)| \geq L^2(1 - \epsilon_\beta)$. The fact that $\fL$ now contains an appropriately sized square follows deterministically by \cite[Lem.~2.6]{CLMST16}.
\end{proof}

\subsection{No large \texorpdfstring{$H+1$}{H+1} level lines below the window}
In this subsection we prove \cref{prop:below-window}. By \cref{lem:prelim-ll-bounds}, we can assume additionally that $\lambda \geq 4(1-\epsilon_\beta)$, and that we can already rule out level lines $\fL$ which do not contain a macroscopic square. This simplified geometry will allows us to study the law of disagreement polymers containing $\fL$ using cluster expansion. Also important is that a level line containing a square is an increasing property. Hence, leveraging monotonicity, we can use a domain enlargement trick to handle interactions between $\fL$ and the domain boundary.

Let $\cE_\square$ be the event that there is an $H+1$ level line $\fL$ containing a square $\square$ of side length $L(1 - \epsilon_\beta^\square)$. Let $\Lambda'$ be the square of side length $5L$, concentric with $\Lambda$. By FKG (and since 0 is the minimum value of the \ZGFF model with a floor at 0), we have $\pi^0_\Lambda \preceq \pi^{H}_{\Lambda'; \Lambda}$ where the latter is the \ZGFF model on $\Lambda'$ with boundary conditions $H$ with a floor at 0 only on sites in $\Lambda$. In particular since $\cE_\square$ is increasing, we have $\pi^0_\Lambda(\cE_\square) \leq \pi^{H}_{\Lambda'; \Lambda}(\cE_\square)$. Abusing notation, we will say that $\fL \in \cE_\square$ if $\fL$ is an $H+1$ level line containing a square of side length $L(1-\epsilon_\beta^\square)$. We will also write $\gamma \in \cE_\square$ if $\gamma$ is a disagreement polymer such that there exists $\fL \in \cE_\square$ and $\fL$ is a subset of the bonds of $\gamma$. Now fix a $\gamma \in \cE_\square$. In order to show that $\pi^{H}_{\Lambda'; \Lambda}(\cE_\square) = o(1)$, we need a precise control over the law of $\gamma$. We begin with a claim which rules out a bad set of $\gamma$ not conducive to an application of \cref{cor:CE-macro-uniformBC-smallL}.
\begin{claim}\label{clm:CE-check-LW}
    Let $\cG$ denote the set of $\gamma$ such that $\sE_\beta(\gamma) \leq 4.1L$. Then, under $\pi^{H}_{\Lambda'; \Lambda}$, the event $\gamma \in \cG \cap \cE_\square$ implies that  $|\gamma| \leq 4.1L$, $h_0 = H+1$, and $h_1 = H$. Moreover, $\pi^{H}_{\Lambda'; \Lambda}(\gamma \in \cE_\square \setminus \cG) = o(1)$.
\end{claim}
\begin{proof}
    By FKG and the decorrelation estimate of \cref{eq:couple-pi-to-inf-vol}, a standard computation gives a lower bound on the floor event:
    \[\hatpi^H_{\Lambda'; \Lambda}(\phi_x \geq 0,\, \forall x \in \Lambda) \geq \prod_{x \in \Lambda} \hatpi^H_{\Lambda'; \Lambda}(\phi_x \geq 0) \geq \tfrac12\exp\big(-|\Lambda|\hatpi_\infty^H(\phi_x < 0)\big) \geq \tfrac12e^{-4(1+\epsilon_\beta)\beta L}\,,\]
    where the last line uses the fact that $\hatpi_\infty^H(\phi_x < 0) = \hatpi_\infty^0(\phi_x \geq h) = (1+o(1))\tfrac{\lambda\beta}{L}$ and the bound on $\lambda$. Hence, it suffices to show that in the no-floor measure, we have $\hatpi^H_{\Lambda'}(\gamma \in \cE_\square \setminus \cG) = o(e^{-4(1+\epsilon_\beta)\beta L})$. But here we can use \cref{prop:CE-law-without-area}. In particular, since $\fI_{\Lambda'}(\gamma) \leq \epsilon_\beta|\gamma|$, we can conclude by a standard Peierls argument mapping $\gamma$ to the disagreement polymer with $4L$ bonds forming the boundary of a square, that $\hatpi^H_{\Lambda'}(\exists \gamma,\, \sE_\beta(\gamma) \geq 4.1L) \leq e^{-4.1\beta L}$. 
    
    Finally, it is immediate that $|\gamma| > 4.1L$ implies $\sE_\beta(\gamma) > 4.1L$. Moreover, if $h_0 \neq H+1$ or if $h_1 \neq H$, the condition that $\gamma \in \cE_\square$ implies that either there is a height gradient of $\geq 2$ along $4(1-\epsilon_\beta)L$ of the bonds in $\gamma$, or there are at least $8(1-\epsilon_\beta)L$ bonds in $\gamma$, both of which imply that $\sE_\beta(\gamma) > 4.1L$.
\end{proof}
For $\gamma \in \cE_\square$, using the bound on $\lambda$ and the large deviation ratios in \cref{thm:LD-DG}, we also have $|D_1 \cap \Lambda| \leq 4\epsilon_\beta^\square L^2$. Hence, for all $\gamma \in \cG \cap \cE_\square$, the conditions of \cref{cor:CE-macro-uniformBC-smallL} are satisfied, and we have 
\begin{align}\label{eq:CE-middle-step}
    \nonumber\pi^H_{\Lambda'; \Lambda}(\cC_{\gamma, H+1}) &\leq \exp\bigg(-\sE^*_\beta(\gamma) + \fI_{\Lambda'}(\gamma) + \tfrac{\uprho_0\lambda\beta}{L}|\Lambda \cap D_0| + \upxi_{1}\sum_{i\geq 2}|D_i \cap F|+ O(L^{1/2+o(1)})\bigg)\\
    &=: \fp_{\Lambda';\Lambda}(\gamma)e^{O(L^{1/2+o(1)})}\,.
\end{align}
With \cref{lem:prelim-ll-bounds,clm:CE-check-LW}, our goal now is to show that \[\sum_{\gamma \in \cG \cap \cE_\square} \fp_{\Lambda';\Lambda}(\gamma)e^{CL^{1/2+o(1)}} = o(1), \quad\mbox{ or equivalently }\quad \sum_{\gamma \in \cG \cap \cE_\square} \fp_{\Lambda';\Lambda}(\gamma) = o(e^{-CL^{1/2+o(1)}})\,.\] 

We begin by showing that the term $\upxi_1\sum_{i\geq 2}|D_i \cap F|$ is negligible. In fact, the proof is essentially the same as in \cref{prop:CE-meso}, only this time we are in a macroscopic setting and need to ensure that we do not violate the additional condition $\cE_\square$. For completeness we provide the details this time, though when similar estimates are needed in future lemmas we will simply refer back to the proof here or in \cref{prop:CE-meso} for details. 

To start, it is easy to show a crude bound on the maximum size of a component $\cD \in \fD(\gamma)$, say by $|\partial \cD| \leq L/5$. Let $u, v$ be the two cut-points delineating $\cD$. Then, $|\partial \cD| \geq 2\norm{u- v}_1$ since $\norm{u- v}_1$ is the length of the shortest path from $u$ to $v$, and $\partial \cD$ consists of two such paths. Observe that $u, v$ must also lie on the level line $\fL$ which $\gamma$ contains. Since $\fL \in \cE_\square$, the minimal length path between all of its cut-points is necessarily $\geq 4L(1-\epsilon_\beta^\square)$. Hence, we have $|\gamma| \geq 4L(1-\epsilon_\beta^\square) + \sum_{\cD \in \fD(\gamma)} |\partial \cD|$, whence the assumption $|\gamma| \leq 4.1L$ immediately implies the bound \begin{equation}\label{eq:rough-max-D}
    \max_{\cD \in \fD(\gamma)}|\partial \cD| \leq L/5\,.
\end{equation}

We next use a Peierls argument to show that for some constant $c > 0$,
\begin{equation}\label{eq:control-finite-comp-square}\sum_{\substack{\gamma \in \cG \cap \cE_\square\\
\max_{\cD \in \fD(\gamma)}|\partial \cD| > \log L}} \fp_{\Lambda';\Lambda}(\gamma) \leq e^{-c\beta \log L}\sum_{\gamma \in \cG \cap \cE_\square} \fp_{\Lambda';\Lambda}\,.
\end{equation}
Again we can fix $b$ and consider the map $T_b$ on $\gamma$ which replaces $\cD_b$ by a (arbitrarily chosen, say the outermost such) minimal length path from $u$ to $v$ that remains within $\cD_b$. As before, we have $\sE^*_\beta(T_b(\gamma)) \leq \sE^*_\beta(\gamma) -\tfrac\beta2|\partial \cD_b|$. Hence, by the decay properties of $\Phi$ and the bound on $\lambda$, we therefore have after accounting for the area terms that 
\[\fp_{\Lambda';\Lambda}(\gamma) \leq \fp_{\Lambda';\Lambda}(T_b(\gamma))e^{\frac{\beta - C}2|\partial \cD_b|+\tfrac{4(1+\epsilon_\beta)\beta}{L}|\cD_b|}\,.\]
By isoperimetry and the crude bound in \cref{eq:rough-max-D}, we have $|\cD_b| \leq |\partial \cD_b|^2/16 \leq |\partial \cD_b|L/80$, so that the above display implies that for some $c > 0$, 
\[\fp_{\Lambda';\Lambda}(\gamma) \leq \fp_{\Lambda';\Lambda}(T_b(\gamma))e^{-c\beta|\partial \cD_b|}\,.\]

Moreover, we claim that the resulting disagreement polymer satisfies $T_b(\gamma) \in \gamma \in \cG \cap \cE_\square$. The bound on $\sE_\beta(T_b(\gamma))$ is easily satisfied since the map $T_b$ only reduces the length of $\gamma$. Showing that $T_b(\gamma) \in \cE_\square$ amounts to showing that we can select a minimal length path from $u$ to $v$ staying within $\cD_b$ that stays outside of the square, $\square$, which $\fL$ contains. Suppose for contradiction that there is no such path. Then, take a shortest path $P$, and let $w, z$ be two points of $P$ on $\partial \square$ which mark an excursion into $\square$ in the sense that in between $w, z$, $P$ lies inside $\square$. Recalling the outer envelope in \cref{def:outer-envelope}, observe that the region sandwiched between $P$ and the arc of $\mathsf{OE}(\gamma)$ between $u, v$ is contained in $\cD_b$. Since $\fL$ lies outside of $\square$, so does $\mathsf{OE}(\gamma)$. Since the arc of $P$ from $w$ to $z$ is inside $\square$, the line segment $\overline{wz}$ (or the $L$ shaped path from $w$ to $z$, if $w, z$ are on two different sides of the square $\square$) is contained in this sandwiched region, and is therefore in $\cD_b$. Hence, the path $P'$ which replaces the arc of $P$ between $w$ and $z$ by $\overline{wz}$ is in $\cD_b$. Since $|P'| < |P|$, this is a contradiction. 

Finally, as before, the number of preimages $\gamma$ under $T_b$ for a given disagreement polymer $\gamma'$ is at most $k^2C^k$. The exact same computation in \cref{eq:map-argument-Di} now concludes the proof of \cref{eq:control-finite-comp-square}. 

With \cref{eq:control-finite-comp-square} proven, we can now restrict our attention to $\gamma$ which additionally satisfies $\max_{\cD \in \fD(\gamma)}|\partial \cD| \leq \log L$. Together with the isoperimetric inequality and the restriction that $|\gamma| \leq 4.1L$, this in particular implies that $\upxi_1\sum_{i\geq 2}|D_i \cap F| \leq O(\log L)$. Hence, our updated goal is now to show that 
\[\sum_{\substack{\gamma \in \cG \cap \cE_\square\\
\max_{\cD \in \fD(\gamma)}|\partial \cD| \leq \log L}} e^{-\sE^*_\beta(\gamma) + \fI_{\Lambda'}(\gamma) + \tfrac{\uprho_0\lambda\beta}{L}|\Lambda \cap D_0|} =: \sum_{\substack{\gamma \in \cG \cap \cE_\square\\
\max_{\cD \in \fD(\gamma)}|\partial \cD| \leq \log L}} \fq_{\Lambda';\Lambda}(\gamma)\leq o(e^{-CL^{1/2+o(1)}})\,.\]
We next continue to rule out a set of unlikely $\gamma$ by Peierls type arguments.

\begin{definition}
    For any two cut-points $v, v' \in \gamma$, let $d_{\gamma}(v, v')$ denote the number of bonds in between $v$ and $v'$. More precisely, removing $v$ and $v'$ splits $\gamma$ into two connected components, and $d_{\gamma}(v, v')$ is the number of bonds in the smaller of the two components. 
\end{definition}

\begin{definition}
    We say that $\gamma$ has a button hole if there are two cut-points $v, v' \in \gamma$, not on opposite sides of the square in the event $\cE_\square$, such that $d_{\gamma}(v, v') \geq \log L$ while $\norm{v - v'}_1 \leq \tfrac12d_{\gamma}(v, v')$.
\end{definition}  
\begin{lemma}
    Write $\gamma \in \cB$ to mean that $\gamma$ has a button hole. Then, 
    \begin{equation}\label{eq:control-button-hole}\sum_{\substack{\gamma \in \cG \cap \cE_\square\\
\max_{\cD \in \fD(\gamma)}|\partial \cD| \leq \log L\\\gamma \in \cB}} \fq_{\Lambda';\Lambda}(\gamma) \leq e^{-c\beta\log L}\sum_{\substack{\gamma \in \cG \cap \cE_\square\\
\max_{\cD \in \fD(\gamma)}|\partial \cD| \leq \log L}} \fq_{\Lambda';\Lambda}(\gamma)\end{equation}
\end{lemma}
\begin{proof}
    Let $\gamma$ be a disagreement polymer in the left sum, and let $v, v'$ be two cut-points resulting in a button hole (if there are multiple candidates for $v, v'$, choose arbitrarily, say the minimal pair under lexicographic ordering). Let $T(\gamma)$ be the result of replacing the portion between $v, v'$ by a shortest length path from $v$ to $v'$ with an edge disagreement of one along the path. This shortest length path can always be chosen such that $T(\gamma)$ still contains a square of side length $L(1-\epsilon'_\beta)$  (we can assume that $v, v'$ are not on opposite sides of the square since this would violate the condition $|\gamma| \leq 4.1L$). By the button hole criterion, we have $\sE_\beta(T(\gamma)) \leq \sE_\beta(\gamma) - \tfrac\beta2d_{\gamma}(v, v')$. Since we are only possibly removing some finite clusters $D_i$, the same inequality holds with $\sE^*_\beta$ instead of $\sE_\beta$. The change in the interaction term is also at most $\epsilon_\beta d_{\gamma}(v, v')$ by the decay properties of $\Phi$. Moreover, $\big||\Lambda \cap D_0(\gamma)| - |\Lambda \cap D_0(T(\gamma))|\big| \leq \min(d_{\gamma}(v, v')^2, \epsilon'_\beta L^2)$, which is always negligible compared to $\tfrac\beta2d_{\gamma}(v, v')$. At the same time, the number of preimages $\gamma$ under the map $T$ is at most $L^2C^{d_{\gamma}(v, v')}$ for a constant $C$, given by the $L^2$ choices for the two cut-points $v, v'$ and the $C^{d_{\gamma}(v, v')}$ choices of how to reconstruct the portion of $\gamma$ between $v$ and $v'$. Hence, we can conclude by a Peierls map argument completely analogous to the computation in \cref{eq:map-argument-Di}. That is, enumerate over the possible preimages $\gamma$ to a given image $\gamma'$ such that $d_{\gamma}(v, v')$ is equal to a fixed value $k \geq \log L$, and sum over $k$.
\end{proof}
Given \cref{eq:control-button-hole}, it now suffices to show that 
\begin{equation}\label{eq:final-reduction-lower-window}\sum_{\substack{\gamma \in \cG \cap \cE_\square\\
\max_{\cD \in \fD(\gamma)}|\partial \cD| \leq \log L\\\gamma\notin \cB}} \fq_{\Lambda';\Lambda}(\gamma)\leq o(e^{-CL^{1/2+o(1)}})\,.
\end{equation}

We finally note that the bound $|\gamma| \leq 4.1L$ and the domain enlargement trick to move $\partial \Lambda'$ far away implies that we can replace $\fI_{\Lambda'}(\gamma)$ by $\fI_{\Z^2}(\gamma)$ at the cost of $o(1)$ in the exponent. We denote the weights with this modification by $\fq_{\Z^2;\Lambda}(\gamma)$ (This step could have been done earlier, but is not needed until now.) 

The benefit of ruling out button-holes is that we can now control the interaction between different parts of $\gamma$ in the following sense. Suppose $\{v_1, \ldots v_s\}$ is a set of ordered cut-points on $\gamma$. Then we can naturally decompose $\gamma = \gamma_1 \circ \gamma_2 \ldots \circ \gamma_s$ such that $\gamma_i$ is the portion of $\gamma$ between $v_i$ and $v_{i+1}$, with the convention that $v_{s+1} = v_1$.
\begin{lemma}\label{lem:bound-interactions}
    There exists a constant $C > 0$ such that for any $\gamma$ with no button holes satisfying $\max_{\cD \in \fD(\gamma)}|\partial \cD| \leq \log L$, and any choice of ordered cut-points $\{v_1, \ldots v_s\}$, we have
    \[|\fI_{\Z^2}(\gamma) - \sum_{i = 1}^s\fI_{\Z^2}(\gamma_i)| \leq sC\log L\,.\]
\end{lemma}
\begin{proof}
    By definition of the $\fI_{\Z^2}$ and the decay properties of $\Phi$, we have
    \[|\fI_{\Z^2}(\gamma) - \sum_{i = 1}^s\fI_{\Z^2}(\gamma_i)| \leq \sum_{i \neq j} \sum_{\substack{\sfW \cap \Delta_{\gamma_i \neq \emptyset}\\{\sfW \cap \Delta_{\gamma_j \neq \emptyset}}}}\Phi(\sfW) \leq \sum_{i=1}^s\sum_{\substack{x \in \gamma_i\\y \in \gamma \setminus \gamma_i}}e^{-c\norm{x-y}_1}\,.\]
    Now fix $i$. Let $W$ be the set of points $x \in \gamma_i$ such that there exists $z \in \gamma \setminus \gamma_i$ with $\norm{z - x}_1 \leq \log L$. By the condition that $\max_{\cD \in \fD(\gamma)}|\partial \cD| \leq \log L$, we can find cut-points $z^\mathsf{cut}, x^\mathsf{cut}$ such that $\norm{z^\mathsf{cut}-x^\mathsf{cut}}_1 \leq 3\log L$. By the button-hole criterion, this implies that $d_{\gamma}(z^\mathsf{cut}, x^\mathsf{cut}) \leq 6\log L$. Now assume that the $v_i$ are ordered clockwise. Starting from $v_i$ and going clockwise, let $x_i^\mathsf{cut}$ be the first cut-point such that $d_{\gamma}(v_i, x_i^\mathsf{cut}) \geq 6\log L$. Let $x_{i+1}^\mathsf{cut}$ be defined analogously by starting from $v_{i+1}$ and going counter-clockwise. By the above argument, every $x \in W$ must either be in between $v_i$ and $x_i^\mathsf{cut}$, or be in between $x_{i+1}^\mathsf{cut}$ and $v_{i+1}$. In particular, $|W| \leq 12\log L$. We can bound
    \begin{equation*}
        \sum_{\substack{x \in W\\y \in \gamma \setminus \gamma_i}}e^{-
    c\norm{x-y}_1} \leq |W|\sum_{y \in \Z^2}e^{-c\norm{x-y}_1} \leq C\log L\,,
    \end{equation*}
    where the point $x$ in the second sum is arbitrary. At the same time, we have
    \begin{equation*}
        \sum_{\substack{x \in \gamma_i \setminus W\\y \in \gamma \setminus \gamma_i}}e^{-
    c\norm{x-y}_1} \leq |\gamma_i|\sum_{\substack{y \in \Z^2\\ d(x, y) \geq \log L}}e^{-c\norm{x-y}_1} \leq e^{-c\log L}\,.
    \end{equation*}
    Combining the above two displays and summing over $i$ completes the proof.
\end{proof}
For brevity, to prove \cref{eq:final-reduction-lower-window}, let $\widetilde\cE_\square$ be the set of all $\gamma \in \cG \cap \cE_\square$ such that $\max_{\cD \in \fD(\gamma)} |\partial \cD| \leq \log L$, and $\gamma$ has no button-holes. We will enumerate over $\gamma \in \widetilde\cE_\square$ by marking a set of special anchor points $\underline{v}$. For a sequence of points $\underline v = \{v_1, \ldots v_s\}$ in the dual graph, define the set $\cP(\underline v)$ as disagreement polymers $\gamma$ such that $\gamma \in \widetilde\cE_\square$, $\mathsf{OE}(\gamma)$ contains all the points $v_i$ in order from 1 to $s$, and each $v_{i+1}$ is the next cut-point in $\mathsf{OE}(\gamma)$ after $v_i$ which has $\norm{\cdot}_1$ distance at least $L^{1/2}$ from $v_i$. In particular, we assume from the condition that the maximum size of any $|\partial \cD| \leq \log L$ that $L^{1/2} \leq \norm{v_{i+1} - v_i}_1 \leq L^{1/2} + \log L$. We may further assume that $\norm{v_s - v_1}_1 \leq L^{1/2}$. Since $|\gamma|\leq 4.1L$, this implies that $s \leq 4.1L^{1/2}$ in order for $\cP(\underline v)$ not to be empty.

Now fix $\underline v$. We will bound the sum \cref{eq:final-reduction-lower-window} first over all $\gamma \in \cP(\underline{v})$. For convenience of notation, let $v_{s+1} := v_1$. Define also $L_{\underline v}$ as the linear interpolation of the points $\underline v$. Recall the definitions of $\tau_\beta$ and $\cP_{\Z^2}(v_i, v_{i+1})$ from \cref{sec:poly-st-wulff}. Using \cref{lem:bound-interactions} and the relationship between $\tau_\beta$ and partition function given by \cite[Prop.~3.14~(ii) and Prop.~4.14]{ChenLubetzky25}, we have a bound on the sum of the weights without the area term:
\begin{align*}
    \sum_{\gamma \in \cP(\underline v)}\exp(-\sE^*_\beta(\gamma) + \fI_{\Z^2}(\gamma)) &\leq e^{sC\log L}\prod_{i = 1}^s\sum_{\gamma_i \in \cP_{\Z^2}(v_i, v_{i+1})}\exp(-\sE^*_\beta(\gamma_i) + \fI_{\Z^2}(\gamma_i))\\
    &\leq e^{sC\log L}\prod_{i = 1}^s\exp(-\tau_\beta(v_{i+1} - v_i))\\
    &= \exp(-\int_{L_{\underline v}}\tau_\beta(\theta_s)ds + sC\log L)\,.
\end{align*}

Define $K_{\underline v}$ as the convex hull of $\underline v$. By convexity of the surface tension, we have
\[\int_{L_{\underline v}}\tau_\beta(\theta_s)ds \geq \int_{\partial K_{\underline v}}\tau_\beta(\theta_s)ds \geq \int_{\partial (K_{\underline v} \cap \Lambda)}\tau_\beta(\theta_s)ds\,.\]
To handle the area term, we observe that $|\Lambda \cap D_0| \leq |K_{\underline v} \cap \Lambda| + sL \leq |K_{\underline v} \cap \Lambda| + 4.1 L^{3/2}$. Hence, we have
\begin{align*}\sum_{\gamma \in \cP(\underline v)}\exp(-\sE^*_\beta(\gamma) &+ \fI_{\Z^2}(\gamma) + \tfrac{\uprho_0\lambda\beta}{L}|\Lambda \cap D_0|) \\
&\leq \exp(-\int_{\partial (K_{\underline v} \cap \Lambda)}\tau_\beta(\theta_s)ds + \tfrac{\uprho_0\lambda\beta}{L}|K_{\underline v} \cap \Lambda| +4.1\uprho_0\lambda\beta L^{1/2} + 4.1CL^{1/2}\log L)\\
&\leq \exp(-\int_{\partial (K_{\underline v} \cap \Lambda)}\tau_\beta(\theta_s)ds + \tfrac{\uprho_0\lambda\beta}{L}|K_{\underline v} \cap \Lambda| +O(L^{1/2}\log L))\,.
\end{align*}
By rescaling to the unit square, we can use \cref{lem:linear-critical-window} to obtain that
\begin{align*}-\int_{\partial (K_{\underline v} \cap \Lambda)}\tau_\beta(\theta_s)ds + \tfrac{\uprho_0\lambda\beta}{L}|K_{\underline v} \cap \Lambda| &\leq -.9\beta(\lambda_* - \uprho_0\lambda)L\,,
\end{align*}
whence we have a combined upper bound of $\exp(-.9\beta(\lambda_* - \uprho_0\lambda)L + O(L^{1/2}\log L))$ for the total weight of all the $\gamma \in \cP(\underline{v})$ contributing to the sum in \cref{eq:final-reduction-lower-window}, for a fixed $\underline v$. Since the number of possible $\underline v$ is bounded by $|\Lambda'|^{s} \leq O(L^{8.2L^{1/2}}) = O(\exp(8.2L^{1/2}\log L))$, we conclude by summing over all choices of $\underline{v}$ that \cref{eq:final-reduction-lower-window} will hold as long as $(\lambda_* - \uprho_0\lambda)L \gg L^{1/2+o(1)}$. Finally, we conclude by checking that for $L \leq L_*^{(H+1)} - (L_*^{(H+1)})^{1/2+o(1)}$, we have 
\[(\tfrac{\lambda_*}{\uprho_0} - \lambda)L \geq \frac{\hatpi_\infty(\phi_o = H+1)}{\beta}(L_*^{(H+1)}-1 - (L_*^{(H+1)} - (L_*^{(H+1)})^{1/2+o(1)}))L = O(L^{1/2+o(1)})\,,\]
concluding the proof of \cref{prop:below-window}.
\subsection{Existence of a large \texorpdfstring{$H+1$}{H+1} level line above the window}
The goal of this subsection is to prove \cref{prop:above-window}. By \cref{lem:prelim-ll-bounds} we can additionally assume that $\lambda <4(1+\tfrac{3}{4\beta})$, and we know that w.h.p.\ only \texttt{large} $H+1$ level lines containing a square $\Lambda_{L(1-\epsilon_\beta^\square)}$ can exist. We follow the strategy of \cite[Section 4.3]{CLMST16}. 

We now proceed via a proof by contradiction. Let $\fS_0$ be the event that there are no \texttt{large} $H+1$ level lines with area larger than $(1-\epsilon_\beta^\square)^2L^2$. We wish to show that $\pi^0_\Lambda(\fS_0) = o(1)$. Recall the shapes defined in \cref{def:cL} and the notational comment of \cref{rem:cL-alt-notation}. Recall also that $\ell^*_n = \ell^*_n(L) := \frac{\sfw_1(\tau_\beta)N_n}{2L\uprho_n}$. 
By \cref{thm:old-level-line-contains-Wulff} and the bounds on $\uprho_n$ in \cref{prop:uprho-bound},
w.h.p.\ there exists an $H$ level line $\fL_1$ containing $L(1-L^{-2/3 - o(1)})\sL(\ell^*_0)$. We can reveal it and by monotonicity it suffices to bound $\pi^H_{\Int(\fL_1)}(\fS_0)$, as $\fS_0$ is decreasing.

Now define $\fS$ as having no \texttt{large} disagreement polymers. The goal is to use \cref{cor:CE-macro-uniformBC-smallL} to show that $\pi^H_{\Int(\fL_1)}(\fS) = o(1)$ by summing over a set of nice $\gamma$. Since \texttt{large} level lines where the interior is lower than the exterior can always be rule out via the standard Peierls argument, this will imply that $\pi^H_{\Int(\fL_1)}(\fS_0) = o(1)$. Define the shape $\cK = \partial((L(1-L^{-1/2+o(1)}))\sL(\ell^*_0))$. Define also ``cigar shapes'' as follows: for any two points $u = (u_1, u_2), v = (v_1, v_2)$, let $\theta$ be the angle of $\overline{uv}$ with the $x$-axis, and $M_{u, v} = v_1 - u_1$. Assume for simplicity that $\theta \in [0, \pi/4]$. Define $\cC^\pm_{u, v}$ by
		\begin{equation}\label{eq:cigar-curves}
			\cC^\pm_{u, v}(t) = \tan(\theta_{u, v})(t-u_1) + \left(\frac{(t - u_1)(M_{u, v} - t + u_1)}{M_{u, v}}\right)^{1/2}(\log L)^2\,.
		\end{equation}
Define the cigar shape $\sC = \sC(u, v)$ as the region in between the curves $\cC^+_{u, v}$ and $\cC^-_{u, v}$. For $u, v$ with an angle not in $[0, \pi/4]$, we can define $\sC(u, v)$ via reflection across $x$ and $y$ axes.

Let $\{v_i\}_{i = 1}^s$ be a (arbitrarily chosen) set of points on $\cK$ such that the distance between any two points is between $3L/s$ and $5L/s$. For simplicity of notation, define $v_{s+1} = v_1$. Our nice set of $\gamma$, denoted $\cG$, will be all $\gamma$ contained in $\bigcup_{i = 1}^s \sC(v_i, v_{i+1})$ with length $|\gamma| \leq 10L$ such that if $D_0$ is the region containing the origin, then $h_0 = H$, if $D_1$ as the region containing the boundary vertices $\partial \Int(\fL_1)$, we have $h_1 = H-1$, and $\max_{\cD \in \fD(\gamma)}|\partial \cD| \leq \log L$. By construction, this implies that $|D_1^\circ \cap F| \leq \epsilon_\beta L^2$. All $\gamma \in \cG$ thus satisfy the conditions of \cref{cor:CE-macro-uniformBC-smallL}, and together with the bound on $|\partial \cD|$ we have
\[\pi^H_{\Int(\fL_1)}(\cC_{\gamma, H+1})/\pi^H_{\Int(\fL_1)}(\fS) \geq \exp(-\sE^*_\beta(\gamma) + \fI_{\Int(\fL_1)}(\gamma) + \tfrac{\uprho_0\lambda\beta}{L}|D_0| + L^{1/2+o(1)})\,.\]
Since $|D_0| \geq |\Int(\cK)| - s(5L/s)^2 = |\Int(\cK)| - 5L^2/s$,
so we can replace $|D_0|$ with $|\Int(\cK)|$ at a cost of $L/s$ as error. Moreover, $\gamma$ stays distance $L^{1/2+o(1)}$ away from $\fL_1$, so we can replace $\fI_{\Int(\fL_1)}(\gamma)$ with $\fI_{\Z^2}(\gamma)$. Finally, letting $\gamma_i$ denote the portion of $\gamma$ between $v_i$ and $v_{i+1}$, we can write $\sE^*_\beta(\gamma) = \sum_i \sE^*_\beta(\gamma_i)$, and by construction of the cigar shapes we have $|\fI_{\Z^2}(\gamma) - \sum_{i=1}^s \fI_{\Z^2}(\gamma_i)| = O(s(\log L)^2)$. Hence, we have
\[\pi^H_{\Int(\fL_1)}(\cC_{\gamma, H+1})/\pi^H_{\Int(\fL_1)}(\fS) \geq \exp(\tfrac{\uprho_0\lambda\beta}{L}|\Int(\cK)| + L^{1/2+o(1)} + O(L/s) + O(s(\log L)^2))\prod_i e^{-\sE^*_\beta(\gamma_i)+ \fI_{\Z^2}(\gamma_i)}\,.\]

We study the product first, which captures the weight of the polymer $\gamma_i$ with no area tilt terms. We claim that 
\[\sum_{\substack{\gamma_i: v_i \mapsto v_{i+1}\\\gamma_i \in \sC(v_i, v_{i+1})\\|\gamma_i| \leq 10L/s\\\max_{\cD \in \fD(\gamma_i)|\partial \cD| \leq \log L}}}e^{-\sE^*_\beta(\gamma_i)+ \fI_{\Z^2}(\gamma_i)} \geq e^{-\tau_\beta(v_{i+1} - v_i) + O(1)}\,.\]
Indeed, if the sum did not have any restrictions besides $\gamma_i : v_i \mapsto v_{i+1}$, this holds by \cite[Prop.~3.14~(ii) and Prop.~4.14]{ChenLubetzky25}. We then note that each of the additional restrictions only affect the sum by a multiplicative factor of $1-o(1)$. Indeed, we can control the effect of restricting to the cigar shape by \cite[Lem.~3.15]{ChenLubetzky25}, control the length by a simple Peierls argument mapping $\gamma_i$ to a minimal length path from $v_i$ to $v_{i+1}$, and control the size of $\cD \in \fD(\gamma_i)$ by the same argument used in the proof of \cref{prop:CE-meso} (only simpler since there are no area terms here). 

Now we can take $s = L^{1/2}$ to absorb all the error terms together, define $\overline\cK$ as the linear interpolation of the points $\{v_i\}$, and combine the above two displays to obtain 
\begin{align}\label{eq:bound-S-prob}
    \sum_{\gamma \in \cG}\frac{\pi^H_{\Int(\fL_1)}(\cC_{\gamma, H+1})}{\pi^H_{\Int(\fL_1)}(\fS)} \geq &\exp\bigg(-\int_{\overline\cK}\tau_\beta(\theta_s)ds+\tfrac{\uprho_0\lambda\beta}{L}|\Int(\cK)| + L^{1/2+o(1)}\bigg)\nonumber\\\nonumber
    &\geq \exp\bigg(-\int_{\cK}\tau_\beta(\theta_s)ds+\tfrac{\uprho_0\lambda\beta}{L}|\Int(\cK)| + L^{1/2+o(1)}\bigg)\\
    &=\exp\big(L\cF^0_{\uprho_0\lambda}((1 - L^{-1/2+o(1)})\sL(\ell^*_0)) + L^{1/2+o(1)}\big)\,,
\end{align}
using the convexity of the surface tension for the second inequality.

Since $\sum_{\gamma \in \cG} \pi^H_{\Int(\fL_1)} \leq 1$, then if we show that the right side of \cref{eq:bound-S-prob} goes to $\infty$ as $L \to \infty$, this implies that $\pi^H_{\Int(\fL_1)} = o(1)$. To reference \cref{lem:linear-critical-window}, it will be easier to use the notational exchange $\sL(\ell^*_0) \equiv \sL(\uprho_0\lambda)$ via the identification in \cref{rem:cL-alt-notation}. 
For any $C \in [0, 1]$ and any $\lambda$, we have $\cF_{\lambda}(C\partial \sL(\lambda')) = C\cF_{\lambda}(\partial\sL(\lambda)) - C(1-C)\lambda\beta|A(\sL(\lambda))|$ by the definition of $\cF_\lambda$, so that when $C = 1 - L^{-1/2+o(1)}$, we have 
\[\cF_{\uprho_0\lambda}((1-L^{-1/2+o(1)})\partial\sL(\uprho_0\lambda)) \geq (1-L^{-1/2+o(1)})\cF_{\uprho_0\lambda}(\partial\sL(\uprho_0\lambda)) - \beta\uprho_0\lambda L^{-1/2+o(1)}\,.\] From \cref{lem:linear-critical-window}, we have that $\cF_{\uprho_0\lambda}(\partial\sL(\uprho_0\lambda)) \geq (\uprho_0\lambda - \lambda_*)(\beta-1)$. Put altogether, we have
\[L\cF_{\uprho_0\lambda}((1-L^{-1/2+o(1)})\partial\sL(\uprho_0\lambda)) \geq L(\uprho_0\lambda - \lambda_*)(\beta - 1)(1 - L^{-1/2+o(1)}) - \beta\uprho_0\lambda L^{1/2+o(1)}\,,\]
and we want to ensure that $L\cF_{\uprho_0\lambda}((1-L^{-1/2+o(1)})\partial\sL(\uprho_0\lambda)) \gg L^{1/2+o(1)}$. For this, it suffices to take $\lambda$ such that $\uprho_0\lambda - \lambda_* \gg L^{-1/2+o(1)}$. We can check that for $L \geq L_*^{(H+1)} + (L_*^{(H+1)})^{1/2+o(1)}$, we indeed have
\[\uprho_0\lambda - \lambda_* \geq \uprho_0\frac{\hatpi_\infty(\phi_o = H+1)}{\beta}(L - L_*^{(H+1)}) \geq O((L_*^{(H+1)})^{1/2+o(1)}/L) = O(L^{-1/2+o(1)})\,.\]
Hence, for such $L$, we can conclude from \cref{eq:bound-S-prob} that $\pi^H_{\Int(\fL_1)}(\fS) = o(1)$, and this concludes the proof of \cref{prop:above-window}.

\section{Limit shape of the top level lines}\label{sec:limit-shape}
In this section we prove \cref{thm:main-thm-limit-shape}, establishing the global limit shape of the top level lines (even in the critical window). In \cref{sec:growth}, we extend \cref{thm:old-level-line-contains-Wulff} to the case where the $H+1$ level line exists, showing one side of the shape theorem, that the level line contains a shape. In \cref{sec:retreat}, we prove the other side, that the level line is contained in a shape. 

Throughout this section, there will be several occasions where we state that a bad event only occurs with probability $o(1)$, and then take a union bound over polynomially many such bad events. This is not an issue, as in fact every $o(1)$ is at most $e^{-c\beta\log L}$ for some constant $c$, the dominant error arising from Peierls maps which delete level lines of size $\log L$.

\subsection{Top level line contains limit shape in the critical window}\label{sec:growth}
The results of \cref{thm:old-level-line-contains-Wulff} show that for any $L$, w.h.p.\ in $\pi^0_{\Lambda}$ the $H, H-1, \ldots H-m$ level lines contain their respective limit shapes, up to a distance of $o(N_n)$ for the $H+1-n$ level line for $n \geq 1$. As we have shown, for certain $L$, there is also a \texttt{large} $H+1$ level line $\fL_0$, possibly with probability bounded away from both 0 and 1. An immediate corollary of \cref{thm:old-level-line-contains-Wulff} is that for $L$ such that $\pi^0_{\Lambda}(\fL_0 = \emptyset) > \delta_L$, the theorem still holds under the measure $\pi^0_{\Lambda}(\cdot \mid \fL_0 = \emptyset)$. 

Here we prove that if there is an $H+1$ level line, then it must also contain its respective limit shape. Note that this easily implies the uniqueness of the \texttt{large} $H+1$ level line when it exists, as we can reveal the innermost such level line and rule out any others via \cref{lem:prob-of-contour}.

Recall again that $\ell^*_n = \ell^*_n(L) = \frac{\sfw_1(\tau_\beta)N_n}{2L\uprho_n}$. Fix any constant $C > 0$, and define the following shapes
\begin{align*}
    &\fR_1 := (1-\tfrac{N_1^{1/3}e^{3C\sqrt{\log L}}}{L})L\sL(\ell^*_1(1+e^{-\frac{C}3\sqrt{\log L}}))\\
    &\fR_0 := (1-2\tfrac{N_0^{1/3}e^{3C\sqrt{\log L}}}{L})L\sL(\ell_{0}^*(1+e^{-\frac{C}3\sqrt{\log L}}))\,.
\end{align*} 
We will show the following proposition.
\begin{proposition}\label{prop:shape-thm-crit-window}
    Let $\cE$ be the event that there exists a \texttt{large} $H+1$ level line which does not contain $\fR_0$ in its interior. Then, $\pi^0_{\Lambda}(\cE) = o(1)$. Consequently, for every $L$ such that $\pi^0_{\Lambda}(\fL_0 \neq \emptyset) > \delta_L$, we have $\pi^0_{\Lambda}(\cE \mid \fL_0 \neq \emptyset) = o(1)$.
\end{proposition}

The proof of \cref{thm:old-level-line-contains-Wulff} does not immediately extend to studying the $H+1$ level line since it assumes the level line exists w.h.p., and needs to be modified to remove this assumption. The difference is the following: if $L$ is such that w.h.p.\ there is a unique \texttt{large} $H+1$ level line, then we can try to show that w.h.p., this level line exists and contains $\fR_0$, and this is a monotone increasing event. If we don't know whether the $H+1$ level line exists, we can only aim to show that the event $\cE$ occurs with probability $o(1)$. However, $\cE$ is no longer an monotone event! Hence, we need to be more careful when we make monotonicity arguments. A reader familiar with \cite[\S4.2]{ChenLubetzky25} will see that this is a technical modification, and the overall proof strategy remains the same.

Recall that $\cW_1(\tau_\beta)$ is the Wulff shape of area 1. For $x \in \Lambda$, define $\cW(x, \ell)$ as the rescaled Wulff shape $L\ell\cW_1(\tau_\beta)$ centered at $x$. Let $\ell_x$ be the largest value of $\ell$ before $\cW(x, \ell)$ reaches within distance $N_0^{1/3}e^{3C\sqrt{\log L}}$ from $\fR_1$. For any $x, \ell$, we define the event $\cE_{x, \ell}$ to be that there 
is an $H+1$ level line which contains $\cW(x, \ell)$ in its interior, but not $\cW'(x, \ell) := (1+L^{-3/4})\cW(x, \ell)$. The main claim is the following:
\begin{claim}\label{clm:crit-window-growth}
    For any domain $\Lambda \supset \fR_1$ and boundary conditions $\eta \geq H$, for all $x, \ell$ such that $\ell^*_0(1+e^{-\frac{C}3\sqrt{\log L}}) \leq \ell \leq \ell_x$, we have $\pi^\eta_\Lambda(\cE_{x, \ell}) = o(1)$.
\end{claim}
\begin{proof}
Fix $x, \ell$ such that $(1+e^{-\frac{C}3\sqrt{\log L}})\ell^*_0 \leq \ell < \ell_x$. Fix an angle $\theta \in [0, \pi/4]$ and let $A, B$ on the bottom right quarter of $\cW(x, \ell)$ be such that $B_1 - A_1 = N_0^{2/3}e^{C\sqrt{\log L}}$ and the angle of $\overline{AB}$ is $\theta$. (By symmetry, this is equivalent to looking at points on the top half of $\cW(x, \ell)$, or fixing instead the vertical distance $B_2 - A_2$ and looking at the right/left halves.) Let $R$ be the rectangle of width $N_0^{2/3}e^{C\sqrt{\log L}}$ and height $2N_0^{2/3}e^{C\sqrt{\log L}}$ with $A, B$ on its sides, such that the distance from $A$ to the bottom of $R$ is $N_0^{1/3}e^{3C\sqrt{\log L}}$. 

Let $\partial^1 R, \partial^2 R$ be the arcs of $\partial R$ above and below $A, B$, respectively. For every vertex $v \in \partial^2 R$, reveal the connected component of sites with height $\leq H-1$ containing it. This reveals an external boundary of sites with height $\geq H$. Note that the standard Peierls argument implies that with probability $1-o(1)$, each connected component revealed has size $\leq \log L$. Similarly, for every vertex $v \in \partial^1 R$ that has not already been revealed, reveal the connected component of sites with height $\leq H$ containing it. Except with probability $o(1)$, the event that there is an $H+1$ level line containing $\cW(x, \ell)$ implies that each of these connected components have size $\leq \log L$. This reveals an external boundary of sites with height $\geq H+1$. Piecing these external boundaries together obtains a circuit of sites $\cC_*$ distanced at most $\log L$ from $\partial R$ with two marked points $A', B'$ with $d(A, A'), d(B, B') \leq \log L$, such that the heights on $\cC$ are $\geq H+1$ on the arc above $A', B'$ and $\geq H$ on the arc below. 

Hence, the event $\cE_{x, \ell}$ implies that at every angle $\theta$, except with probability $o(1)$, we can find such a circuit $\cC_*$. However, it also implies that there is some point $z \in \partial \cW'(x, \ell)$ which is not contained by the $H+1$ level line containing $\cW(x, \ell)$. That is, on $\cE_{x,\ell}$, there is some angle $\theta$ such that for any such circuit $\cC_*$ as above, the connected component of sites with height $\geq H+1$ containing the arc of $\cC_*$ above $A', B'$ does not contain $z$. Note that for any fixed circuit, this latter event is decreasing. We will next argue that it occurs with probability $o(1)$.

Indeed, given $\theta$, we can reveal the outermost such circuit $\cC_*$, which in particular does not reveal anything about its interior. Since we want to upper bound the probability of a decreasing event, we can now use monotonicity to lower the heights along the top and bottom arcs of $\cC_*$ to be exactly $H+1, H$ respectively. By Domain Markov, the law of the heights interior to $\cC_*$ is now a Discrete Gaussian with $H+1, H$ boundary conditions. To satisfy the inputs of \cite[Thm.~4.9]{ChenLubetzky25}, we can then further lower some of the $H+1$ heights to $H$, as done in step 5 of the proof of \cite[Lem.~4.8]{ChenLubetzky25}, to obtain the regularity needed in $\cC_*$ around points where boundary conditions change. The desired bound on the size of the $\geq H+1$ component containing the $H+1$ boundary conditions now follows from \cite[Thm.~4.9]{ChenLubetzky25}\footnote{In particular, a version of \cite[Thm.~4.9]{ChenLubetzky25} where the rectangle has dimensions as described above. As discussed in \cite[Rem. 4.7]{ChenLubetzky25}, this poses no issues. One may also compare the equations leading to \cref{eq:key-eq} with those leading to \cite[Eq. 4.34]{ChenLubetzky25} for details.} together with the computation preceding \cite[Eq. 4.34]{ChenLubetzky25}. Finally, taking a union bound over all possible angles $\theta$ concludes the proof of the claim. 
\end{proof}

\begin{proof}[Proof of \cref{prop:shape-thm-crit-window}]
    By \cref{thm:old-level-line-contains-Wulff,lem:prelim-ll-bounds}, the following hold w.h.p.\ under $\pi^0_{\Lambda}$:
    \begin{enumerate}
        \item There is an $H$ level line containing $\fR_1$ in its interior, and\label{it:H+1-level-cond-1}

        \item If a \texttt{large} $H+1$ level line exists, it also contains $\cW(o, \ell^*_0(1+e^{-\frac{C}3\sqrt{\log L}}))$, where $o$ is the center of $\Lambda$.\label{it:H+1-level-cond-2}
    \end{enumerate}
We want to bound the probability that a \texttt{large} $H+1$ level line exists, contains $\cW(o, \ell^*_0(1+e^{-\frac{C}3\sqrt{\log L}}))$, yet does not contain $\fR_0$. Reveal the outermost $H$ level line containing $\fR_1$. This brings us into the setting of \cref{clm:crit-window-growth}. Now define $\widetilde \fR_0$ as 
    \[\widetilde \fR_0 := \bigcup_{x:\, \ell^*_0(1+e^{-\frac{C}3\sqrt{\log L}}) \leq \ell_x}\cW(x, \ell^*_0(1+e^{-\frac{C}3\sqrt{\log L}}))\,.\]
    With \cref{clm:crit-window-growth} in mind, we first argue that for any domain $V \subset \Lambda$ which contains $\cW(o, \ell^*_0(1+e^{-\frac{C}3\sqrt{\log L}}))$ yet does not contain $\widetilde\fR_0$, there must exist $x$ such that $\cW(x, \ell^*_0(1+e^{-\frac{C}3\sqrt{\log L}})) \subset V$ and $\ell^*_0(1+e^{-\frac{C}3\sqrt{\log L}}) \leq \ell_x$, but $V \not\subset \cW'(x, \ell^*_0(1+e^{-\frac{C}3\sqrt{\log L}}))$. Indeed, one can simply start with $\cW(o, \ell^*_0(1+e^{-\frac{C}3\sqrt{\log L}}))$ and keep shifting the center by one up, down, left, or to the right within $\widetilde\fR_0$ until it intersects a point in $\widetilde\fR_0 \setminus V$. The last shift right before will then be in $V \cap \widetilde\fR_0$, which produces the desired shift $\cW(x, \ell^*_0(1+e^{-\frac{C}3\sqrt{\log L}}))$. 
    
    Thus, by \cref{clm:crit-window-growth} and a union bound over all $x$ such that $\ell^*_0(1+e^{-\frac{C}3\sqrt{\log L}}) \leq \ell_x$, we have that the probability that a \texttt{large} $H+1$ level line exists and does not contain $\widetilde\fR_0$ is $o(1)$. Since $\fR_0 \subset \widetilde\fR_0$, this concludes the proof.
\end{proof}
\subsection{Top level lines are contained within limit shapes}\label{sec:retreat}
In this section, we prove that the top level line is contained in a translation of Wulff shapes. Together with \cite[Thm.~4.4]{ChenLubetzky25} (restated here as \cref{thm:old-level-line-contains-Wulff}) and \cref{prop:shape-thm-crit-window}, this proves \cref{thm:main-thm-limit-shape}.
\begin{theorem}\label{thm:wulff-shape-contains-level-line}
    Fix $\beta$ sufficiently large and $n \geq 0$. Then, the event that $\fL_n$ exists but is not contained in the shape $L\sL(\ell^*_n(1-e^{-\frac{C}3\sqrt{\log L}}))$ occurs with probability $o(1)$ in $\pi^0_{\Lambda}$. Consequently, if $L$ is such that $\pi^0_{\Lambda}(\fL_0 \neq \emptyset) > \delta_L$, then the theorem also holds under $\pi^0_{\Lambda}(\cdot \mid \fL_0 \neq \emptyset)$, and if $\pi^0_{\Lambda}(\fL_0 = \emptyset) > \delta_L$, then the theorem also holds under $\pi^0_{\Lambda}(\cdot \mid \fL_0 = \emptyset)$.
\end{theorem}

We begin with \cref{prop:retreat-gadget} below, a counterpart to \cite[Thm.~4.9]{ChenLubetzky25} which shows that the level line in a smaller domain will not deviate too far from its initial height. To prove this, it is crucial that we have an expression for the law of $\gamma$ with an error of at most a multiplicative $(1+o(1))$ in the larger domain of size $\approx L^{2/3} \times L^{2/3}$, as in \cref{prop:CE-meso}. The previous work \cite{ChenLubetzky25} also looked at the restriction of the level line on an $L^{2/3} \times L^{2/3}$ rectangle, yet the nonnegative conditioning was only enforced on a much smaller $\approx L^{2/3} \times L^{1/3}$ rectangle, reducing the effective size of the domain. This monotonicity trick no longer works here, being in the wrong direction here. Hence, the main obstacle lies in obtaining \cref{prop:CE-meso}.

Fix $n \geq 0$. We define a slight modification of the ``near rectangular'' domains used in \cite[Thm 4.9]{ChenLubetzky25}. Let $R$ be a $N_n^{2/3}e^{C\sqrt{\log L}} \times 3N_n^{2/3}e^{C\sqrt{\log L}}$ rectangle, with two marked points $A, B$ on the left and right sides of $\partial R$, respectively. Assume that if $\theta$ is the angle of $\overline{AB}$, then $|\theta| \leq \pi/4$, and that both $A, B$ are distance at least $N_n^{2/3}e^{C\sqrt{\log L}}$ from the top and bottom sides of $R$. We will consider domains $Q$ satisfying the following:
\begin{enumerate}
    \item \label{it:Q-simp-conn}$Q$ is simply connected,

    \item\label{it:dist-Q-R-bdy} $\dist(\partial Q, \partial R) \leq \log L$,

    \item\label{it:good-boundary-points} There exists $A', B' \in \partial Q$ such that $\sC(\overline{A'B'}) \subset Q$ (see the definition of $\sC$ below \cref{eq:cigar-curves}), and $\max\{d(A, A'), \dist(B, B')\} \leq 2(\log L)^5$.\footnote{In \cite{ChenLubetzky25}, there was an additional requirement that linear cones emanating from $A$ and $B$ did not intersect their respective sides of $\partial Q$. This follows also from our construction of $Q$ in the proof of \cref{clm:retreat-claim}. However, it is not used in the proof of \cite[Thm.~4.9]{ChenLubetzky25}, and hence not needed for \cref{prop:retreat-gadget}, so we do not include it.}

    \item \label{it:Q-bc-arc} The boundary conditions $\xi$ assigns height $H+1-n$ on the top arc from $A'$ to $B'$, and $H-n$ on the bottom arc.
\end{enumerate}

\begin{proposition}\label{prop:retreat-gadget}
    The following holds uniformly over all possible $Q, \xi$ as above. Fix $n \geq 0$. Let $\fL_n$ be the $(H+1-n)$ level line induced by $\xi$. Let $\theta$ denote the angle of $\overline{AB}$ with the $x$-axis. Then, with $\pi^\xi_Q$-probability $1 - o(1)$, $\fL_n$ lies above the point $X= (0, Y - \sigma e^{C\sqrt{\log L}} - \log L)$, where
		\begin{align}\label{eq:def-Y-sigma}
			&Y = -\frac{\uprho_n N_n^{1/3}e^{2C\sqrt{\log L}}}{8(\tau_\beta(\theta) + \tau_\beta''(\theta))\cos(\theta)^3}\,,\\
			&\sigma^2 = \frac{N_n^{2/3}e^{C\sqrt{\log L}}}{4(\tau_\beta(\theta) + \tau_\beta''(\theta))\cos(\theta)^3}\,.\nonumber
		\end{align}
	\end{proposition}
\begin{proof}
    The proof is essentially the same as that of \cite[Thm.~4.9]{ChenLubetzky25}, which we recall showed instead that $\fL_n$ lies below $(0, Y + \sigma e^{C\sqrt{\log L}} + \log L)$. We only comment on the few necessary modifications. First, note that we are imposing the additional assumption that $A, B$ are distance at least $N_n^{2/3}e^{C\sqrt{\log L}}$ from both the top and bottom sides of $R$. We will shortly bound the length of $\gamma$, the level line containing $\fL_n$, such that w.h.p.\ it will not reach the top or bottom of $Q$, preventing any potential pinning issues. (The setting of \cite[Thm.~4.9]{ChenLubetzky25} did not allow for this and relied on a conditioning argument instead, which we no longer need here.)
    
    To set up for an application of \cref{prop:CE-meso}, let $\cG$ denote the set of $\gamma$ with $\sE_\beta(\gamma) \leq 1.1|A-B|_1$. We can show that $\hatpi^\eta_Q(\cG^c) \leq e^{-1.1(\beta - C)|A-B|_1}$ by using the form of the law given by \cref{prop:CE-law-without-area} and applying a Peierls map to a minimal length disagreement polymer. At the same time, we can lower bound $\hatpi^\eta_Q(\phi_x \geq 0,\, \forall x \in Q) \geq e^{-cL^{1/3-o(1)}}$ by the standard FKG computation explained in \cref{clm:CE-check-LW}. Together this implies that $\pi^\eta_Q(\cG^c) = e^{-1.1(\beta - C')|A-B|_1}$, and we can assume we are only considering $\gamma \in \cG$.

    The law of such $\gamma$ is given by \cref{prop:CE-meso}. Compared to \cref{prop:CE1_old}, the area tilt now has an additional prefactor of $\uprho_n$, but the preliminary lemmas needed for \cite[Thm.~4.9]{ChenLubetzky25} (i.e., bounding the size of components $D_i$ in \cite[Lem.~4.12,4.15]{ChenLubetzky25} and bounding the partition function in \cite[Lem.~4.17]{ChenLubetzky25}) allowed for any $\mu > 0$ prefactor, and therefore still hold. We can thus proceed exactly as in the proof of \cite[Thm.~4.9]{ChenLubetzky25}, observing that the proof actually showed the distribution of the height of $\gamma$ at the line $x = 0$ is comparable to a Gaussian centered at $Y$ with variance $\sigma^2$, so controlling the probability that the height exceeds $Y + \sigma e^{C\sqrt{\log L}} + \log L$ is no different than controlling the probability that it is at most $Y - \sigma e^{C\sqrt{\log L}} - \log L$.\footnote{In \cite{ChenLubetzky25}, we arbitrarily decided to bound the size of components using Peierls maps by $(\log L)^2$ instead of $\log L$. This makes no difference.}
\end{proof}

We will prove \cref{thm:wulff-shape-contains-level-line} first for $n = 0$, using an induction over the radius $\ell$ of the Wulff shape. The base case will be proven in \cref{lem:base-case} and the induction step in \cref{lem:induction-step}. Then we will use an induction over $n$ to prove \cref{thm:wulff-shape-contains-level-line} for lower level lines. This involves combining the framework of the disagreement polymer developed in \cite{ChenLubetzky25} with the ``retreat of droplets'' proof method found in \cite[Section 6.2]{CLMST16}. Throughout the rest of the section, we will refer to the error quantity $\epsilon_n := N_n^{-1/3}$.

\begin{definition}
    Define $\cE_n(\ell)$ as the (decreasing) event that there is no \texttt{large} chain of sites with height $\geq H+1-n$ intersecting the exterior of the shape $L(1+\epsilon_n)\sL(\ell)$.
\end{definition}
Observe that to prove \cref{thm:wulff-shape-contains-level-line}, it suffices to show that w.h.p.\ we have $\cE_n(\ell^*_n(1-\tfrac12e^{-\frac{C}3\sqrt{\log L}}))$, as the factor of $1/2$ makes up for the scaling factor of $(1+\epsilon_n)$. More precisely, the additional scaling only comes into effect at the corners of the shape since we are still restricted to $\Lambda$ regardless. It is then easy to check by starting at the corners of the shapes and looking at the intersection points with $\partial \Lambda$ that $(1+\epsilon_n)L\sL(\ell^*_n(1-\tfrac12e^{-\frac{C}3\sqrt{\log L}}))\cap \Lambda\subset L\sL(\ell^*_n(1 - e^{-\frac{C}3\sqrt{\log L}}))$.
\begin{lemma}[Base case for $n = 0$]\label{lem:base-case}
    There exists a constant $s > 0$ such that for $\ell_0 := sN_0/L$, we have that $\cE_0(\ell_0)$ holds w.h.p.
\end{lemma}

\begin{proof}
    We will write the proof in terms of $n$ so it is clear how to generalize it later on, but here $n = 0$. It suffices by symmetry to just look at one corner of the shape $L\sL(\ell_n)$, say the south-west corner, and show that there is no \texttt{large} chain of height $\geq H+1-n$ sites which exits $L\sL(\ell_n)$ in this corner except with probability $o(1)$. As we want to upper bound the probability of an increasing event, we can first make monotone increasing adjustments to the measure. Call $A, B$ the two points where $L\sL(2\ell_n)$ lifts off the west, south boundaries of $\partial \Lambda$, respectively (they are at distance $\leq (1+\epsilon_\beta)sN_n$ from the bottom left corner of $\Lambda$). Let $R$ be the rectangle with width $B_1 - A_1$ and height $sN_n$, such that its top right corner is at $B$. Let $\widetilde \Lambda := \Lambda \cup R$. Finally, let $\xi$ denote boundary conditions on $\partial \widetilde\Lambda$ taking value $H-n$ on the smaller arc of $\partial \widetilde\Lambda$ between $A$ and $B$, and $H+1-n$ otherwise. By monotonicity, we can move to the measure $\pi^\eta_{\widetilde\Lambda}$. Let $\fL$ be the $H+1-n$ level line including the boundary disagreements of $\xi$, and let $\gamma$ be the disagreement polymer containing $\fL$. Let $D_1$ denote the region below $\gamma$ (which has $H-n$ b.c.). It suffices to show that $\gamma$ does not exit $L\sL(\ell_n)$ except with probability $o(1)$, as then a \texttt{large} chain of $\geq H+1-n$ sites can be excluded by a standard Peierls argument on $D_1$. 

Once again, we begin by setting up for \cref{prop:CE-macro-DobrushinBC} with the following claim:
    \begin{claim}\label{clm:CE-check-base-case}
    Let $\cG$ denote the set of $\gamma$ such that $\sE_\beta(\gamma) \leq Le^{\sqrt{\log L}}$, $|D_1| \leq (\tfrac{3\beta}{\hatpi_\infty(\phi_o > H-n)})^2$. Then, $\pi^\eta_{\widetilde\Lambda}(\cG) = 1-o(1)$.
\end{claim}
\begin{proof}
    We first bound the probability that $\cE_\beta(\gamma) > Le^{\sqrt{\log L}}$. Similarly to \cref{clm:CE-check-LW}, by FKG, the decorrelation estimate of \cref{eq:couple-pi-to-inf-vol}, and the bound on the large deviation ratios in \cref{thm:LD-DG}, we have the lower bound
    \[\hatpi^\eta_{\widetilde\Lambda}(\phi_x \geq 0,\, \forall x \in \widetilde\Lambda) \geq \prod_{x \in \Lambda} \hatpi^{H-n}_{\widetilde\Lambda}(\phi_x \geq 0) \geq \tfrac12\exp\big(-|\widetilde\Lambda|\hatpi_\infty^{H-n}(\phi_x < 0)\big) \geq \tfrac12e^{-Le^{c\sqrt{\beta\log L/\log\log L}}}\,.\]
    (The $\sqrt{\log L/\log\log L}$ is not needed for $n = 0$, but is there to clarify how to generalize for larger $n$.) Hence, it suffices to show in the no floor measure that $\hatpi^\eta_{\widetilde\Lambda}(\cE_\beta(\gamma) > Le^{\sqrt{\log L}}) = o(e^{-Le^{c\sqrt{\beta\log L/\log\log L}}})$. Using the form of the law in \cref{prop:CE-law-without-area}, we can conclude by a Peierls map to a minimal length path from $A$ to $B$ that $\hatpi^\eta_{\widetilde\Lambda}(\sE_\beta(\gamma) \geq L{e^{\sqrt{\log L}}}) \leq e^{-c\beta Le^{\sqrt{\log L}}}$. 

    We turn now to bound $|D_1|$. We will in fact show that for some $c, c'$, we have \begin{align*}
    \pi^\eta_{\widetilde\Lambda}(|D_1| > cs^2N_n^2) = e^{-c'sN_n}\,.
\end{align*}
Let $D_1'$ be the connected component of sites with height $\leq H-n$ containing the $H-n$ boundary condition of $\xi$. Then $D_1 \subset D_1'$, and the event $|D_1'| > cs^2N_n^2$ is decreasing. Hence by monotonicity it suffices to drop the floor and bound $\hatpi^\eta_{\widetilde\Lambda}(|D_1'| > cs^2N_n^2)$. Without the floor, we can again appeal to \cref{prop:CE-law-without-area} and apply a Peierls map argument. More specifically, by the isoperimetric inequality and the fact that $|\partial D_1' \setminus \partial \widetilde \Lambda| \leq |\gamma|$, we have that $|D_1'| > cs^2N_n^2$ implies that $|\gamma| \geq (\tfrac{\sqrt{c}}{4}-4(1+\epsilon_\beta))sN_n$. Taking $c$ large enough, we can rule this out by considering the Peierls map which sends $\gamma$ to a minimal length path from $A$ to $B$. 
\end{proof}
    
With the claim proven, we can restrict our attention to $\gamma \in \cG$. Moreover, specific for $n = 0$, we have $|D_0| \leq |\Lambda| \leq (\tfrac{3\beta}{\hatpi_\infty(\phi_o > H+1-n)})^2$. Hence by \cref{prop:CE-macro-DobrushinBC} we have \begin{align}\label{eq:base-case-CE-ref}\pi^\eta_\Lambda(\gamma) &\propto \exp\bigg(-\sE^*_\beta(\gamma) + \fI_{\Lambda}(\gamma) + \tfrac{\uprho_n}{N_n}|D_0| + \uprho_n\hatpi_\infty(\phi_x \geq H+1-n)|(\bigcup_{i \geq 2}D_i)|+ O(L^{1/2+o(1)})\bigg)\\
&=: e^{O(L^{1/2+o(1)})}\fp_{V}^{\eta}(\gamma)
    \end{align}
    Suppose that we did not have the extra $e^{O(L^{1/2+o(1)})}$ error term and the weight was just $\fp^\eta_V(\gamma)$. Then, using the same argument as leading to \cref{eq:control-finite-comp-square}, we can further restrict to $\gamma$ such that $\max_{\cD \in \cD} |\partial \cD| \leq \log L$, so that the term $\uprho_n\hatpi_\infty(\phi_x \geq H)|(\bigcup_{i \geq 2}D_i)| = o(1)$. In more detail, with \cref{eq:base-case-CE-ref}, we can rule out the case that $|\gamma| > 5sN_n$ by taking the Peierls map sending $\gamma$ to the disagreement contour which coincides with the small arc of the boundary from $A$ to $B$ on $\partial \widetilde\Lambda$. Since the minimal length path from $A$ to $B$ has length $2\ell(1+\epsilon_\beta) \leq 2sN_n(1+\epsilon_\beta)$, this implies that we can assume $\max_{\cD \in \fD(\gamma)}|\partial \cD| \leq 5sN_n$ as $\tfrac12|\partial \cD|$ is excess length. With this input, the steps leading up to \cref{eq:control-finite-comp-square} can be followed exactly, considering the map which replaces a component $\cD$ with a minimal length path.
    
    Hence, we will focus on the polymer model of disagreement polymers from $A$ to $B$, restricted to lie in $\widetilde\Lambda$, with weight given by $\exp(-\sE^*_\beta(\gamma) + \fI_{\Lambda}(\gamma) + \tfrac{\uprho_n}{N_n}|D_0|)$, and show that under this model, $\gamma$ exits $L\sL(\ell)$ with probability $o(e^{-L^{1/2+o(1)}})$. Call $\cB$ the set of bad $\gamma$ which exits $L\sL(\ell)$. Denote the probability and expectation with respect to this model by $\bP$ and $\bE$. We can also consider the same model without the area term $\tfrac{\uprho_n}{N_n}|D_0|$, denoted by $\hat\bP$ and $\hat\bE$. Finally, let $\cA(D_0)$ denote the area of $|D_0|$ normalized to be a signed area with respect to the line segment $\overline{AB}$ (where the area above $\overline{AB}$ is signed positive). This normalization doesn't change the measure, so we can replace $|D_0|$ with $\cA(D_0)$ in the definitions of $\bP$ and $\bE$. 

    Automatically, we have that for some absolute constant $c$, $\cA(D_0) \leq cs^2N_n^2$.
    Then, by Jensen's inequality we can write
    \[\bP(\cB) = \frac{\hat\bE[e^{\uprho_n\cA(D_0)/N_n}\one_\cB]}{\hat\bE[e^{\uprho_n\cA(D_0)/N_n}]} \leq \frac{e^{cs^2N_n}\hat\bE[\one_\cB]}{\hat\bE[e^{\uprho_n\cA(D_0)/N_n}]} \leq \frac{e^{cs^2N_n}\hat\bE[\one_\cB]}{e^{\hat\bE[\uprho_n\cA(D_0)/N_n]}}\,.\]
    From \cite[Prop.~3.14, Lem.~3.16]{ChenLubetzky25}, for some constants $c', c'' > 0$, the expectation in the numerator can be upper bounded by $e^{-c'sN_n}$ and the expectation in the denominator can be lower bounded by $-c''(sN_n)^{1/2}$. For $s$ sufficiently small depending only on the constants $c, c'$, the dominant term is $e^{-c'sN_n}$, which proves the desired upper bound on $\bP(\cB)$.
\end{proof}

Given the local control over the level lines in \cref{prop:retreat-gadget}, the induction step below follows as in \cite{CLMST16}, only with certain calculations involving $N_n$ instead of $L$. For completeness, the main ideas of the proof are presented below, and we refer to \cite{CLMST16} for more details. Note that unlike the base case, the proof for the induction step works for any fixed $n \geq 0$.
\begin{lemma}[{\cite[Lem.~6.12]{CLMST16}}, Induction step for $n \geq 0$]\label{lem:induction-step}
    Fix $n \geq 0$. Let $\ell_n = sN_n/L$ for $s$ from \cref{lem:base-case}. If $\ell_n \leq \ell \leq (1-\tfrac12e^{-\frac{C}3\sqrt{\log L}})\ell^*_n$ and $\cE_n(\ell_n)$ holds w.h.p., then so does $\cE_n(\ell + L^{-3/4})$.
\end{lemma}
\begin{proof}
    By symmetry it suffices to consider one of the four corners of $\Lambda$, say the bottom left one. Take the convention that the origin is at the center of $\Lambda$, so in particular the bottom side of $\partial \Lambda$ is on the line $y = -\tfrac{L}{2}$. Let $f_\ell(x)$ be the function on $(-\frac{1+\epsilon_n}{2}L, 0]$ whose graph is the bottom left section of $\partial L(1+\epsilon_n)\sL(\ell)$. Let $\hat{x}(\ell)$ be the solution to $f_\ell(x) = x$, and $x_{\mathsf{L}}(\ell)$, $x_{\mathsf{R}}(\ell)$ the first points where $f_\ell(x)$ is smaller than $-\tfrac L2$, $-\tfrac L2(1+\epsilon_n)$, respectively.

    Now let $\ell' = \ell + L^{-3/4}$, and fix $x \in [\hat{x}(\ell), x_{\mathsf{L}}(\ell')]$. Let $x^\pm := x + \tfrac12N_n^{2/3}e^{C\sqrt{\log L}}$. Let $\theta \in [0, \pi/4]$ satisfy $\tan(\theta) = |f_\ell'(x)|$. By construction, $x_{\mathsf{R}}(\ell) - x_{\mathsf{L}}(\ell') \geq c\sqrt{\epsilon_n}N_n = cN_n^{5/6}$ for some constant $c > 0$, so that in particular $x^+ < x_{\mathsf{R}}(\ell)$.
    Thus, we can use the deterministic bound on Wulff shapes in \cite[Lem.~3.9]{CLMST16} with $d = N_n^{2/3}e^{C\sqrt{\log L}}/L$ and then scale by $L(1+\epsilon_n)$ to obtain that
    \[f_\ell(x) = \tfrac12(f_\ell(x^-) + f_\ell(x^+)) - \frac{\sfw_1(\tau_\beta)N_n^{4/3}e^{2C\sqrt{\log L}}}{L\ell(1+\epsilon_n)16(\tau_\beta(\theta) + \tau_\beta''(\theta))\cos^3(\theta)} + o(1)\,.\]
    Since $\ell' = \ell + L^{-3/4}$, the scaling of the Wulff shape implies that $f_{\ell'}(x) \leq f_\ell(x) + CL^{1/4}$ for some constant $C$. We aim to show that w.h.p., there is no \texttt{large} chain of sites with height $\geq H+1-n$ which drops below the point $(x, f_\ell(x'))$. Motivated by \cref{prop:retreat-gadget}, define
    \[Z_\ell(x):= \tfrac12(f_\ell(x_-) + f_\ell(x_+)) - \frac{\uprho_n N_n^{1/3}e^{2C\sqrt{\log L}}}{8(\tau_\beta(\theta) + \tau_\beta''(\theta))\cos^3(\theta)} - \sigma e^{C\sqrt{\log L}}\,.\]
    Suppose we can show the following claim:
    \begin{claim}\label{clm:retreat-claim}
        For any choice of $x \in [\hat{x}(\ell), x_{\mathsf{L}}(\ell')]$, with probability $1-o(1)$, there is no \texttt{large} chain of sites with height $\geq H+1-n$ which drops below the point $(x, Z_\ell(x))$.
    \end{claim}
    This is sufficient as long as $Z_\ell(x) \geq f_\ell(x) + CL^{1/4}$, and it is a straightforward computation to check that this is satisfied as long as $\ell$ satisfies
    \begin{equation}\label{eq:key-eq}
        \frac{\sfw_1(\tau_\beta)N_n}{2L\ell(1+\epsilon_n)\uprho_n } - 1 \geq e^{-\frac{C}{2}\sqrt{\log L}}c(\beta, \theta)
    \end{equation}
    for some constant $c(\beta, \theta)$ (similar to \cite[Eq. 4.34]{ChenLubetzky25}), or equivalently,
    \[\ell \leq \frac{w_1(\tau_\beta)N_n}{2L(1+\epsilon_n)(1+e^{-\frac{C}{2}\sqrt{\log L}}c(\beta, \theta))\uprho_n }\,.\]
    This in turn is satisfied by $\ell \leq (1-\tfrac12e^{-\frac{C}{3}\sqrt{\log L}})\ell^*_n$. We can then conclude the proof by iterating over all $x \in [\hat{x}(\ell), x_{\mathsf{L}}(\ell')]$ and then using symmetry across the line $y = x$. 

    We now prove \cref{clm:retreat-claim}. This will be a monotonicity argument to justify the application of \cref{prop:retreat-gadget}. Observe that the event in the claim is decreasing, so we are allowed to increase heights en-route to showing it occurs w.h.p. Let $R$ be the rectangle of width $N_n^{2/3}e^{C\sqrt{\log L}}$ and height $3N_n^{2/3}e^{C\sqrt{\log L}}$, centered at $x$. Call $A, B$ the intersection of $\partial L(1+\epsilon_n)\sL(\ell)$ with the left, right sides of $\partial R$ respectively. (If $R$ does not fit inside $\Lambda$, we can again enlarge the domain, taking $\widetilde\Lambda = \Lambda \cup R$ with $0$ boundary conditions.) Call the portion of the boundary $\partial R$ above $A, B$ by $\partial^1 R$, and let $\partial^2 R := \partial R \setminus \partial^1 R$. If we reveal the components of sites with height $\geq H+2-n$ along $\partial^1 R$, the external boundaries of such revealed components stitch together to form a chain of sites with height $\leq H+1-n$. By a Peierls argument, w.h.p.\ these components will have size at most $\log L$. Similarly, if we reveal the components of sites with height $\geq H+1-n$ along $\partial^2 R$, then the external boundaries can be stitched together to form a chain of sites with height $\leq H-n$, and by definition on the event $\cE_n(\ell)$ these components will have size at most $\log L$. Hence, w.h.p, there exists a circuit of sites passing through $A, B$ which has height $\leq H-n$ in the arc below $A, B$ and height $\leq H+1-n$ in the arc above $A, B$. Revealing the outermost such circuit and raising the heights to be equal to $H-n-1$ and $H-n$ respectively, the result is a domain $Q$ with boundary conditions $\xi$ satisfying \cref{it:Q-simp-conn,it:dist-Q-R-bdy,it:Q-bc-arc} with respect to $R$ (and notably, no information about $\phi$ in the interior of $Q$ was revealed). Finally, the technical condition \cref{it:good-boundary-points} can be satisfied by a procedure completely analogous to steps 5, 6 in the proof of \cite[Lem.~4.8]{ChenLubetzky25}. Hence, we are in the position to apply \cref{prop:retreat-gadget}, which immediately implies \cref{clm:retreat-claim}.
\end{proof}

This proves \cref{thm:wulff-shape-contains-level-line} for $n = 0$, and now it remains to extend this to any fixed $n \geq 0$.
\begin{lemma}\label{lem:retreat-n}
     For $n \geq 1$, if $\cE_{n-1}(\ell^*_{n-1}(1-\tfrac12e^{-\frac{C}3\sqrt{\log L}}))$ holds w.h.p., then $\cE_n(\ell^*_n(1-\tfrac12e^{-\frac{C}3\sqrt{\log L}}))$ holds w.h.p.\ as well.
\end{lemma}
\begin{proof}
    Since we already have \cref{lem:induction-step}, the only thing missing is an analog of \cref{lem:base-case} for general $n$. That is, with $s$ as in \cref{lem:base-case} and $\ell_n := sN_n/L$, we want to show that $\cE_n(\ell_n)$ holds with probability $1- o(1)$ for all $n \geq 1$. The difference from the $n = 0$ case is that we now need to work to bound $|D_0| \leq (\tfrac{3\beta}{\hatpi_\infty(\phi_o > H+1-n)})^2$ for the application of \cref{prop:CE-macro-DobrushinBC}, while for $n = 0$ this was satisfied from the trivial bound $|D_0| \leq |\Lambda|$.
 
    For each vertex $v$ in the southwest corner of $L(1+\delta_{n-1})\sL(\ell^*_{n-1}(1-\tfrac12e^{-\frac{C}3\sqrt{\log L}}))$, reveal the connected component of sites with height $\geq H+2-n$. This reveals a boundary of $\leq H+1-n$ sites. Since we are assuming that $\cE_{n-1}(\ell^*_{n-1}(1-\tfrac12e^{-\frac{C}3\sqrt{\log L}})$ holds w.h.p., then w.h.p.\ all of these components have size at most $\log L$. We can stitch these together to form a boundary $\Gamma$ of $\leq H+1-n$ sites, and then by monotonicity we can raise the boundary to height $H+1-n$. Now, we can follow the proof of \cref{lem:base-case} starting with the region below $\Gamma$, as opposed to all of $\Lambda$. We now have $|D_0| \leq CN_{n-1}^2$ for a constant $C$ depending only on $s$ (and hence independent of $\beta$), the only missing input needed for the application of \cref{prop:CE-macro-DobrushinBC}. Moreover, $\Gamma$ is sufficiently far away from the end-points of $\gamma$ so that it does not interfere with any of the random walk estimates used. (More precisely, one would need $\Gamma$ to stay distance $CN_n$ away for a constant $C$ depending on $s$. We have $\Gamma$ is distance $O(N_{n-1})$ away, which is much further.) Hence, the proof of \cref{lem:base-case} shows that with probability $1-o(1)$ we have $\cE_n(\ell_n)$, and then \cref{lem:induction-step} concludes.
\end{proof}

\begin{remark}\label{rem:hausdorff-error-quant}
    The proof of \cref{lem:induction-step} gives details behind the heuristic described in the proof ideas in \cref{sec:proof-ideas}. Here we have an area-tilted random walk at scale $d = N_n^{2/3}f(L)$ for $f(L) = e^{C\sqrt{\log L}}$. The quantity $f_\ell(x)$ is where the boundary of the Wulff shape of size $\ell L$ is, and $Z_\ell(x)$ is where the level line $\fL_n$ is w.h.p.\ above. Let us ignore the centering by $\tfrac12(f_\ell(x^-) + f_\ell(x^+))$ present in both terms. Let us also ignore the factor of $1+\epsilon_n$, which was necessary only for technical reasons, since $\epsilon_n$ is negligibly small. We are left with
    \begin{align*}\bar f_\ell(x) &:= \frac{\sfw_1(\tau_\beta)N_n^{4/3}e^{2C\sqrt{\log L}}}{L\ell16(\tau_\beta(\theta) + \tau_\beta''(\theta))\cos^3(\theta)}\,,\\
    \bar Z(x) &:= \frac{\uprho_n N_n^{1/3}e^{2C\sqrt{\log L}}}{8(\tau_\beta(\theta) + \tau_\beta''(\theta))\cos^3(\theta)} + \sigma e^{C\sqrt{\log L}} =: \mu+f\sigma\,.
    \end{align*}
    $\bar Z(x)$ has a mean term $\mu$, representing the mean of the level line's location above $x$, and a fluctuation term $f\sigma$ representing the most the level line will deviate from $\mu$ except with probability $o(1)$. Then, first observe that $\ell^*_n$ is indeed the value such that $\bar f_{\ell^*_n}(x)=\mu$, giving a first order estimate for the location of $\fL_n$. Secondly, observe that \cref{eq:key-eq} is equivalent to the condition $f_\ell(x) \geq \mu+f\sigma$. Since $f_\ell(x) = C(\beta, \theta)d^2/\ell$ depends inversely on $\ell$, we have altogether that the induction works up to $\ell/\ell_* \leq 1 - O(f\sigma/\mu)$. On the other side, we have a matching bound in \cref{prop:shape-thm-crit-window} which is based off an induction holding up to $\ell/\ell_* \geq 1+ O(f\sigma/\mu)$ (see \cite[Eq. 4.34]{ChenLubetzky25}). 
    
    Moreover, $f\sigma/\mu$ is a power of $f^{-1}$, and this is the form of the Hausdorff distance bound in \cref{thm:main-thm-limit-shape}. Hence, we obtain a better estimate when $f$ is larger. The size we can take for $f$ comes from the bound on $|\partial F| \vee |\partial V|$ in \cref{prop:CE-meso}, which in turn depends on the respective bound in \cref{thm:key-area}, measuring how well we can estimate the probability of staying above a floor. In particular, this area bound will change for different $p$, resulting in the different Hausdorff bounds in \cref{thm:grad-phi-p}.
\end{remark}

\section{Ferrari--Spohn limit law for \texorpdfstring{$\fL_0$}{L0}}\label{sec:limit-law}
In this section we prove \cref{thm:FS-no-exceptional}, which extends the limit law theorem of \cite[Thm.~1.1]{ChenLubetzky25} to handle the $\fL_0$ level line, and to cover all $L$ (whereas the previous result avoided an exceptional set of $L$ close to $L_c^{(h)}$). As in \cref{sec:growth}, this is a matter of modifying the original proof to circumvent an initial lack of monotonicity. The main point is that even though a measure such as $\pi^0_{\Lambda}(\cdot \mid \fL_0 \neq \emptyset)$ does not have FKG, by zooming in to a mesoscopic box, we can still compare level lines in this measure to level lines in a $\ZGFF$ measure which does have FKG. We split the proof into four statements in the proposition below, being the analog of \cite[Thms.~5.1 and~6.1]{ChenLubetzky25} in the settings where we condition on $\fL_0 = \emptyset$ and $\fL_0 \neq \emptyset$. 

Fix $n\geq 0$ and $K>0$. Let $\rho_n(x)$ be the maximum vertical distance of~$\fL_n$ above $x+(\frac{L}2,0)$ for $-N_n^{2/3} \leq x\leq N_n^{2/3}$, and let $\sigma_n$ be the same constant as in \cref{thm:CL25}. Define the process $Y_n(t):=N_n^{-1/3}\rho_n(t N_n^{2/3})$ ($t\in[-K,K]$). 
\begin{proposition}
     Fix $m \geq 0$. For every sequence of $L \to \infty$ such that each $L$ satisfies $\pi^0_{\Lambda}(\fL_0 \neq \emptyset) > \delta_L$, then sampling $\fL_n$ under the measures $\pi^0_{\Lambda}(\cdot \mid \fL_0 \neq \emptyset)$, every weak limit point $(\bY_n(t))_{0 \leq n\leq m}$ of the processes $(Y_n(t))_{0 \leq n\leq m}$ satisfies 
     \begin{align}
         &(\bY_n)_{0 \leq n\leq m} \preceq \bigotimes_{0 \leq n\leq m}\mathsf{FS}_{\sigma_n}\,,\label{eq:FS-UB-top-exists}\\
         &(\bY_n)_{0 \leq n\leq m} \succeq \bigotimes_{0 \leq n\leq m}\mathsf{FS}_{\sigma_n}\,.\label{eq:FS-LB-top-exists}
     \end{align}
    If the sequence $L \to \infty$ is such that each $L$ satisfies $\pi^0_{\Lambda}(\fL_0 = \emptyset)> \delta$, then sampling $\fL_n$ under the measures $\pi^0_{\Lambda}(\cdot \mid \fL_0 = \emptyset)$ we have
    \begin{align}
         &(\bY_n)_{1\leq n\leq m} \preceq \bigotimes_{1 \leq n\leq m}\mathsf{FS}_{\sigma_n}\,,\label{eq:FS-UB-top-not-exists}\\
         &(\bY_n)_{1 \leq n\leq m} \succeq \bigotimes_{1 \leq n\leq m}\mathsf{FS}_{\sigma_n}\,.\label{eq:FS-LB-top-not-exists}
     \end{align}
\end{proposition}
\begin{observation}\label{obs:no-other-level-lines}
    For any $L$, under $\pi^0_{\Lambda}$, by a Peierls argument, w.h.p.\ there are no \texttt{large} level-line down-loops (as per \cref{def:level-lines}). Moreover by \cref{thm:main-thm-limit-shape}, w.h.p.\ there is at most one \texttt{large} level line for heights $1, \ldots, H+1$, and no \texttt{large} $H+2$ level line\footnote{This characterization was already shown for all heights except $H, H+1$ in \cite[Thm.~2]{LMS16}}. Hence, for an appropriate choice of $\delta_L$, these events occur w.h.p.\ also under the measure $\pi^0_{\Lambda}(\cdot \mid \fL_0 \neq \emptyset)$ when $\pi^0_{\Lambda}(\fL_0 \neq \emptyset) > \delta_L$, and under $\pi^0_{\Lambda}(\cdot \mid \fL_0 = \emptyset)$ when $\pi^0_{\Lambda}(\fL_0 = \emptyset) > \delta_L$.
\end{observation}
\begin{proof} As a technical step, let us redefine $\fL_0$ to be the \texttt{large} $H+1$ level line surrounding the origin; by \cref{thm:main-thm-limit-shape} this w.h.p.\ does not change $\fL_0$. We start with \cref{eq:FS-UB-top-exists}. Let $L$ be such that $\pi^0_{\Lambda}(\fL_0 = \emptyset) > \delta_L$. Fixing $n \geq 0$, \cref{prop:shape-thm-crit-window} and the discussion at the beginning of \cref{sec:limit-shape} show that w.h.p.\ under $\pi^0_{\Lambda}(\cdot \mid \fL_0 \neq \emptyset)$, if we look above the bulk of the bottom of $\partial \Lambda$ (say, at least distance $9L/10$ from the corners), then $\fL_n$ is at most distance $N_n^{1/3}(\log L)^{16}$ from the bottom of $\Lambda$ and $\fL_{n+1}$ is at most distance $N_n^{1/3}\exp(-c\sqrt{\beta\frac{\log L}{\log\log L}})$ for some constant $c > 0$. (The control on $\fL_{n+1}$ follows from the bound on \cref{eq:LD-ratio}.)

Now let $R_n$ be the rectangle of size $N_n^{2/3}(\log L)^{25} \times 2N_n^{2/3}(\log L)^{25}$ positioned so that its bottom side distance $2N_n^{1/3} \exp(-c\sqrt{\beta\log L/\log\log L})$ above the bottom side of $\Lambda$, and its center is at $x = 0$. Let $A, B$ be the two points on the left and right sides of $\partial R_n$ such that the distance from $A, B$ to the bottom of $R_n$ is $N_n^{1/3}(\log L)^{16}$. As discussed, w.h.p.\ under $\pi^0_{\Lambda}(\cdot \mid \fL_0 \neq \emptyset)$ the points $A, B$ lie in the interior of $\fL_n$ and $R_n$ lies in the interior of $\fL_{n+1}$.

The main claim to argue is that the stochastic domination lemma of \cite[Lem.~4.8]{ChenLubetzky25} still holds in the conditional measure $\pi^0_{\Lambda}(\cdot \mid \fL_0 \neq \emptyset)$. That is, we will show there is a $\pi^0_{\Lambda}(\cdot \mid \fL_0 \neq \emptyset)$ measurable distribution on simply connected regions $Q \subset R$ and $H+1-n, H-n$ Dobrushin boundary conditions $\xi$ such that w.h.p.\ (in the sampling of $Q$ and $\xi$), the restriction to $Q$ of $\fL_n$ under $\pi^0_{\Lambda}(\cdot \mid \fL_0 \neq \emptyset)$ is stochastically dominated by (i.e., lies below) the $H+1-n$ level line in $\pi^\xi_Q$. Moreover, the distribution on $Q, \xi$ is supported on pairs such that $d(\partial Q, \partial R_n) < \log L$, and the boundary change in $\xi$ occurs in a sufficiently regular part of $\partial Q$ (see \cite[Lem.~4.8~(3)]{ChenLubetzky25} for details). Consequently, to prove \cref{eq:FS-UB-top-exists}, it suffices to prove the bound in $\pi^\xi_Q$ for every such $Q, \xi$ satisfying the aforementioned properties. Crucially, for every fixed $Q, \xi$, the measure $\pi^\xi_Q$ does have FKG.

To prove this claim, call $\partial^1 R_n, \partial^2 R_n$ the top and bottom arcs of $\partial R_n$ going from $A$ to $B$. Consider the following revealing procedure: first reveal all of $\phi\restriction_{\Lambda\setminus R_n}$. Then, for every $x \in \partial^1 R_n$, reveal its (possibly empty) connected component of sites with height $\leq H-n$, which in turn reveals a boundary of $\geq H+1-n$ sites. Similarly, for every $x \in \partial^2 R_n$, reveal its (possibly empty) connected component of sites with height $\leq H-n-1$, which in turn reveals a boundary of $\geq H-n$ sites. Observe that the set of unrevealed sites, call it $Q$, has induced boundary conditions of part $\geq H+1-n$ and part $\geq H-n$. A consequence of \cref{obs:no-other-level-lines} (specifically concerning level-line down-loops), $A, B$ being in the interior of $\fL_n$, and $R_n$ being in the interior of $\fL_{n+1}$, is that throughout this revealing process, w.h.p.\ every revealed component has size at most $\log L$. Hence, $Q$ is a simply connected component with $d(\partial Q, \partial R_n) < \log L$, and the above procedure gives a $\pi^0_{\Lambda}(\cdot \mid \fL_0 \neq \emptyset)$ measurable distribution on sets $Q \subset R$ and (not yet Dobrushin) b.c.\ $\tilde\xi$ such that the restriction to $Q$ of $\fL_n$ in $\pi^0_{\Lambda}(\cdot \mid \fL_0 \neq \emptyset)$ is equal to the  $H+1-n$ level line in $\pi^{\tilde\xi}_Q(\cdot \mid \fL_0 \neq \emptyset)$.\footnote{At this point the $H+1-n$ level line may not be unique in $\pi^{\tilde\xi}_Q$. Hence, it is more accurate to say that if $\cA_1$ is the intersection of the interior of $\fL_n$ with $Q$ under $\pi^0_{\Lambda}(\cdot \mid \fL_0 \neq \emptyset)$ and $\cA_2$ is the connected component of the $\geq H+1-n$ sites containing the top side of $\partial Q$, then $\cA_1 = \cA_2$. This is resolved immediately after when the b.c.\ $\tilde\xi$ are changed to Dobrushin b.c.\ $\xi$.} Moreover, the event $\{\fL_0 \neq \emptyset\}$ is measurable w.r.t.\ $\phi\restriction_{\Lambda\setminus R_n}$, so that in fact $\pi^{\tilde\xi}_Q(\cdot \mid \fL_0 \neq \emptyset) = \pi^{\tilde\xi}_Q$. (Here, we use that $\fL_0$ has been redefined as the \texttt{large} $H+1$ level line containing the origin, so its existence is determined by $\phi\restriction_{\Lambda \setminus R_n}$.) The latter is simply a $\ZGFF$ measure above a floor, and has FKG. Hence, we can now lower the $\geq H+1-n$ heights to be $H+1-n$, and the $\geq H-n$ heights to be $H-n$, in the manner outlined in steps (5), (6) of the proof of \cite[Lem.~4.8]{ChenLubetzky25} to ensure that the boundary conditions change from $H-n$ to $H+1-n$ at a sufficiently regular part of $\partial Q$. This concludes the proof of the claim.

Having reduced to studying $\pi^\xi_Q$, the proof of \cref{eq:FS-UB-top-exists} now follows exactly as in \cite[Thm.~5.1]{ChenLubetzky25}. First the Ferrari--Spohn domination is established for a single level line $\fL_m$. Then we observe that this domination holds for $\fL_{m-1}$ conditionally on $\fL_m$, since w.h.p.\ $\fL_m$ (along with all the lower level lines) was anyways revealed during the exposing of $\phi\restriction_{\Lambda\setminus R_{m-1}}$ in the revealing process above. Continuing in this way until $\fL_0$, this suffices to show the stochastic domination of the joint law as written in \cref{eq:FS-UB-top-exists}. 

The proof of \cref{eq:FS-UB-top-not-exists} is exactly the same, only we consider $n \geq 1$ since there is no longer a $\fL_0$. The proof of \cref{eq:FS-LB-top-exists} is very similar so we will be more brief. This time, we must begin with $n = 0$, but we will leave things in terms of $n$ for clarity on how to handle $n \geq 1$ later. Again we zoom in to a mesoscopic rectangle $R_n$, this time having dimensions $3TN_n^{2/3} \times N_n^{1/3}(\log L)$ for some $T > K$, and placed so the bottom of $\partial R_n$ coincides with the bottom of $\partial \Lambda$ and $R_n$ is centered at $x = 0$. Let $A = (-TN_n^{2/3}, 0)$ and $B = (TN_n^{2/3}, 0)$. The main claim is that there is a $\pi^0_{\Lambda}(\cdot \mid \fL_0 \neq \emptyset)$-measurable distribution on simply connected $Q \subset R$ and $H+1-n, H-n$ Dobrushin b.c.\ $\xi$ such that w.h.p., the restriction to $Q$ of $\fL_n$ under $\pi^0_{\Lambda}(\cdot \mid \fL_0 \neq \emptyset)$ stochastically dominates (i.e., lies above) the $H+1-n$ level line in $\pi^\xi_Q$. This distribution is supported on $Q$ such that $d(\partial Q, \partial R_n) < \log L$, and the boundary conditions change exactly at $A, B$.

To prove the claim, under $\pi^0_{\Lambda}(\cdot \mid \fL_0 \neq \emptyset)$, reveal all of $\phi\restriction_{\Lambda\setminus R_n}$. Then along the top and sides of $\partial R_n$, reveal each connected component of sites with height $\geq H+2-n$, which in turn reveals a boundary of $\leq H+1-n$ sites. W.h.p.\ every component revealed has size at most $\log L$ because we are starting with $n = 0$ and by \cref{obs:no-other-level-lines} there is no \texttt{large} $H+2-n$ level line. Hence we can reduce to the measure $\pi^{\tilde\xi}_Q$ for some simply connected $Q$ and b.c.\ $\tilde\xi$ which is $\leq H+1-n$ on the sides and top of $\partial Q$, and 0 on the bottom flat side of $\partial Q$. Then by monotonicity we can raise heights to obtain Dobrushin boundary conditions $\xi$ satisfying the desired properties.

Having reduced to the measure $\pi^\xi_Q$, the proof of \cref{eq:FS-LB-top-exists} in the case of a single level line ($m = 0$) now follows exactly as in \cite[Thm.~6.1]{ChenLubetzky25}. In particular, the top level line $\fL_0$ lives at scale $N_0^{1/3}$ and w.h.p.\ does not intersect the rectangle $R_1$. Thus, we can now run the proof again for $n = 1$, and this time the revealed components of sites with height $\geq H+2-1$ along the top and sides of $\partial R_1$ are at most size $\log L$ because by \cref{obs:no-other-level-lines} there is at most one \texttt{large $H+2-1$} level line and it does not intersect $R_1$. This establishes the domination of Ferrari--Spohn for $\fL_1$, which holds even conditionally on $\fL_0$ due to the revealing procedure as before. This then shows that w.h.p.\ $\fL_1$ does not intersect $R_2$, and the process can be repeated inductively to obtain \cref{eq:FS-LB-top-exists}.

Finally, the proof of \cref{eq:FS-LB-top-not-exists} is the same, only we start with $n = 1$ instead and the needed input that there is no \texttt{large} $H+2-1$ level line intersecting $R_1$ now follows by the conditioning on $\{\fL_0 = \emptyset\}$ and \cref{prop:shape-thm-crit-window}.
\end{proof}

\section{Extension to \texorpdfstring{$|\nabla\phi|^p$}{grad-phi} models for \texorpdfstring{$1 \leq p < \infty$}{p>1}}\label{sec:extend-main-thm-p}
Here we will prove \cref{thm:grad-phi-p}, extending results to the $|\nabla\phi|^p$ model from \cref{eq:grad-phi-def}. To aid this, the following is a summary of the previous sections and how they relate to each other: 
\begin{itemize}
    \item \cref{sec:prelim} provides \cref{lem:linear-critical-window}, which is not model dependent. Otherwise, the preliminary definitions and cluster expansion results there have already been established for general $p > 1$, see in particular \cite[Rem.~3.4, Prop.~7.12]{ChenLubetzky25}. 

    \item \cref{sec:xi-bbq} provides estimates on the probability of staying above a floor in a given region. The results of this section will require some new arguments to extend to $p >1$.

    \item \cref{sec:refined-CE}  uses the area estimates of \cref{sec:xi-bbq} to study the law of the disagreement polymer in larger domains than covered in \cref{sec:prelim}. The proof of these results given \cref{prop:uprho-bound,thm:key-area} generalize verbatim for $p > 1$. 

    \item \cref{sec:transition-window,sec:limit-shape,sec:limit-law} then prove \cref{thm:main-thm-limit-shape,thm:main-thm-crit-window,thm:FS-no-exceptional} by moving to a polymer model using the results of \cref{sec:refined-CE}, or otherwise using routine FKG and Peierls arguments which hold for all $p \geq 1$.
\end{itemize}

Hence, it remains to prove analogs of \cref{prop:uprho-bound,thm:key-area}, modifying only the $\log$ factors which arise from large deviation probabilities under $\hatpi^{(p)}_\infty$, and then check that modifying these log factors do not cause any problems. We will do this in \cref{sec:extend-to-p} modulo some missing large deviation estimates, and then fill in these missing estimates in \cref{sec:missing-LD-p}. Finally in \cref{sec:SOS-extension} we discuss extensions to the \SOS model ($p = 1$).

\subsection{Extending \cref{prop:uprho-bound,thm:key-area} to \texorpdfstring{$p>1$}{p>1}}\label{sec:extend-to-p}
Here we will show how the results of \cref{sec:xi-bbq} extend for general $p > 1$. We define the natural generalizations
\begin{align}
    \label{eq:xi-n-def-p}
    \upxi^{(p)}_n
    &:= -\frac{1}{\ell_*^2}\log\hatpi^{(p)}_\infty\left( \phi_x \geq -(H+1-n),\,\forall x\in Q_{\ell_*} := \llb 1,\ell_*\rrb^2 \right)\quad\mbox{for}\quad \ell_*:=2^{\lceil\frac12\log_2 L\rceil}\\
    \label{eq:rho-n-def-p}
\uprho^{(p)}_n
&:= \left(\upxi^{(p)}_{n+1} - \upxi^{(p)}_n\right) N_n\,. 
\end{align}
Then we have the following analogs of \cref{prop:uprho-bound,thm:key-area}. Amidst the many $\log L$ and $L^{\delta_p}$ factors, the key properties to note are 
\begin{enumerate}
    \item The error term in the macroscopic range of \cref{thm:key-area-p} is always of the form $L^{1/2 + o(1)}$,

    \item The upper bounds on $\frac{\upxi^{(p)}_n}{\hatpi^{(p)}_\infty(\phi_o > H+1-n)}$ give a margin of at least $L^{-1/2+\delta_p}$ from 1,

    \item This upper bound is transferred to $\uprho^{(p)}_n$ for $1 < p < 2$, but not for $p > 2$.
\end{enumerate}
Hence, \cref{rem:bad-prediction} applies for $1 < p < 2$ as well. The remark should also hold for $p > 2$, as suggested by the upper bound on $\frac{\upxi^{(p)}_n}{\hatpi^{(p)}_\infty(\phi_o > H+1-n)}$, but we are unable to prove a matching upper bound for $\uprho^{(p)}_n$ in this setting. To make progress on this, the first step would be to obtain the rate constant for the large deviation event $\hatpi^{(p)}_\infty(\phi_o \geq h)$, which was proven for $p \leq 2$ in \cite{LMS16} but still missing for $p > 2$.
\begin{proposition}[Cf.\ \cref{prop:uprho-bound}]\label{prop:uprho-bound-p}
        Fix $1 < p < 2$, and $n \geq 0$. There exists a constant $\delta_p > 0$ such that
        \begin{align}
        \label{eq:upxi-bounds-p}
         \frac{\upxi^{(p)}_n}{\hatpi^{(p)}_\infty(\phi_o > H+1-n)} &\in [1 - L^{-\delta_p+o(1)},  1 -L^{-1/2+\delta_p+o(1)}]\,,\\
         \label{eq:uprho-bounds-p}
        \uprho^{(p)}_n &\in [1 - L^{-\delta_p+o(1)},  1 -L^{-1/2+\delta_p+o(1)}]\,.
        \end{align}
        For $p > 2$, there exists $c_0(\beta), c_1(\beta) > 0$ such that
        \begin{align}
        \label{eq:upxi-bounds-p-big}
         \frac{\upxi^{(p)}_n}{\hatpi^{(p)}_\infty(\phi_o > H+1-n)} &\in [1 - e^{-c_0(\beta) \sqrt{\log L}^{\frac{p}{p-1}}},  1 - e^{-c_1(\beta) \sqrt{\log L}^{\frac{p}{p-1}}}]\,,\\
         \label{eq:uprho-bounds-p-big}
        \uprho^{(p)}_n &\in [1-e^{-c_0(\beta) \sqrt{\log L}^{\frac{p}{p-1}}},  1 + L^{-1+o(1)}]\,.
        \end{align}
\end{proposition}
\begin{theorem}[Cf.\ \cref{thm:key-area}]\label{thm:key-area-p}
        Fix $g\geq 0$, $n \geq 0$, and set $h = H+1-n$. The following hold for every subsets $F\subseteq V \subset\Z^2$ where $F$ is connected with $g$ holes. There exists constants $\delta_p, c > 0$ such that
        \begin{enumerate}
            \item\label{it:xi-n-area-meso-p}[\emph{Mesoscopic range}] If $1 < p < 2$ and $|\partial F|\vee |\partial V| \leq O(L^{\frac23+\frac{\delta_p}5})$, or if $p > 2$ and $|\partial F|\vee |\partial V| \leq O(L^{2/3}e^{c\sqrt{\beta\log L}})$, then 
    \begin{align}\label{eq:xi-n-area-estimate-meso-p}\hatpi^{(p), 0}_V(\phi_x \geq -h,\, \forall x \in F) = (1+o(1))\exp\left(-\upxi^{(p)}_n|F|\right)\,.
    \end{align}
    \item\label{it:xi-n-area-macro-p}[\emph{Macroscopic range}] If $1 < p < 2$ and $|\partial F| \vee |\partial V| \leq O( Le^{\sqrt{\log L}})$, then 
    \begin{align}\label{eq:xi-n-area-estimate-macro-lower-p}\quad\;\hatpi^{(p), 0}_V(\phi_x \geq -h,\, \forall x \in F) \geq \exp\left(-\upxi^{(p)}_n|F| + O\big(\sqrt{L}e^{3\sqrt{\log L}}\big)\right)\,,
    \end{align}
    and if $\fS$ is the event that there are no disagreement polymers $\gamma$ in $\phi$ with $|\gamma| \geq \log L$, then        
\begin{align}\label{eq:xi-n-area-estimate-macro-upper-p}\hatpi^{(p), 0}_V(\phi_x \geq -h,\, \forall x \in F \mid \fS) \leq \exp\left(-\upxi^{(p)}_n|F| + O\big(\sqrt{L}e^{3\sqrt{\log L}}\big)\right)\,.
    \end{align}
If $|F| \leq  \big(\frac{3\beta}{\hatpi^{(p)}_\infty(\phi_o > h)}\big)^2 $, then the last upper bound also applies to $\hatpi_V^0(\phi_x\geq -h,\,\forall x\in F)$. The same holds if $p > 2$ and $|\partial F| \vee |\partial V| \leq O(Le^{c\sqrt{\beta\log L}}))$, replacing the error of $O\big(\sqrt{L}e^{3\sqrt{\log L}}\big)$ by $O\big(\sqrt{L}e^{5c\sqrt{\beta\log L}}\big)$.
\end{enumerate}
The above statements also hold when replacing $\hatpi^{(p),0}_V$ by $\hatpi^{(p)}_\infty$.
\end{theorem}

Before we begin the proofs, we explain why the modified $\log$ factors do not cause any problems in the application of \cref{thm:key-area-p}. There are two changes, firstly the change in the requirement on $|\partial F|\vee|\partial V|$, and secondly the error terms in \cref{eq:xi-n-area-estimate-macro-lower-p,eq:xi-n-area-estimate-macro-upper-p}. These changes directly transfer over to generalizations of \cref{sec:refined-CE} for general $p > 1$. 

The second change is never an issue, since these error terms anyways appear in the form $O(L^{1/2+o(1)})$ in all applications in \cref{sec:transition-window,sec:limit-law}.

The first change is relevant in \cref{sec:transition-window,sec:limit-shape,sec:limit-law} when we initially check that we can rule out a-priori any $|\gamma|$ of abnormally large length, so that the regions above and below $\gamma$ (or inside and outside $\gamma$) have boundary lengths satisfying the above requirement. In particular, this is done in \cref{clm:CE-check-LW} and \cref{clm:CE-check-base-case}. In \cref{clm:CE-check-LW} the $\log$ factors do not appear so there is no difference, while the proof of \cref{clm:CE-check-base-case} requires only that the ratio of large deviation rates $\frac{\hatpi^{(p)}_\infty(\phi_o = h-1)}{\hatpi^{(p)}_\infty(\phi_o = h)}$ (given by \cite[Thm.~7.1, Eq.~(7.4)]{ChenLubetzky25}) is smaller than the $\log$ term in the upper bound on $|\partial F| \vee |\partial V|$ in the macroscopic regime. For $1 < p < 2$ this is true as $e^{c\beta(\log L)^{\frac{p-1}{p}}} \ll e^{\sqrt{\log L}}$. For $p > 2$, both are of the form $e^{c\sqrt{\beta\log L}}$, but the constant in \cref{thm:key-area-p} can be taken to be arbitrarily large (it comes from \cref{prop:key-area-estimate-p}, where a larger $c$ gives a weaker result), and in particular larger than $c_0$ from \cite[Thm.~7.1]{ChenLubetzky25}.

Moreover, the new upper bound on $|\partial F|\vee |\partial V|$ in the mesoscopic range determines the size of the rectangles $R$ we can choose in \cref{sec:limit-shape} (in particular, in the ``retreat'' mechanism of \cref{prop:retreat-gadget} and its ``growth'' analog \cite[Thm.~4.9]{ChenLubetzky25}). As mentioned in \cref{rem:hausdorff-error-quant}, this results in the change in the Hausdorff distance bounds in \cref{thm:grad-phi-p}.\footnote{At this point, a keen reader will notice that in fact the requirement on $|\partial F| \vee |\partial V|$ allows for larger domains in $p > 2$ than for $p = 2$, contrary to what should be expected. This is because the bound for $p = 2$ was not sharp; e.g. in the mesoscopic range we could have allowed for all the way up to $O(L^{2/3}e^{c\log L/\log\log L})$ for some constant $c$ depending on previous constants, but the choice of $O(L^{2/3}e^{\sqrt{\log L}})$ suffices and is more reader-friendly.}

We now begin towards the proof of \cref{prop:uprho-bound-p}. Recalling the proof of \cref{prop:uprho-bound}, we will first need bounds on $\xi^{(p)}_{\ell, h}/\hatpi^{(p)}_\infty(\phi_o < -h)$. 
\begin{lemma}[Cf.\ \cref{lem:xi-lower-weak-upper}]\label{lem:xi-lower-weak-upper-p}
Fix $1 < p < 2$. The following holds for all $\beta$ sufficiently large. Let $h = H+1-n$ for fixed $n\geq 0$.  Then for any $\delta > 0$ and all $1 \leq \ell \leq L^{\frac12-\delta}$, we have
\begin{equation}
\label{eq:xi-correct-lower-weak-upper-p}
 1 - L^{-2\delta + o(1)} \leq \frac{\xi^{(p)}_{\ell,h}}{\hatpi^{(p)}_\infty(\phi_o < -h)} \leq 1 + L^{-1+o(1)}\,.
\end{equation}
If $p > 2$ instead, there exists $c(\beta), c_0(\beta)$ such that for $1 \leq \ell \leq \sqrt{L}e^{-c(\beta) \sqrt{\log L}^{\frac{p}{p-1}}}$, we have
\begin{equation}
\label{eq:xi-correct-lower-weak-upper-p-big}
 1 - e^{-c_0(\beta) \sqrt{\log L}^{\frac{p}{p-1}}} \leq \frac{\xi^{(p)}_{\ell,h}}{\hatpi^{(p)}_\infty(\phi_o < -h)} \leq 1 + L^{-1+o(1)}\,.
\end{equation}
Moreover, this holds when replacing $\xi_{\ell,h}$ by $-\frac1{|S|}\log\hatpi^{(p)}_\infty(\phi_x\geq -h,\,\forall x\in S)$ for any set $S\subset \Z^2$, not necessarily connected, of size at most $\ell^2$. In addition, the upper bound holds for all $\ell \geq 1$.
\end{lemma}
\begin{proof}
    The proof of the upper bound is unchanged. The proof of the lower bound for $p > 2$ remains unchanged after plugging in the upper bound on $\hatpi^{(p)}_\infty(\phi_x < -h \mid \phi_y < -h)$ from \cite[Thm.~7.1, Eq.~(7.6)]{ChenLubetzky25}. It is noteworthy that $1<\frac{p}{p-1}<2$ in this range, so that $e^{-c(\beta) \sqrt{\log L}^{\frac{p}{p-1}}} = L^{-o(1)}$. In contrast, for $1 < p < 2$, the bound on $\hatpi^{(p)}_\infty(\phi_x < -h \mid \phi_y < -h)$  is $e^{-c\beta h^p}$ which is a polynomial in $L$. Hence, repeating the proof for $1 < p < 2$ without modification would restrict the allowed $\ell$ to be $\leq L^{1/2 - c}$ for some $c$, and this will turn out to be too restrictive for later. 
    
    Thus, we will modify the proof to allow for a larger range of $\ell$, at the cost of a suboptimal lower bound in \cref{eq:xi-correct-lower-weak-upper-p}. Consider the key equation, \cref{eq:sum-pi(h-h)-ratio}. Plugging in the bound $e^{-c\beta h^p}$ (which replaces the $e^{-c\beta h^2/\log^2h}$ factor), we obtain \[\ell^2(\hatpi^{(p)}_\infty(\phi_o < -h) + L^{-10}) + O((\log L)^2)e^{-c\beta h^p}\,.\]
    Rather than choosing $\ell$ which causes the above two terms to be of the same order, we choose $\ell$ that simply ensures the above is $o(1)$. In particular, since $\hatpi^{(p)}_\infty(\phi_o < -h) \leq \frac1Le^{c\beta(\log L)^{\frac{p-1}p}}$, the conditions on $\ell$ in the lemma ensure that the display is bounded above by $L^{-2\delta + o(1)}$, which as before directly translates to the desired lower bound.
\end{proof}

For a refined upper bound on $\xi^{(p)}_{\ell, h}/\hatpi^{(p)}_\infty(\phi_o < -h)$, we will first need the following conditional large deviation lower bound. As the proof is very different from the $p = 2$ case\footnote{There the claim was stronger and allowed for finite volume settings, but only $V = \Z^2$ was used in application.}, we relegate this to \cref{sec:missing-LD-p}. 
\begin{claim}[Cf.\ \cref{clm:pi(h|h)-lower}]\label{clm:pi(h|h)-lower-p}
Fix $1 < p < 2$. There exists $\delta_p, \bar\delta_p>0$ such that for large $\beta$, if $o'$ is a neighbor of the origin $o$ then
\[ \hatpi^{(p)}_\infty\left(\phi_{o'} \geq h\mid \phi_o \geq h\right) \geq e^{-\beta(\frac12\sE^{(p)}(\phi^*_o) - \bar\delta_p+o(1))h^p} = L^{-1/2+\delta_p+o(1)}\,.\]
If $p > 2$ instead, then 
\[\hatpi^{(p)}_\infty(\phi_{o'} \geq h \mid \phi_o\geq h) \geq e^{-c\beta h^{p/(p-1)}} = e^{-c(\beta) \sqrt{\log L}^{\frac{p}{p-1}}}\,.\]    
\end{claim}

\begin{lemma}[Cf.\  \cref{lem:xi-refined-upper}]\label{lem:xi-refined-upper-p}
Fix $1 < p < 2$, and let $\delta_p$ be from \cref{clm:pi(h|h)-lower-p}. For $h=H+1-n$ with fixed $n\geq 0$, $\beta$ sufficiently large, and $\ell \geq L^{\frac14}$, we have
\[ \frac{\xi^{(p)}_{\ell,h}}{\hatpi^{(p)}_\infty(\phi_o<-h)} \leq 1 - L^{-1/2+\delta_p+o(1)}\,.\]
If instead $p > 2$, then for any constant $\delta > 0$ we can allow $\ell \geq L^\delta$, and we have
\[ \frac{\xi^{(p)}_{\ell,h}}{\hatpi^{(p)}_\infty(\phi_o<-h)} \leq 1 - e^{-c(\beta) \sqrt{\log L}^{\frac{p}{p-1}}}\,.\]
\end{lemma}
\begin{proof}
    The proof remains unchanged from the $p = 2$ case given the new bounds on $\hatpi^{(p),0}_\infty(\phi_{o'} \geq h\mid \phi_o \geq h)$ and $\ell$.
\end{proof}
We also insert here the bound on $\bar\xi^{(p)}_{\ell, h}/\hatpi^{(p)}_\infty(\phi_o = -h)$.
\begin{lemma}[Cf.\ \cref{lem:xi-bar-lower-weak-upper}]
    \label{lem:xi-bar-lower-weak-upper-p}
In the setting of \cref{lem:xi-lower-weak-upper-p}, we have for $1 < p < 2$ that
\begin{equation}
    1 - L^{-2\delta + o(1)}\leq 
    \frac{\bar\xi^{(p)}_{\ell,h}}{\hatpi_\infty^{(p)}(\phi_o = -h)} 
    \leq 1 + L^{-1+o(1)}\,,\end{equation}
and for $p > 2$ that
\begin{equation}
 1 - e^{-c_0(\beta) \sqrt{\log L}^{\frac{p}{p-1}}} \leq \frac{\xi^{(p)}_{\ell,h}}{\hatpi^{(p)}_\infty(\phi_o < -h)} \leq 1 + L^{-1+o(1)}\,.
\end{equation}
\end{lemma}
\begin{proof}
    The proof remains unchanged from the $p = 2$ case, besides the modifications described for \cref{lem:xi-lower-weak-upper-p}.
\end{proof}
This leads to the main proposition controlling the probability of the floor event. (Note that, just as for $p = 2$, we do not need \cref{clm:pi(h|h)-lower-p,lem:xi-refined-upper-p,lem:xi-bar-lower-weak-upper-p} for this. In particular, we can take $\ell_0$ as small as $\log^2 L$, rather than the restriction $\ell \geq L^{1/4}$.)
\begin{proposition}[Cf.\ \cref{prop:key-area-estimate}]\label{prop:key-area-estimate-p}
    Fix $g\geq 0$, $n \geq 0$, and set $h = H+1-n$. Assume $\beta$ is sufficiently large. Let $\delta_p$ be the minimum of the constant from \cref{clm:pi(h|h)-lower-p} and $\frac16$. There exists $c>0$ such that the following holds for every $\log^2L \leq \ell_0 \leq L^{\frac12 - \frac{\delta_p}2}$, and subsets 
    $F \subseteq V\subset \Z^2$ where $F$ has $g$ holes and $|\partial F| \leq L^{3/2}$.
    Let $\Xi = \Xi_1+\Xi_2+\Xi_3$, where for $1 < p < 2$ we define 
    \begin{equation}\label{eq:Xi123-def-p} \Xi_1 =  \frac{|\partial V|}Le^{c\beta(\log L)^{\frac{2p-2}p}}\,,\qquad \Xi_2 = \frac{|F|}
     {\ell_0 L} e^{c\beta(\log L)^{\frac{p-1}{p}}} \,,\qquad \Xi_3 =\frac{|\partial F|\ell_0 }LL^{-\delta_p+o(1)}\,,\end{equation}
     and for $p > 2$ we define
     \begin{equation}\label{eq:Xi123-def-p-big} \Xi_1 =  \frac{|\partial V|}LL^{\epsilon_\beta}\,,\qquad \Xi_2 = \frac{|F|}
     {\ell_0 L} e^{c\sqrt{\beta\log L}}\,,\qquad \Xi_3 =\frac{|\partial F|\ell_0 }Le^{-c(\beta)\sqrt{\log L}^{\frac{p}{p-1}}}\,.\end{equation}
    \begin{enumerate}[(i)]
    \item If $1 < p < 2$ and $|F|\leq \ell_0 L e^{-c\beta(\log L)^{\frac{p-1}{p}}}$, or $p > 2$ and $|F|\leq \ell_0 L e^{-c\sqrt{\beta\log L}}$, then
    \begin{align}\label{eq:area-estimate-meso-p}\hatpi^{(p),0}_V(\phi_x \geq -h,\, \forall x \in F) = \exp\left(-\xi^{(p)}_{\ell_0,h}|F| +O(\Xi) + o(L^{-5}) \right)\,.
    \end{align}
\item Otherwise, if $\fS$ is the event that there are no disagreement polymers $\gamma$ with $|\gamma| \geq \log L$, then
\begin{align}\label{eq:area-estimate-macro-lower-p}
    \hatpi^{(p),0}_V(\phi_x \geq -h,\, \forall x \in F) &\geq \exp\left(-\xi^{(p)}_{\ell_0,h}|F| + O(\Xi)\right)\,,\\
    \label{eq:area-estimate-macro-upper-p}
 \hatpi^{(p),0}_V(\phi_x \geq -h,\, \forall x \in F\mid\fS) &\leq \exp\left(-\xi^{(p)}_{\ell_0,h}|F| + O(\Xi)+o(L^{-5})\right)\,.
\end{align}
If $|F| \leq  \big(\frac{3\beta}{\hatpi^{(p)}_\infty(\phi_o > h)}\big)^2 $, then the last upper bound also applies to $\hatpi^{(p),0}_V(\phi_x\geq -h,\,\forall x\in F)$.
\end{enumerate}
    The same holds under $\hatpi^{(p)}_\infty$, in which case the error term $\Xi_1$ can be omitted from $\Xi$. Separately, in the special case where $F = \llb1,\ell\rrb^2$ with $ \ell_0 + 10\lceil \log L\rceil \mid \ell$, the error term $\Xi_3$ can be omitted from $\Xi$.
\end{proposition}
\begin{proof}
    The proof essentially remains the same after plugging in the new large deviation estimates, so we just comment on the differences. The first error term $\Xi_1$ accounts for points near $\partial V$, which depends on the control of $\hatpi^{(p), 0}_V(\phi_x < -h)$ for $x$ near $\partial V$. For this we plug in the bound from \cite[Lem.~7.10]{ChenLubetzky25} (for $1 < p < 2$, the proof of this lemma actually gives the sharper bound $\hatpi_V^{(p)}(\phi_x \geq h) \leq \hatpi_\infty^{(p)}(\phi_x \geq h)e^{\epsilon_\beta C h^{2p-2}}$, whence the fact that $h = O((\log L)^p)$ gives the term $\frac{1}{L}e^{c\beta(\log L)^{\frac{2p-2}p}}$). 
    
    The second $\Xi_2$ accounts for points in the annuli of squares (recall the proof involves decomposing $V$ into squares), and only involves points far away from $\partial V$ where we can couple to $\hatpi^{(p)}_\infty$. The changed terms arise from the upper bound for $\hatpi^{(p)}_\infty(\phi_o < -h)$, which is obtained by the definition of $H$ and the bound on $\hatpi^{(p)}_\infty(\phi_o = h)/\hatpi^{(p)}_\infty(\phi_o = h-1)$ in \cite[Thm.~7.1]{ChenLubetzky25}. Finally, the third term $\Xi_3$ accounts for squares partially cut off by the boundary. This error depends on the comparison between $\hatpi^{(p)}_\infty(\phi_o < -h)$ and $\xi_{\ell_0, h}$, in particular, the lower bound of \cref{lem:xi-lower-weak-upper-p} applied at $\delta = \delta_p/2$.

    The bound on $|F|$ needed for \cref{eq:area-estimate-meso-p} is to satisfy $|F|\hatpi^{(p)}_\infty(\phi_o < -h) = o(\ell_0)$ (see \cref{eq:area-estimate-small-F}).
\end{proof}
\begin{proof}[Proof of \cref{prop:uprho-bound-p}, except the upper bound of \cref{eq:uprho-bounds-p}]
    Recall that $\ell_*:=2^{\lceil\frac12\log_2 L\rceil} \asymp \sqrt{L}$. As before, the proof is to relate $\xi^{(p)}_{\ell_*,h},\bar\xi^{(p)}_{\ell_*,h}$ to $\xi^{(p)}_{\ell_0,h},\bar\xi^{(p)}_{\ell_0,h}$ for a specific choice of $\ell_0$, and then use \cref{lem:xi-lower-weak-upper-p,lem:xi-refined-upper-p} to conclude for $\upxi^{(p)}$ and \cref{lem:xi-bar-lower-weak-upper-p} for $\uprho^{(p)}$. The proof for $p > 2$ remains unchanged. 
    
    For $1 < p < 2$ however, we must be more careful with our choice of $\ell_0$. (If one were to follow the proof of $p = 2$ as is, the factor of $1+L^{-1/4}$ would overwhelm the upper bound in \cref{lem:xi-refined-upper-p}.) We will take $\ell_0$ to be the largest integer less than $L^{\frac12 - \frac{\delta_p}{2}}$ such that $\ell_0 + 10\lceil\log L\rceil$ divides $\ell_*$. 
    
    In what follows, the constant $c$ may change from line to line. Applying \cref{prop:key-area-estimate-p} in the special case for $\hatpi^{(p)}_\infty$ and $F=\llb1,\ell_*\rrb^2$, we get from \cref{eq:area-estimate-meso-p} that
\[ \hatpi^{(p)}_\infty(\phi_x\geq -h,\,\forall x \in F) = \exp\bigg(-\xi^{(p)}_{\ell_0,h}\ell_*^2 + O\Big(\frac{\ell_*^2}{\ell_0L}e^{c\beta(\log L)^{\frac{p-1}{p}}}\Big)+
o(L^{-5})\bigg)\,.\]
(In the $p = 2$ proof, we chose not to use the divisibility condition as it was unnecessary. It is needed here.) Plugging in $\ell_*$ and $\ell_0$, we get
\begin{equation}\label{eq:xi*-bound-p}\xi^{(p)}_{\ell_*,h} = -\frac1{\ell_*^2}\log\hatpi^{(p)}_\infty(\phi_x\geq -h,\,\forall x\in F) = \xi^{(p)}_{\ell_0,h} + O\Big(\frac{1}{\ell_0L}e^{c\beta(\log L)^{\frac{p-1}{p}}}\Big)\,.\end{equation}
Since $\xi^{(p)}_{\ell_0,h} = (1+o(1))\hatpi^{(p)}_\infty(\phi_o < -h)$ from \cref{lem:xi-lower-weak-upper-p}, we get
\[ \frac{\upxi_n}{\hatpi^{(p)}_\infty(\phi_o<-h)} = \frac{\xi^{(p)}_{\ell_*,h}}{\hatpi^{(p)}_\infty(\phi_o<-h)} = (1+L^{-\frac12+\frac{\delta_p+o(1)}{2}})\frac{\xi^{(p)}_{\ell_0,h}}{\hatpi^{(p)}_\infty(\phi_o<-h)}\,;\]
whence \cref{eq:upxi-bounds-p} follows from the bounds on $\xi^{(p)}_{\ell_0,h}$ in \cref{lem:xi-lower-weak-upper-p,lem:xi-refined-upper-p}. 

Similarly, \cref{eq:xi*-bound-p} implies that $\bar\xi^{(p)}_{\ell_*,h}$ also satisfies
$\bar\xi^{(p)}_{\ell_*,h} = \bar\xi^{(p)}_{\ell_0,h} + O(\frac{1}{\ell_0L}e^{c\beta(\log L)^{\frac{p-1}{p}}})$,
so that the bound $\bar\xi^{(p)}_{\ell_0,h} = (1+o(1))\hatpi^{(p)}_\infty(\phi_o = -h)$ from \cref{lem:xi-bar-lower-weak-upper-p} implies that
\begin{equation}\label{eq:rho_n_via_xibar_ell0-p} \uprho_n = \frac{\bar\xi^{(p)}_{\ell_*,h}}{\hatpi^{(p)}_\infty(\phi_o=-h)} = (1+L^{-\frac12+\frac{\delta_p+o(1)}{2}})\frac{\bar\xi^{(p)}_{\ell_0,h}}{\hatpi^{(p)}_\infty(\phi_o=-h)}\,;\end{equation}
hence, \cref{lem:xi-bar-lower-weak-upper-p} gives the following weak version of \cref{eq:uprho-bounds-p}: 
\begin{equation*}
    \uprho^{(p)}_n \in [1 - e^{-\sqrt{\log L}},  1 + L^{-1+o(1)}]\,.\qedhere
\end{equation*}
\end{proof}
We emphasize that the above calculations show the $1 < p < 2$ case is more delicate than the $p \geq 2$ case. Here the upper bound estimates on $\upxi^{(p)}_n$ (and eventually $\uprho^{(p)}_n$) feature an error term of $1+L^{-\frac12+\frac{\delta_p+o(1)}{2}}$ multiplying a bound of $1 - L^{-1/2+\delta_p+o(1)}$. The only reason the final upper bound stays strictly $< 1$ is because $L^{\delta_p/2} \ll L^{\delta_p}$. Moreover, recall we already had to sacrifice with a suboptimal lower bound in \cref{lem:xi-lower-weak-upper-p} to even have an error term this small to begin with (as otherwise the choice of $\ell_0$ in the above proof would not be allowed). In comparison, for $p \geq 2$ we have an error term of $1 + L^{-1/4}$ multiplying a bound of $1 - L^{-o(1)}$, so the choice of $\ell_0 = L^{1/4}$ there is fairly arbitrary.

\begin{proof}[Proof of \cref{thm:key-area-p}]
The heavy lifting is done in obtaining \cref{prop:key-area-estimate-p}; proving the theorem amounts to verifying that particular choices of $\ell_0$ yield the desired statements. As in \cref{prop:key-area-estimate-p}, we choose $\delta_p$ to be the minimum of the constant from \cref{clm:pi(h|h)-lower-p} and $\frac16$.  

Begin with $1 < p < 2$ and \cref{eq:xi-n-area-estimate-meso-p}. Let $\ell_0 \approx L^{\frac13+\frac{\delta_p}2}$ be the largest integer less than $L^{\frac13+\frac{\delta_p}2}$ such that $\ell_0 + 10\lceil \log L \rceil$ divides $\ell_*$. (The condition $\delta_p \leq \frac16$ now ensures that $L^{\frac13+\frac{\delta_p}2} \leq L^{\frac12-\frac{\delta_p}2}$, so we stay in the regime of \cref{prop:key-area-estimate-p}.) The assumption that $|\partial F|\vee|\partial V| \leq O(L^{\frac23+\frac{\delta_p}5})$ ensures that $|F|\leq \ell_0 L e^{-c\beta(\log L)^{\frac{p-1}{p}}}$. Hence, we may appeal to \cref{eq:area-estimate-meso-p} of \cref{prop:key-area-estimate-p} and check that $\ell_0$ was carefully chosen so that $\Xi = o(1)$. The cost of moving from $\xi^{(p)}_{\ell_0, h}$ to $\xi^{(p)}_{\ell_*, h}$ with a $|F|$ prefactor in the exponent of \cref{eq:xi-n-area-estimate-meso-p} is also $\frac{|F|}{\ell_0L}e^{c\beta(\log L)^{\frac{p-1}p}} = o(1)$.

For \cref{eq:xi-n-area-estimate-macro-lower-p,eq:xi-n-area-estimate-macro-upper-p} (still for $1 < p < 2$), take $\ell_0 \approx L^{1/2}e^{-\sqrt{\log L}}$ to be the largest integer less than $\ell_0 \approx L^{1/2}e^{-\sqrt{\log L}}$ such that $\ell_0 + 10\lceil \log L \rceil$ divides $\ell_*$. We apply \cref{prop:key-area-estimate-p}, and compute $\Xi \leq \sqrt{L}e^{3\sqrt{\log L}}$. Finally, the cost of moving from $\xi^{(p)}_{\ell_0, h}$ to $\xi^{(p)}_{\ell_*, h}$ with a $|F|$ prefactor in the exponent is $\frac{|F|}{\ell_0L}e^{c\beta(\log L)^{\frac{p-1}p}} \leq \sqrt{L}e^{3\sqrt{\log L}}$.

The $p > 2$ case is essentially the same. We can take $\ell_0 \approx L^{1/3}e^{3c\sqrt{\beta\log L}}$ in the mesoscopic case and $\ell_0 \approx L^{1/2}e^{-\sqrt{\log L}}$ in the macroscopic case, where $\approx$ has the same meaning as above and the constant $c$ is from \cref{prop:key-area-estimate-p}. Then as above, one can verify that the assumptions of \cref{prop:key-area-estimate-p} are met, that $\Xi$ satisfies the desired bounds (either $o(1)$ or $\sqrt{L}e^{(5+o(1))c\sqrt{\beta\log L}}$), and that the cost of moving from $\xi^{(p)}_{\ell_0, h}$ to $\xi^{(p)}_{\ell_*, h}$ is the same order as $\Xi$.
\end{proof}

It remains to prove the upper bound of \cref{eq:uprho-bounds-p}. Recall this only applies for $1 < p < 2$. We will need the following conditional large deviation estimate, the proof of which we relegate to \cref{sec:missing-LD-p} as the proof is very different from the $p = 2$ case. One should think of the following as a replacement for \cref{prop:pi-y-h-phi^*-rigid} that is weaker, but sufficient for our purposes.
\begin{claim}\label{clm:pi(h|hh)-upper-p}
    Let $1 < p < 2$. Fix a neighbor of the origin $o'$. There exists $\delta'_p > 0$ and $h_0$ such that for all $h \geq h_0$ and any $y \notin \{o, o'\}$,
    \[\hatpi^{(p)}_\infty(\phi_y = h \mid \phi_o = \phi_{o'} = h) \leq e^{-\beta\delta'_p h^p}\,.\]
\end{claim}
Then, define $\bP^{(p)}_{\ell,h} := \hatpi^{(p)}_\infty(\cdot \mid \sF_{\ell,h})$, where  $\sF_{\ell,h}:= \left\{\phi_x \geq -h,\,\forall x\in Q_{\ell}\right\}$. This is motivated by the fact that we can then write $\bar \xi_{\ell,h} = -\frac{1}{\ell^2}\log\bP_{\ell,h}(\phi_x > -h,\,\forall x\in Q_\ell)$.
\begin{corollary}[Cf.\ \cref{cor:pi(Fc|_h_h)-lower}]
\label{cor:pi(Fc|_h_h)-lower-p}
Fix $1 < p < 2$. Let $h = H+1-n$ for $n$ fixed, and $\ell \leq \sqrt{L} e^{-\sqrt{\log L}}$. If $o'$ is a neighbor of the origin $o$, then for the constant $\delta_p$ from \cref{clm:pi(h|h)-lower-p}, we have
\[ \bP_{\ell,h}\left(\phi_{o'} = -h\mid \phi_o = -h\right) \geq 1 -L^{-1/2+\delta_p+o(1)}\,.\]
\end{corollary}
\begin{proof}
    The proof remains unchanged from the $p = 2$ case, using \cref{clm:pi(h|h)-lower-p} and \cref{clm:pi(h|hh)-upper-p} as inputs instead of \cref{clm:pi(h|h)-lower,prop:pi-y-h-phi^*-rigid}.
\end{proof}

\begin{lemma}[Cf.\ \cref{lem:xi-bar-upper}]\label{lem:xi-bar-upper-p}
Fix $1 < p < 2$, and let $\delta_p$ be from \cref{clm:pi(h|h)-lower}. For $h=H+1-n$ with fixed $n\geq 0$, $\beta$ sufficiently large, and $\ell \geq L^{\frac14}$, we have
\begin{equation*} \frac{\bar\xi_{\ell,h}}{\hatpi_\infty(\phi_o = -h)} \leq 1 -L^{-1/2+\delta_p+o(1)}\,.\end{equation*}
\end{lemma}
\begin{proof}
    The proof remains unchanged from the $p = 2$ case.
\end{proof}

\begin{proof}[Proof of \cref{prop:uprho-bound-p}, upper bound of \cref{eq:uprho-bounds-p}]
    Plug in \cref{lem:xi-bar-upper-p} into \cref{eq:rho_n_via_xibar_ell0-p}.
\end{proof}

\subsection{Missing large deviation estimates for $p > 1$}\label{sec:missing-LD-p}
Finally, we conclude with the proofs of \cref{clm:pi(h|h)-lower-p} and \cref{clm:pi(h|hh)-upper-p}. We will start with \cref{clm:pi(h|h)-lower-p} for $p > 2$, which seeks to prove a lower bound on $\hatpi^{(p)}_\infty(\phi_{o'} \geq h \mid \phi_o\geq h)$. We recall from \cite[Thm.~7.1, Eq.~(7.4)]{ChenLubetzky25} that, for some absolute constant $c_0>0$ and every $h\geq 1$, 
\begin{equation}\label{eq:CL25-pi(h)-ratio-LB}\frac{\hatpi(\phi_o = h )}{\hatpi(\phi_o = h-1)} > \exp(- c_0 \beta h )\,,\end{equation}
and from \cite[Thm.~7.1, Eq.~(7.4)]{ChenLubetzky25}, for another constant $c_1>0$ and every $h\geq 1$,
\begin{equation}\label{eq:CL25-pi(h|h)-UB}\hatpi^{(p)}_\infty(\phi_y = h \mid \phi_x = h) \leq  \exp(- c_1 \beta h^{p/(p-1)})\,.\end{equation}
\begin{proof}[Proof of \cref{clm:pi(h|h)-lower-p}, $p > 2$]
Let $M := (c_0/2)^{1/(p-1)}$ for the constant $c_0$ from \cref{eq:CL25-pi(h)-ratio-LB}, and define
\[ q_M:= \hatpi^{(p)}_\infty \big( \phi_{o'} \geq h - M h^{1/(p-1)} \mid \phi_o \geq h\big)\,.\]
We will argue that
\begin{equation}
    \label{eq:qm-lower-bound}
    q_M > e^{-2 c_0\beta h}\,.
\end{equation}
To see this, assume the opposite, and let $\cS = \left\{ \phi \,:\; \max_{z\sim o} \phi_z \leq h - M^{1/(p-1)}\right\}$. By a union bound,
we have $\hatpi_\infty^{(p)}\left(\cS^c \mid \phi_o \geq h\right) \leq 4 q_M$. Now consider the bijection that decreases $\phi_o$ by $1$. For every $\phi \in \cS \cap \{\phi_o\geq h+1\}$, 
\[ \sum_{z\sim o} (|\phi_o-\phi_z|^p-|\phi_o-1-\phi_z|^p)\geq \sum_{z\sim o} p |\phi_o-1-\phi_z|^{p-1} \geq 4 p M^{p-1} h = 2c_0 h\,,\] using the definition of $M$ in the last step. Thus, $\hatpi_\infty^{(p)}(\cS,\, \phi_o\geq h+1 \mid \phi_o\geq h) \leq e^{-2c_0\beta h}$, and we can conclude that
\begin{align*}
     \hatpi_\infty^{(p)}(\phi_o\geq h+1 \mid \phi_o \geq h) &\leq \hatpi_\infty^{(p)}(\cS,\, \phi_o\geq h+1 \mid \phi_o \geq h) 
     + \hatpi_\infty^{(p)}(\cS^c\mid \phi_o \geq h) \\ &\leq e^{-2c_0\beta h}+ 4 q_M \\ &\leq 6 e^{-2c_0\beta h} \,.\end{align*}
For large $h$, this would contradict the fact that $\hatpi_\infty^{(p)}(\phi_o\geq h+1)/\hatpi_\infty^{(p)}(\phi_o \geq h) \geq (1-\epsilon_\beta)e^{-c_0\beta h}$ via \cref{eq:CL25-pi(h)-ratio-LB} (as well as that $\hatpi_\infty(\phi_o=h) \geq (1-\epsilon_\beta)\hatpi_\infty(\phi_o\geq h)$ by a routine Peierls map). This establishes \cref{eq:qm-lower-bound}. 

To derive the sought bound from \cref{eq:qm-lower-bound}, we write 
\begin{align*}
    \hatpi_\infty^{(p)}(\phi_{o'} \geq h \mid \phi_o \geq h) &= q_M  \prod_{k=h - M h^{1/(p-1)}}^{h-1} \hatpi_\infty^{(p)}(\phi_{o'} \geq k+1 \mid \phi_o \geq h,\,\phi_{o'} \geq k) \\ 
    &\geq q_M \exp\left(-c_0 \beta h \cdot M h^{1/(p-1)} \right) \\
    & \geq \exp\left(- (c_0 M \beta - o(1)) h^{p/(p-1)} \right) \,,
\end{align*}
where the inequality in the second line used FKG to drop the conditioning on $\phi_o\geq h$, followed by an application of \cref{eq:CL25-pi(h)-ratio-LB} to each of the terms in the product (where for each of these $M h^{1/(p-1)}$ terms we further increase $k$ to $h$). Note that, of the two terms in the second line, $q_M \geq e^{-2c_0\beta h}$ is negligible compared to the term $\exp(-c \beta h^{p/(p-1)})$.
\end{proof}

Next, to set up for the proofs of \cref{clm:pi(h|h)-lower-p} for $1 < p < 2$ and \cref{clm:pi(h|hh)-upper-p}, we begin by recalling some preliminaries about $p$-harmonic functions from nonlinear potential analysis. Define the $p$-Laplacian as $(\Delta_p \phi)_x = \frac14 \sum_{y\sim x} |(\nabla \phi)_{xy}|^{p-1}(\nabla \phi)_{xy}$ for $(\nabla \phi)_{xy} = \phi_y-\phi_x$. Define the energy $\sE^{(p)}(f):= \sum_e |\nabla f|^p$, and the $p$-capacity for $A \subset \Z^2$ by
\[ \operatorname{Cap}_p(A):=\inf\{\sE^{(p)}(f):\; f\to 0 \text{ at }\infty,\, f\geq 1 \text{ on }A\}\,.\]  
Now fix a finite set $A$. For $1 < p < 2$, the graph $\Z^2$ is $p$-hyperbolic (see, e.g. \cite[pp.~176--178]{soardi2006potential}), meaning $\operatorname{Cap}_p(\{A\}) > 0$. As a consequence, there is a unique $p$-harmonic function with boundary conditions 1 on $A$ and 0 at infinity, call this $\phi^*_A$. Moreover, we have $\operatorname{Cap}_p(A) = \sE^{(p)}(\phi^*_A)$, and for any $\epsilon > 0$, there exists a radius $r$ and an approximation $\varphi^*_A$ of $\phi^*_A$ such that $\varphi^*_A$ is supported on $U = \bigcup_{x \in A}B_r(x)$ and $\sE(\varphi^*_A) \leq \sE(\phi^*_A) + \frac{\epsilon}{2}$. 

By looking at the maximum of the $p$-harmonic functions, an elementary fact is that
\begin{equation}\label{eq:cap-subadditivity}\operatorname{Cap}_p(A \cup B) \leq \operatorname{Cap}_p(A) + \operatorname{Cap}_p(B)\,.
\end{equation}
Moreover, the $p$-capacity for $1 < p < 2$ is asymptotically additive, i.e.,  \begin{equation}\label{eq:cap-asymp-additivity}\operatorname{Cap}_p(A\cup (B+x))=\operatorname{Cap}_p(A)+\operatorname{Cap}_p(B)+o(1) \quad\mbox{ as }\quad |x|\to\infty
\end{equation}
(this follows from quasi-additivity of the $p$-capacity combined with decay estimates for $p$-Green functions; see, e.g., \cite{HKM93,Mazya11}
for the general framework, and \cite{HolopainenSoardi1997} for the graph setting). 

It is known (\cite[Thm.~5.1, Eq.~(5.1)]{LMS16}) that for every $1<p<2$ there exists some $c_p^*>0$ such that
\begin{equation}\label{eq:pi(h)-for-1<p<2}
\hatpi^{(p)}_\infty(\phi_o \geq h) = e^{-(c_p^* \beta +o(1))h^p}\,,\end{equation}
where the $o(1)$-term goes to $0$ as $h\to\infty$. In fact, the proof shows that $c_p^*$ is precisely $\sE^{(p)}(\phi^*_o)$. The next lemma extends this one-point large deviation result to a general finite set.

\begin{lemma}\label{lem:LD-capacity}
    Let $1 < p < 2$. For any constants $k, \epsilon > 0$, there exists constants $h_0$ such that for all $h \geq h_0$, the following holds. Let $A \subset \Z^2$ have $k$ vertices. Then, 
    \[\exp(-\beta(\sE^{(p)}(\phi^*_A)+\epsilon)h^p) \leq \hatpi^{(p)}_\infty(\phi\restriction_A = h) \leq \exp(-\beta(\sE^{(p)}(\phi^*_A)-\epsilon)h^p)\,.\]
    The same statement holds for $\hatpi^{(p)}_\infty(\phi\restriction_A \geq h)$.
\end{lemma}
\begin{proof}
    As usual, an easy Peierls argument shows that $\hatpi^{(p)}_\infty(\phi\restriction_A = h \mid \phi\restriction_A \geq h) \geq 1-\epsilon_\beta$, so it suffices to prove the statement for $\hatpi^{(p)}_\infty(\phi\restriction_A = h)$. Call the $k$ points of $A$ by $\{x_i\}_{i = 1}^k$. Let $R = Ch^{p-1}$ for a constant $C$ to be chosen later. Let $V = \bigcup_{i}B_R(x_i)$. Let $\cE$ be the event that there is no connected component of sites with height $\geq 1$ with diameter (in $\Z^2$) larger than $R/2$ that intersects $V$. First we prove that
    \begin{equation}\label{eq:UB-by-capacity}
        \hatpi^{(p)}_\infty(\phi\restriction_A = h, \cE) \leq \exp(-\beta(\sE^{(p)}(\phi^*_A)-o(1))h^p)\,.
    \end{equation}
    For every $v \in \partial V$, reveal its connected component of sites with height $\geq 1$. On the event $\cE$, none of these component reach $\{x_i\}$. Thus, this reveals circuits of $\leq 0$ sites around the $\{x_i\}$, where different circuits have disjoint interiors. Group together the points sharing the same circuit into $j \leq k$ groups $A_j$. Note that no information about heights inside the circuit has been revealed. Hence after the revealing, we can forget the event $\cE$, increase the event $\phi\restriction_A = h$ to $\phi\restriction_A \geq h$, use monotonicity to raise the $\leq 0$ sites to height 0, and use Domain Markov to obtain
    \[\hatpi^{(p)}_\infty(\phi\restriction_A \geq h, \cE) \leq \max_{\{V_j\}}\prod_j \hatpi^{(p)}_{V_j}(\phi\restriction_{A_j} \geq h)\,,\]
    where the maximum is over all collections of disjoint subsets $V_j\subset V$ such that each $A_j \subset V_j$. To upper bound the terms $\hatpi^{(p)}_{V_j}(\phi\restriction_{A_j} \geq h)$, observe that each configuration attaining $\phi\restriction_{A_j} \geq h$ has probability at most $e^{-\beta\sE^{(p)}(\phi^*_{A_j})h^p}$. Hence we can write
    \[\hatpi^{(p)}_{V_j}(\phi\restriction_{A_j} \geq h) \leq \hatpi^{(p)}_{V_j}(\max_{x\in A_j}|\phi_x| \geq h^2) + e^{-\beta\sE^{(p)}(\phi^*_{A_j})h^p}(2h^2+1)^{|V_j|} \leq e^{\beta(\sE^{(p)}(\phi^*_{A_j})-o(1))h^p}\,.\]
    The bound \cref{eq:UB-by-capacity} now follows by \cref{eq:cap-subadditivity}.

    Now let $\Gamma_i$ be the outermost 1 level line containing $x_i$. To conclude the desired upper bound of the lemma, it suffices to show that for some $C' > 0$, we have
    \begin{equation}
        \hatpi^{(p)}_\infty(\cE^c\mid \phi\restriction_A = h) \leq \hatpi^{(p)}_\infty(\cE^c \mid \max_i |\Gamma_i| \leq R,\, \phi\restriction_A = h) + \hatpi^{(p)}_\infty(\max_i |\Gamma_i| > R \mid \phi\restriction_A = h) \leq e^{-\beta C'h^{p-1}}\,. 
    \end{equation}
    The first inequality is trivial. An upper bound on $\hatpi^{(p)}_\infty(\max_i |\Gamma_i| > R \mid \phi\restriction_A = h)$ follows via essentially the same proof of \cite[Lem.~5.3]{LMS16}. Firstly, for any neighbor $y$ of $A$, we can bound $\hatpi^{(p)}_\infty(\phi_y \leq -h \mid \phi\restriction_A = h) \leq e^{-c\beta h^p}$ by a Peierls map on the down-loops surrounding $x$. We can afford this as an additive error. Otherwise, if all the neighbors of $A$ have height $\geq -h$, then consider the Peierls map $T$ which lowers the height of all vertices in $\bigcup_i \mathsf{Int}(\Gamma_i)$ by 1, and then raises the height of $\phi\restriction_A$ back up to $h$. The condition on the neighbors of $A$ implies that we have an energy difference of \[\hatpi^{(p)}_\infty(T(\phi)) \geq \hatpi^{(p)}_\infty(\phi)e^{\beta(|\bigcup_i|\Gamma_i| - |\partial A|\beta p(2h)^{p-1})}\,.\]
    Hence, the standard Peierls argument enumerating over $\Gamma_i$ shows that for a sufficiently large constant $C >0$ depending on $|\partial A|$, we have $\hatpi^{(p)}_\infty(\max_i|\Gamma_i| \geq Ch^{p-1} \mid \phi\restriction_A = h) \leq e^{C'\beta h^{p-1}}$ for some other constant $C'$. Finally, an upper bound of the same form on $\hatpi^{(p)}_\infty(\cE^c \mid \max_i |\Gamma_i| \leq R,\, \phi\restriction_A = h)$ follows by observing that the large $\geq 1$ component realizing $\cE^c$ must be in the exterior of all the $\Gamma_i$, so revealing the $\Gamma_i$ reveals $\leq 0$ boundary conditions and $\cE^c$ can be ruled out by a Peierls argument.

    We turn now to showing the lower bound of the lemma. For the choice of $\epsilon$ in the lemma, there exists a radius $r$ and an approximation $\varphi^*_A$ of $\phi^*_A$ such that $\varphi^*_A$ is supported on $U = \bigcup_{x \in A} B_r(x)$ and $\sE(\varphi^*_A) \leq \sE(\phi^*_A) + \frac\epsilon2$. The maximum value attained by $h\varphi^*_A$ is $h$, so we have an integer rounding cost of $\sE(\lfloor h\varphi^*_A\rfloor) \leq \sE(h\varphi^*_A\rfloor) + h^{p-1}r^2$. At the same time, forcing all heights in $U$ to be $0$ under $\hatpi^{(p)}_\infty$ has a probability of at least $(1-\epsilon_\beta)^{|U|}$ by absolute value FKG (valid for $1\leq p\leq 2$). The weight of a specific configuration $\phi$ in $U$, relative to the all-$0$ configuration, is simply $\exp(-\beta E(\phi))$. Hence, in total we have that $\hatpi^{(p)}_\infty(\phi\restriction_A = h) \geq \exp(-\beta(\sE(\phi^*_A)+\frac\epsilon2)h^p -\beta h^{p-1}|U| -\epsilon_\beta |U|)$, where the extra error terms can be absorbed into another factor of $\beta\frac\epsilon2 h^p$ for sufficiently large $h$ since $|U| \leq kr^2$.
\end{proof}

\begin{proof}[Proof of \cref{clm:pi(h|h)-lower-p}, $1 < p < 2$]
Consider the $p$-harmonic function $\phi^*_o$ giving the one point large deviation.
The energy $\sE^{(p)}(\phi^*_o)$ is the sum of the energy over horizontal bonds $\sE^{(p)}_{\sf h}(\phi^*_o)$ and vertical bonds $\sE^{(p)}_{\sf v}(\phi^*_o)$.
W.l.o.g, assume that $\sE^{(p)}_{\sf v}(\phi^*_o) \leq \sE^{(p)}_{\sf h}(\phi^*_o)$. Define $\varphi$ via block-copying $\phi^*_o$ on $2\Z \times \Z$, that is, \[ \varphi_{2k,l} =\varphi_{2k+1,l}=(\phi^*_o)_{k,l}\qquad \mbox{for all $k,l\in\Z$}\,.\]
Then
\[\sE^{(p)}(\varphi) = \sE^{(p)}_{\sf h}(\phi^*_o) + 2\sE^{(p)}_{\sf v}(\phi^*_o) \leq \tfrac32 \sE^{(p)}(\phi^*_o)\,.\]

Since $\varphi$ takes values 1 on $o, o'$ and 0 at infinity, then $\sE^{(p)}(\phi^*_{o, o'}) \leq \sE^{(p)}(\varphi)$. We next verify that $\varphi \neq \phi^*_{o, o'}$, so that there exists some $\delta_p>0$ such that
\begin{equation}\label{eq:doubling-energy-diff}\sE^{(p)}(\phi^*_{o, o'}) \leq \sE^{(p)}(\varphi) - \delta_p \leq \tfrac32 \sE^{(p)}(\phi^*_o) - \delta_p\,.
\end{equation}
We need to verify that the function $\varphi$ is not $p$-harmonic. One elementary way to show this is to take $(k,l)\neq o$, let $x^-=(k-1,l)$, $x=(k,l)$, $x^+=(k+1,l)$, and observe that if $(\Delta_p\varphi)_{2k,l}=0$ then 
$|(\nabla\phi^*_o)_{xx^-}|^{p-2}(\nabla\phi^*_o)_{xx^-}=-\xi$, where $\xi$ is the contribution of the vertical bonds to $(\Delta_p \phi^*_o)_x$, and similarly, if
$(\Delta_p\varphi)_{2k+1,l}=0$ then
$|(\nabla\phi^*_o)_{xx^+}|^{p-2}(\nabla\phi^*_o)_{xx^+}=-\xi$. 
Thus, $(\Delta_p\phi^*_o)_x = -\xi$, and since $\phi^*_o$ is $p$-harmonic, we get $(\nabla\phi^*_o)_{x^-x} = (\nabla\phi^*_o)_{x x^+} = 0$. Iterating this argument (as well as in the vertical axes), while recalling that $\phi^*_o$ is $0$ at infinity, we find $\phi^*_o = \delta_{o}$, a contradiction.

The proof now concludes by combining \cref{eq:doubling-energy-diff,lem:LD-capacity}. 
\end{proof}

\begin{proof}[Proof of \cref{clm:pi(h|hh)-upper-p}]
    The statement is immediately true by \cref{lem:LD-capacity} if $\delta_p$ can depend on $y$, in particular with $\delta_p(y) := \sE^{(p)}(\phi^*_{o, o', y}) - \sE^{(p)}(\phi^*_{o, o'}) - o(1)$. So it remains to show that $\sE^{(p)}(\phi^*_{o, o', y}) - \sE^{(p)}(\phi^*_{o, o'})$ is uniformly bounded away from 0, whence we can take $h$ large enough so that the $o(1)$ is negligible. This in turn follows from \cref{eq:cap-asymp-additivity} combined with the observation that for any $y \notin \{o, o'\}$, we always have $\sE^{(p)}(\phi^*_{o, o', y}) > \sE^{(p)}(\phi^*_{o, o'})$.
\end{proof}

\subsection{Extension to \texorpdfstring{$p = 1$}{p=1}}\label{sec:SOS-extension}
Finally, we discuss the \SOS model. Consider the summary at the beginning of this section. The crucial input needed for the proofs of the main theorems is an analog of \cref{thm:key-area}. But for $p = 1$ this was already proven, with the mesoscopic range in \cite[Prop.~A.1]{CLMST16} and the macroscopic range in \cite[Prop.~2.12]{CLMST16}. (For $p = 1$, there is no need to define $\upxi$ or $\uprho$; one may work directly in terms of $\hatpi^{(1)}_\infty(\phi_o < -h)$.) Hence, the results of \cref{thm:main-thm-limit-shape,thm:main-thm-crit-window} also follow. However, \cref{thm:FS-no-exceptional} does not follow because the quantities $\hatpi^{(p)}_\infty(\phi_o = h)$ and $\hatpi^{(p)}_\infty(\phi_o=h-1)$ are now comparable up to a factor of $C(\beta)$, instead of being on different scales, the latter of which is a crucial part of the proof of \cite{ChenLubetzky25} which \cref{sec:limit-law} is modifying. (In fact the limit law for $p = 1$ is not expected to be Ferrari--Spohn, see, e.g., \cite[\S1]{ChenLubetzky25} for more details). 

\subsection*{Acknowledgments}
This research was supported by the NSF grant \textsc{dms}-2451083.

	\bibliographystyle{abbrv}
	\bibliography{dg_shape_critical_ref}

\end{document}